\documentclass[12pt,a4paper,reqno,oneside]{amsart}
\usepackage[utf8]{inputenc}
\usepackage[english]{babel}

\usepackage{mathtools}
\usepackage{amssymb}
\usepackage{amsmath}
\usepackage[normalem]{ulem}
\usepackage{float}
\usepackage[shortlabels]{enumitem}
\usepackage{tikz}
\usepackage{mathrsfs}
\usepackage{booktabs,tabularx,array}
\usepackage{pgfplots}
\pgfplotsset{compat=1.18} 
\usepackage{hyperref}
\usepackage{cleveref}
\usepackage{autonum}

\makeatletter
\AddToHook{begindocument/end}{
\autonum@generatePatchedReferenceCSL{Cref}
}
\makeatother

\usetikzlibrary{arrows.meta,backgrounds}

\usepackage[left=2cm,right=2cm,top=2.25cm,bottom=2.25cm]{geometry}

\newcommand{\dd}{\, \mathrm{d}}
\newcommand{\R}{\mathbb R}
\newcommand{\g}{{\bf g}}
\newcommand{\s}{\mathfrak{s}}
\newcommand{\fraksc}{\s_{\mathrm{c}}}

\newsavebox{\surfaceglyphbox}
\newfont{\SurfaceMeasureFont}{ec-lmbx10 at 9pt}
\newdimen\surfacewidth
\newdimen\surfaceoffset
\newcommand{\SurfaceText}[3]{
  \global\surfacewidth=0pt
  \foreach \glyph in {#3}{
    \sbox{\surfaceglyphbox}{\SurfaceMeasureFont\boldmath\glyph}
    \global\advance\surfacewidth by\wd\surfaceglyphbox
  }
  \pgfmathsetmacro{\startglyph}{4.5-.5*#2*(\the\surfacewidth)/(1cm)}
  \global\surfaceoffset=0pt
  \foreach \glyph in {#3}{
    \sbox{\surfaceglyphbox}{\SurfaceMeasureFont\boldmath\glyph}
    \pgfmathsetmacro{\gx}{\startglyph+#2*(\the\surfaceoffset)/(1cm)}
    \pgfmathsetmacro{\gd}{.625+.125*(\gx-3)}
    \pgfmathsetmacro{\gz}{(5-#1)/\gd}
    \pgfmathsetmacro{\gb}{.08-.6*.125*(5-#1)/(\gd*\gd)}
    \pgfmathsetmacro{\ge}{.55+.6/\gd}
    \begin{scope}[shift={(\gx,#1,\gz)},cm={1,\gb,-.75,\ge,(0,0)},scale=#2]
      \node[anchor=base west,transform shape,inner sep=0pt,
        text=nonlocal!55!black,font=\footnotesize\bfseries\boldmath]
        at (0,0) {\glyph};
    \end{scope}
    \global\advance\surfaceoffset by\wd\surfaceglyphbox
  }
}

\newsavebox{\surfaceformulabox}
\newcommand{\SurfaceFormula}[3]{
  \sbox{\surfaceformulabox}{\fontsize{9}{11}\selectfont
    \SurfaceMeasureFont\boldmath\color{nonlocal!55!black}$\displaystyle #3$}
  \pgfmathsetmacro{\fw}{(\the\wd\surfaceformulabox)/(1cm)}
  \pgfmathsetmacro{\fs}{min(#2,2.5/\fw)}
  \pgfmathsetmacro{\fstart}{4.5-.5*\fs*\fw}
  \pgfmathsetmacro{\fhalf}{\fw/128+.0006}
  \pgfmathsetmacro{\ftop}{(\the\ht\surfaceformulabox)/(1cm)+.03}
  \pgfmathsetmacro{\fbottom}{-(\the\dp\surfaceformulabox)/(1cm)-.03}
  \foreach \strip in {0,...,63}{
    \pgfmathsetmacro{\fu}{(\strip+.5)*\fw/64}
    \pgfmathsetmacro{\fx}{\fstart+\fs*\fu}
    \pgfmathsetmacro{\fd}{.625+.125*(\fx-3)}
    \pgfmathsetmacro{\fz}{(5-#1)/\fd}
    \pgfmathsetmacro{\fb}{.08-.6*.125*(5-#1)/(\fd*\fd)}
    \pgfmathsetmacro{\fe}{.55+.6/\fd}
    \begin{scope}[shift={(\fx,#1,\fz)},x=1cm,y=1cm,
      cm={1,\fb,-.75,\fe,(0,0)},scale=\fs]
      \clip (-\fhalf,\fbottom) rectangle (\fhalf,\ftop);
      \node[anchor=base west,transform shape,inner sep=0pt,outer sep=0pt]
        at (-\fu,0) {\usebox{\surfaceformulabox}};
    \end{scope}
  }
}

\definecolor{teal}{HTML}{008080}
\definecolor{muted}{HTML}{888888}
\definecolor{rust}{HTML}{B7410E}
\definecolor{red}{RGB}{200,40,40}
\definecolor{blue}{RGB}{40,80,200}
\definecolor{redbar}{RGB}{160,15,35}
\definecolor{bluebar}{RGB}{20,40,165}
\definecolor{bramson}{HTML}{2968AF}
\definecolor{nonlocal}{HTML}{008C8C}
\definecolor{nopull}{HTML}{D69A22}
\definecolor{critical}{HTML}{8D3DAD}
\definecolor{forced}{HTML}{D54B43}

\newtheorem{theorem}{Theorem}[section]
\newtheorem{lemma}[theorem]{Lemma}
\newtheorem{claim}[theorem]{Claim}
\newtheorem{remark}[theorem]{Remark}
\newtheorem{proposition}[theorem]{Proposition}
\theoremstyle{definition}
\newtheorem{definition}[theorem]{Definition}

\numberwithin{equation}{section}

\title[Logarithmic 
corrections
in shifting environments and competition systems]{Logarithmic corrections and interacting fronts in shifting environments and competition systems}

 \author{\'Emeric Bouin}
 \address[E. Bouin]{CEREMADE - Université Paris-Dauphine, PSL Research University, UMR CNRS 7534, Place du Mar\'echal de Lattre de Tassigny, 75775 Paris Cedex 16, France.} 
 \address{DMA - \'Ecole normale supérieure, PSL Research University, UMR CNRS 8553, 45 rue d’Ulm, 75230 Paris Cedex 05, France.} \email{bouin@ceremade.dauphine.fr}

\author{King-Yeung Lam}
\address[K.-Y. Lam]{Department of Mathematics, The Ohio State University, Columbus, OH, 43210, USA.}
\email{lam.184@osu.edu}
\author{Chang-Hong Wu}
\address[C.-H. Wu]{Department of Applied Mathematics, National Yang Ming Chiao Tung University, Hsinchu, 300, Taiwan.}
\email{changhong@math.nctu.edu.tw}

\begin{document}
\begin{abstract}
We investigate refined large-time propagation and front interaction in a two-species Lotka--Volterra competition--diffusion system on the real line. When the faster competitor spreads into unoccupied habitat, its invasion front generates a moving heterogeneous environment for the trailing competitor, whose front position may be influenced by the leading front's Bramson delay.
To isolate and quantify this mechanism, we first study a scalar Fisher--KPP equation in a shifting environment.
Under weighted integrability assumptions on the growth profile, we determine the logarithmic front corrections up to a bounded error and identify the corresponding limiting wave profiles, thereby explicitly demonstrating how the logarithmic displacement of the environment modifies the classical Bramson correction. 
The resulting asymptotics distinguish locally pulled propagation, nonlocal pulling, and two critical transition regimes. 
In the forced-wave regime, the invasion front remains locked to the moving environment, and the dynamics select a unique stationary profile from a one-parameter family. In the nonlocal pulling regime, the favourable habitat ahead accelerates the invasion front even though the front lags behind the environmental shift. At the threshold of nonlocal pulling, the interaction modifies only the logarithmic correction and, at a further critical parameter value, produces an additional \(\ln\ln t\) term. 
We then apply these results to the competition system. In the nonlocally pulled regime, we determine the logarithmic displacement of the trailing invasion front, explicitly quantifying the transmission of the leading species' Bramson delay, and prove convergence to the corresponding travelling-wave profile. The proofs combine spatial correctors associated with Schr\"{o}dinger operators, sharp Gaussian barriers, and perturbative sub- and supersolution constructions for the coupled equations.
Finally, we establish a Liouville-type rigidity principle for a class of monotone parabolic systems, based on quantitative leading-edge control of a single component and mild persistence conditions. This yields uniqueness, up to translation, of monostable travelling waves in both competition regimes.
\end{abstract}

\maketitle
\tableofcontents

\section{Introduction and presentation of main results}

Propagation phenomena in reaction–diffusion equations provide a natural mathematical framework for describing the spatial spread of biological populations, chemical reactions, epidemics, and many other processes in which local growth interacts with spatial dispersal. The central questions are not only how fast an invasion spreads, 
but also the location of its interface to sublinear-in-time accuracy, whether solutions converge to a travelling-wave profile and, if so, how that profile is selected.

Two fundamental prototypes are the Fisher–KPP equation
\begin{equation}\label{e.kkpp}
n_t - d \, n_{xx}=n(r-n),\end{equation}with constant diffusion rate $d>0$ and constant or  heterogeneous growth rate, as well as the diffusive Lotka-Volterra competition system
\begin{equation}\label{e.com}
\begin{cases}
    u_t-u_{xx}=u(1-u-av),\\[5pt]
    v_t-d\, v_{xx}=rv(1-bu-v).
\end{cases}
\end{equation} 

For the scalar equation with constant $r>0$, sufficiently localised initial data generate a front with minimal propagation speed $c=2\sqrt{dr}$ \cite{fisher1937wave,kolmogorov1937study,aronson2006nonlinear}. In one dimension, a celebrated result of Bramson \cite{bramson1983convergence} established the sharper position estimate:
\begin{equation}
x(t) = ct - \frac32 \sqrt{\frac{d}{r}} \ln t + O(1).
\end{equation}

This logarithmic delay, now commonly referred to as the Bramson correction, has become one of the central features of pulled-front propagation. Such a front is selected by the
small-amplitude dynamics far ahead in its leading edge. In contrast, pushed fronts are selected by nonlinear dynamics in the transition region, and typically exhibit only a bounded phase correction. 
The distinction is more subtle in competition systems, where several linearised modes can interact; recent work has examined the pulled-pushed transition as competition parameters vary \cite{wxz2024}.

Heterogeneity introduces another difficulty by breaking translation invariance. We first study a scalar equation with a shifting growth profile $\g(x-s(t))$, where $y=x-s(t)$ is the environmental coordinate, and then apply the resulting estimates to competition--diffusion systems.

By varying the parameters governing the environmental shift $s(t)$, we distinguish five regimes according to the influence of the moving habitat on the invasion front:

\begin{itemize}
    \item[$\to$] \textbf{Classical Bramson regime}
    (\Cref{t.0,thm:standard}).
    For sufficiently slowly shifting environments, we identify conditions
    on the rate of convergence of the profile $\g$ to spatial homogeneity at $+\infty$ under
    which the classical Bramson correction remains unchanged.

    \item[$\to$] \textbf{Forced-wave regime}
    (\Cref{thm:selection}).
    At the critical shift,
    the front remains locked to the moving habitat interface, and we prove that a specific profile is selected from a family of forced waves that are not related by translations.

    \item[$\to$] \textbf{Nonlocal-pulling regime}
    (\Cref{thm:pulling}).
    As the environmental shift becomes faster, the invasion front falls
    behind the moving habitat. Nevertheless, the favourable region
    ahead can enhance propagation through nonlocal pulling, increasing
    the spreading speed and modifying the logarithmic correction.

    \item[$\to$] \textbf{Critical-pulling regime}
    (\Cref{thm:cpulling}).
    At the threshold between nonlocal pulling and no-pulling, the
    distant favourable region no longer increases the leading-order
    spreading speed, but its residual influence modifies the
    logarithmic correction and, at a further critical value, produces an additional
    $\ln\ln t$ term.
    
    \item[$\to$] \textbf{No-pulling regime}
    (\Cref{thm:no-pull}).
    When the favourable region moves sufficiently far ahead, its
    influence on the invasion front becomes negligible, and nonlocal
    pulling no longer occurs.

\end{itemize}

We next turn to the competition system \eqref{e.com} in the monostable regimes $0<a<1$ and $b \in (0,\infty)\setminus\{1\}$. Here $u$ can invade the $v$-only equilibrium. The stable state behind the $u$-front depends on $b$: the species $u$  excludes $v$ when $b>1$, whereas the two species coexist when $0<b<1$.

We are particularly interested in  initially localised populations:
\begin{equation}
   0\le u_0,v_0\le 1,\qquad
u_0,v_0\not\equiv0,\qquad
\operatorname{supp}u_0,\operatorname{supp}v_0
\text{ are compact}.
    \label{eq:compact-data}
\end{equation}

In the absence of competition, the two species spread at the scalar KPP
speeds
    $$c_u=2,
    \qquad
    c_v=2\sqrt{rd}.$$
We focus mainly on the regime $c_v>c_u$, in which the species $v$ first
invades the empty habitat. The species $u$  then propagates through the region
already occupied by $v$ (and thus travels at a speed strictly less than $c_u$) and either eventually replaces it, when $b>1$, or
drives the solution towards the stable coexistence state, when $b<1$. The
solution therefore develops a propagating terrace with two distinct
interfaces: a fast colonisation front connecting $(0,1)$ to $(0,0)$ and a
slower secondary front connecting the stable equilibrium $(u^\star,v^\star)$ to the $v$-only equilibrium $(0,1)$, where 
\begin{equation}\label{e.uv.uv}
(u^\star,v^\star)
=
(1,0)\quad\text{if }b\ge 1,
\qquad\text{or}\qquad
(u^\star,v^\star)
=
\left(
\frac{1-a}{1-ab},
\frac{1-b}{1-ab}
\right)
\quad\text{if }b<1.
\end{equation}

Far ahead of the secondary front, where $u$ is exponentially small, the $v$-component is
therefore expected to behave to leading order like a Fisher--KPP front
propagating with speed $c_v$. In turn, the equation for $u$ may be viewed
heuristically as a Fisher--KPP equation in a shifting environment, with the
effective growth coefficient $1-av$ generated by the propagation of the
competing species. This scalar problem isolates the mechanism through which
the logarithmic delay of the fast front is transmitted to the secondary
front. As the analysis below shows, the resulting logarithmic correction to
the secondary front is not, in general, given by the classical Bramson
coefficient.

A key contribution of the present work is to disentangle the mechanisms governing the interaction between the two invasion fronts in the Shigesada--Kawasaki problem. Girardin and Lam~\cite{girardin2019invasion} determined the leading-order spreading speeds through a global construction of coupled sub- and supersolutions, in which the dynamics of the two fronts are treated simultaneously. While effective for establishing spreading speeds, this approach does not readily isolate the mechanism by which the faster front influences the slower one. Our analysis instead proceeds in stages, identifying the exponentially small density of the slower species at the location of the faster front as the quantity that mediates their interaction. More precisely, setting
\begin{equation}
s(t)=c_vt-\eta_v\ln (1+t),
\qquad
c_v=2\sqrt{dr},
\qquad
\eta_v=\frac32\sqrt{\frac dr},
\end{equation}
we establish
\begin{equation}\label{e.107.1}
u(t,s(t))
\asymp
t^{-\frac32+\frac{c_v\eta_v}{2}}
e^{-\left[\tfrac{c_v^2}{4}-1\right]t}.
\end{equation}
This estimate provides the link between the two fronts and allows us to analyse the subsequent propagation of the slower species through an associated scalar equation in a shifting environment. By separating these mechanisms, we obtain a more transparent route from the interaction of the fronts to
localisation of their level sets up to a bounded error
and convergence to the corresponding travelling-wave profiles. Identifying the limiting profiles requires an additional Liouville-type rigidity argument, developed for competition systems and more general monotone parabolic systems in \Cref{app:rigid}.

\subsection{Relation to previous work}

Kolmogorov et al.~\cite{kolmogorov1937study} first showed that, for
Heaviside initial data, solutions of the Fisher--KPP equation converge in
shape to a travelling wave after a suitable time-dependent shift. This was
later extended to broader classes of monotone initial data
\cite{Rothe}. Uchiyama~\cite{uchiyama1978behavior} obtained the critical
logarithmic correction up to an $O(\ln\ln t)$ error, while
Bramson~\cite{bramson1983convergence} established the sharp asymptotics
through the probabilistic representation of the KPP equation by branching
Brownian motion. Related results were obtained by Lau~\cite{lau1985nonlinear}
using maximum-principle arguments. Bramson-type logarithmic corrections have also been established for spatially periodic environments
\cite{hamel2016logarithmic}, providing a precursor to the study of
more general spatial heterogeneity.

Models of population dynamics in shifting environments were introduced by
Potapov and Lewis~\cite{potapov2004climate} and Berestycki et al.
\cite{berestycki2009can}. Depending on the environmental shift rate, the
population front may propagate ahead of the moving habitat, remain locked to
it, or lag behind it. Of special interest are forced waves, which are stationary solutions in the moving environmental frame. For environments with a local maximum, Berestycki and Fang \cite{berestycki2018forced} established existence and multiplicity results for forced waves and investigated their attractivity, including the selection of minimal forced waves in certain parameter regimes. Here we treat the borderline case of a monotone environment, where there is no minimal forced wave, and extend the ideas in Hamel et al. \cite{hamel2013short,hamel2016logarithmic} and Nolen et al. \cite{nolen2017convergence} to prove selection of a suitable forced wave in the shifted frame.

Beyond the forced wave regime, nonlocal pulling may occur:
favourable habitat located an $O(t)$ distance ahead of the invasion front can
enhance propagation and select a supercritical speed
\cite{holzer2014accelerated,girardin2019invasion}.
Speed selection in such problems has been studied extensively, in
particular through Hamilton--Jacobi limits \cite{lam2026survey,Lam2022asymptotic}. 
For nonmonotone environmental profiles, leading-order spreading speeds have been studied recently in \cite{girardin2024spreading,lam2024asymptoticspreadingkppreactive}. 
The study of Bramson-type logarithmic corrections in shifting environments has only recently been initiated in \cite{LamWu_preprint}, where the environmental profile $\g$ is assumed to be piecewise constant, allowing the use of explicit heat-kernel estimates.

A further contribution of this work is a Liouville-type rigidity principle that may be of independent interest. In the scalar Fisher--KPP setting, the PDE proof of the Bramson correction in \cite{hamel2013short} uses a sliding argument to characterise limiting entire solutions in a moving frame, based on the Liouville theorem of Berestycki and Hamel~\cite{beresycki2007generalized}. For monotone systems, uniqueness of bistable travelling waves up to translation is classical; see Volpert et al.~\cite[pp.~15, Theorem~2.1]{Volpert1994travelling}, as well as \cite{matano2019asymptotic} for an anisotropic model. Related uniqueness results for monostable competition systems have been established in specific settings through comparison and sliding arguments, often under assumptions involving detailed spatial asymptotics or linear determinacy \cite{huang2007,Leung2011travelling,Zhao2011existence}. Our rigidity theorem applies to a general class of monostable systems, including \eqref{e.com}, and requires only rough control of the leading-edge decay to initiate sliding. Uniqueness of travelling waves up to translation then follows as a consequence.

Seminal existence results for travelling waves in Lotka--Volterra competition systems were obtained by Tang and Fife \cite{tang1980propagating}, Gardner \cite{gardner1982existence}, Conley and Gardner \cite{conley1984application}, Volpert et al.~\cite{Volpert1994travelling}, and Kan-On \cite{kan1995parameter,kan1997fisher}. Spreading from compactly supported data was studied in \cite{lewis2002spreading,li2005spreading}, while 
exponentially decaying data and nonlocal interactions between separated fronts were investigated in \cite{liu2021asymptotic}. Our results refine front localisation up to a bounded error and establish convergence to travelling-wave profiles for initial data with algebraically weighted exponential tails. 

The co-invasion setting, in which initially localised populations generate interacting fronts, was motivated by Shigesada and Kawasaki
\cite{shigesada1997biological}. Leading-order spreading speeds were established by Girardin and Lam \cite{girardin2019invasion} for $0<a<1<b$ and by Liu et al. \cite{liu2019asymptotic} for $0<a,b<1$; see also  \cite{lin2012asymptotic}. For comparison, the bistable case $a,b>1$ was studied in \cite{carrere2018spreading,peng2021sharp}, 
but it is quite different from the monostable case. See also \cite{stavrianidi2026longtimebehaviorsystemsfisherkpp} for front interaction arising from mutation of species. We extend the monostable spreading-speed description by resolving the logarithmic displacement of the interacting fronts.

\subsection{Main results for the scalar equation}

Before stating our main results, 
we recall the following standard asymptotic notations 
for two positive functions $p$, $q$ of a variable $z$. We write $p(z)\lesssim q(z)$ if $\limsup\limits_{z\to\infty}\frac{p(z)}{q(z)}<\infty$. We use the notation $p(z)\asymp q(z)$ whenever $p(z)\gtrsim q(z)$ and $p(z)\lesssim q(z)$.

We study propagation for the scalar equation in a shifting environment \cite{potapov2004climate,
berestycki2009can,berestycki2018forced,lam2026survey}, where the population is modelled by 
\begin{equation}\label{eq:mainshift}
 \begin{cases}
      n_t - n_{xx} = n( \g(x- s(t)) - n) \quad \text{ for }t>0,~x \in \mathbb{R},\\
      n(0,\cdot) = n_0,
 \end{cases}  
\end{equation}
where $n_0 \in \mathsf{L}^\infty(\mathbb R; [0,1])$ satisfies, for some $x_0>0$, 
\begin{equation}\label{e.ini}
\liminf_{x\to-\infty}n_0(x)>0, \quad n_0(x)=0 \quad \text{ for all } x\geq x_0. 
\end{equation}
The climate shift is given by
\begin{equation}
s(t) = \s t - \eta \ln(1+t),
\end{equation}
where $\s$ and $\eta$ are given constants: $\s$ represents the leading-order speed of the climate drift, while $\eta$ quantifies the delay with respect to the leading order. 

For a level $\kappa>0$, define the rightmost level-set position
\begin{equation}\label{location-u}
\xi_{\kappa}(t)=\sup\left\{x\in\mathbb{R}:n(t,x)\ge \kappa\right\}.
\end{equation}
Our goal is to locate these level sets up to a bounded error. 
As we shall see, the invasion front may advance ahead of the environmental interface, remain locked to it, or lag behind it.

Existing results describe persistence, extinction and spreading speeds in several classes of shifting environments, generally at leading order. We refine this description by determining logarithmic corrections to front locations and, where appropriate, convergence after recentering to homogeneous travelling waves $\Phi_\nu$ connecting $\nu>0$ to $0$:
\begin{equation}\label{TW0}
c_\star\Phi_{\nu}'+\Phi_\nu''+\Phi_\nu(\nu-\Phi_\nu)=0 \quad \text{ with } \lim_{-\infty}\Phi_\nu = \nu,\quad \lim_{+\infty}\Phi_\nu = 0.
\end{equation}
For each pair of environmental parameters $(\s,\eta)$, we associate a selected leading-order speed $c_\star=c_\star(\s,\eta)$ and exponential decay rate $\lambda_\star=\lambda_\star(\s,\eta)$ with the corresponding travelling-wave profile and $\nu\in\{1-a,1\}$. Specifically, we set (see \cite{Lam2022asymptotic,lam2024asymptoticspreadingkppreactive})
\begin{equation}\label{eq:lambdastar}
\lambda_\star := 
\begin{cases}
\sqrt{1-a}, &\text{when } \s > 2 ( \sqrt{1-a} + \sqrt{a} ),\\[5pt]
\frac{\s}{2} - \sqrt{a} &\text{when } \s \in (2, 2 ( \sqrt{1-a} + \sqrt{a} )],\\[5pt]
\begin{cases}
1-\sqrt{a},&\text{if } \eta<\frac{3}{2}\\
1,&\text{if } \eta \ge\frac{3}{2}
\end{cases} &\text{when } \s = 2,\\
1&\text{when } \s < 2,
\end{cases}
\end{equation}
and 
\begin{equation}\label{eq:cstar}
c_\star := 
\begin{cases}
2 \sqrt{1-a}, &\text{if } \s > 2 ( \sqrt{1-a} + \sqrt{a} ),\\[5pt]
\lambda_\star + \frac{1-a}{\lambda_\star} &\text{if } \s \in (2, 2 ( \sqrt{1-a} + \sqrt{a} )],\\[5pt]
2&\text{if } \s \leq 2.
\end{cases}
\end{equation}
For convenience, we also denote
\begin{equation}
    \fraksc :=2(\sqrt a+\sqrt{1-a}).
\end{equation}
When $\s=2$, the limiting homogeneous travelling wave is associated with $\nu=1-a$ if $\eta<\frac32$, and with $\nu=1$ if $\eta>\frac32$. At the critical value $\eta=\frac32$, a distinct forced-wave regime emerges between these two homogeneous travelling-wave regimes. Here, the invasion front remains locked to the moving environmental interface, and the dynamics select a particular forced wave from a one-parameter family; see \Cref{thm:selection}. This intermediate regime marks the transition between nonlocal pulling from the favourable region ahead and classical Fisher--KPP propagation within that region. The transition is also reflected in the discontinuity of the selected decay rate $\lambda_\star$, which jumps from $1-\sqrt a$ to $1$ as $\eta$ crosses $\frac32$, while the propagation speed remains $c_\star=2$.

Suppose $\s < 2$ and the growth rate is given by a heterogeneous function $\g \in \mathsf{L}^\infty(\mathbb{R})$ satisfying
\begin{equation}\label{e.spatial1}
  \lim_{x \to +\infty} \g(x) = 1.
\end{equation}
Then \cite[Theorem 5]{Lam2022asymptotic} shows that the spreading speed is still $c^\star=2$.

The following theorem establishes stability of the classical Bramson correction.

\begin{theorem}\label{t.0} Suppose $\s <2$ and $\eta \in \mathbb R$. 
Let $n$ be a solution of \eqref{eq:mainshift} with initial data satisfying \eqref{e.ini}. Assume that there exists $x_0 >0$ such that $\g \in \mathsf{L}^\infty(\mathbb{R})$ satisfies
\begin{equation}
{\inf_{\mathbb R} \g >0,\qquad}
   \sup_{x \geq x_0}\g \leq 1, \qquad \sigma^\star(x):= 1-\inf_{x' \geq x}\g(x') \in \mathsf{L}^1((x_0,\infty)).
\end{equation} 
Then there exists a bounded function $\Lambda: \mathbb{R}^+ \to \mathbb{R}$ such that for any
$\gamma>0$,
\begin{equation}\label{e.1016.2a}
\lim_{t\to\infty} \sup_{x \geq s(t)+\gamma \ln t}  \left| n(t,x) -  \Phi_1^\star\left(x-2t + \frac32 \ln t + \Lambda(t) \right) \right| = 0,
\end{equation}
where $ \Phi_1^\star$ is the critical travelling wave that satisfies \eqref{TW0} with $\nu =1$, $\lambda_\star = 1$ and $c_\star = 2$.
\end{theorem}

The intuition is that, for any $c \in (0,2-\s)$, the solution $n(t,x)$ satisfies a perturbed equation:
\begin{equation}
n_t - n_{xx} \geq n(1 - \sigma^\star(ct) - n) \quad \text{ for }t \geq 0,~ x - s(t) \geq ct.
\end{equation}
Here the additional growth coefficient $t\mapsto \sigma^\star(ct)$ is integrable and thus introduces only a bounded error in the invasion front moving frame; see \cite{Ducrot2015} for results in multi-dimensions in the case $\s=0$. However, the condition $1-\g \in \mathsf{L}^1(\R^+)$ has to be strengthened to $\left( y \mapsto y(1-\g(y))\right) \in \mathsf{L}^1(\R^+)$ when the distance between $s(t)$ and the invasion front is bounded or of order $\ln t$; see condition (G2) and \Cref{thm:standard}.

In the remainder of this section, we assume that the environmental profile $\g$ is nondecreasing, i.e. 
the habitat ahead of a right-moving environmental interface is more favourable than that behind it. 
Let $\g_+,\g_-$ denote the limits of $\g$ at $\pm \infty$, respectively. 
Up to some rescaling, one may assume that $\g_+ = 1$ and $\g_- = 1 -a$, for some $a \in (0,1)$. Since $\g$ increases from $1-a$ to $1$, the more favourable habitat lies to the right of the moving interface.
\begin{itemize}
\item[{\rm(G1)}] The function $\g\in \mathsf{L}^\infty(\mathbb{R})$ is monotonically increasing and
    \begin{equation}
    \g(-\infty) =1-a \quad \text{ and }\quad \g(+\infty) = 1 \quad \text{ for some }a \in (0,1).
    \end{equation}
     \item[{\rm (G2)}] We impose the following quantitative decay conditions at spatial infinity.
        \begin{equation}
        |\g(y)-(1-a)| \in \mathsf{L}^1((-\infty,0)) \quad \text{ and }\quad  y|1-\g(y)| \in \mathsf{L}^1((0,\infty)).\end{equation} 
\end{itemize}

When the environment shifts sufficiently slowly, the population spreads ahead of the shift, and one is led to the standard Fisher-KPP situation; see \Cref{f.1}. Note that the first part of the condition (G2) plays no role in this theorem.
\begin{theorem}\label{thm:standard}
{Suppose $\s=2$ and $\eta\in  \left(\frac{3}{2},\infty\right)$.}

Then the rightmost $\kappa$-level position defined in \eqref{location-u} satisfies
\begin{equation}\label{level-set}
 \xi_{\kappa}(t)=2 t - \frac{3}{2}\ln (1+t)+O(1) \quad \text{ for each }\kappa \in (0,1),
\end{equation}
and 
there exists a bounded function $\Lambda : \mathbb{R}^+ \mapsto \mathbb{R}$ such that for any $\gamma>0$, 
\begin{equation}\label{standard-cov}
\lim_{t \to \infty} \, \sup_{x \geq s(t) + \gamma  \ln t }\Big|n(t,x) - \Phi_{1}^\star \left(x-2 t +\frac{3}{2}  \ln t+\Lambda(t)\right)\Big| = 0,
\end{equation}
where $\Phi_{1}^\star$ is a (critical) travelling wave that satisfies \eqref{TW0} with $\nu=1$, $\lambda_\star=1$, and $c_\star=2$. In particular, for any $\kappa \in (0,1)$, the level sets lie at a logarithmically diverging distance ahead of $s(t)$:
\begin{equation}
\xi_{\kappa}(t) - s(t) = \left( \eta - \frac32  \right)\ln t + O(1).
\end{equation}
\end{theorem}

Arguing as in \cite{hamel2013short}, one can strengthen the conclusions \eqref{e.1016.2a} and \eqref{standard-cov} in \Cref{t.0,thm:standard} to
\begin{equation}
{\lim_{R \to \infty}\lim_{t\to\infty} \sup_{x \geq s(t)+ R} \left| n(t,x) -  \Phi_1^\star\left(x-2t + \frac32 \ln t + \Lambda(t) \right) \right| = 0}.
\end{equation}

\begin{figure}[H]

\begin{tikzpicture}[scale=1.15]

\draw[->, thick] (-6,0) -- (6.2,0) node[right] {$x$};
\draw[->, thick] (-5.8,-0.4) -- (-5.8,3.2);

\draw[red, dashed] (-5.8,2.6) -- (6,2.6);
\node[red, left] at (-5.9,2.6) {$1$};

\draw[red, dashed] (-5.8,2.0) -- (6,2.0);
\node[red, left] at (-5.9,2.0) {$1-a$};

\draw[blue, dashed] (-5.8,1.1) -- (6,1.1);
\node[blue, left] at (-5.9,1.1) {$\kappa$};

\draw[red, very thick]
(-5.8,2.0)
-- (-2.5,2.0)
to[out=0,in=180] (0.25,2.6)
-- (6,2.6);

\node[red] at (-4.8,2.25) {$\g(x-s(t))$};

\draw[blue, very thick]
(-5.8,1.95)
to[out=0,in=180] (-2,1.95)
to[out=0,in=180] (1,2.55)
to[out=180,in=180] (3,2.55)
to[out=0,in=180] (6,0.05);

\node[blue] at (1.5,2.3) {$n(t,x)$};

\draw[dotted] (4.6,0) -- (4.6,1.15);
\draw[dotted] (-1.5,0) -- (-1.5,2.6);

\filldraw[blue] (4.6,0) circle (2.5pt);
\node[blue,below] at (4.6,0) {$\xi_{\kappa}(t)$};

\filldraw[red] (-1.5,0) circle (2.5pt);
\node[red,below] at (-1.5,0) {$s(t)$};

\end{tikzpicture}
\caption{Schematic illustration of the large-time propagation described in \Cref{thm:standard}.} \label{f.1}
\end{figure}
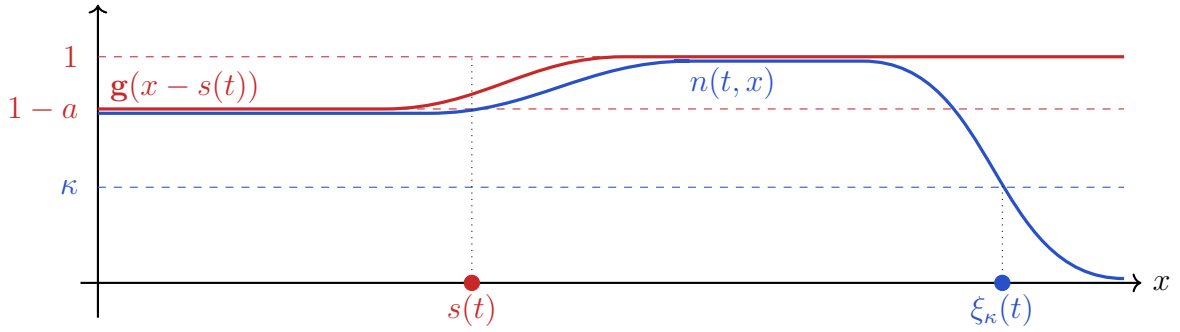

We next state the theorem at $(\s,\eta)=(2,\tfrac32)$. The family of stationary forced-wave profiles is illustrated in \Cref{fig:profiles}.

\begin{theorem}\label{thm:selection}
Assume $\s =2$ and $\eta = \frac32$. For every $A>0$, there is a unique positive bounded solution $U_A$ to
\begin{equation}\label{eq:profile}
U_A''+2U_A'+U_A(\g-U_A)=0\quad\text{in }\R
\end{equation}
such that
\begin{equation}\label{eq:profile-limits}
\lim_{\xi\to-\infty} U_A =1-a,\qquad \lim_{\xi\to+\infty}  U_A=0,\qquad
\lim_{\xi\to+\infty}\frac{U_A(\xi)}{\xi e^{-\xi}}=A.
\end{equation}
Let $n$ be the solution to \eqref{eq:mainshift}. Then there exists $A_\infty\in(0,\infty)$ such that, for every $R>0$,
\begin{equation}\label{eq:selection}
\lim_{t\to\infty}\sup_{|x|\leq R}
\left\vert n(t,s(t)+x)-U_{A_\infty}(x)\right\vert=0.
\end{equation}
Furthermore, the value of $A_\infty$ is given by 
\begin{equation}\label{eq:amplitude-formula}
A_\infty=\frac1{2\sqrt\pi}\lim_{t\to\infty}
(1+t)^{-\frac32}\int_\R \mathfrak{m}(x)e^x n(t,s(t)+x)\dd x, 
\end{equation}
where the limit $A_\infty$ exists and is positive, and where
$\mathfrak{m}$ is the unique positive solution to
\begin{equation}
\mathfrak{m}''=(1-\g)\mathfrak{m},\qquad \mathfrak{m}(-\infty)=0,\qquad \mathfrak{m}'(+\infty)=1. 
\end{equation}
Note that the existence and uniqueness of $\mathfrak{m}$ are guaranteed thanks to {\rm (G1)--(G2).}
\end{theorem}

\begin{figure}
    \centering
    \includegraphics[width=0.7\linewidth]{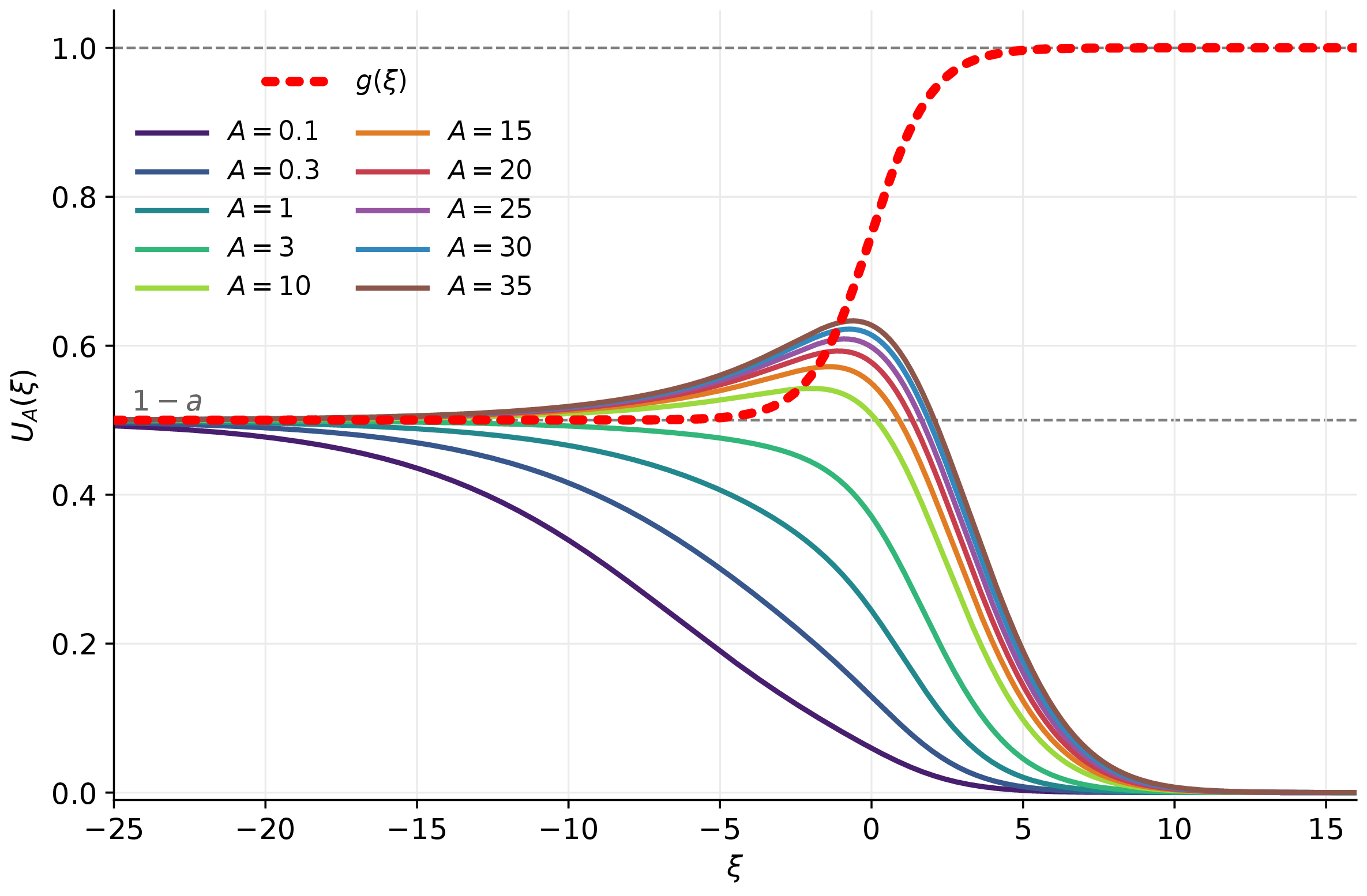}
    \caption{Forced wave profiles $U_A$ at the critical speed $\s=2$, computed
by a shooting method for \eqref{eq:profile} with 
$\g(\xi)=1-\frac{1}{2(1+e^\xi)}$.
The profiles satisfy $U_A(-\infty)=\frac12$ and are parametrised by 
$A$, defined by
$U_A(\xi)\sim A\xi e^{-\xi}$ as $\xi\to+\infty$.
The values of $A$ are indicated in the legend.
The bolded red dotted curve represents $\g$, while the grey dashed lines
show the limiting states $1-a=\frac12$ and $1$.}
    \label{fig:profiles}
\end{figure}
In the supercritical pulling regime, the front speed (which exceeds the critical speed of $2\sqrt{1-a}$ in the less favourable region behind the environmental shift) and the logarithmic correction are both influenced by the interaction with the environmental interface.

\begin{theorem}\label{thm:pulling}
Assume  
\begin{equation}
{\rm(i)} \quad (\s,\eta) \in \left(2,\,\fraksc\right)\times \mathbb R,\quad \text{ or }\quad  {\rm(ii)} \quad (\s,\eta) \in \{2\} \times \left(-\infty,\frac{3}{2}\right).
\end{equation}
Then there exists a bounded function $\Lambda : \mathbb{R}^+ \mapsto \mathbb{R}$ such that 
\begin{equation}
\lim_{t \to \infty} \,\sup_{x\in \mathbb{R} }\Big| n(t,x) - \Phi_{1-a}^\star\left(x-c_{\star} t +\frac{1}{\lambda_{\star}}\left(\frac32 - \eta  \sqrt{a}\right) \ln t+\Lambda(t)\right)\Big|=0,
\end{equation}
where $\Phi_{1-a}^\star$ is a (supercritical) travelling wave that satisfies \eqref{TW0} with $\nu=1-a$, $\lambda_\star=\tfrac{\s}{2}-\sqrt{a}$, and $c_\star=\lambda_\star + \tfrac{1-a}{\lambda_\star}$.
In particular, for each $\kappa \in (0,1-a)$, the level set $\xi_{\kappa}(t)$ defined in \eqref{location-u} is lagging behind the environmental shift and 
\begin{equation}\label{e.level.3}
s(t) - \xi_\kappa(t) \gtrsim \begin{cases}(\s - c_\star)t  &\text{ in case (i),}\\[5pt]
    \frac{1}{1-\sqrt{a}} \left(\frac32 - \eta\right) \ln t &\text{ in case (ii).}
\end{cases}
\end{equation}
\end{theorem}

Next, we state the result in the critical pulling case. In this case, the
polynomial prefactor in the leading edge of the minimal travelling wave changes. This may modify the logarithmic coefficient even though the first-order spreading speed remains unchanged. At a critical parameter, the leading logarithmic contributions balance and an additional $\ln\ln t$ correction emerges.

\begin{theorem}\label{thm:cpulling}
Assume $\s = \fraksc$ and $\eta \in \mathbb R$. 

Set
\begin{equation}\label{m^w-2}
\ell(t) := 
\begin{cases}
 \frac{3}{2\lambda_{\star}} \ln t , &\text{ if }\eta < - \frac{1}{2\sqrt{a}},\\[5pt]
 \frac{3}{2\lambda_{\star}} \ln t - \frac{1}{\lambda_{\star}} \ln \ln t &\text{ if } \eta =-\frac{1}{2\sqrt{a}},\\[5pt]
\frac{\frac{5}{2} -\eta \sqrt{a}}{2\lambda_{\star}} \ln t &\text{ if } \eta  > -\frac{1}{2\sqrt{a}}.
\end{cases}
\end{equation}
Then there exists a bounded function $\Lambda : \mathbb{R}^+ \mapsto \mathbb{R}$ such that 
\begin{equation}
\lim_{t \to \infty}\sup_{x\in \mathbb{R} }\Big|n(t,x) - \Phi_{1-a}^\star\left(x- c_{\star} t +\ell(t)+\Lambda(t)\right)\Big|= 0,
\end{equation}
where $\Phi_{1-a}^\star$ is a (critical) travelling wave that satisfies \eqref{TW0} with $\nu=1-a$, $\lambda_\star=\sqrt{1-a}$, and $c_\star=2\sqrt{1-a}$.
In particular, for each $\kappa \in (0,1-a)$, \begin{equation}\xi_{\kappa}(t)=c_{\star} t - \ell(t)+O(1).\end{equation}
\end{theorem}

The critical phase correction can also be interpreted through an effective exponent, in the spirit of Bramson's analysis \cite{bramson1983convergence} (see also \cite{alfaro2024bramson,zhang2025sharpasymptoticskppequation} for PDE arguments in the style of \cite{hamel2013short}). 
Indeed, a useful way to interpret the logarithmic corrections obtained in a shifting
environment is to view the moving environmental interface as a mechanism
which dynamically generates an effective tail as seen by the population residing in the less favourable region. As will appear later, up to multiplicative
constants and bounded spatial shifts, the population density near the
interface has the characteristic size
\begin{equation}
   n(t,s(t))
    \asymp
    t^{-\frac32+\frac{\s\eta}{2}}
    e^{-\left[\frac{\s^2}{4}-1\right]t}.
\end{equation}

The factor $t^{-\frac32}$ records the diffusive loss produced by the Dirichlet
leading-edge dynamics, whereas the factor $t^{\frac{\s\eta}{2}}$ records the
logarithmic displacement of the environmental interface. 
This estimate must then be matched with the dynamics in the region behind the interface, where the linearised equation is 
\begin{equation} w_t=w_{xx}+(1-a)w. \end{equation} 
For this homogeneous equation, a tail of the form $ w_0(x)\asymp x^q e^{-\lambda x}$ with $0<\lambda\leq \sqrt{1-a}$
generates, at the linear level, 
\begin{equation}\label{eq:linear-effective-tail} w(t,x) \asymp (x-2\lambda t)^q \exp\left(-\lambda x+(\lambda^2+1-a)t\right) \qquad \text{ where }\lambda x- (\lambda^2 +1-a)t \gtrsim t. 
\end{equation} 
At the environmental interface, one has \begin{equation} s(t)-2\lambda t = (\s-2\lambda)t-\eta\ln(t+1). \end{equation} The exponential behaviour in \eqref{eq:linear-effective-tail} agrees with $n(t,s(t))$ when 
\begin{equation}\label{eq:lambda-star-effective-tail} \lambda=\lambda_\star:=\frac{\s}{2}-\sqrt a \quad \text{ and }\quad q = q_B. 
\end{equation} 
Indeed, $-\s\lambda_\star+\lambda_\star^2+1-a = 1-\frac{\s^2}{4}$ leads to $\s-2\lambda_\star=2\sqrt a$.
Consequently, evaluating \eqref{eq:linear-effective-tail} at $x=s(t)$ gives $$t^{q+\lambda_\star\eta} e^{-\left[\frac{\s^2}{4}-1\right]t},$$ up to a multiplicative constant. The matching happens if and only if $q+\lambda_\star\eta = -\frac32+\frac{\s\eta}{2}$, or equivalently, $ q_{\mathrm{B}}=-\frac32+\sqrt a\,\eta$. 
Thus, the exponent $q_B$ measures the polynomial thickness of the effective tail generated by the shifting interface. Its two contributions have a clear origin: the term $-\frac32$ comes from the Dirichlet heat-kernel decay, while the term $\sqrt a\,\eta$ describes the transfer of the logarithmic phase of the environmental interface to the population front.

The next result concerns the non-pulling case. In this regime, the fast interface is too far away to influence the secondary front at leading order. One recovers a Bramson-type delay associated with the growth rate $1-a$. 
 
\begin{theorem}\label{thm:no-pull}
Assume $\s > \fraksc$ and $\eta \in \mathbb R$. Then there exists a bounded function $\Lambda : \mathbb{R}^+ \mapsto \mathbb{R}$ such that 
\begin{equation}\label{thm1.7-conv}
\lim_{t \to \infty}\sup_{x\in\mathbb{R}}  \Big|n(t,x )  -\Phi_{1-a}^\star\left(x- c_{\star} t+ \frac{3}{2\lambda_{\star}}\ln t + \Lambda(t)\right) \Big| = 0,
\end{equation} 
where $\Phi_{1-a}^\star$ is a (critical) travelling wave that satisfies \eqref{TW0} with $\nu=1-a$, $\lambda_\star=\sqrt{1-a}$, and $c_\star=2\sqrt{1-a}$.
 In particular, for each $\kappa \in (0,1-a)$, $\xi_{\kappa}(t)=c_{\star} \,t - \frac{3}{2\lambda_{\star}}\ln t+O(1)$.
\end{theorem}

The preceding results also show that
the spreading front is unaffected by the environmental shift when $s(t)$ moves sufficiently fast, in the sense that
   \begin{equation}
   \liminf_{t \to \infty} \left[\frac{s(t) - \fraksc t -\frac{1}{2\sqrt a} \ln (1+t)}{\ln(1+t)}\right] 
    >\varepsilon
   \quad \text{ for some }\varepsilon>0.
   \end{equation} 
In this case, the distance between the two fronts satisfies
\begin{equation}
s(t) -\xi_\kappa(t) > 2\sqrt a t +\left(\frac1{2\sqrt{a}} + \frac3{2\sqrt{1-a}} \right) \ln(1+t) + \varepsilon \ln(1+t)  \quad \text{ for }t\gg 1.
\end{equation}

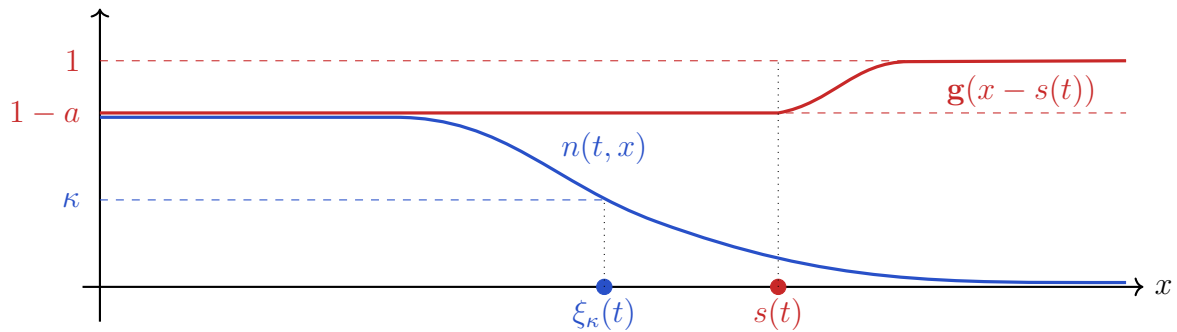
\begin{figure}[H]

\begin{tikzpicture}[scale=1.15]

\draw[->, thick] (-6,0) -- (6.2,0) node[right] {$x$};
\draw[->, thick] (-5.8,-0.4) -- (-5.8,3.2);

\draw[red, dashed] (-5.8,2.6) -- (6,2.6);
\node[red, left] at (-5.9,2.6) {$1$};

\draw[red, dashed] (-5.8,2.0) -- (6,2.0);
\node[red, left] at (-5.9,2.0) {$1-a$};
\draw[blue, dashed] (-5.8,1) -- (0,1);
\node[blue, left] at (-5.9,1.0) {$\kappa$};

\draw[red, very thick]
(-5.8,2.0)
-- (2.0,2.0)
to[out=10,in=180] (3.5,2.59)
-- (6,2.6);

\node[red] at (4.8,2.25) {$\g(x-s(t))$};

\draw[blue, very thick]
(-5.8,1.95)
to[out=0,in=180] (-2.4,1.95)
to[out=0,in=160] (0.6,0.75)
to[out=340,in=180] (6,0.05);

\node[blue] at (0,1.6) {$n(t,x)$};

\filldraw[blue] (0,0) circle (2.5pt);
\node[blue,below] at (0,0) {$\xi_\kappa(t)$};

\filldraw[red] (2,0) circle (2.5pt);
\node[red,below] at (2,0) {$s(t)$};

\draw[dotted] (0,0) -- (0,1);
\draw[dotted] (2,0) -- (2,2.6);
\end{tikzpicture}
\caption{Schematic illustration of the large-time propagation described in \Cref{thm:pulling}, \Cref{thm:cpulling}, \Cref{thm:no-pull}. Depending on the situation, the gap between $s(t)$ and $\xi_\kappa(t)$ is either linear, logarithmic, or of constant order.}
\end{figure}

Let us comment on the terminology associated with each regime. In the climate-shift problem, \textit{pulling} refers to the situation in which the invasion front is selected by the linearised dynamics near the unstable state $n=0$, rather than by the nonlinear core of the population profile, as in standard Fisher-KPP dynamics. When the growth rate equals $1-a$ behind the moving environmental interface and $1$ ahead of it, \textit{nonlocal pulling} occurs if the front remains in the less favourable region but is accelerated by the exponentially small amount of population that spawns from the distant favourable region ahead. The selected speed is then determined not only by the local growth rate $1-a$, but also by the velocity $\s$ of the environmental interface. More precisely, in the \textit{supercritical pulling regime}
$$
2<\s< \fraksc,
$$
the remote favourable region selects the decay exponent
$$
\lambda_\star=\frac{\s}{2}-\sqrt a
<\sqrt{1-a},
$$
and hence the speed
$$
c_\star=\lambda_\star+\frac{1-a}{\lambda_\star}
> 2\sqrt{1-a}.
$$
Thus, “supercritical” refers here to the fact that the selected wave in the less favourable medium has a speed strictly larger than its minimal KPP speed. 

At the threshold $\s=\fraksc$, one has $\lambda_\star=\sqrt{1-a}$ and $c_\star=2\sqrt{1-a}$; this is the \textit{critical pulling regime}. The distant favourable region then no longer changes the first-order propagation speed, but its marginal influence still modifies the logarithmic correction: for instance, when $\eta=0$, the usual Bramson delay $-\frac{3}{2\lambda_{\star}}\ln t$ is replaced by $-\frac{5}{4\lambda_{\star}}\ln t$, while particular values of $\eta$ may produce an additional $\ln\ln t$ correction. 

Finally, when $\s>\fraksc$, the favourable region moves too far ahead to affect the front even at logarithmic order, and the population recovers the classical locally pulled Fisher--KPP dynamics in the medium of growth rate $1-a$. Thus, the expression “no pulling” in this classification means no pulling by the remote favourable region; the resulting front may nevertheless remain pulled in the usual local KPP sense.

We briefly recall the Dirichlet-kernel mechanism behind the classical Bramson correction before explaining its adaptation to heterogeneous growth. The PDE approach of Hamel, Nolen, Roquejoffre and Ryzhik \cite{hamel2013short,hamel2016logarithmic} replaces the nonlinear leading-edge dynamics with a linear equation on a half-line endowed with an absorbing boundary. In the critical frame $y=x-c_\star t$, this equation is
$$
w_t-c_\star w_y=w_{yy}+rw,\qquad y>0,\qquad w(t,0)=0.
$$

Factoring out an exponential prefactor via $p(t,y)=e^{\lambda_\star y}w(t,y)$ reduces this equation to the heat equation on the half-line with a homogeneous Dirichlet condition. The explicit Dirichlet heat kernel then gives
$$
w(t,y)\asymp\,{t^{-\frac32}}{y}\,e^{-\lambda_\star y-\frac{y^2}{4t}},
\qquad 0\leq y\lesssim\sqrt t.
$$
The factor $t^{-\frac32}$ arises from cancellation in the reflected heat kernel. If the absorbing boundary is shifted to
$c_\star t-\frac{\rho}{\lambda_\star}\ln(t+t_0)$,
the corresponding linear solution carries the prefactor $t^{\,\rho-\frac32}$; hence $\rho=\frac32$ is the unique choice for which the solution remains of order one at a fixed distance from the moving boundary. Matching this leading-edge behaviour with the decay of the critical travelling wave selects the position
$$
c_\star t-\frac{3}{2\lambda_\star}\ln t.
$$
In \cite{hamel2013short}, rigorous upper and lower barriers are obtained from suitably shifted Dirichlet problems and then matched to the nonlinear travelling-wave tail. Comparison traps the fixed level sets within a bounded distance of this position.

For a heterogeneous profile $\g$, the corresponding linearised equation does not admit an explicit heat kernel. To overcome this difficulty, we introduce spatial correctors and construct Gaussian-type sub- and supersolutions for the heterogeneous linear problem, retaining the sharp decay estimates at the moving interface $x=s(t)$; see \Cref{prop:supersolutionshiftallbeta,prop:subsolutionshiftallbeta}. These estimates quantify the interaction between the two fronts and determine the logarithmic correction to the secondary front's position. A further key ingredient is the construction of a Gaussian-type subsolution for the full nonlinear equation that, unlike conventional decaying barriers for linearised problems, remains nontrivial uniformly in time. This allows us to directly track the moving level sets up to a bounded error and, together with the subsequent rigidity analysis, establish convergence to the selected travelling-wave profile.

\begin{table}[H]
\centering
\label{tab:scalar-regimes}
\small
\renewcommand{\arraystretch}{1.35}
\setlength{\tabcolsep}{4pt}
\begin{tabularx}{\textwidth}{@{}
>{\raggedright\arraybackslash}p{0.19\textwidth}
>{\raggedright\arraybackslash}X
>{\centering\arraybackslash}p{0.10\textwidth}
>{\centering\arraybackslash}p{0.18\textwidth}
>{\raggedright\arraybackslash}p{0.3\textwidth}@{}}
\toprule
Regime & Parameters & Speed $c_\star$ &
Phase correction $A(\s,\eta)$& Selected profile\\
\midrule

Classical Bramson
&
$\s<2$, or \quad 
$\s=2$, $\eta>\frac32$
&
$2$
&
$\frac32\ln t$
&
Critical KPP wave $\Phi_1^\star$
\\
\hline
Forced wave
&
$\s=2$, $\eta=\frac32$
&
$2$
&
$\frac32\ln t$
&
Selected forced wave $U_{A_\infty}$
\\
\hline

Nonlocal pulling
&
$2<\s<\fraksc$, or
$\s=2$, $\eta<\frac32$
&
$\lambda_\star+
\frac{1-a}{\lambda_\star}$
&
$\frac{\frac32-\eta\sqrt a}{\lambda_\star}
\ln t$
&
Supercritical KPP wave $\Phi_{1-a}^\star$
\\
\hline

Critical pulling
&
$\s=\fraksc$
&
$2\sqrt{1-a}$
&
$\ell(t)$;  see \eqref{m^w-2}
&
Critical KPP wave $\Phi_{1-a}^\star$
\\
\hline

No pulling
&
$\s>\fraksc$
&
$2\sqrt{1-a}$
&
$\frac{3}{2\sqrt{1-a}}\ln t$
&
Critical KPP wave $\Phi_{1-a}^\star$
\\
\bottomrule
\end{tabularx}
\medskip
\caption{Summary of the five propagation regimes.
The front position satisfies
$\xi_\kappa(t)=c_\star t-\ell(t)+O(1)$
for the admissible levels $\kappa$ in each regime.}
\end{table}

\begin{figure}[htbp]
\centering
\begin{tikzpicture}[scale=1.2]
\begin{axis}[
    width=10cm,
    height=6cm,
    xmin=1.35,
    xmax=3.5,
    ymin=1.15,
    ymax=3.25,
    axis lines=left,
    xtick={2,2.73205},
    xticklabels={$2$,$\fraksc$},
    ytick={1.3,1.5,1.73205,3},
    yticklabels={$\frac{5}{4\sqrt{1-a}}$,$\color{teal}{\frac32} $,$\frac{3}{2\sqrt{1-a}}$,$\frac{3}{2(1-\sqrt a)}$},
    xlabel style={at={(current axis.right of origin)}, anchor=west, inner sep=0pt, xshift=5pt, yshift=6pt},
    clip=false
]
\addplot[teal,very thick,domain=1.35:2]{1.5};
\addplot[teal,very thick,domain=2:2.7320508,samples=80]{3/(2*(x/2-.5))};
\addplot[teal,very thick,domain=2.7320508:3.5]{sqrt(3)};
\addplot[muted,densely dashed] coordinates {(2,1.15)(2,3)};
\addplot[muted,densely dashed] coordinates {(2.73205,1.15)(2.73205,1.73205)};
\addplot[only marks,mark=o,teal,mark size=3pt,mark options={fill=white}] coordinates {(2,1.5)(2.73205,1.73205)};
\addplot[only marks,mark=*,teal,mark size=3pt] coordinates {(2,3)};
\addplot[only marks,mark=*,rust,mark size=3.5pt] coordinates {(2.73205,1.3)};
\end{axis}
\end{tikzpicture}
\caption{For $\eta=0$, the figure summarises the logarithmic correction coefficient in Theorems~\ref{t.0}--\ref{thm:no-pull}. The horizontal axis is $\s$, and the vertical axis 
is the logarithmic correction coefficient $A(\s,\eta)$.
It is drawn for the parameter range
$\frac{5}{4\sqrt{1-a}}<\frac{3}{2}$ (a schematic representation).
Although the spreading speed $c_\star$ varies continuously with $\s$ (see \eqref{eq:cstar}), the correction coefficient is discontinuous at both $\s=2$ and $\s=\fraksc$.
}
\label{fig:phase-correction}
\end{figure}
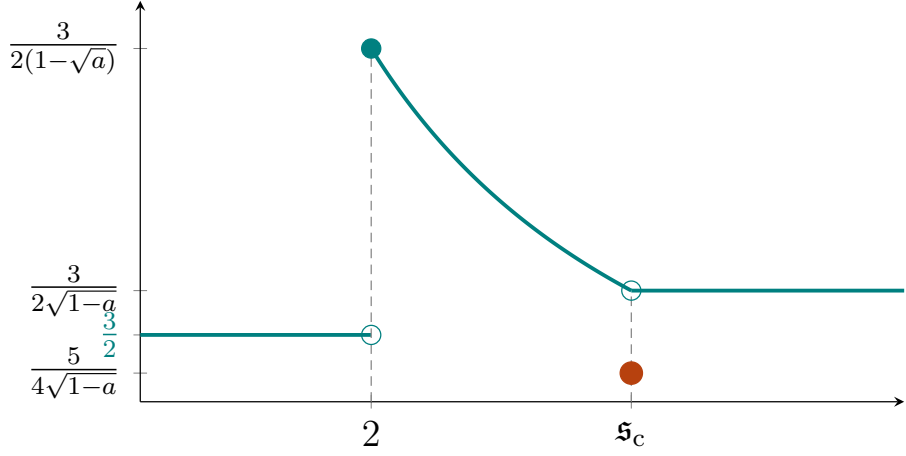

\begin{figure}[H]
\centering
\resizebox{\linewidth}{!}{
\begin{tikzpicture}[
  x={(1cm,.08cm)},y={(.75cm,-.55cm)},z={(0cm,.6cm)},
  >=Stealth,every node/.style={font=\small},
  declare function={n(\x,\y)=(5-\y)/(.625+.125*(\x-3));}]

\fill[black!2] (0,0,0)--(9,0,0)--(9,4.4,0)--(0,4.4,0)--cycle;
\foreach \x in {0,1.5,3,4.5,6,7.5,9}
  \draw[black!10,thin] (\x,0,0)--(\x,4.4,0);
\foreach \y in {0,1,2,3.5,4.4}
  \draw[black!10,thin] (0,\y,0)--(9,\y,0);
\draw[black!35,densely dashed] (3,0,0)--(3,4.4,0);
\draw[black!35,densely dashed] (6,0,0)--(6,4.4,0);

\filldraw[fill=bramson!28,fill opacity=.65,draw=bramson!65]
  (0,0,2.4)--(3,0,2.4)--(3,4.4,2.4)--(0,4.4,2.4)--cycle;
\foreach \x in {.6,1.2,1.8,2.4}
  \draw[bramson!28,thin] (\x,0,2.4)--(\x,4.4,2.4);
\foreach \y in {.55,1.1,1.65,2.2,2.75,3.3,3.85}
  \draw[bramson!28,thin] (0,\y,2.4)--(3,\y,2.4);

\filldraw[fill=nopull!32,fill opacity=.55,draw=nopull!80!black]
  (6,0,3)--(9,0,3)--(9,4.4,3)--(6,4.4,3)--cycle;
\foreach \x in {6.6,7.2,7.8,8.4}
  \draw[nopull!40,thin] (\x,0,3)--(\x,4.4,3);
\foreach \y in {.55,1.1,1.65,2.2,2.75,3.3,3.85}
  \draw[nopull!40,thin] (6,\y,3)--(9,\y,3);

\foreach \i in {0,...,9}{
  \foreach \j in {0,...,10}{
    \pgfmathsetmacro{\xa}{3+.3*\i}
    \pgfmathsetmacro{\xb}{\xa+.3}
    \pgfmathsetmacro{\ya}{.4*\j}
    \pgfmathsetmacro{\yb}{\ya+.4}
    \filldraw[fill=nonlocal!35,fill opacity=.68,
      draw=nonlocal!38,line width=.25pt]
      (\xa,\ya,{n(\xa,\ya)})--(\xb,\ya,{n(\xb,\ya)})--
      (\xb,\yb,{n(\xb,\yb)})--(\xa,\yb,{n(\xa,\yb)})--cycle;
  }
}

\draw[bramson,thick,densely dashed] (3,0,2.4)--(3,3.5,2.4);
\draw[bramson,very thick] (3,3.5,2.4)--(3,4.4,2.4);
\draw[nonlocal,very thick] (3,0,{n(3,0)})--(3,3.5,{n(3,3.5)});
\draw[nonlocal,thick,densely dashed] (3,3.5,{n(3,3.5)})--(3,4.4,{n(3,4.4)});
\draw[nonlocal,thick,densely dashed] (6,0,5)--(6,4.4,.6);
\draw[nopull!85!black,thick,densely dashed] (6,0,3)--(6,4.4,3);

\draw[white,line width=3.4pt] (6,0,3)--(6,1,3)--(6,4.4,1.3);
\draw[critical,line width=1.8pt] (6,0,3)--(6,1,3)--(6,4.4,1.3);
\filldraw[fill=black,draw=white,line width=.7pt] (6,1,3) circle (2.7pt);
\filldraw[fill=forced,draw=white,line width=.8pt] (3,3.5,2.4) circle (3pt);

\draw[nonlocal!80!black,line width=1.1pt,domain=3:6,samples=28]
  plot (\x,2,{n(\x,2)});
\draw[black!50,densely dashed] (3,2,2.4)--(3,2,{n(3,2)});
\draw[critical,densely dashed] (6,2,2.5)--(6,2,3);
\filldraw[draw=bramson,fill=white,line width=.9pt] (3,2,2.4) circle (2.2pt);
\filldraw[draw=white,fill=nonlocal,line width=.7pt] (3,2,{n(3,2)}) circle (2.4pt);
\filldraw[draw=nonlocal,fill=white,line width=.9pt] (6,2,3) circle (2.2pt);
\filldraw[draw=white,fill=critical,line width=.7pt] (6,2,2.5) circle (2.4pt);

\draw[->] (0,4.4,0)--(9.7,4.4,0) node[right] {$\s$};
\draw[->] (0,0,0)--(0,4.9,0) node[below right] {$\eta$};
\draw[->] (0,0,0)--(0,0,8.8) node[above] {$A(\s,\eta)$};
\draw (3,4.4,0)--(3,4.57,0) node[below] {$2$};
\draw (6,4.4,0)--(6,4.57,0) node[below] {$\s_c$};
\foreach \y/\lab in {1/\eta_c,2/0,3.5/{\frac32}}{
  \draw (0,\y,0)--(-.13,\y,0) node[left] {$\lab$};
}

\begin{scope}[shift={(1.5,2.05,2.4)},cm={.75,-.55,1,.08,(0,0)}]
  \node[transform shape,scale=.92,align=center,text=bramson!75!black,
    font=\small\bfseries\boldmath,inner sep=0pt] at (0,0)
    {Classical Bramson\\[5pt]$\displaystyle A=\frac32$};
\end{scope}
\begin{scope}[shift={(7.55,2.3,3)},cm={.75,-.55,1,.08,(0,0)}]
  \node[transform shape,align=center,text=nopull!55!black,
    font=\small\bfseries\boldmath,inner sep=0pt] at (0,0)
    {No pulling\\[5pt]$\displaystyle A=\frac{3}{2\sqrt{1-a}}$};
\end{scope}

\SurfaceText{.39}{1}{N,o,n,l,o,c,a,l,{\hspace{.3em}},p,u,l,l,i,n,g}
\SurfaceFormula{1.02}{.84}{A=\frac{\frac32-\eta\sqrt a}{\lambda_\star}}

\node[anchor=west,text=critical,font=\small\bfseries\boldmath,
  align=left,inner sep=2pt] (clabel)
  at (8.8,.2,5.1) {Critical pulling\\$\eta=\eta_c:= - \frac{1}{2 \sqrt{a}}$};
\draw[->,critical,thick] (clabel.west) to[bend right=12] (6,1,3);
\node[anchor=west,text=forced,font=\small\bfseries,
  inner sep=2pt] (flabel) at (.9,-.6,5.5) {Forced wave};
\draw[->,forced,thick] (flabel.south east) to[out=-35,in=125] (3,3.5,2.4);
\end{tikzpicture}
}
\caption{Schematic illustration of the value of the logarithmic coefficient.
This plot illustrates where the correction
coefficient $A(\s,\eta)$ is continuous or discontinuous. Colours identify the regimes. Solid boundary curves give attained values and dashed boundary
curves give one-sided limits. The darker curve on the green surface marks the section
$\eta=0$; open and filled circles distinguish limits from attained values.
At $\s=2$, the nonlocal value is attained for $\eta<\frac32$ and the classical
value for $\eta\geq\frac32$. The purple curve gives the critical values.}
\label{fig:phase-correction-3d}
\end{figure}
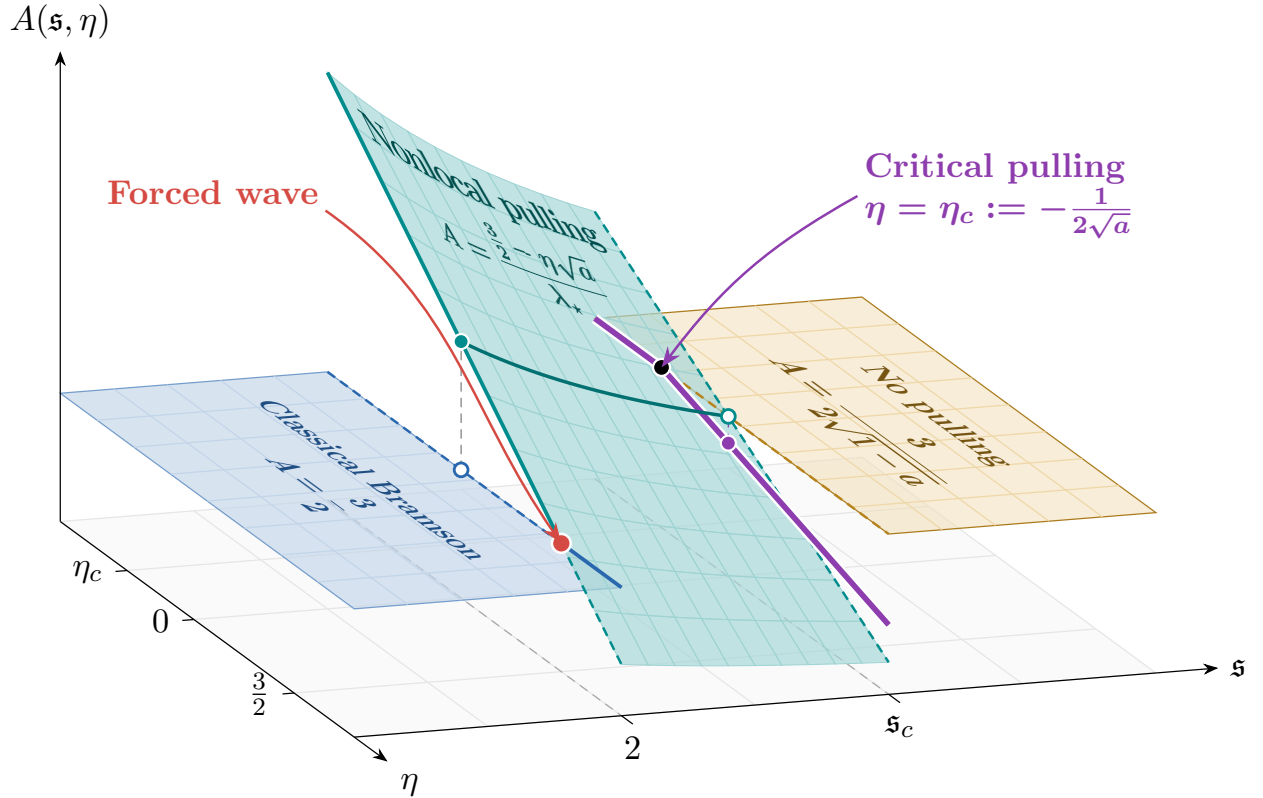

\subsection{Travelling waves and invasion dynamics of the competition system and interaction of interfaces}

We now return to the competition system \eqref{e.com} and consider travelling waves of the form
\begin{equation}
u(t,x)=U_c(\xi),\quad v(t,x)=V_c(\xi), \quad \xi=x-ct,
\end{equation}
where $c$ is the propagation speed. Substituting this ansatz into
\eqref{e.com}, we find that the wave profile $(U_c,V_c)$ satisfies
\begin{equation}\label{TW-sys}
\begin{cases}
U_c''+cU_c'+U_c(1-U_c-aV_c)=0,&\quad \xi\in\mathbb{R},\\
dV_c''+cV_c'+rV_c(1-V_c-bU_c)=0,&\quad \xi\in\mathbb{R},\\
\lim_{\xi \to -\infty}(U_c,V_c)(\xi)=E^-,\quad \lim_{\xi \to +\infty}(U_c,V_c)(\xi)=E^+.
\end{cases}
\end{equation}

Kan-On \cite{kan1997fisher} established a critical number $c^0_{\rm LLW} \in [2\sqrt{1-a},2]$ such that \eqref{TW-sys} has a travelling wave solution connecting $(u^\star, v^\star)$ to $(0,1)$ if and only if $c \in [c^0_{\rm LLW},\infty)$, where $(u^\star,v^\star)$ denotes the stable state given in \eqref{e.uv.uv}. Furthermore, we say that the minimal travelling wave is pulled (or linearly determined) if $c^0_{\rm LLW} = 2\sqrt{1-a}$, while it is pushed (or nonlinearly determined) if $c^0_{\rm LLW} > 2\sqrt{1-a}$. The question of linear determinacy was first studied by Hosono \cite{hosono1998minimal} and Lewis et al. \cite{lewis2002spreading}.

Let $d,r>0$, $0<a<1$ and {$b >0$}. We consider entire classical solutions
$0<u,v<1$ in the moving frame $y=x-ct$ of
\begin{equation}\label{eq:LV}
\begin{cases}
u_t=u_{yy}+cu_y+u(1-u-av),\\[5pt]
v_t=dv_{yy}+cv_y+rv(1-v-bu),
\end{cases}
\qquad (t,y)\in\mathbb{R}^2.
\end{equation}
A travelling wave of speed $c$ is thus a stationary solution of \eqref{eq:LV}. 
We first establish uniqueness up to translation of such waves in both monostable parameter regimes, extending 
\cite[Proposition~A.7]{girardin2019invasion}.
\begin{theorem}[Uniqueness of travelling fronts]\label{t.uniqueness}
Let $a \in (0,1)$ and $b \in (0,1)\cup(1,\infty)$ and let $c>0$ be given.  
If $(U,V)$ and $(\widetilde U, \widetilde V)$ are bounded nonnegative travelling waves satisfying \eqref{TW-sys} and connecting $(u^\star, v^\star)$ to $(0,1)$, then $U' <0 < V'$ on $\mathbb{R}$ and 
\begin{equation}
(U, V) \equiv  (\widetilde U(\cdot\,-\sigma), \widetilde V(\cdot\,-\sigma))
\qquad \text{ for some }\sigma \in \mathbb R.
\end{equation}
\end{theorem}

The uniqueness theorem follows from the rigidity result below, which assumes only two-sided control of the leading-edge tail of $u$ and suitable persistence at spatial infinity. Unlike sliding methods that compare both components with translates of a known front  \cite{Leung2011travelling,Wang2019uniqueness,Zhao2011existence,wu2023sharp}, our argument exploits stability of the remaining asymptotic modes. This allows us to initiate the sliding argument without detailed decay expansions for $v$ or for the left tail of $u$. The structural extension to monotone systems is given in  \Cref{app:rigid}.

\begin{theorem}\label{thm:liouville}
Let $0<a<1$ and $b\in (0,\infty)\setminus\{1\}$. Suppose there exist $c >0$ and a bounded, positive solution $(u,v)$ of \eqref{eq:LV} satisfying the following properties:
\begin{itemize}
    \item[{\rm(i)}]  \text{(Asymptotics of $u$ at $y=+\infty$)} 
    \begin{equation}\label{e.917.2}
u(t,y) \asymp  (\ell_0 + \ell y)e^{-\lambda y}
\qquad \text{ uniformly in } t\in\mathbb{R},\ y\ge R,
\end{equation}
for some $R,\lambda>0$, and $\ell_0,\ell \geq 0$ such that $\ell_0 + \ell >0$. 
    \item[{\rm(ii)}]  \text{(Uniform positivity at opposite spatial ends)}
\begin{equation}\label{e.917.-1}
\inf_{t\in \mathbb{R},~ y \ge A_0} v(t,y) >0\quad \text{ and }\quad  \inf_{t\in \mathbb{R},~ y \leq -A_0} u(t,y) >0 \quad \text{ for some }A_0>0.
\end{equation}
\end{itemize}
In case $b\in (0,1)$, we assume in addition that 
\begin{itemize}
    \item[{\rm(iii)}] \text{(Additional positivity condition in the coexistence case)}
    \begin{equation}\label{e.917.-2}
\inf_{t\in \mathbb{R},~ y \le -A'_0} v(t,y) >0\quad \text{ for some }A'_0>0.
\end{equation}
\end{itemize}
Then there exists $\sigma\in\mathbb{R}$ such that
\begin{equation}
(u(t,y),v(t,y))=(U_c(y-\sigma),V_c(y-\sigma))
\qquad \text{ for }~(t,y)\in\mathbb{R}^2,
\end{equation}
where $(U_c,V_c)$ is a monotone front connecting the stable state $(u^\star, v^\star)$ to $(0,1)$.
\end{theorem}
In particular, \Cref{thm:liouville} shows that the existence of an entire solution satisfying its hypotheses forces $c\ge c_{\rm LLW}^0$. Moreover, the decay exponent $\lambda$ in \eqref{e.917.2} is necessarily the characteristic exponent of the corresponding travelling wave, namely
\begin{equation}\label{e.921.5}
\lambda= 
\begin{cases}\frac12 ({c - \sqrt{c^2-4(1-a)}}), &\text{when } c > c^0_{\rm LLW},\\[5pt]
\frac12 ({c + \sqrt{c^2-4(1-a)}})&\text{when } c = c^0_{\rm LLW},
\end{cases}
\end{equation}

The invasion dynamics of \eqref{e.com} have been extensively studied. For compact perturbations of  $(0,1)$ with $u_0\not\equiv 0$, Lewis, Li and Weinberger \cite{lewis2002spreading,li2005spreading} identified the minimal wave speed  ($c^0_{\mathrm{LLW}} \in [2\sqrt{1-a}, 2]$) as the spreading speed of the Cauchy problem:
\begin{align}
&\lim_{t\to\infty}\sup_{|x|<ct} \left(|u(t,x)-u^{\star}|+|v(t,x) - v^{\star}|\right)=0\quad \mbox{for each $0<c<c^0_{\mathrm{LLW}}$,}\\
&\lim_{t\to\infty}\sup_{|x|>ct} \left(|u(t,x)|+|v(t,x)-1|\right)=0\quad \mbox{for each $c>c^0_{\mathrm{LLW}}$,}
\end{align}
where $(u^{\star},v^{\star})$ is the stable state given in \eqref{e.uv.uv}.

For compactly supported populations, or for initial data satisfying 
\begin{equation}\label{ic:GL2019}
\begin{cases}
0\le u_0,v_0\le1,\qquad
\quad v_0\not\equiv0,
\qquad \operatorname{supp}v_0\text{ is compact},\\[5pt]
 \liminf\limits_{x\to-\infty}u_0(x)>0,\quad  u_0(x)=0\quad \text { for all large } x,
\end{cases}
\end{equation}
a second invasion interface may arise as the slower species $u$  advances through the region already colonised by $v$. This co-invasion scenario was proposed by Shigesada and Kawasaki \cite{shigesada1997biological} in their study of postglacial tree migrations  \cite{davis1981quaternary}. The corresponding spreading speeds were determined for strong-weak competition by Girardin and Lam \cite{girardin2019invasion} using coupled barriers and for weak competition by Liu et al. \cite{liu2019asymptotic} using Hamilton-Jacobi methods. Both settings allow nonlocal pulling, whereby the fast front attracts the trailing front.

In the strong--weak case $0<a<1<b$, 
 when $rd<1$, the slower $v$-species becomes extinct and the $u$-species spreads at speed $2$, 
 under the initial-data assumptions of \cite{girardin2019invasion}.

Under the corresponding assumptions on the initial data, when $rd>1$, the species $v$  first colonises the empty habitat with speed $c_v^\star=2\sqrt{rd}$, while the species $u$  invades the region occupied by $v$ at a slower speed $c_u^\star$. More precisely, the following spreading behaviour holds:

\begin{equation}
\begin{cases}
\lim\limits_{t\to\infty}\sup\limits_{0\leq x<ct} \left(|u(t,x)-u^{\star}|+|v(t,x)-v^{\star}|\right)=0,\quad &\mbox{for any $c<c_{u}^{\star}$,}\\[5pt]
\lim\limits_{t\to\infty}\sup\limits_{c_1 t<x< c_2t} \left(|u(t,x)|+|v(t,x)-1|\right)=0,\quad &\mbox{for any $c_{u}^{\star}<c_1<c_2<2\sqrt{rd}$,}\\[5pt]
\lim\limits_{t\to\infty}\sup\limits_{x>ct} \left(|u(t,x)|+|v(t,x)|\right)=0,\quad &\mbox{for any $c>2\sqrt{rd}$}.    
\end{cases}
\end{equation}
Here $(u^\star,v^\star)$ is the stable equilibrium defined in \eqref{e.uv.uv}: it is $(1,0)$ in the strong--weak case and the positive coexistence equilibrium in the weak competition case. The selected speed $c_u^\star$ is given by
\begin{equation}\label{c-u-star}
    c^{\star}_u:= \begin{cases}
    c^0_{\mathrm{LLW}} &\text{ if }\sqrt{rd} \geq \sqrt{1-a} + \sqrt a,\\[5pt]
    \max\left\{c^0_{\mathrm{LLW}}, c_{\rm GL}\right\}  &\text{ if }1<\sqrt{rd} < \sqrt{1-a}+ \sqrt{a},
    \end{cases}
\end{equation}
where
\begin{equation}
c_{\mathrm{GL}} := \sqrt{rd}-\sqrt a + \frac{1-a}{ \sqrt{rd}-\sqrt a},  
\end{equation}
which corresponds to \eqref{eq:cstar}  with $\s = 2\sqrt{rd}$.
In particular, under the condition 
\begin{equation}\label{eq:NP}
   1 < \sqrt{rd} < \sqrt{1-a} + \sqrt{a} \quad \text{ and } \quad c_{\mathrm{GL}} > c^0_{\rm LLW}, 
\end{equation}
we have \begin{equation}
c^\star_v = 2\sqrt{dr} >2 \qquad  c^\star_u = c_{\mathrm{GL}} \in (2\sqrt{1-a}, c^\star_v),\end{equation} 

Under \eqref{eq:NP}, the trailing front moves faster than the minimal wave in the homogeneous habitat occupied by $v$, despite remaining behind the leading front of $v$. This enhancement, known as nonlocal pulling, was identified for a decoupled system in \cite{holzer2014accelerated} and subsequently for the coupled model in \cite{girardin2019invasion}; see also the survey \cite{lam2026survey}.

We first consider the case where $u$ is the invasive species and $v$ is the resident species. In this setting, assume $u_0$ is front-like with a polynomially weighted exponential tail, in the sense that 
\begin{align}\label{ic-lv-u1}
0\leq u_0\leq 1 \text{ in } \mathbb{R},\quad 
\liminf_{x\to-\infty}u_0(x)>0, \quad u_0(x) \asymp x^{q} e^{-\lambda x} \text{ for all } x\geq x_0. 
\end{align}
for some $x_0>1$ and some $q\in\mathbb{R}$,
while the density $v_0$ of the resident species satisfies 
\begin{align}\label{ic-lv-v}
k_0\leq v_0\leq 1 \quad\text{ in } \mathbb{R}, \text{ for some } k_0>0.
\end{align}
For each $\lambda\in(0,\sqrt{1-a}]$, we associate the speed
\begin{equation}\label{c-lambda}
c_\lambda:=\lambda+\frac{1-a}{\lambda}.
\end{equation}
We further define
\begin{equation}\label{lambda-LLW}
\lambda_{\mathrm{LLW}}:=
\frac{c^0_{\mathrm{LLW}}-\sqrt{(c^0_{\mathrm{LLW}})^2-4(1-a)}
}{2}\in(0,\sqrt{1-a}],
\end{equation}
which is the smaller positive root of
\begin{equation}
\lambda^2-c^0_{\mathrm{LLW}}\lambda+1-a=0.
\end{equation}
Since $\lambda\mapsto c_\lambda$ is strictly decreasing on $(0,\sqrt{1-a}]$,
\begin{equation}\label{equiv-c-lambda-LLW}
c_\lambda>c^0_{\mathrm{LLW}}\quad\Longleftrightarrow\quad
0<\lambda<\lambda_{\mathrm{LLW}}.
\end{equation}
For $\kappa >0$, define the rightmost $\kappa$-level position
\begin{align}
 \xi_{\kappa}[u](t):=\sup\left\{x\in\mathbb R:\ u(t,x)\geq\kappa\right\}.
\end{align}

We now state the corresponding convergence result for these algebraically weighted exponential tails.

\begin{theorem}\label{thm:syst1}
Let $0<a<1$ and $b\in (0,\infty)\setminus \left\{1\right\}$ be given and let  $(u,v)$ be the solution of \eqref{e.com} with initial data satisfying
\eqref{ic-lv-u1} for some $\lambda >0$, $x_0>1$ and $q\in\mathbb{R}$, and \eqref{ic-lv-v}. Suppose   \begin{equation}
\lambda \in (0,\sqrt{1-a}), \qquad c_\lambda>c^0_{\mathrm{LLW}}.\end{equation} Then 
for any $\delta\in(0,u^{\star})$, the level set of $u$ satisfies
\begin{equation}
\xi_\delta[u](t)=c_\lambda t+\frac{q}{\lambda}\ln t+O(1)
\quad\text{as }t\to\infty.
\end{equation}
Moreover, there exists a bounded function
$\Lambda:\mathbb R_+\to\mathbb R$ such that
\begin{align}
\lim_{t\to\infty}\sup_{x\in\mathbb{R}}
&\Bigg\{\left|u(t,x)-U_{c_\lambda}\left(
x-c_{\lambda} t-\frac{q}{\lambda}\ln t-\Lambda(t)
\right)\right\vert \nonumber\\
&+
\left|v(t,x)-V_{c_\lambda}\left(x-c_\lambda t-\frac{q}{\lambda}\ln t-\Lambda(t)\right)\right\vert\Bigg\}=0,\label{sys1-TW-conv}
\end{align}
where $(U_{c_\lambda},V_{c_\lambda})$ is the travelling wave solution to \eqref{TW-sys} with $c=c_\lambda$, $E^-=(u^{\star},v^{\star})$ and $E^+=(0,1)$, which is unique up to translation.
\end{theorem}

\Cref{thm:syst1} makes explicit how the initial leading edge determines the refined propagation law: the decay rate $\lambda$ selects the speed $c_\lambda$, whereas the algebraic factor $x^q$ induces the phase correction $\frac{q}{\lambda}\ln t$. This gives a delay for $q<0$, no logarithmic displacement for $q=0$, and an advance for $q>0$, provided $c_\lambda > c^0_{\rm LLW}$. Unlike spectral convergence results for sufficiently localised perturbations of a fixed wave 
\cite[Chapter~5]{Volpert1994travelling}, our maximum-principle approach accommodates these nonstationary shifts; compare also \cite[Theorem 1.1]{wu2023sharp} for the case without an algebraic phase correction.

We next treat nonlocal pulling when both species are initially localised, building on the leading-order barrier constructions of \cite{girardin2019invasion}. Define
\begin{equation}
c^{\star}_v:=2\sqrt{rd},\quad \lambda_v:=\sqrt{\frac{r}{d}},\quad \eta_{v}:=\frac{3}{2\lambda_v}=\frac{3}{2}\sqrt{\frac{d}{r}},
\end{equation}
\begin{equation}
\lambda_{\mathrm{GL}}:=\frac{c^{\star}_v}{2}-\sqrt{a}=
\sqrt{dr}-\sqrt{a}, \quad
c_{\mathrm{GL}}:=\lambda_{\mathrm{GL}}
+\frac{1-a}{\lambda_{\mathrm{GL}}}.
\end{equation}

\begin{theorem}\label{thm:syst2}
Assume $0<a<1$, $b\in (0,\infty)\setminus\left\{ 1\right\}$ and the nonlocal-pulling regime \eqref{eq:NP}:
\begin{equation}
1<\sqrt{rd}<\sqrt{1-a}  + \sqrt a, \qquad c_{\mathrm{GL}}>c^0_{\mathrm{LLW}}.\end{equation}
Let $(u,v)$ be the solution of \eqref{e.com} with initial data 
satisfying \eqref{ic:GL2019}\footnote{The same conclusion holds also when $u_0$, $v_0$ are nonnegative, nontrivial, and compactly supported.}.
Then for any $\kappa\in(0,u^\star)$, the rightmost $\kappa$-level  position satisfies
\begin{equation}
\xi_\kappa[u](t)=c_{\mathrm{GL}}t+\frac{q_u}{\lambda_{\mathrm{GL}}}\ln t+O(1)
\quad\text{as }t\to\infty,
\end{equation}
where $(u^\star,v^\star)$ is the stable state given in \eqref{e.uv.uv} and 
\begin{equation}
q_u:=-\frac32+\sqrt a\,\eta_{v}
=\frac32\left(\sqrt{\frac{ad}{r}}-1
\right).
\end{equation}
Equivalently,
\begin{equation}
\xi_\kappa[u](t)=c_{\mathrm{GL}}t+\frac{
\frac32\left(\sqrt{\frac{ad}{r}}-1\right)
}{\sqrt{dr}-\sqrt a
}\ln t+O(1)
\quad\text{as }t\to\infty.
\end{equation}
Moreover, there exists a bounded function
$\Lambda:\mathbb R_+\to\mathbb R$ such that for each $\eta'>0$,
\begin{align}\label{sys2-TW-conv}
\lim_{t\to\infty}\sup_{0\leq x<c^{\star}_v t-(\eta_{v} + \eta')\ln(1+t)}
\Bigg\{&\left|u(t,x)-U_{c_{\mathrm{GL}}}\left(
x-c_{\mathrm{GL}} t-\frac{q_u}{\lambda_{\mathrm{GL}}}\ln t-\Lambda(t)
\right)\right\vert
\nonumber\\
&+
\left|v(t,x)-V_{c_{\mathrm{GL}}}\left(x-c_{\mathrm{GL}} t-\frac{q_u}{\lambda_{\mathrm{GL}}}\ln t-\Lambda(t)\right)\right\vert\Bigg\}=0,
\end{align}
where $(U_{c_{\mathrm{GL}}},V_{c_{\mathrm{GL}}})$ satisfies \eqref{TW-sys} with $E^-=(u^\star, v^\star)$ and $E^+=(0,1)$.
\end{theorem}
We finish by highlighting the main technical inputs and novelties of the proofs. To extend the piecewise-constant results of \cite{LamWu_preprint} to general profiles $\g$, it is essential to retain the heterogeneity in the linearised problem. Under the weighted integrability hypothesis (G2), we construct spatial correctors by analysing the stationary zero-energy spectral problems of certain one-dimensional Schrödinger operators on $\mathbb{R}$ or on half-lines.
In contrast to the homogeneous Dirichlet case, where the corrector is given by $\underline m(y)=y$,
the correctors for the inhomogeneous case are asymptotically linear, convex functions whose
convergence towards their asymptotes need not have a prescribed rate. Our Gaussian barriers incorporate these correctors directly, carrying a defect
\begin{equation}
    y\underline m'(y)-\underline m(y)=\int_0^y z \, \underline m''(z)\, \dd z,
\end{equation}
which need not be bounded. Coupled adjustments to the Gaussian centre and diffusion time absorb this error, while weighted integrability controls the remaining amplitude correction. This preserves the characteristic Dirichlet factor $t^{-\tfrac32}$ without imposing higher moment conditions on $\g$ at infinity. For the nonlinear lower barrier, a further correction is needed: at $\s =2$, the linear profile near the interface is of order $(1+t)^{\eta - \tfrac32}$, and a purely time-dependent correction can create a non-integrable loss when $\eta \geq \tfrac12$. To overcome this difficulty, identified in \cite{hamel2013short}, we instead exploit spatial localisation by taking a corrector satisfying
$$m''=(1-\g+h)m,$$
where the positive, rapidly decaying function $h$ provides an additional margin to take into account the KPP nonlinearity.

At the forced-wave threshold $(\s,\eta)=(2,\tfrac32)$, the remaining question is which stationary profile is selected. We identify the front through the limit of a weighted moment, chosen  to cancel the heterogeneous linear term. Following the approach of Nolen et al. \cite{nolen2017convergence}, we derive estimates at scale $\sqrt t$ and transfer them to bounded distances from the interface. We demonstrate that the uniform moment bounds in the original variables $(t,x) = (t,y + s(t))$, together with a compactness argument, imply convergence in the self-similar variables $(\tau,z) = (\ln(1+t), \tfrac{y}{\sqrt{1+t}})$ to a multiple of the principal eigenfunction of the associated Dirichlet Fokker-Planck operator. 

For the competition system, a logarithmic recentering introduces a velocity error of order $(1+t)^{-1}$ in the travelling-wave equations. The perturbation corrector constructed in \Cref{sec:supersolsystem}, following \cite{girardin2019invasion,LamWu_preprint}, compensates for this error while preserving the leading exponential tail. Together with the scalar estimates, this corrector yields front localisation up to a bounded error in \Cref{sec:proofsyst}. In the nonlocally pulled regime, the resulting exponent $q_u = -\tfrac32 + \sqrt a \eta_v$, predicted heuristically in \cite{LamWu_preprint}, captures both leading-edge diffusion and the Bramson delay of the faster competitor, refining the earlier spreading-speed results. The argument also applies to algebraically weighted exponential initial tails. 

Finally, the rigidity theorem in \Cref{sec:liouville}  converts these front estimates into profile convergence. The sliding argument uses quantitative decay only for the invading component at $+\infty$, supplemented by persistence in the other spatial regions, rather than bounds by translates for both components \cite[Lemma 2.10]{Wang2019uniqueness}; see also \cite{wu2023sharp}. It yields uniqueness in the two monostable competition regimes and removes the additional ordering assumption on decay exponents in \cite[Proposition A.7]{girardin2019invasion}. \Cref{app:rigid} identifies the corresponding stability and irreducibility hypotheses for more general monotone parabolic systems. 

The remainder of the paper is organised as follows. Sections~\ref{sec:corrector}--\ref{sec:proofshift} develop the scalar correctors, barriers and proofs of \Cref{t.0,thm:standard,thm:pulling,thm:cpulling,thm:no-pull}; \Cref{sec:selection} establishes the forced-wave selection theorem. The competition-system constructions and proofs appear in \Cref{sec:supersolsystem,sec:proofsyst}. This is followed by the self-contained rigidity and uniqueness results in \Cref{sec:liouville}. The appendices contain an auxiliary system supersolution (\Cref{app:auxiliary}) whose construction is more standard, background results from \cite{LamWu_preprint} and other sources (\Cref{app:previouspaper}), and an extension of rigidity to general monotone parabolic systems (\Cref{app:rigid}).

\section{Interlude: existence and properties of correctors}\label{sec:corrector}
In this section, we construct the corrector functions $\underline{m}$, $\mathfrak{m}$ and $\overline{m}$, and establish properties needed for the construction of sub- and supersolutions. 
Although related constructions are standard in the theory of one-dimensional Schr\"{o}dinger operators, we provide the proofs for completeness.

\begin{proposition}\label{prop:subcorrector}
Let ${\sigma} \in \mathsf{L}^1(\mathbb{R}^+,y\,\dd y)$. Then the initial value problem 
\begin{align}\label{m-eq}
    \underline{m}''(y)={\sigma(y)}\underline{m}(y),\quad y>0,\quad \underline{m}(0)=0,\quad \underline{m}'(0)=1,
\end{align}
has a unique solution $\underline{m}\in \mathscr{C}^1(\mathbb{R}^+)\cap \mathsf{W}_{\rm loc}^{2,1}(\mathbb{R}^+)$.
If $\sigma> 0$ almost everywhere on $(0,\infty)$, then $\underline{m}$ is positive and strictly convex. Moreover, $\ell:= \underline{m}'(+\infty)\in (1,\infty)$ exists and satisfies
\begin{equation}\label{m-behavior-a}
\underline{m}(y) \sim \ell y, \qquad  y \leq \underline{m}(y)\leq \ell y \quad \text{for all } y \geq 0.
\end{equation}
If ${\sigma}$ is continuous, then $\underline{m}\in \mathscr{C}^2(\mathbb{R}^+)$. 
\end{proposition}

\begin{proof}[{\bf Proof of \Cref{prop:subcorrector}}]
Integrating \eqref{m-eq} twice, we have
\begin{align}\label{m-inte}
 \underline{m}'(y)=1+\int_0^y {\sigma}(\xi)\underline{m}(\xi)\, \dd \xi,\quad  \underline{m}(y)=y+\int_0^y(y-\xi){\sigma}(\xi)\underline{m}(\xi)\, \dd \xi.    
\end{align}
Hence, the existence and uniqueness of $\underline{m}$ follow by Picard iteration applied to the second Volterra integral equation. 

We next establish a global linear-growth bound. The second equality in \eqref{m-inte} implies that, for $y>0$,
\begin{equation}
\frac{|\underline{m}(y)|}{y}\leq 1+\frac{1}{y}\int_0^y(y-\xi)|{\sigma}(\xi)|\, |\underline{m}(\xi)|\, \dd \xi
\leq 1+\int_0^{y}\xi|{\sigma}(\xi)|\frac{|\underline{m}(\xi)|}{\xi} \, \dd \xi,
\end{equation}
where we used $\frac{y-\xi}{y}\leq 1$ (here $\xi\in[0,y]$).
Therefore, since $\left( \xi \mapsto \xi {\sigma}(\xi) \right) \in \mathsf{L}^1(\mathbb{R}^+)$,  Gr\"{o}nwall's inequality yields  
\begin{align}\label{m/1+y-upp}
   \frac{|\underline{m}(y)|}{y}\leq \exp\left(\int_0^{\infty}\xi|{\sigma}(\xi)|\, \dd \xi\right) <+\infty,\quad \text{ for } y> 0.
\end{align}
From now on, assume in addition that ${\sigma} >0$ for a.e. $y\geq 0$. 
It follows from \eqref{m-inte} that $\underline{m}(y)\geq y \geq 0$ and thus $\underline{m}''(y) > 0$
for a.e. $y\geq 0$ and thus $\underline{m}'$ is strictly increasing. From the bounds \eqref{m-inte} and \eqref{m/1+y-upp}, we deduce that $\ell:= \lim_{\xi \to \infty}\underline{m}'(\xi) \in (1,\infty)$ exists and $y \leq \underline{m}(y)\leq \ell y$ for all $y\geq0$, which is \eqref{m-behavior-a}.
\end{proof}

\begin{proposition}\label{lem:m-correction0930} Assume {\rm(G1)--(G2)}.
There exists a unique solution $\mathfrak{m}$ of
\begin{equation}\label{eq:h}
\mathfrak{m}''=(1-\g)\mathfrak{m},\qquad \mathfrak{m}(-\infty)=0,\qquad \mathfrak{m}'(+\infty)=1. 
\end{equation}
Moreover, $\mathfrak{m}$ is positive and strictly increasing, and there exists $c_1>0$ such that
\begin{equation}\label{m-asymp0930}
  \begin{cases}
  \mathfrak{m}(y)\sim c_1e^{\sqrt{a}y},\quad 
  \mathfrak{m}'(y)\sim c_1\sqrt{a}e^{\sqrt{a}  y}& \quad \text{ as } y\to-\infty,\\
  \mathfrak{m}(y)\sim y,\quad 
  \mathfrak{m}'(y)\sim 1& \quad \text{ as } y\to+\infty, 
  \end{cases}
\end{equation}
and for some $c_2,c_3>0$,
\begin{equation}\label{m-bounds0930}
  \frac{\mathfrak{m}'(y)}{\mathfrak{m}(y)}
       \ge\frac{c_2}{1+y_+},\quad
       0<\mathfrak{m}'(y)\le1\quad \text{ for } y\in\mathbb{R}; \quad 
  \mathfrak{m}(y)\le c_3(1+y)\quad\text{ for } y\geq0, 
\end{equation}
where $y_+=\max\{y,0\}$.
\end{proposition}

\begin{proof}[{\bf Proof of \Cref{lem:m-correction0930}}]
The proof of Lemma~\ref{lem:m-correction0930} is similar to that of Proposition~\ref{prop:subcorrector}.
 First, using $q=1-a-\g \in \mathsf{L}^1((-\infty,0))$, we construct the solution $\mathfrak{m}$ with the Volterra integral equation
    \begin{equation}
\mathfrak{m}(y)=e^{\sqrt a y}
+\frac1{\sqrt a}\int_{-\infty}^y
\sinh\left(\sqrt a(y-z)\right)\,
q(z)\mathfrak{m}(z)\,\dd z. 
\end{equation}
The resulting solution $\mathfrak{m}$ has the following asymptotics before normalization:
\begin{equation}
\mathfrak m(y)\sim e^{\sqrt a y}, \qquad \mathfrak{m}'(y) \sim \sqrt ae^{\sqrt a y},
\qquad y\to-\infty.
\end{equation}
Convexity, which follows from \eqref{eq:h} and the nonnegativity of $\sigma=1-\g$, then extends the positivity of $\mathfrak{m}$ and $\mathfrak{m}'$ to all of $\mathbb{R}$. As in the proof of \Cref{prop:subcorrector}, integrating $\mathfrak{m}''=\sigma \mathfrak m$ twice gives
\begin{align}\label{m-inte'}
 \mathfrak{m}'(y)= \mathfrak{m}'(0)+\int_0^y {\sigma}(\xi)\mathfrak{m}(\xi)\, \dd \xi,\qquad  \mathfrak{m}(y)=\mathfrak{m}(0) + \mathfrak{m}'(0)y+\int_0^y(y-\xi){\sigma}(\xi)\mathfrak{m}(\xi)\, \dd \xi.   
\end{align}
The remaining properties of $\mathfrak{m}$  follow, as before, from $\sigma \in \mathsf{L}^1((0,\infty);y \dd y)$ as before. Finally, normalise by replacing $\mathfrak m(y)$ by $\frac{1}{\mathfrak m'(+\infty)} \mathfrak{m}(y)$.
\end{proof}

\begin{proposition}\label{prop:supercorrector}
Assume that ${\sigma} \in \mathsf{L}^1(\mathbb{R})$ and that there exists $\theta>0$ such that
\begin{equation}
-\theta^2 < {\sigma} \leq 0 \quad \text{ on } \, \mathbb{R}^-, \quad \text{ and }\quad  {\sigma} \geq 0  \quad \text{ on } \, \mathbb{R}_\star^+.
\end{equation}
Then, for any $\mu \in [\theta, \infty)$, the problem
\begin{equation}\label{e.0817.1}
    \begin{cases}
        \overline{m}'' +2\mu \overline{m}' = {\sigma}\overline{m} \quad \text{ on } \mathbb{R},\\[5pt]      
\displaystyle\lim_{-\infty}  \overline{m}=1, \quad \lim_{-\infty}  \overline{m}' = 0.
    \end{cases}
\end{equation}
has a unique solution $\overline{m}_{\mu}\in\mathscr{C}^1(\mathbb{R})\cap \mathsf{W}_{\mathrm{loc}}^{2,1}(\mathbb{R})$. Furthermore, $\overline{m}_\mu>0$ on $\mathbb{R}$, and the following
properties hold: 
\begin{enumerate}
\item For each $M>0$ and all $\mu \in [\theta,\theta + M]$,
\begin{equation}\label{e.0819.1} 
\overline{m}_{\mu}(y) \asymp_M 1 \quad \text{ for  } y\to \pm\infty,
\end{equation}
\item  For any $\mu  \in [\theta, \infty)$, $\overline{m}_{\mu}' \leq 0$ for $y \leq 0$.
\item If, in addition, ${\sigma}= 0$ a.e. on $\mathbb{R}_+$, then $\overline{m}'_\mu \leq 0$
for $y \in \mathbb{R}$, and
\begin{align}\label{id-m(3)}
 \overline{m}_{\mu} (y) = - \frac{\overline{m}_{\mu}'(0)}{2\mu} e^{-2\mu y} 
+ \inf_{\mathbb{R}} \overline{m}_{\mu}\quad \text{ for }y >0.   
\end{align}
If, furthermore, ${\sigma}<0$ a.e. on $\R_-^\star$, then $\overline{m}_\mu'<0$ on $\mathbb{R}$.
\end{enumerate}
\end{proposition}
\begin{proof}[{\bf Proof of \Cref{prop:supercorrector}}]
For notational simplicity, we write $\overline{m} = \overline{m}_{\mu}$.

Existence and uniqueness follow once again by Picard iteration for a Volterra equation. To see this, rewrite the first equation in \eqref{e.0817.1} in divergence form, 
\begin{equation}
(e^{2\mu y}\overline{m}')'(y) = e^{2\mu  y} {\sigma(y)}\overline{m}(y), \qquad y \in \R.
\end{equation}
By 
integrating from $-\infty$ and using the initial data and $\mu>0$, we obtain
\begin{equation}\label{e.m.2}
\overline{m}'(y) = \int_{-\infty}^y e^{-2\mu (y-s)} {\sigma}(s)\overline{m}(s)\,\dd s.
\end{equation}
Integrating once more, exchanging the order of integration, and simplifying, we obtain
\begin{equation}\label{e.m.3}
\overline{m}(y) = 1+ \int_{-\infty}^y 
 H_\mu(y-s)
{\sigma}(s)\overline{m}(s)\,\dd s,
\end{equation}
where
\begin{equation}
H_\mu(\zeta):=\frac{1}{2\mu }( 1- e^{-2\mu \zeta})\leq \frac{1}{2\mu},\quad  \zeta\geq 0.
\end{equation}

This is a standard Volterra integral equation. Hence, for any $\mu>0$ and any integrable ${\sigma}$, the existence and uniqueness of $\overline{m}$  follow by Picard iteration. 

The upper bound 
\begin{align}\label{bound-m1}
 \|\overline{m}\|_{\mathsf{L}^\infty(\mathbb{R})} \leq \exp\left( \frac1{2\mu }\|{\sigma}\|_{\mathsf{L}^1(\mathbb{R})}\right)   
\end{align}
follows from applying Gr\"{o}nwall's inequality to the following consequence of \eqref{e.m.3}:
\begin{equation}
|\overline{m}(y)| \leq  1+ \int_{-\infty}^y \frac{1}{2\mu } |{\sigma}(s)|\, |\overline{m}(s)|\,\dd s.
\end{equation}
This bound also implies that
$\lim_{+\infty} \overline{m}' = 0$. Indeed, observe that 
$$|\overline{m}'(y)| \leq \|\overline{m}\|_\infty \int_{-\infty}^y e^{-2\mu (y-s)} |{\sigma}(s)|\,\dd s,$$
and the conclusion follows from the standard decay properties of convolutions, since $\sigma\in\mathsf{L}^1(\mathbb{R})$. 

We next show that, for $\mu\geq\theta$, one has $\overline{m}>0$ on $\mathbb{R}$ and $\overline{m}'\leq0$ on $(-\infty,0]$. Set \begin{equation}\phi(y):=e^{\mu y}\overline{m}(y).\end{equation}

Since $\overline{m}\sim 1$ is bounded as $y \to -\infty$, $\phi \sim e^{\mu y}$ as $y \to -\infty$. Also, using $e^{-\mu y}
\phi'(y)
=\mu\overline{m}(y)+\overline{m}'(y)\to\mu>0$ as $y\to-\infty$,
we have $\phi>0$ and $\phi'>0$ for $y \ll -1$. If $\phi$ has its leftmost zero at $y_0$, then $\phi>0$ in $(-\infty, y_0)$. From $\phi''=(\mu^2 + {\sigma})\phi$, we see that 
$\phi$ is strictly increasing and cannot vanish at $y_0$, a contradiction. Consequently,
\begin{equation}
\phi>0,\quad \phi'>0 \text{ in } \mathbb{R}.
\end{equation}
In particular, $\overline{m} >0$ on $\mathbb{R}$.
It then follows from \eqref{e.m.2} and ${\sigma}\leq 0$ that $\overline{m}'\leq 0$ in $(-\infty,0]$. 

We now show that $\inf_{\mathbb{R}} \overline{m} >0$ for each $\mu \in [\theta,\infty)$. Since $\overline{m}' \leq 0$ in $(-\infty,0]$, it remains to show that $\overline{m}$ is bounded below by a positive constant for $y \geq 0$.  Define
\begin{equation}
w(y) := \frac{1}{2}\left(\phi(0) + \frac{\phi'(0)}{\mu}\right) e^{\mu y} + \frac{1}{2}\left(\phi(0) - \frac{\phi'(0)}{\mu}\right) e^{-\mu y}. 
\end{equation}
Then $(w(0),w'(0)) = (\phi(0),\phi'(0))$ and
\begin{equation}
-\phi'' + \mu^2 \phi = - {\sigma}\phi \leq 0 = - w'' + \mu^2 w \quad\text{ on } \R^+.
\end{equation}
By the ODE comparison principle, we have $\phi\geq w$ on $\R_+^\star$, and hence
\begin{equation}
\overline{m}(y) \geq \frac{1}{2}\left(\phi(0) + \frac{\phi'(0)}{\mu}\right)  + \frac{1}{2}\left(\phi(0) - \frac{\phi'(0)}{\mu}\right)  e^{-2\mu y} \geq \min\left\{\frac{1}{2}\left(\phi(0) + \frac{\phi'(0)}{\mu}\right), \phi(0)\right\} >0,
\end{equation}
for $y \geq 0$ where we used $\phi(0), \phi'(0)>0$ proved earlier.

We have proved, for each $\mu \geq \theta$, the boundedness of $\overline{m}$ from above and below by positive constants, which possibly depend on $\mu$. Moreover, it follows from compactness (and uniqueness of $\overline{m}$) that such bounds are uniform on $[\theta,\theta+ M]$, for each $M>0$. This proves assertions (1) and (2).

Finally, we prove part (3). Suppose that ${\sigma}=0$ almost everywhere on $\mathbb{R}_+$. Then 
\begin{align}\label{Id-m1}
    \overline{m}'(y)=\overline{m}'(0)e^{-2\mu y} \quad \text{ for } y\geq 0.
\end{align}
Since $\overline{m}'(0)\leq 0$, we have $\overline{m}'\leq0$ on $\mathbb{R}$.
In addition, integrating \eqref{Id-m1} over $[0,y]$ gives
\begin{align}
\overline{m}(y)=\overline{m}(0)+\frac{\overline{m}'(0)}{2\mu}\left(1-e^{-2\mu y}\right)=-\frac{\overline{m}'(0)}{2\mu}e^{-2\mu y}+\lim_{z\to+\infty}\overline{m}(z).
\end{align}
Since $\overline{m}$ is nonincreasing, we immediately obtain \eqref{id-m(3)}. 

If, furthermore, ${\sigma}<0$ for almost every $y<0$, then \eqref{e.m.2} gives $\overline{m}'(y)<0$ for $y\leq0$, in particular $\overline{m}'(0)<0$.  The above discussion implies $\overline{m}'<0$ on $\mathbb{R}$. 
Therefore, assertion~3 holds and the proof of \Cref{prop:supercorrector} is complete.
\end{proof}

\section{Construction of supersolutions}\label{sec:supersolshift}

In this section, we construct a supersolution to \eqref{eq:mainshift} for $\s \geq 2$. Start from \eqref{eq:mainshift}. 
In the shifted variable $y=x-s(t)$, equation \eqref{eq:mainshift} becomes
\begin{equation}
\zeta_t - \dot s \zeta_y - \zeta_{yy} = (\g - \zeta) \zeta \quad \text{ for }t>0,~y \in \mathbb{R},
\end{equation}
Following the standard approach (see \cite{hamel2013short,hamel2016logarithmic}), 
we factor out the leading exponential decay and the algebraic contribution of the logarithmic shift by setting
\begin{equation}
\zeta(t,y) = (1 +t)^{\frac{\s \eta}{2}}e^{-\left[\frac{\s^2}{4}-1\right]t} e^{-\frac{\s}{2}y}\Psi(t,y).
\end{equation}
We obtain the following equation for $\Psi$:
\begin{align}
\Psi_t + \left( \s - \dot s \right) \Psi_y - \Psi_{yy}
&= \left(\g -1 - \zeta \right) \Psi\qquad \text{ for }t>0,~y \in \mathbb{R}.
\end{align}
Under (G1), the right-hand side is nonpositive.

\begin{proposition}\label{prop:supersolutionshiftallbeta}
Let $\s\geq2$ and $\eta\in\mathbb{R}$.  Then the function
\begin{equation}
\overline{\Psi}(t,y):= 
   \frac{y}{\tau(t)^{\frac32}} \exp \left( - \frac{(y + 2|\eta|)^2}{4\tau(t)} -\frac{4(\eta^2 +2)}{\tau(t)^{\frac12}}\right) \text{ for }t>0,~ y >0,
\end{equation}
where $\tau(t) = 1+ t + \sqrt{1+t}$,
satisfies the differential inequality
\begin{equation}
\overline{\Psi}_t - \overline{\Psi}_{yy} + \frac{\eta}{1+t} \overline{\Psi}_y {+(1-\g)\overline\Psi}\geq 0 \quad \text{ for } t >0,~ y >0.
\end{equation}
Consequently,
\begin{equation}\label{e.102.zeta}
\overline\zeta(t,y) := (1 +t)^{\frac{\s \eta}{2}}e^{-\left[\frac{\s^2}{4}-1\right]t} e^{-\frac{\s}{2}y}\overline\Psi(t,y)
\end{equation}
satisfies 
\begin{equation}
\overline\zeta_t - \dot s \overline\zeta_y - \overline\zeta_{yy} \geq (\g - \overline\zeta) \overline\zeta \quad \text{ for }t>0,~y >0.
\end{equation}
\end{proposition}

\begin{proof}[{\bf Proof of \Cref{prop:supersolutionshiftallbeta}}]
Setting $z=y+2|\eta|$, direct differentiation gives 
\begin{equation}
\frac{\overline\Psi_t}{\overline\Psi} = -\frac{3\tau'}{2\tau} + \frac{z^2}{4\tau^2}\tau' + 2(\eta^2 +2) \frac{\tau'}{\tau^{\frac{3}{2}}}, \qquad \frac{\overline\Psi_y}{\overline\Psi} = \frac1{y} - \frac{2z}{4\tau}, \qquad 
\frac{\overline\Psi_{yy}}{\overline\Psi} = -\frac1{2\tau} - \frac{z}{y\tau} + \frac{z^2}{4\tau^2}.
\end{equation}
Therefore, 
\begin{align}
&\frac{1}{\overline{\Psi}}\left[\overline{\Psi}_t-\overline{\Psi}_{yy}+\frac{\eta}{1+t}\overline{\Psi}_y\right]\notag \\
&= \frac{2(\eta^2 +2)\tau'}{\tau^{\frac32}}
 + (\tau' -1)\frac{z^2}{4 \tau^2} - \frac32 \frac{\tau'}{\tau} +
\left(
 \frac{z}{\tau}+\frac{\eta}{1+t}\right) \frac{1}{y}
+\frac{1}{2\tau}
-\frac{z}{2\tau}
 \frac{\eta}{1+t}.
\label{Psi-eq'}
\end{align}
For $t>0$, we have
\begin{equation}
1+t<\tau<2(1+t),
\qquad\quad 
 \tau'-1 =  \frac{1}{2\sqrt{1+t}},\qquad \quad 
\frac{z}{\tau}+\frac{\eta}{1+t}\geq \frac{y}{\tau}.
\end{equation}
Using also $1-\g\geq0$, we obtain
\begin{align}
&\frac{1}{\overline{\Psi}}\left[ \overline{\Psi}_t-\overline{\Psi}_{yy}+\frac{\eta}{1+t}\overline{\Psi}_y+(1- \g)\overline{\Psi}\right] \notag \\
&\geq
\frac{2(\eta^2 +2)\tau'}{\tau^{\frac32} } + \frac{1}{2\sqrt{1+t}} \cdot \frac{z^2}{4\tau^2} - \frac{3}{2\tau}\left(1+  \frac{1}{2 \sqrt{1+t}}\right) + \frac{1}{\tau} + \frac{1}{2\tau} - \frac{\eta z}{2\tau(1+t)}
 \notag \\
 &=\frac{2(\eta^2 +2)}{\tau^{\frac32} }\left( 1 + \frac{1}{2\sqrt{1+t}}\right) + \frac{1}{2\sqrt{1+t}} \cdot \frac{z^2}{4\tau^2} - \frac{3}{4\tau \sqrt{1+t}} - \frac{\eta z}{2\tau(1+t)}  \notag \\
  &\geq \frac{2(\eta^2 +2)}{\tau^{\frac32} }
 + \frac{1}{2\sqrt{1+t}}\left[\frac{z^2}{4\tau^2} - 2 \frac{z}{2\tau}\frac{\eta}{\sqrt{1+t}} - \frac{3}{2\tau}\right] 
 \notag\\
  &=\frac{2(\eta^2 +2)}{\tau^{\frac32} }
  + \frac{1}{2\sqrt{1+t}}\left[\left(\frac{z}{2\tau} - \frac{\eta}{\sqrt{1+t}}\right)^2 - \frac{\eta^2}{1+t}- \frac{3}{2\tau}\right] 
  \notag\\
  &\geq \frac1{(1+t)^{\frac32} }\left[\frac{\eta^2 +2}{2^{\frac12} } - \frac{\eta^2}{2} - \frac{3}{4}\right] \geq 0.
\end{align}
The asserted supersolution inequality for $\overline\zeta$ now follows from the exponential transformation. 
\end{proof}

\begin{proposition}\label{lem:2A}
Fix $\mu\in(0,\sqrt{a})$ such that
\begin{equation}
-\mu^2 < 1-a-\g(0) \leq 1-a-\g(y) \leq 0 ,\quad y\leq 0.
\end{equation}
Define
\begin{align}\label{sigma-0911}
{\sigma(y)}:=\begin{cases}
1-a-\g(y)\leq 0,&\quad y\leq0,\\
0,&\quad y>0. 
\end{cases}
\end{align}
Let $\overline m_\mu$ be the corrector from
\Cref{prop:supercorrector} associated with $\mu$ and ${\sigma}$.
Let $\psi$ be a bounded, positive supersolution of 
\begin{equation}
\psi_t \geq \psi_{xx} + \psi(1-a-\psi) \quad \text{ on } \R_+^\star \times \R,
\end{equation}
such that, for some $T>0$,
\begin{equation}\label{e.3heat}
\inf_{x \leq s(t)}\left( \frac{\psi_x(t,x)}{\psi(t,x)}  +\frac{\dot s(t)}{2} \right) > \mu \quad \text{ for }t \geq T,    
\end{equation}
Then there exists $k>1$ such that $\overline n(t,x):=k\overline m_\mu(x-s(t))\psi(t,x)$ satisfies
\begin{equation}
    \overline{n}_t - \overline{n}_{xx} \geq \overline{n}(\g(x-s(t)) - \overline{n} )\quad \text{ for }t>T,~ x <s(t).
\end{equation}
\end{proposition}

\begin{proof}[{\bf Proof of \Cref{lem:2A}}]
Set $y=x-s(t)<0$ and write
$m=\overline m_\mu$. Since
\begin{equation}
m''+2\mu m'=(1-a-\g)m,
\end{equation}
a direct calculation gives
\begin{align}
&\overline n_t-\overline n_{xx}
-\g\overline n+\overline n^2\\
&\quad =
km\left(\psi_t-\psi_{xx}-(1-a)\psi\right)
-km'(\dot s\,\psi+2\psi_x-2\mu\psi)
+k^2m^2\psi^2\\
&\quad \ge
-2k\psi m'
\left(\frac{\dot s}{2}+\frac{\psi_x}{\psi}-\mu\right)
+km(km-1)\psi^2,
\end{align}
where we used
$\psi_t-\psi_{xx}-(1-a)\psi\ge-\psi^2$.

By \Cref{prop:supercorrector},
\begin{equation}
m'(y)\le0,\qquad
m(y)\ge m(0)>0,\qquad y\le0.
\end{equation}
Choose $k>1$ such that $km(0)\ge1$.
Then $km(y)-1\ge0$ for every $y\le0$.
Moreover, \eqref{e.3heat} implies
\begin{equation}
\frac{\dot s(t)}2+
\frac{\psi_x(t,x)}{\psi(t,x)}-\mu>0
\qquad(t\ge T,\ x\le s(t)).
\end{equation}
Both terms in the preceding lower bound are
therefore nonnegative, as required.
\end{proof}

\begin{lemma}\label{lem:3A}
Let $\lambda\in(0,\sqrt{1-a})$. Define
\begin{align}\label{c-k-def}
c_{\lambda}  = \lambda + \frac{1-a}{\lambda},\qquad k_{\lambda}=\frac{\lambda}{c_{\lambda}-2\lambda}.
\end{align}
For $\delta>0$, set
\begin{align}\label{Phi0911}
\overline{\Phi}_\delta(t,\xi) = \left( 1 - \varepsilon(t) k_{\lambda}\xi \right) e^{-\lambda \xi},\qquad \text{ where }\quad 
\varepsilon(t) = \frac{5 \sqrt\delta}{1 + \delta t}.
\end{align}
Then there exists $\hat\delta>0$ such that for any $\delta \in (0,\hat\delta]$ and any function $\rho$ satisfying 
\begin{align}\label{rho-cond}
|\rho'(t)| \leq \frac{\varepsilon(t)}{5}  \qquad \text{ for }~t\geq 0,
\end{align}
the function $\overline{\Phi}_\delta$ is positive for $t \geq 0,~ \xi \in [0,\delta t]$ and 
\begin{equation}\label{e.910.3a}
(\overline{\Phi}_\delta)_t - (c_{\lambda}  + \rho'(t))(\overline{\Phi}_\delta)_\xi - (\overline{\Phi}_\delta)_{\xi\xi} - (1-a)  \overline{\Phi}_\delta  \geq 0, \quad \text{ for }t \geq 0,~ \xi \in [0,\delta t]. 
\end{equation}
Moreover, there exists $C_{\lambda} = 10 k_\lambda$ independent of $\delta$, such that
\begin{equation}\label{e.0910.3}
-C_\lambda\sqrt{\delta}\leq\frac{(\overline{\Phi}_\delta)_\xi(t,\xi)} {\overline{\Phi}_\delta(t,\xi)} +\lambda < 0 \quad \text{ for }t \geq 0,~ \xi \in [0,\delta t].
\end{equation}
\end{lemma}

\begin{proof}[{\bf Proof of \Cref{lem:3A}}]
Since $c_{\lambda}-2\lambda>0$, we have $k_{\lambda}>0$. 
For $\xi\in[0,\delta t]$, $\varepsilon(t)(1+\xi)\leq 5\sqrt{\delta}$. We may choose $\hat\delta>0$ sufficiently small that $1-k_{\lambda}\varepsilon(t)\xi\geq \frac{1}{2}$ and hence $(\overline{\Phi}_\delta)>0$
in $[0,\infty)\times[0,\delta t]$.

A direct computation gives
\begin{align}
    \frac{(\overline{\Phi}_\delta)_\xi}{\overline{\Phi}_\delta}&= \frac{-\lambda - k_{\lambda} \varepsilon(t) (1-\lambda \xi)}{1-k_{\lambda} \varepsilon(t) \xi}= -\lambda-\frac{k_{\lambda} \varepsilon(t)}{1-\varepsilon(t) k_{\lambda}\xi},
\end{align}
which implies \eqref{e.0910.3}, where we also used 
$1-\varepsilon(t) k_{\lambda}\xi\geq \frac{1}{2}$ and $\varepsilon(t)\leq 5\sqrt{\delta}$
in the domain.

We next verify the differential inequality. 
\begin{align}
    e^{\lambda\xi}(\overline{\Phi}_\delta)_t &= -\varepsilon'(t) k_{\lambda}\xi \\
    e^{\lambda\xi}(\overline{\Phi}_\delta)_{\xi} &= -\varepsilon(t) k_{\lambda} - \lambda \left( 1 - \varepsilon(t) k_{\lambda}\xi \right) \\
    e^{\lambda\xi}(\overline{\Phi}_\delta)_{\xi\xi} &= 2 \lambda \varepsilon(t) k_{\lambda}+ \lambda^2 \left( 1 - \varepsilon(t) k_{\lambda}\xi \right)
\end{align}
Consequently,
\begin{align}
&\left((\overline{\Phi}_\delta)_t - (c_{\lambda}  + \rho'(t))(\overline{\Phi}_\delta)_\xi - (\overline{\Phi}_\delta)_{\xi\xi} - (1-a) \overline{\Phi}_\delta\right)e^{\lambda\xi} \\
&= -\varepsilon'(t) k_{\lambda}\xi + (c_{\lambda} + \rho'(t)) \left( \varepsilon(t) k_{\lambda} + \lambda \left( 1 - \varepsilon(t) k_{\lambda}\xi \right)\right) \\
& \qquad \qquad \qquad- \left( 2 \lambda \varepsilon(t) k_{\lambda}+ \lambda^2 \left( 1 - \varepsilon(t) k_{\lambda}\xi \right)\right) - (1-a) \left( 1 - \varepsilon(t) k_{\lambda}\xi \right) \\
&\geq
{\rho'(t)\varepsilon k_{\lambda}}+
\rho'(t) \lambda \left( 1 - \varepsilon(t) k_{\lambda}\xi \right) + \varepsilon(t) \lambda
{\geq -|\rho'(t)|(\varepsilon k_{\lambda}+\lambda)+\varepsilon \lambda},  
    \end{align}
where we used \eqref{c-k-def}, $0<1-\varepsilon(t)k_\lambda\xi\leq1$, and $\varepsilon'(t)\leq0$.
    
Using $|\rho'(t)|\leq \frac{\varepsilon(t)}{5} \leq \sqrt\delta$,
we further obtain 
\begin{equation}
    -|\rho'(t)|(\varepsilon k_{\lambda}+\lambda)+\varepsilon \lambda \geq -\frac{\varepsilon}{5}(\varepsilon k_\lambda + \lambda) + \varepsilon\lambda \geq \varepsilon\lambda\left(\frac45 - \frac{k_\lambda}{\lambda} \sqrt\delta\right)\geq 0.
\end{equation}
This proves \eqref{e.910.3a} for $t>0$, $0\leq \xi \leq \delta t$, provided $\delta$ is small enough.
\end{proof}

\begin{proposition}\label{prop:supersolshiftbeta2}
 
Assume $\s=2$ and $\eta < \frac32$, and set $\lambda_\star=1-\sqrt{a}$. Then there exist  $\delta>0$, $\mu\in(0,\sqrt{a})$, and $\overline{t}>1$ such that, for every $t_0 \geq \overline{t}$ and every $\overline{k}_2 >0$,
\begin{equation}
\overline\Phi(t,x):= \overline{k}_2\overline{m}_\mu(x-s(t)) \overline\Phi_{\delta}(t-t_0,x-b_{t_0}(t)) 
\end{equation}
satisfies  
\begin{align}\label{e.929.1}
\overline{\Phi}_t -\overline{\Phi}_{xx} - \g(x-s(t)) \overline{\Phi} + \overline{\Phi}^2 \geq 0 \qquad \text{ in }\Omega_0:= \left\{(t,x):~ t \geq t_0,~ b_{t_0}(t) \leq x \leq s(t)\right\},
\end{align}
where $\overline\Phi_{\delta}$ is given in \eqref{Phi0911} with $\lambda=\lambda_\star$ and 
$\overline m_\mu$ is the corrector given by \Cref{prop:supercorrector} associated with ${\sigma}$ defined by $\sigma:=(1-a-\g)\mathbf{1}_{\mathbb{R}^-}$,  
and
\begin{equation}
b_{t_0}(t) = s(t_0) + 2(t-t_0) - \frac{1}{\lambda_\star}\left(\frac32 - \eta\sqrt a\right)\ln\frac{1+t}{1+t_0}.
\end{equation}
Moreover, the following boundary bounds hold:
\begin{equation}
     \inf_{t \geq t_0} (1+t)^{\tfrac32 - \eta}\overline\Phi(t,s(t)) >0, \qquad  \inf_{t \geq t_0} \overline\Phi(t,b_{t_0}(t)) >0.
\end{equation}
\end{proposition}

\begin{proof}[{\bf Proof of \Cref{prop:supersolshiftbeta2}}]

Choose $\mu= \frac12(1-\lambda_\star)=\frac12\sqrt{a}$.
 By replacing $s(t)$ by $s(t) + x_1$ for some $x_1 \in \mathbb{R}$, we may assume without loss of generality that 
    \begin{equation}
    -\mu^2 {<} 1-a - \g(y) \leq 0,\quad y\leq 0.
    \end{equation}
Let  $\hat{\delta} >0$ and $C_{\lambda_\star}$ be as given in \Cref{lem:3A}. Choose $\delta\in(0,\hat{\delta})$ sufficiently small so that
\begin{equation}\label{c-star-delta}
C_{\lambda_\star}\sqrt{\delta}<\frac{\sqrt{a}-\mu}{4},
\quad
\frac{|\frac{3}{2}-\sqrt{a}\eta |}{\lambda_\star}\delta
\leq\frac{\sqrt\delta}{2}.
\end{equation}
Next, choose $\overline t>1$ sufficiently large that
\begin{equation}\label{t1-cond}
\frac{|\eta|}{2(1+\overline{t})}<\frac{\sqrt{a}-\mu}{4},
\quad
\frac{\frac{3}{2}-\eta}{\lambda_\star(1+\overline{t})}\leq\delta,
\quad \frac{|\frac32-\eta\sqrt{a}|}{\lambda_\star(1+\overline{t})}\leq\frac{\sqrt{\delta}}{2}.
\end{equation}

Fix $t_0\geq \overline{t}$.
For $t\geq t_0$, we set $\tau:=t-t_0$.
Define
\begin{equation}
\rho(\tau):=b_{t_0}(t_0+\tau)-b_{t_0}(t_0)-2\tau,
\end{equation}
so that
\begin{equation}
\rho(\tau)=-\frac{\frac{3}{2}-\sqrt{a}\eta}{\lambda_\star}
\ln\frac{1+t_0+\tau}{1+t_0}.
\end{equation}
By \eqref{c-star-delta} and \eqref{t1-cond}, for $\tau\geq 0$,
\begin{align}
\frac{\left\vert \frac{3}{2}-\sqrt{a}\eta \right\vert}{\lambda_\star}(1+\delta\tau)
=\frac{\left\vert \frac{3}{2}-\sqrt{a}\eta \right\vert}{\lambda_\star}+
\frac{\left\vert \frac{3}{2}-\sqrt{a}\eta \right\vert}{\lambda_\star}\delta\tau\leq{\sqrt\delta}(1+t_0)+{\sqrt\delta}\tau.
\end{align}
It follows that
\begin{equation}
|\rho'(\tau)|=\frac{\left\vert \frac{3}{2}-\sqrt{a}\eta \right\vert}{\lambda_\star(1+t_0+\tau)}
\leq \frac{\sqrt{\delta}}{1+\delta\tau},
\quad \tau\geq0,
\end{equation}
which is exactly condition \eqref{rho-cond} of \Cref{lem:3A}.

Since $\eta<\frac{3}{2}$, we have $s(t)-b_{t_0}(t)\geq 0$. Moreover, by \eqref{t1-cond}, 
\begin{align}
s(t)-b_{t_0}(t)=\frac{\frac{3}{2}-\eta}{\lambda_\star}\ln\frac{1+t}{1+t_0}
\leq \delta(t-t_0),\quad t\geq t_0.
\end{align}
Thus, $\Omega_0$ is contained in the domain of validity of \Cref{lem:3A}.

Define
\begin{equation}
\psi(t,x):=\overline{\Phi}_\delta\left(t-t_0,x-b_{t_0}(t)\right).
\end{equation}
Since $b_{t_0}'(t)=2+\rho'(t-t_0)$,
\Cref{lem:3A} gives
$\psi_t-\psi_{xx}-(1-a)\psi\geq0$ in $\Omega_{0}$.
Moreover, by the choice $2\mu = \sqrt a$,  \eqref{e.0910.3}, \eqref{c-star-delta} and \eqref{t1-cond},
\begin{equation}\label{e.929.2}
\frac{\psi_x}{\psi}+\frac{\dot s}{2} \geq  
\left(-\lambda_\star-C_{\lambda_\star}\sqrt\delta\right)
+1-\frac{\eta}{2(1+t)}>
\mu, \qquad \text{ for }(t,x)\in\Omega_0.
\end{equation}
We now show that $\overline\Phi$ satisfies the supersolution inequality \eqref{e.929.1} in $\Omega_0$. 

Set
\begin{equation}
W(t,x):=\overline m_\mu(x-s(t))\psi(t,x),
\qquad y:=x-s(t)\le0.
\end{equation}
Since
\begin{equation}
\overline m_\mu''+2\mu\overline m_\mu'
=(1-a-\g)\overline m_\mu
\quad\text{on }(-\infty,0),
\end{equation}
direct differentiation gives
\begin{align}
&W_t-W_{xx}-\g(y)W\\
&\quad=
\overline m_\mu(y)
\left(\psi_t-\psi_{xx}-(1-a)\psi\right)
-2\psi\,\overline m_\mu'(y)
\left(\frac{\dot s(t)}2+\frac{\psi_x}{\psi}-\mu\right).
\end{align}
By \Cref{lem:3A}, the first term is nonnegative
throughout $\Omega_0$. The second is also nonnegative,
since $\overline m_\mu'\le0$ on $(-\infty,0]$
by \Cref{prop:supercorrector}, and
\eqref{e.929.2} holds. Consequently,
\begin{equation}
W_t-W_{xx}-\g(x-s(t))W\ge0
\qquad\text{in }\Omega_0.
\end{equation}
For any $\overline k_2>0$, setting
$\overline\Phi=\overline k_2W$ yields
\begin{equation}
\overline\Phi_t-\overline\Phi_{xx}
-\g(x-s(t))\overline\Phi+\overline\Phi^2
=
\overline k_2
\left(W_t-W_{xx}-\g(x-s(t))W\right)
+\overline k_2^2W^2
\ge0.
\end{equation}
This proves \eqref{e.929.1}.

It remains to establish the boundary estimates. At $x=b_{t_0}(t)$, we have
\begin{equation}
\overline{\Phi}(t,b_{t_0}(t))= \overline{k}_2 \overline{m}_\mu(b_{t_0}(t)-s(t))\geq  \overline{k}_2\left(\inf_{y\leq0}\overline{m}_\mu(y)\right)>0.
\end{equation}
Finally, by decreasing $\delta$ in \eqref{Phi0911} if necessary, we have 
\begin{equation}
\overline\Phi_\delta(\tau,\xi)\geq\frac{1}2e^{-\lambda_\star\xi},\qquad
0\leq\xi\leq\delta\tau.
\end{equation}
Consequently,
\begin{align}
\overline\Phi(t,s(t))
&= \overline{k}_2\overline{m}_\mu(0)
\overline{\Phi}_\delta\left(t-t_0,s(t)-b_{t_0}(t)\right)\\
&\geq\frac{ \overline{k}_2\overline{m}_\mu(0)}{2}
e^{-\lambda_\star\left(s(t)-b_{t_0}(t)\right)}=\frac{ \overline{k}_2\overline{m}_\mu(0)}{2}
\left(
\frac{1+t}{1+t_0}
\right)^{-\frac32+\eta},
\end{align}
by the definition of $b_{t_0}$. 
This completes the proof.
\end{proof}

\section{Construction of a subsolution}\label{sec:subsolshift}

In this section, we construct a subsolution to \eqref{eq:mainshift} when $\s > 2$ or when $\s = 2$ and $\eta \in (-\infty, \frac32]$. Note such a subsolution, for the case $\s=2$ and $\eta = \frac32$, automatically takes care of the remaining case $\s = 2$ and $\eta \in (\frac32,\infty)$, in which the invasion front is moving strictly faster than the moving environmental interface $s(t)$.

Writing $n(t,x) = \zeta(t,y)$ in the moving frame $y=x-s(t)$, we rewrite  \eqref{eq:mainshift} as
\begin{equation}\label{e.zetaa}
\zeta_t - \dot s \zeta_y - \zeta_{yy} = (\g - \zeta) \zeta \quad \text{ for }t>0,~y \in \mathbb{R},
\end{equation}
Following \cite{hamel2013short,hamel2016logarithmic}, we factor out the leading 
exponential factors and polynomial corrections in time. Specifically, we write
\begin{equation}\label{e.zeta.1}
\underline\zeta(t,y)= \mathcal{R}(t,y)\underline\Psi(t,y) \qquad \text{ with } \quad \mathcal{R}(t,y) =(1 +t)^{\frac{\s \eta}{2}}e^{-\left[\frac{\s^2}{4}-1\right]t} e^{-\frac{\s}{2}y}.\end{equation}
Then $\underline\zeta$ is a subsolution of \eqref{e.zetaa} if and only if $\underline\Psi$ satisfies
\begin{align}
\underline\Psi_t + \left( \s - \dot s \right) \underline\Psi_y - \underline\Psi_{yy} &\leq  \left(\g -1 - \mathcal{R}(t,y)\underline\Psi\right) \underline\Psi\qquad \text{ for }t>0,~y \in \mathbb{R}.
\end{align}
We now construct such a subsolution $\underline\Psi(t,y)$.

\begin{lemma}\label{lem:1Aa} Let $K>0$ be given.
Let $\underline{m}$ be the corrector given by \Cref{prop:subcorrector}, with  
\begin{equation}
{\sigma} = 1- \g + h>0,\quad \mbox{ where }
h(y)=  (1+y)e^{-y}.
\end{equation}
Denote $\ell = \underline{m}'(+\infty)\in (1,\infty)$ and 
 define
    \begin{align}
    &\delta(t) = \frac{1}{\sqrt{1+t}} \int_0^{\sqrt{1+t}} z \sigma(z)\underline{m}(z) \dd z, \quad 
    & \tau(t) = 1+t - K\int_0^t \delta(s) \dd s, ~~\quad \\
    &y(t) = |\eta| + \frac{\int_0^{\sqrt{1+t}} z \sigma(z)\underline{m}(z) \dd z}{\underline{m}'(\sqrt{1+t})}, & e(t) = -\frac52 K {\int_{0}^t \frac{\delta(s)}{\max\{1,\tau(s)\}} \, \dd s},
\end{align}
where $\tau, e$ depend on $K$. Then, there exists $T_1=T_1(K)>0$ such that
\begin{equation}
\tau(t) >1 ~ ~ \text{ in }[T_1,\infty),\qquad \lim_{t\to+\infty}\dot\tau(t) = 1 \quad \text{ and }\quad  \lim_{t\to\infty} e(t)\in (-\infty,0)~\text{exists},
\end{equation}
and there exists $0<c_1<1$ independent of $K$ such that for $t\in [T_1,\infty)$,
\begin{equation}\label{e.913.2}
 \delta(t) = o(1), \qquad \delta(t) \ge \frac{{c_1}}{\sqrt{1+t}}, \qquad \dot y(t) \leq  \frac{3\delta}{\sqrt{1+t}},  
 \end{equation}
\begin{equation}\label{e.913.5}
\frac{|y(t)|}{\sqrt{1+t}} =o(1), \qquad  \frac{\dot{y}^2}{\delta} \le  10\frac{\delta}{\tau} , \qquad \frac{1}{(1+t)^2 \delta} \leq  \frac{1}{\tau^2\delta} \leq \frac{1}{c_1}\frac{\delta}{\tau}.
\end{equation}
In addition, the following estimate holds:
\begin{equation}\label{e.928.1m}
    \frac{1}{y}\int_0^y z \sigma(z)\underline{m}(z) \dd z \leq \ell y^2\frac{\delta}{\tau}  \qquad \text{ for }t \geq T_1, ~y \geq 
    \sqrt{1+t}.
\end{equation}
\end{lemma}

\begin{proof}[{\bf Proof of \Cref{lem:1Aa}}]
We first establish \eqref{e.928.1m}. Set  $\mathcal{G}(y) = \frac{1}{y^3}\int_0^y z^2 \sigma (z) \dd z$.
Since $\sigma$ is positive, nonincreasing, and belongs to $\mathsf{L}^1((0,\infty),y\dd y)$, we have
\begin{align}
   \mathcal{G}'(y)& = \frac{y^3  
   {\sigma(y)}
   -  \int_0^y 3 \,z^2  
   {\sigma(z)}
   \dd z }{\vert y\vert^4}  
   \leq 0.
\end{align}
For $y\geq\sqrt{1+t}$ and sufficiently large $t$, using $\tau(t)\leq1+t$ and $z\leq\underline{m}(z)$, we obtain
\begin{align}
    \mathcal{G}(y) &\leq \mathcal{G}
    (\sqrt{1+t})
    \leq  (1+t)^{-\tfrac32} \int_0^{\sqrt{1+t}} z^2 \sigma(z)\,\dd z
    \le \frac{\delta}{\tau}.
\end{align}
Since $\underline{m}(z)\leq\ell z$ gives $\frac{1}{y^3}\int_0^y z\sigma(z)\underline{m}(z)\,\dd z\leq\ell\mathcal{G}(y)$,
 this proves \eqref{e.928.1m}.

Observe that $\int_0^{\sqrt{1+t}} z \underline{m}''(z)\,\dd z \geq \int_0^1 z^2 \sigma(z)\,\dd z>0$. Hence
\begin{equation}\label{e.928.1}
    \delta(t) \geq \frac{c_1}{\sqrt{1+t}} \quad \text{ for }t\geq 0, \qquad  \text{ for }0<c_1 < \int_0^1 z^2 \sigma(z)\,\dd z.
\end{equation}
Next, we prove $\delta(t) = o(1)$ as a consequence of Lebesgue's dominated convergence theorem:
\begin{equation}\label{e.913.7}
    |\delta(t)| =  \int_0^\infty z \sigma(z) \underline{m}(z) \frac{\mathbf{1}_{z \leq \sqrt{1+t}}}{\sqrt{1+t}}\,\dd z \lesssim \int_0^\infty \frac{z}{\sqrt{1+t}} \mathbf{1}_{z \leq \sqrt{1+t}} \, z \sigma(z) \,\dd z.
\end{equation}
The right-hand side tends to zero as $t\to\infty$. Consequently, $\int_0^t\delta(s)\dd s=o(t)$ and $\dot\tau(t)\to1$ for every $K>0$. We may therefore choose $T_1>0$, depending on $K$, sufficiently large that

\begin{equation}\label{e.928.2}
\frac{9}{10}(1+t) \leq \tau(t) \leq(1+t) \quad \text{ for }t\ge T_1.
\end{equation}

Next, we show that $e(t)$ is uniformly bounded. By its definition, it suffices to prove that $\delta/\tau$ is integrable over $[T_1,\infty)$. To this end, we use Tonelli's theorem. Setting
\[
I(R)=\int_0^R z\sigma(z)\underline{m}(z) \,\mathrm{d}z,
\]
we obtain
\begin{align}
    \int_{T_1}^\infty \frac{\delta(t)}{\tau(t)} \dd t  &\lesssim  \int_{T_1}^\infty 
    \frac{I(\sqrt{1+t})}{(1+t)^{\frac{3}{2}}} \dd t \lesssim \int_{\sqrt{1+T_1}}^\infty \frac{I(R)}{R^2} \dd R \lesssim \left|\frac{I(\sqrt{1+T_1})}{\sqrt{1+T_1}} + \int_{\sqrt{1+T_1}}^\infty \sigma(z) \underline{m}(z)\, \dd z\right|
\end{align}
The last integral is finite since $\underline{m}(z)\approx z$, and $\int_0^\infty z \sigma(z)\,\dd z <+\infty$. Thus, $\tfrac{\delta}{\tau}$ is integrable and $\lim\limits_{t\to\infty}e(t)$ exists and is negative.

Next, we claim that
\begin{equation}\label{e.913.7a}
\dot{y} \leq  \frac{3\delta(t)}{2(1+t)^{\frac12}} \quad \text{ and }\quad \frac{y(t)}{\sqrt{1+t}} = o(1).
\end{equation}
The second assertion follows from
\begin{equation}
\frac{y(t)}{\sqrt{1+t}} = \frac{|\eta|}{\sqrt{1+t}} + \frac{\delta(t)}{\underline{m}'(\sqrt{1+t})} \to 0,
\end{equation}
since $\underline{m}'\geq1$ and $\delta(t)\to0$.
For the first assertion, set $R=\sqrt{1+t}$. Then
\begin{align}
\dot{y}(t) = \frac{\sigma(R) \underline{m}(R)^2}{2R \underline{m}'(R)^2} \qquad \text{ for a.e.}~  t.
\end{align}
By convexity and $\underline{m}(0)=0$, we have $\underline{m}(R)\leq R\underline{m}'(R)$. Together with the monotonicity of $\sigma$, this yields
\begin{align}
R\dot{y}(t)&\leq  \frac{R^2}{2}\sigma(R) \leq \frac{3}{2R}\int_0^R z^2\sigma(z) \dd z \leq \frac{3}{2}\delta(t).
 \end{align}
This proves the first assertion in \eqref{e.913.7a}.

Thus \eqref{e.913.2} follows from \eqref{e.928.1}, \eqref{e.913.7}, and \eqref{e.913.7a}. Similarly, \eqref{e.913.5} follows from \eqref{e.913.2}, \eqref{e.928.2}, and \eqref{e.913.7a}, after decreasing $c_1$ if necessary.\end{proof}

\begin{definition}
For $\underline{C}>0$ and $K>0$, let $T_1$, $\tau(t)$, $y(t)$, $e(t)$ be as in \Cref{lem:1Aa}. Since $\tau(t)>1$, we can define
    \begin{equation}
\underline{\Psi}_{\underline{C},K}(t,y):= \begin{cases}
\displaystyle  \underline{C}  \frac{\underline{m}(y)}{\tau(t)^{\frac32}} \exp \left( - \frac{(y - y(t))^2}{4\tau(t)} +e(t)\right) &\text{ for } {t\geq T_1},~ y >0,\\[5pt]
    0  &\text{ for }t \geq T_1,~ y \leq 0,
\end{cases}
\end{equation}
where $\underline{m}$ is given by \Cref{prop:subcorrector} applied with ${\sigma} = 1- \g + h>0$ and $h(y):=  (1+y)e^{-y}$. 
\label{d.1}
\end{definition}

\begin{proposition}\label{prop:subsolutionshiftallbeta}
Assume that either $\s > 2$ and $\eta \in \mathbb{R}$, or $\s =2$ and $\eta \le  \frac32$. Let $\underline{\Psi}_{\underline{C},K}$ be defined as in \Cref{d.1}.   
Then there exist constants $\underline{C}^\star$ and $K_*$ such that, for any $\underline{C} \in (0, \underline{C}^\star)$ and $K \geq K_*$, the function
$\underline{\Psi} = \underline\Psi_{\underline{C},K}$ 
satisfies 
\begin{equation}\label{e.929.10}
\underline{\Psi}_t - \underline{\Psi}_{yy} + \frac{\eta}{1+t} \underline{\Psi}_y + (1- \g )\underline{\Psi} \leq - \mathcal{R}\underline{\Psi}^2 \quad \text{ for } t >T,~ y \in \mathbb{R}
\end{equation}
for some {$T \geq T_1$}, depending on $K$, where $\mathcal{R}(t,y)$ is defined in  \eqref{e.zeta.1}.
Furthermore, the auxiliary functions $\tau$, $y$ and $e$, given in \Cref{lem:1Aa}, which also depend on $K$, satisfy
\begin{equation}
\lim_{t\to\infty} e(t)\in(-\infty,0), \qquad y(t) = o(\sqrt{t}), \qquad  \tau(t)\geq 1~\text{ for }t\geq T,\qquad 
\frac{\tau(t)}{1+t} \to 1.
\end{equation}
In particular, the solution $n$ of \eqref{eq:mainshift} with initial condition \eqref{e.ini} obeys 
\begin{equation}\label{e.929.11}
n(t,s(t)) \gtrsim t^{-\tfrac32 + \tfrac{\s \eta}2} e^{- \left[\tfrac{\s^2}4-1 \right]t} \quad \text{ for }t \gg 1.    
\end{equation}
\end{proposition}

\begin{proof}[{\bf Proof of \Cref{prop:subsolutionshiftallbeta}}]
For simplicity, write $m=\underline{m}$ throughout the proof. Let $K>0$ be chosen later. By \Cref{lem:1Aa}, the function $e$ is bounded, $\tau(t)>1$, and $\tau(t)\sim t$. Since $\int_0^t\delta(s)\dd s\to+\infty$ and $y(t)^2/\tau(t)\to0$, we may choose $T\geq\max\{T_1,1\}$ sufficiently large that
\begin{align}\label{tau-cond}
 t-\tau(t) \to +\infty,\qquad  \frac9{10}(1+t)\le\tau(t)<t,\qquad
\frac{y(t)^2}{\tau(t)}\le1 \quad\text{for all }t\ge T.  
\end{align}
In particular, $\tau(t)>1$ and
$0\le y(t)\le\sqrt{1+t}$ for $t\ge T$.

Since $\sigma=1-\g+h>0$, we have $m''(y)=\sigma(y)m(y)>0$ for almost every $y>0$.

By the notion of a generalised subsolution \cite[Definition 1.1]{lam2022introduction}, it suffices to verify the differential inequality for $y>0$, where $\underline{\Psi}>0$, because $m(0)=0$. Direct differentiation gives
\begin{align}
\frac{\underline{\Psi}_t}{\underline{\Psi}} &= -\frac{3}{2} \frac{\dot \tau}{\tau}    + \dot e - \left( \frac{-2(y - y(t)) \dot y(t) \tau(t) - (y - y(t))^2 \dot \tau}{4\tau(t)^2} \right) \\
\frac{\underline{\Psi}_y}{\underline{\Psi}} &= \frac{m'}{m} - \frac{(y-y(t))}{2 \tau} \\
\frac{\underline{\Psi}_{yy}}{\underline{\Psi}} &= \left(\frac{m'}{m} - \frac{(y-y(t))}{2 \tau}\right)^2 + \left(\frac{m'}{m}\right)' - \frac{1}{2 \tau} = \left(\frac{m'}{m} - \frac{(y-y(t))}{2 \tau}\right)^2 + \frac{m''m - \vert m' \vert^2}{m^2} - \frac{1}{2 \tau}\\
&= -\frac{m'}{m} \frac{(y-y(t))}{ \tau} +\frac{(y-y(t))^2}{4 \tau^2} + \frac{m''}{m} - \frac{1}{2 \tau} 
\end{align}
Combining these identities yields
\begin{align}
&\frac{1}{\underline{\Psi}}\left[\underline{\Psi}_t-\underline{\Psi}_{yy}+\frac{\eta}{1+t}\underline{\Psi}_y+(1- \g)\underline{\Psi} \right]\notag \\
&= -\frac{3}{2} \frac{\dot \tau}{\tau}    + \dot e +  \left(\dot y(t) - \frac{\eta}{1+t}\right) \frac{y - y(t)}{2 \tau}  - (1-\dot \tau) \frac{(y - y(t))^2 }{4\tau(t)^2} \\
& \qquad \qquad \qquad + \frac{1}{2 \tau}+ \left(\frac{\eta \tau }{1+t}+ y-y(t) \right) \frac{m'}{\tau m}  -h.
\end{align}
Here we used $m''=(1-\g+h)m$.

Using the identity $ym'(y)=m(y)+\int_0^y z\,m''(z)\dd z$, we obtain
\begin{align}
&\frac{1}{\underline{\Psi}}\left[\underline{\Psi}_t-\underline{\Psi}_{yy}+\frac{\eta}{1+t}\underline{\Psi}_y+(1- \g)\underline{\Psi} \right] \\ 
&\leq \frac{3}{2} \frac{(1-\dot \tau)}{\tau}    + \dot e  + \left(\dot y(t) - \frac{\eta}{1+t}\right) \frac{y - y(t)}{2 \tau} -  (1-\dot \tau)\frac{(y - y(t))^2 }{4\tau(t)^2}\\
&\quad + \left(\frac{\eta \tau }{1+t}-y(t) + \frac{\int_0^y z \, m''(z) \dd z}{m'(y)}  \right) \frac{m'}{\tau m} -h \label{e.928m}
\end{align}

\medskip

\noindent {\bf\# Zone 1 :} $0<y \leq \sqrt{1+t}$.
\begin{align}
\frac{1}{\underline{\Psi}}\left[\underline{\Psi}_t-\underline{\Psi}_{yy}+\frac{\eta}{1+t}\underline{\Psi}_y+(1- \g)\underline{\Psi} \right]\notag 
&\leq \frac{3}{2} \frac{(1-\dot \tau)}{\tau}    + \dot e  + \left(\dot y(t) - \frac{\eta}{1+t}\right) \frac{y - y(t)}{2 \tau} \\
&\quad + \left(\frac{\eta \tau }{1+t}-y(t) + \frac{\int_0^y z \, m''(z) \dd z}{m'(y)}  \right) \frac{m'}{\tau m} -h
\end{align}
since $\dot\tau\leq1$. Moreover,
\begin{align}
   \frac{\dd }{\dd y} \left( \frac{\int_0^y z \, m''(z) \dd z}{m'(y)} \right)& = \frac{ym''(y)m'(y) - \int_0^y z \, m''(z) \dd z m''(y)}{\vert m'(y)\vert^2}\\
   &=\frac{ym'(y) - \int_0^y z \, m''(z) \dd z}{\vert m'(y)\vert^2}m''(y) = \frac{m''(y)m(y)}{\vert m'(y)\vert^2} \geq 0.
\end{align}
Since $0<\tau(t)\leq1+t$, we have $\eta\tau(t)/(1+t)\leq|\eta|$, regardless of the sign of $\eta$. By the definition of $y(t)$, the last bracket in \eqref{e.928m} is therefore nonpositive in this region. Consequently,
\begin{align}
&\frac{1}{\underline{\Psi}}\left[\underline{\Psi}_t-\underline{\Psi}_{yy}+\frac{\eta}{1+t}\underline{\Psi}_y+(1- \g)\underline{\Psi} \right] \leq \frac{3}{2} \frac{(1-\dot \tau)}{\tau}    + \dot e +  \left(\dot y(t) - \frac{\eta}{1+t}\right) \frac{y - y(t)}{2 \tau}  -h
\end{align}
By \eqref{e.913.2} and the choice of $T$,
$$\delta(t) = o(1), \qquad \delta \ge \frac{c_1}{\sqrt{1+t}}, \qquad 0 \leq \dot y(t) \le\frac{3\delta}{\sqrt{1+t}}, \qquad 
0\leq y(t)\leq \sqrt{1+t}
$$
and hence
\begin{equation}
  \left|\dot y(t)\frac{y - y(t)}{2 \tau}\right| \leq \sqrt{1+t} \frac{\dot y(t)}{2 \tau} \le  \frac{3\delta}{2\tau},
  \end{equation}

\begin{equation}
\left| \frac{\eta}{1+t}\frac{y - y(t)}{2 \tau} \right| \leq \frac{|\eta|}{2(1+t)\tau} \left(\sqrt{1+t} + y(t)\right) \leq 
{\frac{|\eta|}{c_1}\frac{\delta}{\tau}}
\end{equation}
where $c_1$ is given by \eqref{e.913.2}. It follows that
\begin{equation}\label{e.928.5a}
     \left| \left(\dot y(t) - \frac{\eta}{1+t}\right) \frac{y - y(t)}{2 \tau} \right\vert \leq {\left(\frac{3}{2}+\frac{|\eta|}{c_1}\right)} \frac{\delta}{\tau},  \text{ where the bound is  independent of }K
\end{equation}
By the definitions of $\tau$ and $e$, we have $\dot\tau=1-K\delta$ and $\dot e=-\frac52K\delta/\tau$, so
\begin{equation}\label{e.928.5b}
      \frac32 \frac{(1-\dot\tau)}{\tau}  +\dot e = - K\frac{\delta}{\tau}.
\end{equation}
Combining \eqref{e.928.5a} and \eqref{e.928.5b}, and taking $K$ sufficiently large, yields
\begin{align}
&\frac{1}{\underline{\Psi}}\left[\underline{\Psi}_t-\underline{\Psi}_{yy}+\frac{\eta}{1+t}\underline{\Psi}_y+(1- \g)\underline{\Psi} \right] \leq -h. \label{e.928.n}
\end{align}

\medskip

\noindent {\bf \# Zone 2 :} $y \geq \sqrt{1+t}$.

Using \eqref{e.928.5b} and $\eta \tau/(1+t)\leq |\eta|\leq y(t)$
in \eqref{e.928m}, we obtain 
\begin{align}
&\frac{1}{\underline{\Psi}}\left[\underline{\Psi}_t-\underline{\Psi}_{yy}+\frac{\eta}{1+t}\underline{\Psi}_y+(1- \g)\underline{\Psi} \right] + h + K\frac{\delta}{\tau} \notag \\
&\leq -  (1-\dot \tau)\frac{(y - y(t))^2 }{4\tau(t)^2} + \left(\dot y(t) - \frac{\eta}{1+t}\right) \frac{y - y(t)}{2 \tau} + \left(\frac{\int_0^y z \, m''(z) \dd z}{m'(y)}  \right) \frac{m'}{\tau m}  \\
&=   
-(1-\dot \tau)\left[ \frac{(y - y(t))^2 }{8\tau(t)^2}
+ \left(\frac{1}{\sqrt 2}\frac{y - y(t)}{2 \tau} - \frac{\dot y(t) - \frac{\eta}{1+t}}{\sqrt 2(1-\dot \tau)}\right)^2\right] +\frac{{\left(\dot y(t) - \frac{\eta}{1+t}\right)^2}}{2(1-\dot \tau)}\\
&\qquad + \frac{1}{\tau m} \int_0^y z \sigma(z){m}(z) \dd z  \qquad \qquad \hfill\text{(Completing the squares and using $m'' = \sigma m$.)} \\
&\leq    
-(1-\dot \tau)\left[ \frac{\frac{y^2}{2} + (\frac{y}{\sqrt 2} - \sqrt 2 y(t))^2 - y(t)^2 }{8\tau(t)^2}
\right]+ \frac{{\left(\dot y(t) - \frac{\eta}{1+t}\right)^2}}{2(1-\dot \tau)}  +\frac{1}{\tau m} \int_0^y z \, \sigma(z){m}(z) \dd z  \quad \\
&\leq   - \frac{(1-\dot\tau) }{16} \frac{y^2}{\tau^2} + \frac{(1-\dot\tau)}{8} \frac{y(t)^2}{\tau^2}  + \frac{1}{1-\dot\tau}\left({|\dot y(t)|^2} +  \frac{\eta^2}{(1+t)^2}\right)  + \frac{1}{\tau y} \int_0^y 
{z}\sigma(z){m}(z)\dd z, 
\end{align}
where the last inequality uses $m(y)\geq y$.
Using $1-\dot\tau=K\delta$ and \eqref{e.928.1m}, we find
\begin{align}
&\frac{1}{\underline{\Psi}}\left[\underline{\Psi}_t-\underline{\Psi}_{yy}+\frac{\eta}{1+t}\underline{\Psi}_y+(1- \g)\underline{\Psi} \right] + h \\
&\leq -{ K \frac{\delta}{\tau}} - \frac{K }{16} \frac{y^2}{\tau}\frac{\delta}{\tau} + \frac{K}{8} \frac{y(t)^2}{\tau}\frac{\delta}{\tau}  + \frac{1}{K}\frac{|\dot y|^2}{\delta} + \frac{\eta^2}{K} \frac{1}{\delta(1+t)^2}  + \ell \frac{y^2 \delta}{\tau^2 } \\
&=- \frac{y^2 \delta}{\tau^2} \left[ \frac{K}{16} -\ell \right] - 
\frac{\delta}{\tau} \left[ {K} - \frac{K}{8} \frac{y(t)^2}{\tau} - \frac{1}{K} \frac{|\dot y|^2}{\delta} \frac{\tau}{\delta} - \frac{\eta^2}{K} \frac{1}{\delta(1+t)^2}\frac{\tau}{\delta}  \right] \\
&= - \frac{y^2 \delta}{\tau^2} \left[ \frac{K}{16} -\ell \right] - 
\frac{\delta}{\tau} \left[ {K} - 
{\frac{K}{8}}
- \frac{10}{K}  - \frac{\eta^2}{K} \frac{1}{c_1}  \right] \qquad \text{ by \eqref{e.913.5} and \eqref{tau-cond}.} 
\end{align}
Since $c_1$ and $\eta$ are independent of $K$, choosing $K$ sufficiently large ensures that \eqref{e.928.n} also holds in Zone 2 for $t\geq T$.

It remains to choose $\underline{C}>0$ sufficiently small to control the nonlinear term, namely, to ensure $\mathcal{R}\underline\Psi\leq h$, where $\mathcal{R}$ is defined in \eqref{e.zeta.1}. Together with the preceding estimates, this yields \eqref{e.929.10}:
\begin{equation}
\frac{1}{\underline{\Psi}}\left[\underline{\Psi}_t-\underline{\Psi}_{yy}+\frac{\eta}{1+t}\underline{\Psi}_y+(1- \g)\underline{\Psi} \right] \leq 
-h \leq - \mathcal{R} \underline\Psi.
\end{equation}
Indeed, for $t\geq T$ and $y>0$, we have
\begin{equation}
\frac{\mathcal{R}(t,y)\underline{\Psi}(t,y)}{h(y)}
=\underline{C}\frac{m(y)}{1+y}
 \frac{(1+t)^{\frac{\s\eta}{2}}}{\tau(t)^{\frac{3}{2}}}
 \exp\left\{
 -\left(\frac{\s^2}{4}-1\right)t
 -\left(\frac{\s}{2}-1\right)y
 -\frac{(y-y(t))^2}{4\tau(t)}+e(t)
 \right\}.
\end{equation}
Using $\s\geq 2$, $m(y)\leq \ell y$, $e(t)\leq 0$ and \eqref{tau-cond}, we further have 
\begin{equation}
\frac{\mathcal{R}(t,y)\underline{\Psi}(t,y)}{h(y)}\leq \left(\frac{10}{9}\right)^{\frac{3}{2}}\ell \underline{C}(1+t)^{\frac{\s\eta}{2}-\frac{3}{2}}\exp\left[-\left(\frac{\s^2}{4}-1\right)t\right].
\end{equation}

When $\s>2$, the exponential decay dominates any polynomial growth. When $\s=2$, the bound is uniform for $\eta\leq\frac32$. Thus, after choosing $\underline{C}>0$ sufficiently small, we have $\mathcal R\underline\Psi/h\leq1$ for all $t\geq T$, proving \eqref{e.929.10}.

It remains to establish \eqref{e.929.11}. For every fixed $T>0$, the solution satisfies
\begin{equation}
n(T,x) \gtrsim e^{-\frac{(x-A)^2}{4T}} \qquad \mathrm{ for~some } \, A \in \mathbb R,~ \text{ for }x  \gg 1.
\end{equation}
The constructed function $\underline n$, defined below, satisfies an analogous Gaussian upper bound with $T$ replaced by $\tau(T)\in(0,T)$; see \eqref{tau-cond}. Consequently,
 $\frac{\underline{n}(T,x)}{n(T,x)}\to0$ as $x\to+\infty$, where 
\begin{equation}
\underline{n}(T,x)
:=\mathcal{R}(T,x-s(T))\underline{\Psi}(T,x-s(T)).
\end{equation}
After decreasing $\underline{C}$ if necessary, we obtain
$\underline{n}(T,x)\le n(T,x)$ for all $x\in\mathbb R$. The parabolic comparison principle on $[T,\infty)\times\mathbb R$ then gives
\begin{equation}
n(t,y + s(t)) \geq \mathcal{R}(t,y)\underline\Psi(t,y) \qquad \text{ in }[T,\infty)\times \mathbb{R}.\end{equation}
In particular,
\begin{equation}
n(t,s(t)+1)\gtrsim (1+t)^{-\frac{3}{2}+\frac{\s \eta}{2}}\exp\left[-\left(\frac{\s^2}{4}-1\right)t\right],\quad t\gg1.
\end{equation}
The parabolic Harnack inequality then yields \eqref{e.929.11}, completing the proof.
\end{proof}

\section{The moving habitat problem with logarithmic correction: \texorpdfstring{\\}{e} proofs of Theorems
\texorpdfstring{\ref{t.0}}{t.0},
\texorpdfstring{\ref{thm:standard}}{thm:standard}, \texorpdfstring{\ref{thm:pulling}}{thm:pulling}, \texorpdfstring{\ref{thm:cpulling}}{thm:cpulling} and \texorpdfstring{\ref{thm:no-pull}}{thm:no-pull}}\label{sec:proofshift}

In this section, we take advantage of the constructions made in all previous sections to perform the proofs of most of the results concerning the spreading estimate for the climate-shift equation.

\begin{proof}[{\bf Proof of \Cref{t.0}}] By a translation, we may assume without loss that $\inf_{y \geq 0}\g>0$. 
Set
\begin{equation}
X(t):=2t-\frac32\ln(1+t).\end{equation}
Take a solution $\hat{n}$ of the homogeneous KPP equation $\hat{n}_t - \hat{n}_{xx} = \hat{n}(1-\hat n)$ with initial data  (with $M, K \gg 1$) such that
\begin{equation}
\hat{n}_0 = M\mathbf{1}_{(-\infty, K]}, \qquad  \hat{n}_0 \geq n_0.
\end{equation}
The upper bound of level sets of $\hat n$ is well known to be within a bounded distance from $X(t)$ \cite{bramson1983convergence,hamel2013short}. 
If  $\g(x-s(t))\leq 1$  globally, then $n_t - n_{xx} \leq n(1-n)$ and $n \leq \hat n$ by comparison; the upper bound of the level set follows. Since $\s<2$, we can choose a larger $K$ so that 
\begin{equation}
\hat n(t,x) \geq \frac12 \quad \text{ for }t \geq 0,\quad x - s(t) \leq x_0.
\end{equation}
Then, for $
A\ge \max\{1,\,1+2(\|g\|_\infty-1)_+\}$, 
it is clear that 
\begin{equation}
(A\hat{n})_t - (A\hat{n})_{xx} - A\hat{n}(\g -A\hat n) \geq A \hat n\left( (A-1)\hat n + 1 - \g(x-s(t))\right) \geq 0 \quad \text{ for }t\geq 0,~ x\in \mathbb{R}.
\end{equation}
Then $n \leq A \hat n$ in $\mathbb{R}^+\times \mathbb{R}$ and again the upper bound of the level set 
\begin{equation}
\xi_\kappa(t) \leq  X(t) + O(1) \qquad \text{ for }\kappa \in (0,1),
\end{equation}
follows from that of $\hat n$. Moreover, using \cite[Proposition 2.3]{hamel2013short}, we have
\begin{equation}\label{e.109.1}
\sup_{0 \leq y \leq \sqrt t} (1+y)^{-1} e^{y}   n(t,X(t) +y) \lesssim 1 \qquad \text{ for }t\gg 1. 
\end{equation}

For the lower bound, let
\begin{equation}
\underline B(t,x)
:=\underline\zeta_0(t,x-X(t)),
\end{equation}
where
$\underline\zeta_0 
=\mathcal R \underline\Psi $
is the subsolution constructed in
\Cref{prop:subsolutionshiftallbeta} for the
homogeneous growth profile $\g\equiv1$,
with $\s=2$ and $\eta=\frac32$.
In particular,
\begin{equation}
\underline B_t-\underline B_{xx}
\le\underline B(1-\underline B) \qquad \text{ in }[T,\infty)\times \mathbb R,
\end{equation}
in the generalised sense, and
$
\operatorname{supp}\underline B(t,\cdot)
\subseteq[X(t),\infty).$

Since $\s<2$, we have
\begin{equation}
X(t)-s(t)
=(2-\s)t+
\left(\eta-\frac32\right)\ln(1+t)
\longrightarrow+\infty.
\end{equation}
For sufficiently large $t$, define
\begin{equation}
h(t):=
\sup_{x \, \in \,\operatorname{Supp}\underline B(t,\cdot)}
\left(1-\g(x-s(t))\right)_+.
\end{equation}
Then
\begin{equation}
0\le h(t)\le
\sigma^\star\left(X(t)-s(t)\right).
\end{equation}
Since $X -s$ is eventually increasing
with derivative bounded below by a positive
constant, the integrability of $\sigma^\star$
implies
\begin{equation}
\int_{T}^\infty h(t)\,\dd t<\infty.
\end{equation}

We now choose a single comparison time
$T_1\ge T$ sufficiently large that
$\underline B(t,\cdot)$ is defined for $t\in [T_1,\infty)$
and
$
1<\tau(T_1)<T_1,$ 
where $\tau$ is the diffusion-time correction
appearing in the construction of
$\underline\zeta_0$.
Define, for $t\ge T_1$,
\begin{equation}
\beta(t):=
\exp\left(-\int_{T_1}^t h(r)\,\dd r\right),
\qquad
\underline n_1(t,x) :=
\beta(t) \underline B(t,x).
\end{equation}
The integrability of $h$ gives
$
0<\beta_\infty:=
\exp\left(-\int_T^\infty h(r)\,\dd r\right)
\le\beta(t)\le1.$ Also, wherever $\underline B>0$, we have
\begin{align}
&\beta^{-1}
\left[
(\underline n_1)_t-(\underline n_1)_{xx}
-\underline n_1(\g(x-s(t))-\underline n_1)
\right]\\
&\quad=
\underline B_t-\underline B_{xx}
-\underline B\left(
\g(x-s(t))+h(t)-\beta\underline B
\right)\\
&\quad\le
\underline B
\left[
1-\g(x-s(t))-h(t)
+(\beta-1)\underline B
\right]
\le0.
\end{align}
The inequality also holds in the generalised
sense across the boundary of the support.
Hence $\underline n_1$ is a subsolution of
\eqref{eq:mainshift} on $[T_1,\infty)\times\mathbb R$.

We next compare $\underline n_1(T_1,\cdot)$
with $n(T_1,\cdot)$. Since $n_0$ is nontrivial
and the reaction coefficient is bounded,
comparison with a linear heat equation yields
constants $A\in\mathbb R$ and $c_{T_1}>0$ such that
\begin{equation}
n(T_1,x)\ge
c_{T_1}\exp\left(-\frac{(x-A)^2}{4T_1}\right)
\qquad\text{for all sufficiently large }x.
\end{equation}
On the other hand, by the explicit expression
of $\underline\zeta_0$, there exist
$C_{T_1}>0$ and $D_{T_1}\in\mathbb R$ such that
\begin{equation}
\underline n_1(T_1,x)
\le C_{T_1}(1+x)
\exp\left(
-\frac{(x-D_{T_1})^2}{4\tau(T_1)}
\right)
\qquad\text{for all sufficiently large }x.
\end{equation}
Since $\tau(T_1)<T_1$, it follows that
\begin{equation}
\lim_{x\to+\infty}
\frac{\underline n_1(T_1,x)}{n(T_1,x)}=0.
\end{equation}
Furthermore,
$\underline n_1(T_1,x)=0$ for $x\le X(T_1)$,
whereas $n(T_1,x)>0$ for every $x\in\mathbb R$.
By continuity, there exists
$\varepsilon\in(0,1]$ such that
\begin{equation}
\varepsilon\underline n_1(T_1,x)
\le n(T_1,x),
\qquad x\in\mathbb R.
\end{equation}

Since $\varepsilon\le1$ and the reaction term
is of Fisher--KPP type,
$\varepsilon\underline n_1$ remains a
subsolution. The comparison principle gives
\begin{equation}\label{eq:single-time-comparison}
n(t,x)\ge
\varepsilon\underline n_1(t,x)
=\varepsilon\beta(t)\underline B(t,x),
\qquad t\ge T_1,\quad x\in\mathbb R.
\end{equation}
In particular, for each fixed $y>0$,
the explicit expression of $\underline\zeta_0$
and the properties
\begin{equation}
\tau(t)\sim t,\qquad
y(t)=o(\sqrt t),\qquad
e(t)\longrightarrow e_\infty>-\infty
\end{equation}
imply
\begin{equation}\label{e.109.2}
    \inf_{ 1\leq y \leq \sqrt{1+t}} (1+y)^{-1}e^y n(t,X(t)+y) \gtrsim \inf_{ y \geq 1} \frac{\underline{m}(y)}{1+y} \gtrsim 1.
\end{equation}

Thus the lower barrier supplies the desired
lower localisation at the critical Bramson
position. Together with the corresponding
estimate behind the interface, this yields
\begin{equation}
\xi_\kappa(t)\ge
X(t)+O(1), \qquad \text{ for }\kappa\in(0,1).
\end{equation}

Finally, we prove the convergence to the wave profile \eqref{e.1016.2a}, which is equivalent to the existence of some $R\in (0,\infty)$ such that, for a given constant $\gamma>0$, 
\begin{equation}
    \lim_{t\to\infty} \min_{|z|\leq R} \sup_{y \geq s(t) - X(t) + \gamma \ln t} \left|n(t,y + X(t)) - \Phi^\star_1(y + z) \right| =0.
\end{equation}
Fix $\gamma>0$ and let $\{(t_j,y_j)\}$ be a sequence such that $t_j \to \infty$ and $y_j \geq - X(t_j)+ s(t_j) + \gamma  \ln t$. It suffices to show that 
\begin{itemize}
    \item[(i)] $n(t_j, y_j + X(t+ t_j)) \to 0$ if $y_j \to +\infty$;
    \item[(ii)] $n(t_j, y_j + X(t+ t_j)) \to 1$ if $y_j \to -\infty$;
    \item[(iii)] if $y_j$ is uniformly bounded, then up to passing to a subsequence, the sequence of functions
\begin{equation}\label{e.kpppp}
\tilde{n}_j(t,y) = n(t+t_j, y + X(t+ t_j))
\end{equation}
satisfies
$\tilde{n}_j (t,y) \to \Phi^\star_1(y +z)$ in $\mathscr{C}_{\rm loc}(\mathbb{R}^2)$;\footnote{This is equivalent to $\tilde{n}_j (t,y+y_j) \to \Phi^\star_1(y+y_j +z)$ in $\mathscr{C}_{\rm loc}(\mathbb{R}^2)$.}
\item[(iv)] such a phase correction $z$ is bounded uniformly independent of the sequence $(t_j,y_j)$.
\end{itemize}

Consider first the case  $y_j \to +\infty$; then define
\begin{equation}
    n_j(t,y) = n(t+t_j, y + y_j + X(t+ t_j)).
\end{equation}
Since $X(t)-\xi_\kappa(t) = O(1)$ for all $\kappa \in (0,1)$,  we may pass to a sequence so that $n_j \to w_\infty$ in $\mathscr{C}_{\rm loc}(\mathbb R^2)$ and  $\liminf_{j \to \infty} y_j - \xi_\kappa(t_j) \to \infty$ for all $\kappa \in (0,1)$. Then $w_\infty$ is a nonnegative entire solution of
\begin{equation}\label{e.kppp}
    w_t - 2 w_y - w_{yy} = w(1-w) \qquad \text{ in }\mathbb{R}^2,
\end{equation}
and such that $w_\infty(0,0) = 0$. By the strong maximum principle, it follows that $w_\infty \equiv 0.$ This shows the convergence in case $y_j \to +\infty$.

Next, assume $y_j \to -\infty$. Now
\eqref{e.109.2} implies $\inf_{t \geq 1} n(t,X(t))>0$. Combining with \eqref{e.ini} and $\inf\g >0$, it follows from the parabolic maximum principle that
\[
\inf_{t >1,~ x \leq X(t)} n(t,x) >0.
\]
Then we can again pass to a subsequence such that $n_j \to w_\infty$ in $\mathscr{C}_{\rm loc}(\mathbb R^2)$ such that $w_\infty$ satisfies \eqref{e.kppp} and $\inf_{\mathbb{R}^2}w_\infty >0$. A standard classification of entire solutions of \eqref{e.kppp} gives $w_\infty \equiv 1$.

Finally, in the case where $\{y_j\}$ is uniformly bounded, we consider $\tilde{n}$ as given in \eqref{e.kpppp}.
Passing to a sequence, we may assume $\tilde{n}_j \to \tilde{w}_\infty$.

Moreover, $\tilde{w}_\infty$ satisfies \eqref{e.kppp} with the following tail bounds thanks to \eqref{e.109.1} and \eqref{e.109.2}:
\begin{equation}
 \frac{1}{M_0}(1+y)e^{-y} \leq  \tilde{w}_\infty(t,y) \leq M_0 (1+y)e^{-y} \qquad \text{ in }\mathbb{R}\times [1,\infty),
\end{equation}
where $M_0$ is independent of the sequence $(t_j,y_j)$. We can then apply \Cref{p.109.1}, which identifies the limit as a translate of the critical travelling wave, i.e.
\[
\tilde{w}_\infty(t,y) = \Phi_1^\star(y + z) \qquad \text{ for some }z \in \mathbb{R}.
\]
Moreover, since $\Phi_1^\star(y) \sim (1+y) e^{-y}$, it follows that $\frac{1}{M_0} \leq e^{-z} \leq M_0$, so that $|z| \leq \ln M_0$ is uniformly bounded independently of the sequence $(t_j,y_j)$. This proves (iii) and (iv) and hence the convergence to profile \eqref{e.1016.2a}. In fact, the bounded phase function $\Lambda(t)$ can be defined by the unique root $\Lambda$ of
\[
n(t,X(t)+1) = \Phi^\star_1(\Lambda +1),
\]
which is well defined since $n(t,X(t)+1) \in (0,1)$ for all $t\gg 1$. 
This completes the proof.

\end{proof}

\begin{proof}[{\bf{Proof of \Cref{thm:standard}}}]
The upper bound of the level set follows by repeating the argument in the proof of Theorem \ref{t.0}. We only need to establish the lower bound.

Set
\begin{equation}
X(t):=2t-\frac32\ln(1+t),\qquad
\underline B(t,x):=\underline\zeta^\star(t,x-X(t)),
\end{equation}
where
$\underline\zeta^\star
=\mathcal R\underline\Psi$
is the subsolution constructed in
\Cref{prop:subsolutionshiftallbeta} with heterogeneous growth profile $\g$ and
with $\s=2$ and $\eta=\frac32$. (This is in contrast to what we did in the proof of \Cref{t.0}, when we took $\g_{\rm hom}\equiv 1.$ We used the additional integrability assumption for $1-\g$ to obtain the corrector function.) 
Then Proposition \ref{prop:subsolutionshiftallbeta} gives
\begin{equation}
\underline B_t-\underline B_{xx} - \underline B (\g(x-s(t)) - \underline{B}) \leq \underline B_t-\underline B_{xx}-
 \underline B\left(g(x-X(t))-\underline B\right) \leq 0,
\end{equation}
using also the facts that $g$ is non-decreasing and $\eta\ge\frac{3}{2}$.

By repeating the preceding proof, we therefore obtain the lower estimate for the level set,
\begin{equation}
\xi_\kappa(t) \geq 2t - \frac32 \ln(1+t) + O(1).
\end{equation}
The convergence to the critical homogeneous
travelling-wave profile follows by the
same compactness and Liouville argument
as in the proof of \Cref{t.0}.
\end{proof}
Note that the above level set estimate also applies for the case $s(t) = 2t -  \tfrac32 \ln(1+t)$. However, the convergence to the wave profile is much more involved, and we postpone the proof of \Cref{thm:selection} to \Cref{sec:selection}.

\begin{proof}[{\bf Proof of  \Cref{thm:pulling} and \Cref{thm:cpulling}}]

Define 
\begin{equation}
 \lambda_\star = \frac{\s}2 - \sqrt a,\qquad c_\star = \lambda_\star + \frac{1-a}{\lambda_\star},
\end{equation}
\begin{equation}
    X(t) =\begin{cases}
      c_\star t-\frac1{\lambda_\star}\left(\frac32 - \eta \sqrt{a}\right) \ln (1+t) &\text{ if }2<\s <2\sqrt{a} + 2\sqrt{1-a},~\text{ or }\s =2,~ \eta <\frac32,\\[5pt]
        c_\star t-\ell(t) &\text{ if }\s = 2\sqrt{a} + 2\sqrt{1-a},
    \end{cases}
\end{equation} 
 with $\ell$ as in \eqref{m^w-2}.
Let $n$ be the solution to \eqref{eq:mainshift}.

\medskip

\noindent{{\bf \# Step 1.}} Start with the lower estimate. By \Cref{prop:subsolutionshiftallbeta}, we have
\begin{equation}\label{e.0823.9}
    n(t,s(t)) \gtrsim t^{-\frac32 + \frac{\s \eta}{2}} e^{-\left[\frac{\s^2}{4}-1\right]t}.
\end{equation}
Since $n_t \geq n_{xx} + n(1-a - n)$ in the sectorial region $\left\{x < s(t)\right\}$, we can apply \Cref{lem:liminf-case} to conclude immediately that 
\begin{equation}\label{e.iii}
\liminf_{t\to\infty }\left[\inf_{ {-\infty} < x \leq s(t)}\left(n(t,x )  -\Phi_{1-a}^\star\left(x- X(t)-C\right)\right)\right] \geq  0.
\end{equation}

\medskip

\noindent{{\bf \# Step 2.}}  We now turn to the upper estimate. 

Our strategy is to glue the supersolution from \Cref{prop:supersolutionshiftallbeta} with the one constructed in  (i) \Cref{lem:2A} in case $\s >2$; or (ii) \Cref{prop:supersolshiftbeta2} in case $\s=2$.

\noindent By \Cref{prop:supersolutionshiftallbeta}, for $t$ large enough, and any $\bar y \geq 0$
\begin{equation}
\overline \zeta(t,y):= (1+t)^{\frac{\s \eta}{2}}e^{-\left[\frac{\s^2}{4}-1\right]t} e^{-\frac{\s}{2}(y + \bar y)}\overline{\Psi}\left(t,y + \bar y \right)
\end{equation}
satisfies 
\begin{equation}\label{e.930.2c}
\overline\zeta_t - \dot{s} \overline\zeta_y- \overline\zeta_{yy} \geq (\g(y+ \bar y) - \overline\zeta) \overline\zeta \geq (\g(y) - \overline\zeta) \overline\zeta,\quad \text{ for }~ y \geq -\bar y.
\end{equation}
Moreover, we have
\begin{equation}\label{e.911.2}
    \lim_{t\to\infty} (1+t)^{\frac32 - \frac{\s\eta}{2}} e^{\left(\frac{\s^2}{4}-1\right)t}\zeta(t,y ) = F(y):=(y+ \bar y) e^{-(y + \bar y)\frac{\s}{2}} \quad \text{ for each }y \geq -\bar y. 
\end{equation}

\medskip
{\bf \#\# Case 1: $2<\s \leq \fraksc$.} 
\medskip
Set 
\begin{equation}
q_B = -\frac32 + \sqrt a \eta, \qquad \lambda_\star = \frac{\s}2 - \sqrt a \in (1-\sqrt a, \sqrt{1-a}],\qquad \mu = \frac{\sqrt a}2.
\end{equation}
By replacing $\g$ by a translation, and fixing $x_0>1$ large, we may assume without loss that
\begin{equation}
-\mu^2 < 1-a-\g(0) \leq 1-a-\g(y) \leq 0 \quad \text{ holds for }y \leq 0, \qquad  \frac{2|q_B|}{x_0} < \frac{\sqrt a}2.
\end{equation}
Let $\overline{m}_\mu$ be given by \Cref{prop:supercorrector} with $\sigma = (1-a-\g\mathbf{1}_{(-\infty,0]}$, and let $\psi$ be the solution to the homogeneous FKPP equation $\psi_t - \psi_{xx} = \psi(1-a-\psi)$ in $[0,\infty)\times \mathbb R$ where the initial datum $\psi_0 \in \mathsf{L}^\infty(\mathbb R)$ is smooth, 
and satisfies
\begin{align}
&\psi_0(x) = \alpha x^{q_B} e^{-\lambda_\star x} \quad \quad ~~\text{ for }x \geq x_0,\qquad \liminf_{x \to -\infty}\psi(0,x) >0 \quad \text{ for }x <-x_0,\\
   & \frac{\psi'_0(x)}{\psi_0(x)} \geq -\frac{2|q_B|}{x_0} - \lambda_\star \quad \text{ for } x \in \mathbb R,
\end{align}
where $\alpha>0$ will be chosen later.
Since $\lambda_\star \in (0,\sqrt{1-a}]$, we can apply \Cref{lem:C1}(i) to obtain 
\begin{equation}\label{e.930.4a}
    \lim_{t\to\infty} \sup_{\frac{x}{t} \in [c_\star + \delta, c_\star + \frac{1}{\delta}]} \left| (x-2\lambda_\star t)^{-q_B} e^{\lambda_\star x- (\lambda_\star^2 + 1-a)t } \psi(t,x)- \alpha\right| = 0 \quad \text{ for each } \delta \in (0,1),
\end{equation}
which implies, denoting $A_\psi = (\s -2\lambda_\star)^{q_B}$, 
\begin{equation}\label{e.930.2a}
 \lim_{t\to\infty}(1+t)^{\tfrac32 - \tfrac{\s \eta}2} e^{\left[\tfrac{\s^2}4 -1\right]t} \psi(t,s(t) + y)       = \alpha A_\psi e^{-\lambda_\star y} \qquad \text{ for each }y \in \mathbb{R}.
\end{equation}
We claim that there exists $T_1>0$ such that
\begin{equation}\label{c.930.1}
   \inf_{x \leq s(t)} \left(\frac{\psi_x}{\psi} + \frac{\dot s}{2} \right)> \mu \qquad \text{ for }t \geq T_1.
\end{equation}
This follows from applying the parabolic maximum principle to $W := \frac{\psi_x}{\psi}$, which satisfies 
\begin{equation}
W_t = W_{xx} + 2WW_x -\psi W, \quad W(0,x) \geq -\frac{2|q_B|}{x_0}- \lambda_\star.
\end{equation}
By our choice of $\mu$, we may fix $T_2>0$ such that 
\begin{equation}\label{e.26.11}
\inf_{x \in \mathbb{R}} \left(\frac{\psi_x}{\psi}(t,\cdot) + \frac{\dot s(t)}2 \right)\geq - \frac{2|q_B|}{x_0}-\left(\frac{\s}{2} - \sqrt{a}\right)  + \frac{\s}2 - \frac{\eta}{2(1+t)}  > \mu \quad \text{ for }t \geq T_2.
\end{equation}
Fix $\alpha>0$ small and then $\bar{y} > \frac2{\s}$ large so that
\begin{equation}\label{e.930.2}
\frac2{\s e} > \alpha A_\psi, \qquad 
    \bar y \exp\left( -\frac{\s \bar y}2 \right) < \alpha A_\psi \bar{m}_\mu(0) e^{- \lambda_\star \bar y + \frac{2\lambda_\star}{\s}}.
\end{equation}
Define
\begin{equation}
\overline n(t,y) = \overline{m}_\mu(y)\psi(t,s(t) + y + \bar y - \frac{2}{\s}).
\end{equation}
By \Cref{lem:2A}, for any $K>1$ large enough, $K\overline n$
is a supersolution to 
\begin{equation}
(K\overline{n})_t - \dot{s} (K\overline{n})_y- (K\overline{n})_{yy} \geq (K\overline{n})(\g - K\overline{n}) \quad \text{ for }t \gg 1,~y \leq 0. \label{e.930.2b}
\end{equation}
Then we verify that for any $K>1$ large enough, the following is a supersolution:
\begin{equation}
    \hat n(t,y):=
    \begin{cases}
        K\overline\zeta(t,y),
        & y\geq0,\\[5pt]
        K\min\bigl\{
            \overline\zeta(t,y),\,\overline n(t,y)
        \bigr\},
        & -\bar y+\dfrac{2}{\s}\leq y<0,\\
        K\overline n(t,y),
        & y<-\bar y+\dfrac{2}{\s}.
    \end{cases}
\end{equation}
Notice that $K\overline\zeta$ and $K\overline n$ are classical
supersolutions of \eqref{eq:mainshift}, written in the $(t,y)$ moving frame,
on the domains
\begin{equation}
    \left\{(t,y):~ t\geq T,~y\geq-\bar y+\tfrac{2}{\s}\right\}
    \qquad\text{and}\qquad
    \{(t,y):~ t\geq T,~y\leq0\},
\end{equation}
respectively, by \eqref{e.930.2c} and \eqref{e.930.2b}. We glue these two
locally defined classical supersolutions on their overlapping region by
taking their pointwise minimum. The resulting function $\hat n$ is
continuous, by \eqref{e.911.2} and \eqref{e.930.2a}, although it need not be
smooth across the interfaces where the two branches meet. It nevertheless
defines a generalised supersolution; see, for instance,
\cite[Definition~1.1.1]{Lam2022asymptotic}. 

Both $\overline\zeta(t,y)$ and $\overline{n}(t,y)$  satisfy the differential inequalities in their respective domains (thanks to \eqref{e.930.2c} and \eqref{e.930.2b}).

Set $y_L:=-\bar y+\frac{2}{\s}<0$. By
\eqref{e.911.2} and \eqref{e.930.2a},
for every fixed $y\in[y_L,0]$,
\begin{align}
&\lim_{t\to\infty}
(1+t)^{\frac32-\frac{\s\eta}{2}}
e^{\left(\frac{\s^2}{4}-1\right)t}
\left(\overline\zeta(t,y)-\overline n(t,y)\right)\\
&\quad=
(y+\bar y)e^{-\frac{\s}{2}(y+\bar y)}
-\alpha A_\psi\overline m_\mu(y)
e^{-\lambda_\star(y+\bar y-\frac{2}{\s})}.
\end{align}
Based on our choice of parameters in \eqref{e.930.2}, at $y=0$ this limit equals
\begin{equation}
\bar y e^{-\s\frac{\bar y}{2}}
-\alpha A_\psi\overline m_\mu(0)
e^{-\lambda_\star\bar y+2\frac{\lambda_\star}{\s}}
<0,
\end{equation}
whereas at $y=y_L$ it equals
\begin{equation}
\frac{2}{\s e}
-\alpha A_\psi\overline m_\mu(y_L)>0.
\end{equation}
Consequently, there exists $T \gg 1$ such that 
\begin{equation}
\overline\zeta(t,0)<\overline n(t,0),
\qquad
\overline\zeta(t,y_L)>\overline n(t,y_L) \qquad \text{ for }t \geq T.
\end{equation}
Thus the pointwise minimum agrees with
$\overline\zeta$ near $y=0$ and with
$\overline n$ near $y=y_L$.
It follows that the piecewise function
$\hat n$ is continuous and is a generalised supersolution
on $[T,\infty)\times \mathbb R$.

Since $\tau(T)>T$,
the Gaussian tail of $\overline\zeta(T,\cdot)$
decays more slowly than that of $n(T,\cdot)$
as $y\to+\infty$. 
Together with positivity and continuity
on bounded intervals, these properties
allow us to increase $K$ so that
\begin{equation}
n(T,s(T)+y)\le\hat n(T,y),
\qquad y\in\mathbb R.
\end{equation}

Hence, we can apply comparison to deduce that $n(t,y + s(t)) \leq \hat n(t,y)$ for $t\gg 1$, and $y \in \mathbb{R}$. 
In particular, whenever $\s = \lim_{t\to\infty} \dot{s} >2$, we have
\begin{equation}\label{e.102.4}
n(t,s(t)+ y ) \lesssim (1 + y) t^{-\tfrac32 + \tfrac{\s \eta}2} e^{- \left[ \frac{\s^2}4 -1\right]t} e^{-\frac{\s}2 y} \qquad \text{ for }y\geq 0.
\end{equation}

Moreover, 
the level set of $n(t,y + s(t))$ is bounded from above by that of $\hat{n}(t,y)$ which is at a bounded distance from $\psi(t,y + s(t))$. By \Cref{lem:C1}(ii), we deduce
\begin{equation}
\xi_\kappa(t) \leq X(t) + O(1).
\end{equation}
Since the convergence to a travelling wave can be derived from the level set estimate as in the proof of \Cref{thm:standard},  we have proved  \Cref{thm:cpulling}, as well as \Cref{thm:pulling} in the case $\s \in (2,\fraksc)$.

\begin{remark}\label{e.1011.5}
    We make two observations here for later purposes. First, using  
    \Cref{lem:C1}(i) and \eqref{e.102.4}, for each $\overline{c} \in (2, \s)$, we have
\begin{equation}\label{e.102.13}
    n(t, \overline{c} t) \lesssim \psi(t,\overline{c}t +O(1) )\lesssim t^{q_B} e^{-\lambda_\star(\bar c - c_\star) t} \qquad \text{ with }\quad q_B = -\frac32 + \sqrt a \eta. 
\end{equation}
Next, based on \eqref{e.102.4} and \Cref{prop:subsolutionshiftallbeta}, for any $(\s,a)$ such that $2<\s< \fraksc$,
    \begin{equation}
n(t,s(t)) \asymp_{\s,a} (1+t)^{-\tfrac32 + \tfrac{\s \eta}2} e^{- \left[ \tfrac{\s^2}4-1\right]t} 
    \end{equation}
\end{remark}

\medskip

{\bf \#\# Case 2: $\s =2$ and $\eta < \tfrac32$.}
\medskip
\noindent In this case, we need to use the refined supersolution $\overline\Phi$ from \Cref{lem:3A} for $t > T$ large enough. Choose 
\begin{equation}
\lambda = 1-\sqrt a \quad \text{ and }\quad q= -\frac32 + \eta \sqrt a
\end{equation}
with $\delta \in (0,\hat\delta]$ where $\hat\delta$ is specified by \Cref{prop:supersolshiftbeta2} as soon as $\lambda=1-\sqrt a$ is fixed.

Fix $t_0$ sufficiently large so that
\Cref{prop:supersolshiftbeta2} applies,
and set
\begin{equation}
\overline\Psi_1(t,y)
:=\overline\Phi(t,y+s(t)),
\end{equation}
where $\overline\Phi$ is constructed
using the reference time $t_0$ and is defined in \begin{equation}
\Omega_0 = \left\{(t,x):~ t\geq t_0,~ b_{t_0}(t) \leq x\leq s(t)\right\}.\end{equation} 

Choose $\bar y>1$ and then $T>t_0$
sufficiently large that
\begin{equation}
b_{t_0}(T)-s(T)<1-\bar y.
\end{equation}
Since $s(t)-b_{t_0}(t)$ is increasing,
we have
\begin{equation}
b_{t_0}(t)-s(t)<1-\bar y,
\qquad t\ge T.
\end{equation}
Consequently, the entire overlap
$1-\bar y\le y\le0$ lies in the domain
of $\overline\Psi_1$ for all $t\ge T$.

Define
\begin{equation}
    \hat{n}(t,y) :=  \begin{cases}
    K\overline\zeta(t,y) &\text{ for } y \geq 0,\\
    K\min\left\{ \overline\zeta(t,y),~  \overline\Psi_1(t,y)\right\} &\text{ for } -\bar{y} + 1 \leq y < 0,\\
    K\overline\Psi_1(t,y) &\text{ for }b_{t_0}(t) - s(t) \leq y <-\bar{y} + 1 .
    \end{cases} 
\end{equation}

As in Case 1, the glued function is a generalised
supersolution in
\begin{equation}
\mathcal D_T
:=\{(t,x):t\ge T,\ x\ge b_{t_0}(t)\}.
\end{equation}
By increasing its multiplicative constant,
we can arrange
\begin{equation}
\hat n(T,x-s(T))\ge n(T,x),
\qquad x\ge b_{t_0}(T),
\end{equation}
and
\begin{equation}
\hat n(t,b_{t_0}(t)-s(t))
\ge1\ge n(t,b_{t_0}(t)),
\qquad t\ge T.
\end{equation}
The comparison principle in the moving
domain $\mathcal D_T$ therefore gives
\begin{equation}
n(t,x)\le\hat n(t,x-s(t)),
\qquad (t,x)\in\mathcal D_T.
\end{equation}
We deduce the remaining inequality in \eqref{e.level.3}:
\begin{equation}
\xi_\kappa(t) \leq b_{t_0}(t)  + O(1)\quad \text{ for }t \gg 1.
\end{equation}
Since $|X(t) - b_{t_0}(t)|$ is uniformly bounded, 
this proves the estimate of level sets.

\medskip

\noindent{{\bf \# Step 3.}} To prove convergence to the travelling-wave profile,  consider any subsequential limits of
\begin{equation}
w(t,y) = \lim_{t_j \to +\infty} n(t+ t_j, y + X(t + t_j)),
\end{equation}
then
\begin{equation}
w_t - w_{yy} - c_\star w_y = w(1-a -w)  \qquad \text{ in }\mathbb{R}^2
\end{equation}
By Step 1, we have
\begin{equation}\label{e.116.1a}
w(t,y) \geq \Phi^\star_{1-a}(x+C_1) 
\quad \text{ in }\mathbb{R}\times [1,\infty),
\end{equation}
whereas in Step 2, we have, in Case 1, 
\begin{align}
n(t+ t_j, y + X(t + t_j)) \leq \bar{n} (t+ t_j, y  -s(t+t_j)+ X(t+t_j)) \lesssim\psi(t+ t_j, y + X(t+t_j)),
\end{align}
for $0 \leq y \leq s(t+t_j) - X(t+t_j)$, where $s(t) - X(t) \to +\infty$.
Hence, we use the convergence of $\psi(t+t_j,y+X(t+t_j))$ to the travelling-wave profile $\Phi_{1-a}$ (due to \Cref{lem:C1}(ii)) to deduce that 
\begin{equation}\label{e.116.1b}
    w(t,y)\le \Phi^\star_{1-a}(x - C_2) \quad \text{ in }\mathbb R\times [1,\infty); 
\end{equation}
whereas in Case 2,
\begin{equation}
n(t+ t_j, y + X(t + t_j)) \leq
    \overline\Phi(t+ t_j, y + X(t + t_j)) \quad \text{ for } ~ b_{t_0}(t+t_j) - X(t+t_j) \leq y \leq s(t+t_j) - X(t+t_j).
\end{equation}
Here, using $s(t) - X(t) \to +\infty$, and in addition that $b_{t_0}(t) - X(t)$ remains bounded, we deduce that there exists $y_0 >0$ such that, uniformly in $t\in \mathbb{R}$, 
\begin{equation}\label{e.116.1c}
 w(t,y)\lesssim e^{-\lambda_\star y} \text{ in } \mathbb R\times [y_0,\infty).
\end{equation}
The convergence to the travelling-wave profile follows from \eqref{e.116.1a}, \eqref{e.116.1b} in case 1 (resp. \eqref{e.116.1a}, \eqref{e.116.1c} in case 2), upon applying \Cref{p.109.1}.
\end{proof}

\begin{proof}[{\bf Proof of \Cref{thm:no-pull}}]
We deduce \Cref{thm:no-pull} from \Cref{thm:cpulling}.
On the one hand, the upper bound for the level-set position $\xi_\kappa(t)$ follows from \Cref{thm:cpulling} used with the shift $t \mapsto \fraksc t-\eta_c\ln(1+t)$, for some $\eta_c<-\frac1{2\sqrt a}$. The level set in that case bounds $\xi_\kappa(t)$ from above, thanks to
\begin{equation}
g(x-s(t)) <g(x-\fraksc t+\eta_c\ln(1+t) + C)  \quad \text{ for all }t\geq 0,~ x \in \mathbb R,
\end{equation}
for some $C>0$.
Hence, setting $X(t) = 2\sqrt{1-a} t - \frac{3}{2\sqrt{1-a}}\ln (1+t)$, we have
\begin{equation}
\xi_\kappa(t) \leq X(t) + O(1) \quad \text{ for }t \gg 1.
\end{equation}
On the other hand,  observe that $n_t - n_{xx} \geq n(1-a-n)$, so the level set is also bounded below by the homogeneous problem analysed in \cite{bramson1983convergence,hamel2013short}:
\begin{equation}
\xi_\kappa(t) \geq 2\sqrt{1-a}  t- \frac{3}{2\sqrt{1-a}}\ln t + O(1) \quad \text{ for }t \gg 1.
\end{equation}
The above reasoning also implies that any subsequential limits of 
\begin{equation}
w_j(t,y) = n\left(t + t_j, y + X(t+t_j)\right)
 \qquad t_j \to +\infty,
 \end{equation}
are trapped between two translates of travelling wave solutions $\Phi^\star_{1-a}(y \pm C)$. The convergence to travelling profiles then follows from \Cref{p.109.1}. 
This completes the proof.
\end{proof}

\section{Selection of a forced wave at the borderline regime:   \texorpdfstring{$\s = 2$}{s = 2} and \texorpdfstring{$\eta = \frac{3}{2}$}{eta = {3}{2}}}\label{sec:selection}

In this section, we extend aspects of some works by Hamel, Nolen, Roquejoffre and Ryzhik \cite{hamel2013short,hamel2016logarithmic}, Nolen, Roquejoffre and Ryzhik \cite{nolen2017convergence}, and Berestycki and Fang \cite{berestycki2018forced}. Unless otherwise stated, we enforce $\s =2$ and $\eta = \frac32$ throughout this section.

\begin{lemma}\label{lem:profiles}
The decaying positive bounded profiles solving  \eqref{eq:profile} are parameterised by their $A\in(0,\infty)$ in
\eqref{eq:profile-limits}. Moreover,
\begin{equation}\label{eq:W-tail}
(e^\xi U_A)'\to A,\qquad e^\xi U_A(\xi)=A\xi+o(\xi)\quad \text{ as } \xi \to+\infty,
\end{equation} 
and
\begin{equation}\label{eq:W-bounds}
\sup_{\xi\in\R}\frac{e^\xi\left(U_A(\xi)+|U_A'(\xi)|\right)}{\mathfrak{m}(\xi)}<\infty.
\end{equation}
\end{lemma}
\begin{proof}[{\bf Proof of \Cref{lem:profiles}}] This goes back to \cite{berestycki2018forced}. The existence is \cite[Theorem 1.1 and Theorem 1.4]{berestycki2018forced}. 
Note that
\begin{align}\label{e.104.1}
\left(e^\xi U_{A}\right)_{\xi \xi} &= (1 - \g + U_{A})e^\xi U_{A}  .
\end{align}
Since $U_A$ is bounded, $e^\xi U_A,  (e^\xi U_A)' \to 0$ at $-\infty$. Integrating over $\R$, one obtains 
\begin{equation}
    A := (e^\xi U_A)'(+\infty) = \int_\R (1 - \g(\xi) + U_{A}(\xi))e^\xi U_{A}(\xi) \, \dd \xi.
\end{equation}
The map $U_A \mapsto A$ is thus monotonically increasing: it provides a parametrisation.
The asymptotics can now be  derived as in {\Cref{prop:subcorrector}}, where we view $e^\xi U_A$ and $1-\g + U_A$ as playing the roles of $\underline{m}$ and $\sigma$ in $\eqref{m-eq}$, respectively. To repeat the argument therein, it remains to prove an \textit{a priori} exponential bound on $U_{A}$ (which gives finiteness of the first moment). To obtain this, just observe that 
\begin{align}
\frac{\dd}{\dd \xi}\frac{\left(e^\xi U_{A}\right)_{\xi}}{e^\xi U_{A}}& = \frac{\left(e^\xi U_{A}\right)_{\xi \xi}\left(e^\xi U_{A}\right)-\left(e^\xi U_{A}\right)_{\xi}^2}{\left(e^\xi U_{A}\right)^2} =  \frac{(1 - \g + U_{A}) \left(e^\xi U_{A}\right)^2-\left(e^\xi U_{A}\right)_{\xi}^2}{\left(e^\xi U_{A}\right)^2}\\
&= (1 - \g + U_{A}) - \left(\frac{\left(e^\xi U_{A}\right)_{\xi}}{e^\xi U_{A}}\right)^2.
\end{align}
Since, for large $\xi$, say $\xi \geq \xi_0$, $(1 - \g + U_{A})(\xi) \leq \frac14$, we deduce that $\frac{\left(e^\xi U_{A}\right)_{\xi}}{e^\xi U_{A}} \leq \frac12$ for $\xi$ even larger, since $\xi \mapsto \frac12 \tanh\left(\tanh^{-1}\left(2\frac{\left(e^\xi U_{A}\right)_{\xi}}{e^\xi U_{A}}(\xi_0)\right)+ \frac{\xi}{2}\right)$ is then an upper barrier on $[\xi_0,\infty)$ once $\frac{\left(e^\xi U_{A}\right)_{\xi}}{e^\xi U_{A}}(\xi_0) < \frac12$. 
\end{proof}

\medskip

We are ready to prove \Cref{thm:selection}.

\begin{proof}[{\bf Proof of \Cref{thm:selection}}]

\noindent{{\bf \# Step 1. The value of $A_\infty$.}} 

\medskip

Writing $\eqref{eq:mainshift}$ in the moving coordinate $y=x-s(t)$, we obtain
\begin{equation}
\zeta_t - \dot s \zeta_y - \zeta_{yy} = (\g - \zeta) \zeta \quad \text{ for }t>0,~y \in \mathbb{R}.
\end{equation}
Following the standard leading-edge transformation (see \cite{hamel2013short,hamel2016logarithmic}), 
we shall pull out an exponential factor by setting \begin{equation}
\zeta(t,y) = e^{-y}\Psi(t,y).\end{equation}
We then obtain the following equation for $\Psi$:
\begin{align}\label{eq:Psi}
\Psi_t - \Psi_{yy} + \frac32\frac{1}{1+t} (\Psi_y - \Psi) + \left(1-\g \right) \Psi + e^{-y}\Psi^2&= 0,\qquad \text{ for }t>0,~y \in \mathbb{R}.
\end{align}

Let  $\overline{\Psi}(t,y)$ be defined in Proposition~\ref{prop:supersolutionshiftallbeta} with $\s=2$ and $\eta=\frac{3}{2}$.
We next set 
\begin{equation}
{\Psi}_1(t,y):=(1+t)^{\tfrac32}\overline{\Psi}(t,y+1),\quad y\geq 0,
\end{equation}
so that ${\Psi}_1(t,0)>0$. Since $\g$ is monotone increasing,
\begin{equation}
(\Psi_1)t - (\Psi_1)_{yy} + \frac32\frac{1}{1+t}((\Psi_1)_y-\Psi_1) + \left(1-\g \right) \Psi_1\geq [\g(y+1)-\g(y)]\Psi_1\geq 0\quad y\geq 0.
\end{equation}
Thus
we can easily  see that $K{\Psi}_1$ is a supersolution of \eqref{eq:Psi} for any $K>1$.
Furthermore, observing that $1+t\leq \tau(t):=1+t+\sqrt{1+t}\leq 2(1+t)$, we see that the  exact form of  $\overline{\Psi}$ in Proposition~\ref{prop:supersolutionshiftallbeta} implies that $\inf_{t\geq 0}{\Psi}_1(t,0)>0$. Therefore, we can choose $K>0$ sufficiently large such that 
\begin{align}
K {\Psi}_1(t,0)\geq 1  \geq \Psi(t,0)\quad \text{ for } t\geq 0;\quad 
K{\Psi}_1(0,y)\geq \Psi(0,y)\quad \text{ for } y\geq 0,
\end{align}
where we also used that $\Psi(0,y)$ is compactly supported. Therefore, we can apply the comparison principle to conclude that
$\Psi\leq K{\Psi}_1$ on $[0,\infty)^2$. 

Next, choose $T_0>0$ such that 
\begin{equation}
\tau(t) \leq \tfrac32 (1+t)\quad  \text{ for }t \geq T_0 \qquad  \text{(using}~ \tau(t) \sim (1+t)).\end{equation}
These estimates and the exact form of  $\Psi_1(t,y)$ and $\tau(t) \sim (1+t)$ imply that 
there exists $C>0$ such that 
\begin{equation}
\Psi(t,y)\leq K\frac{(1+t)^{\frac{3}{2}}(y+1)}{\tau(t)^{\frac{3}{2}}}
  \exp\left[-\frac{(y+4)^2}{4\tau(t)}
                  -\frac{17}{\sqrt{\tau(t)}}\right]
  \lesssim (1+y)e^{-\frac{y^2}{6(1+t)}},
\end{equation}
for $t>T_0$ and $y>0$. 
On the other hand, since $0\leq n\leq 1$, we have
\begin{equation}
0\leq \Psi(t,y)=e^{y}n(t,s(t)+y)\leq e^{y},\qquad \text{ for }t\in \mathbb R,~ y < 0.
\end{equation}
Consequently, there is a constant $C$ independent of $t \in [T_0,\infty)$ such that
\begin{equation}\label{Psi-bound0930}
  0\le\Psi(t,y)\lesssim
  \begin{cases}
    (1+y)\exp\left(-\dfrac{y^2}{6(1+t)}\right)&\quad \text{ for } (t,y)\in [T_0,\infty) \times \mathbb{R}^+,\\
    e^y&\quad \text{ for } (t,y)\in\mathbb{R}^+ \times \mathbb{R}^-.
  \end{cases}
\end{equation}

Define
\begin{equation}\label{eq:M}
M(t)=(1+t)^{-\frac32}\int_\R \mathfrak{m}(y)\Psi(t,y)\dd y.
\end{equation}
By Lemma~\ref{lem:m-correction0930} and the bound \eqref{Psi-bound0930}, there exists $K_0>0$ such that $0\leq M(t)\leq K_0$ for all $t\geq 0$.
We claim that $M(t)$ converges. By direct computation, we obtain
\begin{equation}\label{eq:Mprime}
M'(t)=(1+t)^{-\frac32}\left[
\frac{3}{2(1+t)}\int_\R \mathfrak{m}'(y) \Psi(t,y) \dd y
-\int_\R \mathfrak{m}(y)e^{-x}\Psi(t,y)^2\dd y\right],
\end{equation}
where we have used  integration by parts and 
$\mathfrak{m}''=(1-\g)\mathfrak{m}$. 

By Lemma~\ref{lem:m-correction0930} and the bound \eqref{Psi-bound0930} again,
\begin{align}
(1+t)^{\tfrac52}M'(t)\lesssim  \int_\R \mathfrak{m}'(y) \Psi(t,y) \dd y &\lesssim 1 +\int_{0}^{\infty}(1+x)e^{-\frac{y^2}{6(1+t)}}\dd y\\ &= 1 +\int_{0}^{\infty}e^{-\frac{y^2}{6(1+t)}}\dd y + \int_{0}^{\infty}ye^{-\frac{y^2}{6(1+t)}}\dd y \\
&\lesssim 1+t.
\end{align}
Hence, $M'(t) \lesssim (1+t)^{-\tfrac32}$. To summarise,
\begin{equation}
0 \leq M(t)\leq K_0, \qquad  \int_0^\infty \max\{0,M'(t)\}\,dt <+\infty,
\end{equation}
which implies that $ A_\infty:= \frac{1}{2\sqrt\pi}\lim\limits_{t\to\infty}M(t)$ exists, and belongs to $[0,\infty)$.

It remains to prove that this limit is positive. For this, we use the subsolution from \Cref{prop:subsolutionshiftallbeta} to estimate
\begin{align}\label{eq:1M}
M(t) = (1+t)^{-\frac32} \int_{\mathbb R} \mathfrak m (y)\Psi(t,y)\,\dd y &\geq \int_0^\infty \mathfrak{m}(y)\underline\Psi(t,y)\dd y \\
&\geq \tau(t)^{-\frac32}\int_0^{+\infty} y^2\exp \left( - \frac{(y - y(t))^2}{4\tau(t)}\right) \dd y \gtrsim 1,
\end{align}
where we performed the change of variables $y={z}\sqrt{1+t}$.

\medskip

\noindent{{\bf \# Step 2. Transforming into self-similar variables.}}

\medskip

We change variables to self-similar variables (diffusive scaling) on the positive half-line. 

\begin{equation}
    \tau = \ln(1+t), \qquad z=\frac{y}{(1+t)^{\frac{1}{2}}}, \qquad \Psi_{\mathrm{SS}}(\tau,z)=e^{-\frac{\tau}{2}}\Psi(e^\tau-1,e^{\frac{\tau}{2}}z).
\end{equation}
(The time variable $\tau$ here is to be distinguished from the $\tau(t)$ which only appears in \# Step 1.)

Then \eqref{eq:Psi} becomes 
\begin{equation}\label{eq:PsiSS}
(\Psi_{\mathrm{SS}})_\tau=(\Psi_{\mathrm{SS}})_{zz}+\frac{z}2(\Psi_{\mathrm{SS}})_z+ \Psi_{\mathrm{SS}} 
-e^{\tau}(1-\g)(e^{\frac{\tau}{2}}z) \Psi_{\mathrm{SS}} -\frac{3}{2}e^{-\frac{\tau}{2}}(\Psi_{\mathrm{SS}})_z
-e^{\frac{3}{2}\tau}e^{-z e^{\frac{\tau}{2}}} \Psi_{\mathrm{SS}} ^2.
\end{equation}
By \eqref{Psi-bound0930}, there is a constant $C_0>0$  
such that
\begin{align}\label{cptness-1}
  0\leq \Psi_{\mathrm{SS}}(\tau,z) \leq C_0(e^{-\frac{\tau}{2}}+z)e^{-\frac{z^2}{6}},\qquad \tau >\ln(T_0 +1),~z\geq 0.  
\end{align}
Observe {from (G2)} that the monotonicity and moment assumptions of $1-\g$ imply that $e^{\tau}(1-\g)(e^{\frac{\tau}{2}}z)$ goes to zero locally uniformly on $\R_+^\star$:
\begin{equation}\label{cptness-2}
e^{\tau}(1-\g)(e^{\frac{\tau}{2}}z) = (e^{\frac{\tau}{2}})^2 (1-\g)(e^{\frac{\tau}{2}}z)\leq\frac8{3z^2}\int_{\frac{e^{\frac{\tau}{2}}z}{2}}^{e^{\frac{\tau}{2}}z} y(1-\g)(y)\dd y\longrightarrow_{\tau \to \infty}0.
\end{equation}
Therefore, the last three terms in \eqref{eq:PsiSS} all vanish, in the weak sense, {on compact intervals of $z$ as $\tau \to +\infty$.
}

For any $\theta \in \R$, and $\tau_n \geq - \theta$ a sequence of times going to $+\infty$, one may consider $$\Psi_{\mathrm{SS}}^{(n)}: (\theta,z) \mapsto \Psi_{\mathrm{SS}}(\theta + \tau_n,z).$$ By the sliding method (parabolic regularity), 
every sequence of time translates has a convergent subsequence $\Psi_{\mathrm{SS}}^{(n)}$ in $\mathscr{C}_{\rm loc}(\R \times \R_+^\star)$ with the limit being an entire solution to 
\begin{equation}\label{eq:PsiSSinfty}
\begin{cases}
(\Psi_{\mathrm{SS}}^{\infty})_\theta=(\Psi_{\mathrm{SS}}^{\infty})_{zz}+\frac{z}2(\Psi_{\mathrm{SS}}^{\infty})_z+\Psi_{\mathrm{SS}}^{\infty},& \qquad  \theta \in \R, z>0,      \\[5pt]
\Psi_{\mathrm{SS}}^{\infty}(\theta,0) = 0,& \qquad \theta \in \R,\\
   0\leq \Psi^{\infty}_{\mathrm{SS}}(\theta,z)\lesssim ze^{-\frac{z^2}{6}},& \qquad \theta \in \mathbb{R},~z\geq 0.
    \end{cases}
\end{equation}
where $C_0>0$ is independent of $\theta$ and the Dirichlet boundary condition is inherited 
from \eqref{cptness-1}. We will show that $\Psi^{\infty}_{\mathrm{SS}}$ is independent of $\theta$ and hence it is a constant multiple of the principal eigenfunction of  \eqref{eq:PsiSS}.

\medskip

\noindent{{\bf \# Step 3. The selection of $A_\infty$.}} 
In this step, we prove that
\begin{equation}\label{moment-Id}
    \int_{\R^+} z \Psi_{\mathrm{SS}}^{\infty}(\theta,z) \, \dd z = 2\sqrt{\pi}A_\infty, \qquad \text{for each } \theta \in \R.
\end{equation}

First, we claim
\begin{equation}\label{M-t}
M(t)=\int_0^\infty\frac{\mathfrak{m}(\sqrt{1+t} \,z)}{\sqrt{1+t}}
\Psi_{\mathrm{SS}}(\tau,z)\dd z+O(t^{-\frac32}). 
\end{equation}
To see \eqref{M-t}, we first focus on $y<0$. Since $\Psi(t,y)\leq e^{y}$ and $0<\mathfrak{m}$ is monotone increasing,
\begin{equation}
0\le (1+t)^{-\frac{3}{2}}\int_{-\infty}^0\mathfrak{ m}(y)\Psi(t,y)\dd y
\le (1+t)^{-\frac{3}{2}}\mathfrak{m}(0)\int_{-\infty}^0e^y \dd y=\mathfrak{m}(0)(1+t)^{-\frac{3}{2}}.
\end{equation}

For $y>0$, let $y=z\sqrt{1+t}$. A direct computation gives 
\begin{equation}
(1+t)^{-\frac{3}{2}}\int_0^\infty\mathfrak{m}(y)\Psi(t,y)\dd y=\int_0^\infty\frac{\mathfrak{m}(z\sqrt{1+t})}{\sqrt{1+t}}\Psi_{SS}(\ln(1+t),z)\dd z.
\end{equation}
This gives \eqref{M-t}.

Now, taking the limit in \eqref{M-t}, we need two properties of $\mathfrak{m}$ from Lemma~\ref{lem:m-correction0930}:
\begin{align}\label{key-m}
    \mathfrak{m}(y)\lesssim (1+y),\quad y\geq 0;\quad \mathfrak{m}(y)\sim y\quad \text{ as } y\to\infty.
\end{align}

Fix $z>0$. By \eqref{key-m}, we have 
\begin{align}
&\frac{\mathfrak{m}(z\sqrt{1+t})}{\sqrt{1+t}}
=z\frac{\mathfrak m(z\sqrt{1+t})}{z\sqrt{(1+t)}}
\to z,\quad \text{ as } t\to\infty,\\
&0\leq\frac{\mathfrak{m}(z\sqrt{1+t} )}{\sqrt{1+t}}\lesssim\frac{(1+z\sqrt{1+t})}{\sqrt{1+t}}=(1+t)^{-\frac{1}{2}}+z,\quad z>0
\end{align}

Together with \eqref{cptness-1}, we have 
\begin{equation}
\frac{\mathfrak{m}(z\sqrt{1+t})}{\sqrt{1+t}}\Psi_{SS}(\ln(1+t),z)\lesssim (1+z)^2e^{-\frac{z^2}{6}},\quad t\geq 0,\ z\geq0.
\end{equation}
Thus, we can use Lebesgue's dominated convergence theorem in \eqref{M-t} to derive \eqref{moment-Id}.

\medskip

\noindent{{\bf \# Step 4. Characterisation of $\Psi^\infty_{\rm SS}(\theta,z)$.}} In this step, we use the spectral decomposition argument used in the proof of \cite[Lemma~2.2]{hamel2013short} to show that $\Psi^{\infty}_{\mathrm{SS}}$ is stationary and satisfies
\begin{equation}\label{exact-PSIss}
    \Psi_{\mathrm{SS}}^{\infty}(\theta,z) = A_\infty z e^{-\frac{z^2}{4}},  \quad z \geq  0. 
\end{equation}
Since this limit is independent of the subsequence $\tau_n \to \infty$ in the definition of $\Psi^\infty_{\rm SS}$, we have
\begin{equation}\label{e.104.5}
  \lim_{\tau \to +\infty}  \Psi_{\rm SS}(\tau,z)  = A_\infty z e^{-\tfrac{z^2}4} \qquad \text{ in }\mathscr{C}([0,\infty)).
\end{equation}

We proceed to prove \eqref{exact-PSIss}. Set
\begin{equation}
w(\theta,z):=e^{\frac{z^2}{8}}\Psi^\infty_{\rm SS}(\theta,z).
\end{equation}
It follows from \eqref{eq:PsiSSinfty} that
\begin{equation}\label{w-L2}
    \sup_{\theta\in\mathbb R}
    \|w(\theta,\cdot)\|_{L^2(0,\infty)}<\infty,
\end{equation}
and that $w$ is a weak $L^2$-solution of
\begin{equation}\label{e.104.2}
    w_\theta+\mathcal Mw=0,\qquad
    w(\theta,0)=0,\qquad
    \theta\in\mathbb R,\ z>0,
    \qquad
    \mathcal M:=-\partial_{zz}+\frac{z^2}{16}-\frac34.
\end{equation}

We introduce the associated quadratic form
\begin{equation}
Q(w):=\langle \mathcal Mw,w\rangle_{L^2(0,\infty)}
=\int_0^\infty
\left(
    |w_z|^2+
    \left(\frac{z^2}{16}-\frac34\right)w^2
\right)\,\dd z,
\end{equation}
with form domain
\begin{equation}
\mathcal D(Q)
=
\left\{
w\in H_0^1(0,\infty): zw\in L^2(0,\infty)
\right\}.
\end{equation}
The operator $\mathcal M$ is self-adjoint and nonnegative, with
\begin{equation}
\ker\mathcal M=\operatorname{span}\{e_0\},
\qquad
e_0(z)=\frac{z e^{-\frac{z^2}{8}}}{\sqrt{2\sqrt{\pi}}}.
\end{equation}
The operator $\mathcal M$ is the Dirichlet half-line harmonic oscillator,
shifted by $-\frac{3}{4}$. Using the standard spectral decomposition of the harmonic
oscillator in Hermite functions (see, e.g., \cite[Theorem 8.5]{MR3243083}),
and restricting the odd eigenfunctions to $(0,\infty)$, one obtains
$
\sigma(\mathcal M)=\{0,1,2,\ldots\}.$ 

Hence if we define
\begin{equation}
w^\perp = w - \langle w, {\bf e}_0\rangle {\bf e}_0, 
\end{equation}
then
\begin{equation}\label{e.104.3}
    Q(w)\ge \|w^\perp\|_{L^2((0,\infty))}^2,
    \qquad
    \text{ for }w\in H^1_0((0,\infty)),~ zw \in L^2((0,\infty)).
\end{equation}

Multiplying \eqref{e.104.2} by $w^\perp$ and integrating by parts, we have
\begin{equation}
\frac{d}{d\theta} \|w^\perp(\theta,\cdot)\|^2_{L^2} = -2Q(w^\perp (\tau,\cdot)) \leq  -2\|w^\perp(\theta,\cdot)\|^2_{L^2}\quad \text{ for }\theta \in \mathbb{R},
\end{equation}
where we used \eqref{e.104.3}.
Combining this inequality with \eqref{w-L2}, we deduce that $w^\perp(\theta,\cdot)  \equiv 0$ for all $\theta \in \mathbb{R}$, and thus
\begin{equation}
w(\theta,z) = \widetilde C z e^{-\tfrac{z^2}8} \qquad \text{ for some constant }\widetilde C.
\end{equation}
Using \eqref{moment-Id}, we derive \eqref{exact-PSIss}.

Note that it would be possible to derive the bound directly by using the heat kernel in the half space, together with explicit \textit{a priori} bounds on it. 

\medskip

\noindent{{\bf \# Step 5. The convergence step (so-called ``tail-to-interior" as in \cite{nolen2017convergence})}}

\medskip

We now transfer the estimate at the diffusive scale to an estimate on compact sets, using a suitable supersolution. Define the difference between the solution and the wave, renormalised by $\mathfrak{m}$, in the spirit of \cite{nolen2017convergence}:
\begin{equation}
    \mathcal{E}(t,y) = \frac{
    {\Psi(t,y)}- e^yU_{A_\infty}(y)}{\mathfrak{m}(y)}
\end{equation}
The goal of this step is to show that, 
for each fixed $\epsilon$, there exists $\underline{T}(\epsilon)>0$ such that for all $T \geq \underline{T}$, 
\begin{align}\label{E-bound1}
\left\vert \mathcal{E}(T,y) \right\vert {\leq M} \left(e^{-\epsilon (1-\sqrt{a}) \sqrt{1+\tfrac{{T}}{2}}} + \Lambda_T(\epsilon)+ \epsilon^2 +e^{-\mu_T \tfrac{T}2}\right).      
\end{align}
We will then take $T\to+\infty$ and then $\epsilon \to 0$.

For this purpose, write $\Psi = \mathfrak{m}\mathcal{E} + e^yU_{A_\infty}(y)$ and substitute into \eqref{eq:Psi},
\begin{align}
\Psi_t &= \mathfrak{m}\mathcal{E}_t\\
\Psi_y &= \mathfrak{m}'\mathcal{E} +\mathfrak{m}\mathcal{E}_y + (e^yU_{A_\infty})_y\\
\Psi_{yy} &= \mathfrak{m}''\mathcal{E} +2\mathfrak{m}'\mathcal{E}_y + \mathfrak{m}\mathcal{E}_{yy}+ (e^yU_{A_\infty})_{yy}
\end{align}
which leads to 
\begin{align}
&\mathcal{E}_t - \mathcal{E}_{yy} - \left( 2 \frac{\mathfrak{m}'}{\mathfrak{m}} - \frac32\frac{1}{1+t}\right) \mathcal{E}_y  + \left( \frac32\frac{1}{1+t} \left[\frac{\mathfrak{m}'}{\mathfrak{m}}-1 \right]\right)\mathcal{E} \\
&= \frac{(e^yU_{A_\infty})_{yy}-\left(1-\g \right) ( e^yU_{A_\infty}(y))-  e^{-y}\Psi^2}{\mathfrak{m}} -\frac32\frac{1}{1+t}\frac{e^y(U_{A_\infty})_{y}}{\mathfrak{m}},\\
&= \frac{e^y U_{A_\infty}^2-  e^{-y}\Psi^2}{\mathfrak{m}} -\frac32\frac{1}{1+t}\frac{e^y(U_{A_\infty})_{y}}{\mathfrak{m}}= -\mathcal{E}( U_{A_\infty}+ e^{-y}\Psi) -\frac32\frac{1}{1+t}\frac{e^y(U_{A_\infty})_{y}}{\mathfrak{m}},
\end{align}
recalling $(e^yU_{A_\infty})_{yy}=\left(1-\g +U_{A_\infty}\right) ( e^yU_{A_\infty}(y))$ and
\begin{align}
e^y U_{A_\infty}^2-  e^{-y}\Psi^2 & = e^{-y}( (e^y U_{A_\infty})^2- \Psi^2) \\
&= e^{-y}( e^y U_{A_\infty} - \Psi)( e^y U_{A_\infty}+ \Psi) = -\mathfrak{m}\mathcal{E}( U_{A_\infty}+ e^{-y}\Psi).     
\end{align} 
In the end,  
\begin{align}
&\mathcal{E}_t - \mathcal{E}_{yy} - \left( 2 \frac{\mathfrak{m}'}{\mathfrak{m}} - \frac32\frac{1}{1+t}\right) \mathcal{E}_y  + \left( \frac32\frac{1}{1+t} \left[\frac{\mathfrak{m}'}{\mathfrak{m}}-1 \right] + U_{A_\infty}+ e^{-y}\Psi\right)\mathcal{E} = -\frac32\frac{1}{1+t}\frac{e^y(U_{A_\infty})_{y}}{\mathfrak{m}}.
\end{align}
Observe that, by the Kato inequality, 
\begin{align}
&\left(\left\vert \mathcal{E} \right\vert\right)_t - \left(\left\vert \mathcal{E} \right\vert\right)_{yy} - \left( 2 \frac{\mathfrak{m}'}{\mathfrak{m}} - \frac32\frac{1}{1+t}\right) \left(\left\vert \mathcal{E} \right\vert\right)_y  \\
&\qquad \leq -\left( \frac32\frac{1}{1+t} \left[\frac{\mathfrak{m}'}{\mathfrak{m}}-1 \right] + U_{A_\infty}+ e^{-y}\Psi\right)\left\vert \mathcal{E} \right\vert  + \frac32\frac{1}{1+t}\left\vert \frac{e^y(U_{A_\infty})_{y}}{\mathfrak{m}} \right\vert \\& \qquad \leq  \frac32\frac{1}{1+t} \left\vert \mathcal{E} \right\vert + \frac{C}{1+t}.
\end{align}
Hence,
\begin{equation}\label{e.PPP}
 \mathcal{P}\left\vert \mathcal{E} \right\vert \leq \frac{C}{1+t},\qquad \text{ where}\quad    \mathcal{P}:=\partial_t-\partial_{yy}
-\left(2\frac{\mathfrak{m}'}{\mathfrak{m}}-\frac{3}{2(1+t)}\right)\partial_y-\frac{3}{2(1+t)},
\end{equation}
Let us construct a function $\overline{\mathcal{E}}$ such that
\begin{align}
& \mathcal{P}\overline{\mathcal{E}}\geq  \frac{C}{1+t}.
\end{align}
on the parabolic domain $$\Omega_{{T}} = \left\lbrace (t,y) \in \R^+ \times \R  , \;t \in  \left[\frac{{T}}{2},{T}\right] \quad \text{and} \quad y \in [-\epsilon \sqrt{1+t},\epsilon \sqrt{1+t}] \right\rbrace.$$

Observe that, for any $t \in [\frac{{T}}{2},{T}]$,
\begin{equation}\label{e.104.8a}
    \left\vert \mathcal{E}(t,-\epsilon \sqrt{1+t}) \right\vert =  \left\vert  \frac{n(t,s(t)-\epsilon \sqrt{1+t}) - U_{A_\infty}(-\epsilon \sqrt{1+t})}{e^{\epsilon \sqrt{1+t}} \mathfrak{m}(-\epsilon \sqrt{1+t})} \right\vert \lesssim e^{-\epsilon (1-\sqrt{a}) \sqrt{1+t}}, 
\end{equation}
uniformly in $t$ and
\begin{align}
     \left\vert \mathcal{E}  (t,\epsilon \sqrt{1+t})\right\vert &=  \left\vert \frac{\Psi(t,\epsilon \sqrt{1+t}) - e^{\epsilon \sqrt{1+t}}U_{A_\infty}(\epsilon \sqrt{1+t})}{\mathfrak{m}(\epsilon \sqrt{1+t})} \right\vert\\
    &= \frac{\epsilon \sqrt{1+t}}{\mathfrak{m}(\epsilon \sqrt{1+t})}  \left\vert \frac{\Psi(t,\epsilon \sqrt{1+t}) - e^{\epsilon \sqrt{1+t}}U_{A_\infty}(\epsilon \sqrt{1+t})}{\epsilon \sqrt{1+t}}\right\vert
\end{align}
But, (using \eqref{e.104.5} in particular)
\begin{align}
&\frac{\epsilon \sqrt{1+t}}{\mathfrak{m}(\epsilon \sqrt{1+t})} = 1 + o_t(1),\\ 
&\frac{\Psi(t,\epsilon \sqrt{1+t})}{\epsilon \sqrt{1+t}} = \frac{\Psi_{\mathrm{SS}}\left({\ln(1+t)},\epsilon\right)}{\epsilon} \to A_\infty e^{-\frac{\epsilon^2}{4}}, \quad
\frac{e^{\epsilon \sqrt{1+t}}U_{A_\infty}(\epsilon \sqrt{1+t})}{\epsilon \sqrt{1+t}}  
\longrightarrow A_\infty,
\end{align}
as $t\to\infty$ and hence 
also uniformly for $t\in[\frac{T}{2},{T}]$ as $T\to\infty$.
Consequently, defining 
\begin{align}\label{Lambda-T}
\Lambda_{T}(\epsilon):=\sup_{\tau \in \left[\frac{{T}}{2},{T}\right]} |\mathcal E(\tau,\epsilon\sqrt{1+\tau})|,
\end{align}
we deduce, for each small $\epsilon>0$,
\begin{align}\label{Lambda-limit}
\Lambda_{{T}}(\epsilon)\to A_\infty(1-e^{-\frac{\epsilon^2}{4}})
\quad\text{as }T\to\infty.    
\end{align}
The above fact will not be needed immediately, but will become useful at the end of the proof.

Note that, for any $\epsilon' \in (-1,1)$,
\begin{equation}\label{e.104.8b}
    \left\vert \mathcal{E} \right\vert\left(\tfrac{{T}}{2}, \epsilon' \sqrt{1+\tfrac{{T}}{2}}\right) = \left|\frac{ \Psi\left(\frac{{T}}{2},\epsilon'\sqrt{1+\tfrac{{T}}{2}} \right) - e^{\epsilon' \sqrt{1+\tfrac{{T}}{2}}} U_{A_\infty}\left(\epsilon' \sqrt{1+\tfrac{{T}}{2}}\right)}{\mathfrak{m}\left(\epsilon' \sqrt{1+\tfrac{{T}}{2}}\right)} \right|\lesssim 1.
\end{equation}

Define 
$$\overline{\mathcal{E}}(t,y) := M\left(e^{-\epsilon (1-\sqrt{a}) \sqrt{1+\tfrac{{T}}{2}}} + 
\Lambda_{T}(\epsilon)
+ \epsilon^2 + 
e^{-\mu_{T}(t-\frac{{T}}{2})}
\right)\mathcal{E}_{\mathrm{cos}}(y) 
,$$
where $M>0$ will be determined later,
\begin{equation}
\mu_{T}:=\frac{\pi^2}{128\epsilon^2(1+T)},\quad 
\mathcal{E}_{\mathrm{cos}}(y) :=\cos\left(\frac{\pi}{8} \frac{y+\epsilon \sqrt{1+{T}}}{\epsilon \sqrt{1+{T}}}\right).
\end{equation}

This function has just a slight dependency with respect to $t$.

Note that, on $\Omega_T$, we have
\begin{equation}
0\le\frac{\pi}{8}\frac{y+\epsilon\sqrt{1+T}}{\epsilon\sqrt{1+T}}\le\frac{\pi}{4},
\quad
\frac1{\sqrt2}\le\mathcal E_{\mathrm{cos}}\le1.
\end{equation}

Observe that the bounds \eqref{e.104.8a} and \eqref{e.104.8b} are independent of $\epsilon$ and $T$. Hence, by construction, there exists $M_1$ independent of $\epsilon$ such that for $M \geq M_1$,
\begin{align}
&\left\vert\mathcal{E}(\tau,-\epsilon \sqrt{1+\tau}) \right\vert \leq \overline{\mathcal{E}}(\tau,-\epsilon \sqrt{1+\tau}),  \quad \left\vert\mathcal{E}(\tau,\epsilon \sqrt{1+\tau}) \right\vert \leq \overline{\mathcal{E}}(\tau,\epsilon \sqrt{1+\tau}), \quad \tau \in \left[\frac{{T}}{2},{T}\right],  \\
&\left\vert \mathcal{E} \right\vert\left(\tfrac{{T}}{2}, \epsilon' \sqrt{1+\tfrac{{T}}{2}}\right) \leq  \overline{\mathcal{E}}\left(\tfrac{{T}}{2}, \epsilon' \sqrt{1+\tfrac{{T}}{2}}\right), \quad  \epsilon' \in (-\epsilon,\epsilon).
\end{align}
A direct computation gives 
\begin{align}
\mathcal{P}\mathcal{E}_{\mathrm{cos}}&= (\mathcal{E}_{\mathrm{cos}})_t - (\mathcal{E}_{\mathrm{cos}})_{yy} - \left( 2 \frac{\mathfrak{m}'}{\mathfrak{m}} - \frac32\frac{1}{1+t}\right) (\mathcal{E}_{\mathrm{cos}})_y - \frac32\frac{1}{1+t} \mathcal{E}_{\mathrm{cos}}\\
&= \left(\frac{\pi^2}{64 \epsilon^2 (1+{T})} - \frac32\frac{1}{1+t} \right)\mathcal{E}_{\mathrm{cos}} + \frac{\pi}{8 \epsilon (1+{T})^{\frac12}} \left( 2 \frac{\mathfrak{m}'}{\mathfrak{m}} - \frac32\frac{1}{1+t}\right) \sin\left(\frac{\pi}{8} \frac{y+\epsilon \sqrt{1+{T}}}{\epsilon \sqrt{1+{T}}}\right),\label{Ecos-ineq}
\end{align}
where we also used $(\mathcal{E}_{\mathrm{cos}})_t=0$.

Now fix $0<\epsilon\le\epsilon_0$, where
$\epsilon_0^2\le \frac{\pi^2}{384}$.
By Lemma~\ref{lem:m-correction0930},
$\frac{\mathfrak{m}'(y)}{\mathfrak{m}(y)}
\gtrsim \frac{1}{1+y_+}$. Since $t\ge \frac{T}{2}$ and $y_+\le\epsilon_0\sqrt{1+T}$, we further have 
\begin{equation}
2\frac{\mathfrak{m}'}{\mathfrak{m}}
-\frac{3}{2(1+t)}\geq\frac{2c_2}{1+\epsilon_0\sqrt{1+T}}-\frac3{1+T}\ge0
\end{equation}
for all large $T$, uniformly in $0<\epsilon\le\epsilon_0$. 
Thus, the second term in \eqref{Ecos-ineq} is nonnegative. Furthermore,
\begin{align}
\frac{\pi^2}{64\epsilon^2(1+T)}-\frac{3}{2(1+t)}\geq
\frac{\pi^2}{64\epsilon^2(1+T)}
-\frac3{1+T}\geq
\frac{\pi^2}{128\epsilon^2(1+T)}=\mu_T.    
\end{align}

We obtain
\begin{equation}
\mathcal{P}\mathcal E_{\mathrm{cos}}
\ge\mu_T\mathcal E_{\mathrm{cos}} \quad \Longrightarrow \quad \mathcal{P} \left( e^{-\mu_T(t-\frac{T}{2})} \mathcal{E}_{\mathrm{cos}} \right) \geq 0 \quad \text{ in }\Omega_T.
\end{equation}
Due to this inequality, we obtain the desired differential inequality: 
\begin{align}
\mathcal{P}\overline{\mathcal E}
&\ge M\mu_T\left(e^{-\epsilon(1-\sqrt a)\sqrt{1+\frac{T}{2}}}
+\Lambda_T(\epsilon)+\epsilon^2\right)\mathcal {E}_{\mathrm{cos}}\geq\frac{M\pi^2}{128\sqrt2(1+T)}\geq
\frac{C}{1+t}  \qquad \text{ in }\Omega_T,
\end{align}
provided $M= \max\{M_1,M_2\}$, with   $M_2 = \frac{256\sqrt2\,C}{\pi^2}$. 

As a conclusion of that step, by the parabolic comparison principle in $\Omega_T$, and then setting $t=T$, we obtain in particular that for each small and fixed $\epsilon$, there exists $\underline{T}(\epsilon)>0$ such that for all $T \geq \underline{T}$,
\begin{align}\label{E-bound}
\left\vert \mathcal{E}({T},y) \right\vert \leq \overline{\mathcal{E}}({T},y)   {\leq M} \left(e^{-\epsilon (1-\sqrt{a}) \sqrt{1+\tfrac{{T}}{2}}} + 
    {\Lambda_T(\epsilon)}+ \epsilon^2 
    +e^{-\mu_T \tfrac{T}2}
    \right),   
\end{align}
where we used $\mathcal{E}_{\rm cos} \leq 1$ and set $t=T$. 
\medskip

\noindent{{\bf \# Step 6. Conclusion.}}
\medskip

\noindent Fix $R>0$. {For each fixed $0<\epsilon\le\epsilon_0$, the interval $[-R,R]\subset[-\epsilon\sqrt{1+T},\epsilon\sqrt{1+T}]$ for all sufficiently
large $T$.}

Taking $T\to\infty$ in \eqref{E-bound} and using \eqref{Lambda-limit}, it follows that
\begin{equation}
\limsup_{T\to\infty}\sup_{|y|\le R}|\mathcal{E}(T,y)|
\leq M\left(A_\infty(1-e^{-\frac{\epsilon^2}{4}})+\epsilon^2 + e^{-\frac{\pi^2}{256 \epsilon^2}}
\right).
\end{equation}
where we used $\mu_T T \to \frac{\pi^2}{128 \epsilon^2}$ as $T\to\infty$.
Since $M$ is independent of $\epsilon$, letting $\epsilon\searrow 0 $ gives
$\lim\limits_{t\to\infty}\sup\limits_{|y|\le R}|\mathcal{E}(t,y)|=0$.

This implies 
\begin{equation}
\sup_{|y|\leq R}
\left\vert n(t,s(t)+y)-U_{A_\infty}(y)\right\vert \leq \left(\sup_{|y|\leq R} \mathfrak{m}(y)e^{-y}\right) \cdot \sup_{|y|\leq R} \left\vert\mathcal{E}(t,y)\right\vert {\to0}, 
\end{equation}
as $t\to\infty$.
This completes the proof.
\end{proof}

\section{Construction of supersolutions to the competition system}\label{sec:supersolsystem}

In this section, we construct an explicit supersolution to \eqref{e.com} in the sense of systems.   
It is worth recalling that this system is not cooperative; however, the system satisfied by $(u,1-v)$ is. As a consequence, by a supersolution to \eqref{e.com} we mean a supersolution to the equation for $u$ and a subsolution to the equation for $v$; see \cite[Section 2.1]{girardin2019invasion} and \cite[Chapter 3.3]{lam2022introduction} for the formulation of generalised sub- and supersolutions for competition systems.

We will treat both monostable cases $0<a<1<b$ and $0<a,b<1$. We first give some definitions and lemmas that work for both cases, and then discuss separately the construction for the two cases.

\begin{definition}
    We say that $(\overline{U},\underline{V})$ is a supersolution of \eqref{e.com} presented in the shifted frame $\xi = x-ct - \rho(t)$ in the domain $\Omega$ if the pair satisfies the following differential inequalities there:
\begin{equation}\label{e.com.xi}
\begin{cases}
\mathcal{N}_1[\overline{U},\underline{V}]:= \overline{U}_t - (c+\rho'(t))\overline{U}_\xi - \overline{U}_{\xi\xi} - \overline{U}(1-\overline{U} - a\underline{V}) \geq 0, \\[8pt]
\mathcal{N}_2[\overline{U},\underline{V}]:= \underline{V}_t - (c+\rho'(t))\underline{V}_\xi - d\underline{V}_{\xi\xi} - r\underline{V}(1-b\overline{U} - \underline{V}) \leq 0. \end{cases}
\end{equation}
\end{definition}

The basic building blocks are the travelling wave solutions of a perturbed version of \eqref{e.com}, introduced in \cite{girardin2019invasion}. 
Precisely, for given $0<a<1$, $b \in (0,\infty)\setminus\{1\}$ and $\alpha>0$ sufficiently small (that is $\alpha\in(0,\frac12)$ in case $0<a<1<b$ and \eqref{e.919.7} in case $0<a,b<1$), consider
\begin{equation}\label{e.com.eta}
u_t=u_{xx}+u\left(1-\dfrac{u}{1+\alpha}-av\right) \quad \text{ and }\quad 
v_t=dv_{xx}+rv(1-v-bu). 
\end{equation}

\begin{definition}\label{d.0829.1} For $0<a<1<b$  (resp. $0<a,b<1$), 
let $c^\alpha_{\mathrm{LLW}} \in [2\sqrt{1-a}, 2]$ be the minimal speed such that a travelling wave connecting 
\begin{equation}
(1+\alpha,0)~ \text{ ~ to ~}~ (0,1) \quad \qquad \left(\text{resp. } \quad (U^\star,V^\star):=\left(\frac{1-a}{\tfrac{1}{1+\alpha}-ab},\frac{\tfrac{1}{1+\alpha}-b}{\tfrac{1}{1+\alpha}-ab}\right) ~ \text{ ~ to ~}~ (0,1) \right),
\end{equation}
which exists if and only if $c \geq c^\alpha_{\rm LLW}$ \cite[Chapter 2, Theorem 2.2]{Volpert1994travelling}. 
For each $c > c^\alpha_{\mathrm{LLW}}$, let $(U_{c}^{\alpha},V_{c}^{\alpha})$ be the monotone travelling-wave profile of the $\alpha$-perturbed system connecting $(1+\alpha,0)$ to $(0,1)$ (resp. $(U^\star,V^\star)$ to $(0,1)$) with speed $c$. 
Its profile equations are
\begin{equation}\label{e.TWeta}
\begin{cases}
(U_c^\alpha)''
+c(U_c^\alpha)'
+U_c^\alpha
\left(
1-\dfrac{U_c^\alpha}{1+\alpha}-aV_{c}^{\alpha}
\right)=0, &\text{ for }\xi \in \mathbb{R},\\[5pt]
d(V_c^\alpha)''
+c(V_c^\alpha)'
+rV_c^\alpha(1-bU_c^\alpha-V_c^\alpha)=0, &\text{ for }\xi \in \mathbb{R}.
\end{cases}
\end{equation}
\end{definition}

\begin{lemma}\label{l.0907.1}
Let $c^0_{\rm LLW}$ denote the minimal wave speed when $\alpha = 0$. For each $c > c^0_{\rm LLW}$, there exists $\alpha >0$ such that $c > c^\alpha_{\rm LLW}$, i.e.   such that $(U^\alpha_c, V^\alpha_c)$ exists and that 
\begin{equation}\label{e.0907.1}
\lim_{\xi \to +\infty}    e^{\lambda_c \xi}U^\alpha_c(\xi)= a_0>0  \quad \text{ with }\quad \lambda_c = \frac{c- \sqrt{c^2 - 4(1-a)}}2.
\end{equation}
Moreover, fixing $a_0 \geq 1$ large enough, we have
\begin{equation}\label{e.0816.1}
\left\{
\begin{aligned}
\inf_{\xi \leq 1} U_c^\alpha(\xi)
&\geq 1 + \frac{2\alpha}{3},
\qquad
\sup_{\xi \leq 1} V_c^\alpha(\xi)
\leq \frac{\alpha}{8},
\\[8pt]
\inf_{\xi \in (-\infty,1]}
\frac{(V_c^\alpha)'(\xi)}{V_c^\alpha(\xi)}
&> \mu_1
:= \frac{1}{2d}
\left[
-c + \sqrt{
c^2 + 4rd\left(
b\left(1+\tfrac{3}{4}\alpha\right)-1
\right)}
\right].
\end{aligned}
\right.
\end{equation}
in case $0<a<1<b$;
\begin{equation}\label{e.0816.1'}
\inf_{\xi \leq 1} U_c^\alpha(\xi) \geq \frac{1-a}{\tfrac{1}{1+\frac{\alpha}{2}}-ab} \quad \text{ and }\quad \sup_{\xi \leq 1} V_c^\alpha(\xi) \leq \frac{\tfrac1{1+\frac{\alpha}{2}} -b}{\tfrac1{1+\frac{\alpha}{2}} -ab} 
\end{equation}
in case $0<a,b<1$.
\end{lemma}
Importantly, we also normalise other travelling waves with $c > c^\alpha_{\rm LLW}$ by \eqref{e.0907.1}, so that the derivative of $(U^\alpha_c,V^\alpha_c)$ with respect to $c$ is well defined.

\begin{proof}[{\bf Proof of \Cref{l.0907.1}}]
The existence of $(U^\alpha_c, V^\alpha_c)$ follows from the continuous dependence of $\alpha \mapsto c^\alpha_{\rm LLW}$  \cite[Chapter 3, Theorem 4.2]{Volpert1994travelling}. Note that the linearisation of the first equation at $(0,1)$ is independent of $\alpha \geq 0$, and therefore \eqref{e.0907.1} holds. Finally, \eqref{e.0816.1} follows from the standard asymptotics of travelling waves as $\xi \to -\infty$.
\end{proof}

Before going to the construction, we record some properties satisfied by $(U_{c}^{\alpha},V_{c}^{\alpha})$.
\begin{lemma}\label{lem.0815.1}
The family $(U_{c}^{\alpha},V_{c}^{\alpha})$ is strictly monotone in $\xi$, i.e.
\begin{equation}\label{e.4aa}
(U_{c}^{\alpha})' <0 < (V_{c}^{\alpha})' \quad \text{ for }\xi \in \mathbb{R}.
\end{equation}
Next, define
\begin{equation}\label{P-Q-def}
P(\xi):=A(U^\alpha_c)'-\partial_cU_{c}^{\alpha}(\xi),
\qquad
Q(\xi):=A(V^\alpha_c)'-\partial_cV_{c}^{\alpha}(\xi),
\end{equation}
then there exists $A>0$ such that
\begin{equation}
P< 0 \text{ in } [-1,\infty), \qquad Q> 0
\quad {\text{ in } [-1,1].} 
\end{equation}
and
\begin{equation}\label{e.4}
P(\xi)
=
O\left((1+\xi)e^{-\lambda_c\xi}\right),
\qquad
P'(\xi)
=
O\left((1+\xi)e^{-\lambda_c\xi}\right) \quad \text{ as } \xi\to+\infty.
\end{equation}
In fact,  there exists $C_0>1$ such that
\begin{align}
&U_{c}^{\alpha} + |P|+|P'| \le
C_0(1+|\xi|)(-(U_{c}^{\alpha})') \quad \mathrm{ for } \; \xi  \geq -1,\label{e.9} \\[5pt] 
&(1-V_{c}^{\alpha}) + |Q|+|Q'| \leq C_0(1+|\xi|)(V_{c}^{\alpha})' \quad \mathrm{ for } \; \xi  \geq -1\label{e.9b}, 
\end{align}
and in particular, we have
\begin{equation}\label{e.9c}
{\mathbf{1}_{\xi \geq -1}} U_{c}^{\alpha} |P| + {
\frac{P^2}{1+\alpha}+a\left|PQ
\right\vert}
\le
C_0(-(U_{c}^{\alpha})') \qquad |Q^2+bPQ| \le
C_0(V_{c}^{\alpha})' ,
\end{equation}
for $\xi \geq -1$. 
\end{lemma}

\begin{proof}[{\bf Proof of \Cref{lem.0815.1}}] 
The strict spatial monotonicity follows from
\cite[Chapter 2, Theorem 2.2]{Volpert1994travelling}
and \cite{kan1995parameter}.

Upon introducing $W=1-V$, the
travelling-wave system becomes cooperative
and irreducible. Its translational mode
has components of the same strict sign.
The standard nondegeneracy theory for
monotone heteroclinic waves gives (see, e.g., \cite[Chapter~4, \S5]{Volpert1994travelling})
\begin{equation}
\ker\mathcal L_c
=\operatorname{span}
\{((U_c^\alpha)',-(V_c^\alpha)')\}
\end{equation}
in the cooperative variables.
After fixing the normalisation
\eqref{e.0907.1}, the standard
implicit-function theorem yields smooth
dependence on $c$.

Hence, $\partial_c U^\alpha_c$ is well-defined.

Differentiating \eqref{e.0907.1} with respect to $c$ gives
\begin{equation}
\partial_cU_{c}^{\alpha}(\xi)
\sim
-a_0\,\lambda_c'\,\xi e^{-\lambda_c\xi}
\qquad\text{as }\xi\to+\infty.
\end{equation}
Therefore,
\begin{equation}\label{e.P-leading}
-\partial_cU_{c}^{\alpha}(\xi)
\sim a_0\,
\lambda'_c\,\xi e^{-\lambda_c\xi}
=
-\frac{a_0\lambda_c^2}{1-a-\lambda_c^2}\,
\xi e^{-\lambda_c\xi},
\end{equation}
where
\begin{equation}
\lambda'_c
=
-\frac{\lambda_c}{c-2\lambda_c}
=
-\frac{\lambda_c^2}{1-a-\lambda_c^2}
<0.
\end{equation}
This proves $P<0$ for $\xi \gg 1$. Taking $A>0$ large enough, we can ensure that $P<0$ in $[-1,\infty)$ 
and also $Q>0$ in $[-1,1]$.

The estimates in \eqref{e.4} follow from the corresponding
travelling-wave asymptotics. Since
\begin{equation}
-(U_{c}^{\alpha})'(\xi)\sim a_0\,\lambda_c\,e^{-\lambda_c\xi}
\qquad\text{as }\xi\to+\infty,
\end{equation}
and all quantities involved are bounded on compact subsets,
\eqref{e.9}--\eqref{e.9b} follow after increasing $C_0$
if necessary.
\end{proof}
We are now ready to construct our supersolution (in the $(t,\xi)$ variables) step by step. We first treat the more delicate case $0<a<1<b$. 

\subsection{Supersolution for the case \texorpdfstring{$0 < a < 1 < b$}{0 < a < 1 < b}}

To increase transparency, we work in the Bramson frame $\xi=x-ct -\rho(t)$ and we 
split the entire construction into three steps. In \Cref{lem:U0V0}, we add an $O(t^{-1})$ correction to a travelling wave solution so it moves in the Bramson frame with nontrivial sublinear speed. 
In Lemmas \ref{lem.bar} and \ref{lem.qsmall}, we then perform two more truncations to obtain the final form of the supersolution for the competition system.

\begin{lemma}\label{lem:U0V0}
For each $c > c^0_{\rm LLW}$, there exists $\hat\delta>0$ small such that for any  $\delta \in (0,\hat\delta]$, and for any
 $\rho\in \mathscr{C}^1([0,\infty))$ satisfying
\begin{equation}\label{e.0816.2}
    |\rho'(t)| \leq \frac{\sqrt\delta}{1 +\delta  t} \quad \text{ for }t\geq 0,
\end{equation}
then there exists
a supersolution
$(\overline{U}_0,\underline V_0)$ presented in the shifted frame $\xi = x-ct - \rho(t)$ 
satisfying \eqref{e.com.xi} in
\begin{equation}
\Omega
=
\left\{
(t,\xi):
t\ge 0,\;
0\le \xi\le\delta t
\right\}.
\end{equation}
Moreover, with $\alpha>0$ as given in \Cref{l.0907.1}, the following properties are satisfied:
\begin{itemize}
    \item[{\rm(i)}] On $\left\{(t,\xi):~t>0,~0\leq \xi \leq 1\right\}$, \begin{equation}\label{e.sup.1} 
\overline{U}_0 > 1 + \frac{\alpha}{2} \quad \text{ and }\quad \underline{V}_0 < \frac{\alpha}4.
\end{equation}
\begin{equation}\label{e.0831.2}
\frac{\underline{V}_{0,\xi}}{\underline{V}_0} > \mu_1:=\frac{1}{2d}\left[-c+ \sqrt{c^2 + 4rd\left(b\left(1+\tfrac34\alpha\right)-1\right)}\right] \end{equation}

\item[{\rm(ii)}] $\overline{U}_0$ has the same spatial decay as the travelling-wave profile $U_c^\alpha$:
\begin{equation}\label{e.sup.2}
\limsup_{t\to\infty} \sup_{ \frac{\delta}{2}t \leq \xi \leq \delta t} e^{\lambda_{c} \xi}\overline{U}_0 < +\infty \quad \text{ with }\quad \lambda_c = \tfrac12 \left(c-\sqrt{c^2 - 4\left(1-a\right)}\right). 
\end{equation}

\item[{\rm(iii)}] $\overline{U}_0 >0$ and $\underline{V}_0 >0$ in $\Omega$.
\end{itemize}

\end{lemma}

\begin{proof}[{\bf Proof of \Cref{lem:U0V0}}]
Let $\alpha$ be as given in \Cref{l.0907.1}.
For $t\geq 0$ and $\xi \in \mathbb{R}$, define 
\begin{equation}\label{e.bar}
\overline{U}_0(t,\xi)
=
U_c^\alpha(\xi)+\frac{5\sqrt \delta }{1+\delta t}P(\xi),
\qquad
\underline V_0(t,\xi)
=
V_c^\alpha(\xi)+\frac{5\sqrt \delta }{1+\delta 
t}Q(\xi),
\end{equation}
where $(P,Q)$ is defined in \eqref{P-Q-def}.

We will later show that it forms a supersolution when restricted to $\Omega$. The motivation of this expansion comes from \cite{hamel2013short}, where it is shown that the logarithmically delayed invasion wave, when viewed in the self-similar variables \begin{equation}
(\tau,y) = \left(\ln(t_0+t),\frac{x-\xi_\kappa(t)}{\sqrt{t_0 + t}}\right),\end{equation} 
exhibits exponential convergence to an equilibrium state in the rescaled time $\tau = \ln(t_0 +t)$. This establishes that the optimal rate of convergence to the wave profile in the original time variable $t$ is of order $(t_0 + t)^{-1}$ for some large $t_0$, which is roughly equivalent to $\frac{\sqrt\delta}{1+\delta t}$ for fixed, small  $\delta$.

We first verify the assertions (i) and (ii) before verifying the differential inequality for supersolution.  To prove assertion (i), we use \eqref{e.0816.1} and choose $\delta \in (0,\hat\delta)$, with $\hat\delta$ being a small enough constant, to ensure that \eqref{e.sup.1} and \eqref{e.0831.2} hold uniformly for $t \geq 0$ and $|\xi|\leq 1$, as well as 
\begin{equation}\label{e.0815.3}
0<\sqrt\delta <  \frac{1}{55C_0 (1+r)}.
 \end{equation}
 Assertion (ii) follows from  
\eqref{e.0907.1} and \eqref{e.4}. Taking $\hat\delta>0$ smaller, we can ensure (iii) holds.

Next, we proceed to show that $(\overline U_0, \underline V_0)$ is a supersolution. 
Differentiating \eqref{e.TWeta} with respect to $c$, and using the fact
that $((U_c^\alpha)',(V_c^\alpha)')$ solves the homogeneous linearised
system, gives
\begin{equation}\label{e.Peq}
P''
+cP'
+
\left(
1-\frac{2U_c^\alpha}{1+\alpha}-aV_c^\alpha
\right)P
-aU_c^\alpha Q
=
(U_c^\alpha)',
\end{equation}
\begin{equation}\label{e.Qeq}
dQ''
+cQ'
+r(1-2V_c^\alpha-bU_c^\alpha)Q
-rbV_c^\alpha P
=
(V_c^\alpha)'.
\end{equation}
Denoting $\varepsilon(t) = \frac{5\sqrt \delta }{1+ \delta  t}$, then
\begin{equation}\label{e.ratio}
\varepsilon'(t) = -\frac{\sqrt\delta}5 \varepsilon(t)^2 \quad  \text{ and }\quad \sup_{0 \leq \xi \leq \delta t} \varepsilon(t)\cdot (1+\xi) \leq 5\sqrt\delta.
\end{equation}
A direct calculation using \eqref{e.TWeta}, \eqref{e.Peq} and $ \varepsilon'P \geq 0$ gives
\begin{align}\label{e.N1}
\mathcal{N}_1[\overline{U}_0,\underline V_0]
={}&
\frac{\alpha}{1+\alpha}\overline{U}_0^2
-\left(
 \rho'+ \varepsilon
\right)(U_c^\alpha)'
-{ \varepsilon \rho'}P'
+  \varepsilon' P
+ \varepsilon^2
\left(
\frac{P^2}{1+\alpha}+aPQ
\right)\\
\geq{}& \frac{\alpha}{1+\alpha}\overline{U}_0^2
-\left(
 \rho'+ \varepsilon
\right)(U_c^\alpha)'
-{ \varepsilon \rho'}P'
+a \varepsilon^2
PQ.
\end{align}
By \eqref{e.0816.2} and   $(U_c^\alpha)'<0$,
\begin{equation}\label{e.main}
-\left(
 \rho'+ \varepsilon
\right)(U_c^\alpha)'
\ge
\frac{4}{5}  \varepsilon(-(U_c^\alpha)').
\end{equation}

Dropping the first term on the right-hand side of \eqref{e.N1}
(which is nonnegative),  and using \eqref{e.9}--\eqref{e.9c}, we obtain from \eqref{e.N1}:
\begin{align}
\mathcal{N}_1[\overline{U}_0,\underline V_0]
&\ge 
\frac{4}{5}  \varepsilon
(-(U_c^\alpha)')
- \sqrt\delta C_0   \varepsilon(-(U_c^\alpha)') 
-  \varepsilon^2 C_0 (-(U_c^\alpha)')\\
&\ge
-(U_c^\alpha)' \varepsilon
\left[\frac45 -  C_0(\sqrt\delta + 5\sqrt\delta 
)\right].
\end{align}

Using  
\eqref{e.0815.3}, we thus have
\begin{equation}\label{used-prop7.8}
\mathcal{N}_1[\overline{U}_0,\underline V_0]
\ge
\frac{\sqrt\delta}{1+\delta t}(-(U_c^\alpha)')
\ge0,
\qquad
t\ge 0,\quad 0\le\xi\le\delta t.
\end{equation}

For the second component, \eqref{e.Qeq} gives
\begin{align}\label{e.N2}
\mathcal{N}_2[\overline{U}_0,\underline V_0]
={}&
-\left(
 \rho'+ \varepsilon
\right)(V_c^\alpha)'
- \varepsilon{ \rho'}Q'
+ \varepsilon'Q +{r \varepsilon^2}
\left(Q^2+bPQ\right).
\end{align}
Therefore, (recalling now  $(V_c^\alpha)'>0$,  \eqref{e.9b}--\eqref{e.9c})
\begin{align}
\mathcal{N}_2[\overline{U}_0,\underline V_0]
&\le
-\frac45  \varepsilon(V_c^\alpha)'
+ \frac15  \varepsilon^2 C_0(1+\xi) (V_c^\alpha)' + \sqrt\delta  \varepsilon^2 C_0(1+\xi)(V_c^\alpha)'
+ r \varepsilon^2 C_0(V_c^\alpha)'\\
&\le
{ \varepsilon}(V_c^\alpha)'
\left[-\frac4{5} + C_0( \sqrt\delta + 5\sqrt\delta + 5r \sqrt\delta)
\right].
\end{align}
Using the choice \eqref{e.0815.3}, we have
\begin{equation}\label{e.0907.3}
    \mathcal{N}_2[\overline{U}_0,\underline V_0]
\le
-\frac{\sqrt\delta}{1 +\delta t}(V_c^\alpha)' 
\le0,
\qquad \text{ for }
t\ge 0,\quad 0\le\xi\le\delta t.
\end{equation}
Hence $(\overline{U}_0,\underline V_0)$ is a supersolution of \eqref{e.com}
in $\Omega$.
\end{proof}

While $(\overline{U}_0,\underline{V}_0)$ is only defined in $t>0$, $0\leq \xi \leq \delta t$, we continue to build a supersolution in the larger domain $t>0$, $-\infty < \xi < \delta t$. Our strategy is to first truncate the component $U$ at $1+\frac{\alpha}4$ and then truncate the component $V$ to vanish for $\xi \ll 1$. The latter requires particular care and is only an issue for the case $b>1$.

\begin{lemma}\label{lem.bar}
Under the hypotheses of \Cref{lem:U0V0}, there exists $\hat \delta>0$ small such that for any $\delta \in (0,\hat\delta]$, there exists
a supersolution
$(\overline{U},\underline{V}_{1})$ of \eqref{e.com.xi} in $\{(t,\xi): t\geq0, \, -\infty<\xi\leq \delta t\}$
satisfying property {\rm(ii)} of \Cref{lem:U0V0} and, moreover,
\begin{itemize}
    \item[{\rm(i$'$)}] There exists a number $\underline\xi<0$ such that \begin{equation}\label{e.sup.1'} 
\overline{U} \equiv 1+\frac{\alpha}{4}\quad \text{ and }\quad \underline{V}_{1} \equiv 0 ~\text{ for }t>0,~\xi \leq \underline\xi ,
\end{equation}
where $\alpha$ is as given in \Cref{l.0907.1} and  $\underline\xi$ is a fixed constant depending on $\delta$.
\item[{\rm(ii$'$)}] $\overline{U}$ has the same spatial decay as the travelling-wave profile $U_c^\alpha$:
\begin{equation}\label{e.sup.2'}
\limsup_{t\to\infty} \sup_{\delta \frac{t}{2} \leq \xi \leq \delta t} e^{\lambda_{c} \xi}\overline{U} < +\infty \quad \text{ with }\quad \lambda_c = \tfrac12(c- \sqrt{c^2 - 4(1-a)}), 
\end{equation}
\end{itemize}
\end{lemma}

\begin{proof}[{\bf Proof of \Cref{lem.bar}}]

Start by observing simply that by defining
\begin{equation}
\overline{U}
=
\begin{cases}
   \displaystyle \min\left\{1 + \frac{\alpha}{4},\overline{U_0}\right\} &\text{ for }\xi \geq 0,\\[5pt]
\displaystyle 1 + \frac{\alpha}{4} &\text{ for }\xi < 0.
\end{cases}\end{equation}
From property (i) of \Cref{lem:U0V0}, 
it is clear that $\overline{U}$ is continuous: 
\begin{equation}\label{e.bar22}
\overline{U}(t,\xi) \equiv 1+\frac{\alpha}{4}\quad \text{ for }0\leq \xi \leq 1.
\end{equation}
Moreover, a direct verification of the differential inequalities shows that
$(\overline U,\underline V_0)$ remains a supersolution in the domain $\Omega=\left\{(t,x):~t\geq 0,\,0 \leq \xi \leq \delta t\right\}.$ Since this verification is immediate, we omit the details.

\medskip

We shall now truncate $\underline{V}_1$ to vanish for $\xi \ll -1$ (this is inspired by \cite{girardin2019invasion}). Define the continuous function
\begin{equation}
    \theta(\xi) =\begin{cases}
        (1+\delta) e^{\mu_+ \xi} -\delta e^{\mu_- \xi} >0  &\text{ for }\xi > \underline\xi,\\
        0 &\text{ for }\xi \leq \underline\xi,
    \end{cases}
\end{equation}
where
\begin{align}\label{e.0831.5}
&\underline\xi:= \frac{d}{\sqrt{(c-5\sqrt\delta)^2 + 4rd(b(1+\tfrac{\alpha}2})-1)} \ln \frac{\delta}{1+\delta},\\   
&\mu_\pm = \frac{1}{2d}\left(-\left(c-5\sqrt\delta\right) \pm \sqrt{\left(c-5\sqrt\delta\right)^2 + 4rd\left(b\left(1+\tfrac{\alpha}{
2}\right)-1\right)}\right)\notag.
\end{align}
By construction, $\theta$ is continuous in $\mathbb R$ and satisfies 
$$d\theta'' +(c -5\sqrt{\delta}) \theta'- r\theta (b(1+\tfrac{\alpha}{2})-1 ) = 0 \qquad \text{ for } \xi > \underline\xi.$$ 
By choosing $\delta_2>0$ smaller if necessary,
there exists $\mu_2 \in (0,\mu_1)$\footnote{Recall that $\mu_1 = \frac1{2d}\left( -c + \sqrt{c^2 + 4rd(b(1+3\frac{\alpha}{4})-1)}\right)$ is defined in \Cref{lem:U0V0}.} such that for any $\delta \in (0,\delta_2)$, we have
\begin{equation}\label{e.0907.2}
\sup_{|\xi|\leq 1}\frac{\theta'(\xi)}{\theta(\xi)} \leq \mu_2 < \mu_1.
\end{equation}
Indeed, it is clear that $\theta'>0$ wherever $\theta>0$. Moreover, if we send $\delta\searrow 0$, then 
\begin{equation}
\frac{\theta'(\xi)}{\theta(\xi)} \to \frac{1}{2d}\left(-c + \sqrt{c^2 + 4rd\left(b\left(1+\frac{\alpha}{
2}\right)-1\right)}\right)\in (0,\mu_1) \quad \text{  in } \mathscr{C}_{\text{loc}}(\mathbb{R}).
\end{equation}
This proves \eqref{e.0907.2} holds for all $\delta$ sufficiently small.

\begin{claim}\label{c.0907.1} 
There exists a $\hat\xi:[0,\infty) \to [0,1]$ such that 
\begin{equation}
\hat\xi(0) = 0, \quad \hat\xi'>0,  \quad 0 < \hat\xi(t) < \min\left\{1,\delta t\right\} \quad \text{ for }t>0,
\end{equation}
and 
\begin{equation}
G(t,\xi):= \ln \underline{V}_0(t,\xi) - \ln [\underline{V}_0(0,0)\theta(\xi))]  
\begin{cases}
 <0 \quad &\mathrm{ for }  \quad
 0 \leq \xi < \hat\xi(t), \\[5pt]
 =0 &\mathrm{ for }\quad \xi = \hat\xi(t), \\[5pt]
 >0 &\mathrm{ for } \quad\hat\xi(t) < \xi < 1.
\end{cases}
\end{equation}
\end{claim}
Indeed, thanks to \eqref{e.0831.2} and \eqref{e.0907.2}, there is at most one root since
\begin{equation}\label{e.0831.3}
G_\xi(t,\xi) \geq \mu_1 - \mu_2 >0 \quad  \text{ for }t>0,~\xi \in [0,1].
\end{equation}
Since $G(0,0) = 0$, we can apply the implicit function theorem to obtain the existence of $\hat\xi(t)$. Moreover, differentiating $G(t,\hat\xi(t)) \equiv 0$ we have
\begin{equation}\label{e.0831.4}
-\hat\xi'(t) G_\xi =  G_t= \frac{(\underline{V}_0)_t}{\underline{V}_0} = - \frac{5 \delta^{\frac{3}{2}}}{(1+\delta t)^2}\frac{Q(\xi)}{\underline{V}_0}<0.
\end{equation}
Combining with \eqref{e.0831.3}, we have $\hat\xi'>0$. Moreover, using the fact that $\underline{V}_0$ is bounded from below by a positive constant uniformly in $|\xi| \leq 1$, we have
\begin{equation}
0 < \hat\xi'(t) \leq \frac{C \delta^{\frac{3}{2}}}{(1+\delta t)^2} = C \sqrt\delta \frac{d}{dt}\left(1- \frac{1}{1+\delta t}\right).
\end{equation}
Upon integration and using $\hat{\xi}(0)=0$, we obtain
\begin{equation}
0<\hat{\xi}(t)
\leq C\sqrt{\delta}
\left(1-\frac{1}{1+\delta t}\right)
=
\frac{C\delta^{\frac{3}{2}}t}{1+\delta t}.
\end{equation}
Possibly choosing $\hat \delta>0$ smaller so that
$C\sqrt{\hat\delta}<1$, we obtain (for every $\delta \in (0,\hat\delta]$)
\begin{equation}
0<\hat{\xi}(t)<\min\{1,\delta t\},
\qquad t>0.
\end{equation}
This proves the claim.
\medskip

Define
\begin{equation}\label{e.bar2}
\underline{V}_1(t,\xi)= \begin{cases}
\underline{V}_0(t,\xi) &\text{ for }\xi \ge  \hat\xi(t),\\
\underline{V}_0(0,0)\theta(\xi) &\text{ for }\xi < \hat\xi(t).
\end{cases}
\end{equation}
It is clear that, near $\xi=\hat{\xi}(t)$, $\underline{V}_1$ is continuous and is the maximum of $\underline{V}_0$ and $\underline{V}_0(0,0)\theta$.  Then $(\overline{U},\underline{V}_{1})$ is a supersolution in the region $\hat\xi(t)<\xi<\delta t$, thanks to \Cref{lem:U0V0}. 

It remains to check the region $-\infty<\xi \leq \hat\xi(t)$. Indeed, verification of $\mathcal{N}_1\geq 0$ is simple, since $\overline{U}\equiv 1 +\frac{\alpha}{4}$ for $\xi \leq \hat\xi(t)$. We have
\begin{equation}
\mathcal{N}_1[\overline{U},\underline{V}_1] = -\overline{U} (1 - \overline{U} - a \underline{V}_1) \geq  - \left(1+\frac{\alpha}{4}\right)\left(1 - \left(1+\frac{\alpha}{4}\right) - 0\right) >0.
 \end{equation}

Before verifying $\mathcal{N}_2 \leq 0$, note that $\theta$ and $\underline{V}_0$ are both monotone increasing in $\xi$, so that\footnote{We used $\hat\xi(t) <1$ and \eqref{e.sup.1}.}
\begin{equation}
 0\leq \underline{V}_1(t,\xi) \leq \underline{V}_1(t,1) = \underline{V}_0(t,1) \leq \frac{\alpha}{4} \quad \text{ for }-\infty< \xi \leq 1.
\end{equation}
For $\xi \leq 
\underline\xi$, $\underline{V}_1\equiv 0$ so $\mathcal{N}_2 \equiv 0$ holds trivially. It remains to consider the case when $\xi \in (\underline\xi, \hat\xi(t))$:
\begin{align}
   \frac{1}{\underline{V}_0(0,0)} \mathcal{N}_2[\overline{U},\underline{V}_1] &= -(c + \rho'(t)) \theta' - d\theta'' - r\theta (1- b(1+\tfrac{\alpha}{4}) - \underline{V}_1) \notag\\
&\leq  -(c + \rho'(t)) \theta' - d\theta'' - r\theta (1- b(1+\tfrac{\alpha}{4}) - \tfrac{\alpha}{4}) \notag \\
&=  -(c + \rho'(t)) \theta' - d\theta'' - r\theta (1- b(1+\tfrac{\alpha}{2}) )- r(b-1) \theta \frac{\alpha}{4} \notag \\
&= -(5\sqrt\delta + \rho'(t)) \theta' - r(b-1) \theta \frac{\alpha}{4} \leq 0,
\label{e.0907.5}
\end{align}
where the last equality is due to the differential equation satisfied by $\theta$.

Finally, we verify the assertions (i$'$) and (ii$'$). Assertion (i$'$) follows from \eqref{e.bar22} and \eqref{e.bar2}, while   
Assertion  (ii$'$) still holds as $(\overline{U},\underline{V}_1) = (\overline{U}_0,\underline{V}_0)$ for $\xi \geq \delta \frac{t}{2}$ and $t\gg 1$. 
\end{proof}

\begin{proposition}
\label{lem.qsmall}
Under the hypotheses of \Cref{lem:U0V0}, 
for each 
\begin{align}\label{gamma-choice}
 \gamma \in \left(0,\min\left\{r,r(b-1)
\right\}\right),   
\end{align}

there exists a small constant $0<{\chi^{\star}_0}\leq 1$ (depending on $\delta,\gamma$) such that for any ${\chi_0} \in (0,{\chi^{\star}_0}]$, the functions
\begin{equation}\label{e.101.1}
(\overline{U}, \underline{V})
=(\overline{U}, (\underline{V}_1-\chi)_+),
\quad \text{ with }\quad 
\chi(t)=\chi_0e^{-\gamma t},
\end{equation}
form a supersolution to \eqref{e.com.xi}  on
$\{(t,\xi): t\geq0, \, -\infty<\xi\leq \delta t\}$. Moreover, the following properties hold:
\begin{itemize}
    \item[{\rm(i$''$)}] There exists $\underline\xi<0$ such that \begin{equation}\label{e.sup.1''} 
\overline{U} \equiv 1+\frac{\alpha}{4}\quad \text{ and }\quad \underline{V} \equiv 0 ~\text{ for }t>0,~\xi \leq \underline\xi ,
\end{equation}
where $\alpha$ is as given in \Cref{l.0907.1} and  $\underline\xi$ is a fixed constant depending on $\delta$.
\item[{\rm(ii$''$)}] $\overline{U}$ has the same spatial decay as the travelling-wave profile $U_c^\alpha$:
\begin{equation}\label{e.sup.2''}
\limsup_{t\to\infty} \sup_{\delta \frac{t}{2} \leq \xi \leq \delta t} e^{\lambda_{c} \xi}\overline{U} < +\infty \quad \text{ with }\quad \lambda_c = \tfrac12(c- \sqrt{c^2 - 4(1-a)}), 
\end{equation}

 \item[{\rm(iii$''$)}] On the ray $\left\{(t,x):~ t>0,~\xi = \delta t\right\}$, we have
\begin{equation}\label{e.sup.3''}
\underline{V} \leq 1- \chi_0 e^{-\gamma t} \quad \text{ for }t \geq 0.
\end{equation}

\end{itemize}
\end{proposition}

\begin{proof}[{\bf Proof of \Cref{lem.qsmall}}] It is clear that properties (i$''$)--(iii$''$) hold.

\noindent{\# First component.}
If $\overline U=1+\frac{\alpha}{4}$, then
\begin{equation}
\mathcal N_1[\overline U,\underline V]
=\left(1+\frac{\alpha}{4}\right)
\left(\frac{\alpha}{4}+a\underline V\right)>0.
\end{equation}
Otherwise, $\overline U=\overline U_0$ and,
by \eqref{e.sup.1}, necessarily $\xi>1$.
Hence $\underline V_1=\underline V_0$ and
\begin{align}
\mathcal N_1[\overline U,\underline V]
&\ge
\mathcal N_1[\overline U_0,\underline V_0]
-a\chi(t)\overline U_0\\
&\ge
\frac{\sqrt\delta}{1+\delta t}
(-(U_c^\alpha)')
-aC_0\chi_0e^{-\gamma t}
(-(U_c^\alpha)').
\end{align}
Here we used
$\overline U_0\le U_c^\alpha
\le C_0(-(U_c^\alpha)')$ for $\xi\ge1$,
which follows from the travelling-wave
asymptotics and strict monotonicity.
Therefore $\mathcal N_1\ge0$, provided
\begin{equation}
aC_0\chi_0^\star
\sup_{t\ge0}e^{-\gamma t}(1+\delta t)
\le\sqrt\delta.
\end{equation}

 \medskip

\noindent \# Second component. For the second component, note that
\begin{align}
\mathcal{N}_2[\overline{U}, \underline{V}_1]
&\le \begin{cases}
\displaystyle- r(b-1) \underline{V}_1 \frac{\alpha}{4} &\text{ for }\underline\xi < \xi < \hat\xi(t)\\[5pt]
-\displaystyle\frac{\sqrt\delta}{1+ \delta t}(V_c^\alpha)' &\text{ for }\hat\xi(t) < \xi < \delta t.
\end{cases}
\end{align}
We shall treat the zones $\xi > \hat\xi(t)$ and its complement separately. 

\#\# First consider the case $\xi > \hat\xi(t)$. If $\underline{V}\equiv 0$, there is nothing to prove. Otherwise $\underline{V}=\underline{V}_1 -\chi$ has a positive value and $\chi'=-\gamma \chi$, and hence 
\begin{align}\label{e.qN2}
\mathcal{N}_2[\overline{U}, \underline{V}]
={}&
\mathcal{N}_2[\overline{U}, \underline{V}_1]
+\gamma \chi
+r\chi(1-2\underline{V}_1-b\overline{U})
+r\chi^2.
\end{align}

Since $(\overline U, \underline{V}_1) \to (0,1)$ as $\xi \to \infty$ (among points $(t,\xi)$ in the wedge), we may 
fix a constant $L>0$ sufficiently large so that
$$
1-2{\underline V_1}-b\overline{U}
\le
-\frac12 \left(1+\frac{\gamma}{r}\right)
\qquad\text{for }\xi\ge L, ~ \text{ uniformly in }t\geq 0.
$$
Then, for $\xi\ge L$, (recalling that $0<\gamma<r$)
\begin{equation}\label{e.912.1}
\gamma \chi
+r\chi(1-2{\underline V_1}-b\overline{U}) + r\chi^2
\le
\chi \left(-\frac{r-\gamma}{2} + r \chi_0 \right) \le 0,
\end{equation}
since $\chi_0 \in (0,\chi^{\star}_0]$ and we required $r\chi^{\star}_0\le\frac{r-\gamma}{4}.
$  
Hence, \eqref{e.qN2} leads to
\begin{equation}
\mathcal{N}_2[\overline{U}, \underline{V}]
\le
\mathcal{N}_2[\overline{U},\underline{V}_1]
\le0 \quad \text{ for }t \geq 0,~ L\le\xi\le\delta t.
\end{equation}

On the fixed interval $0\le\xi\le L$, the $\chi$-terms in
\eqref{e.qN2} are bounded by
\begin{equation}
\gamma \chi
+r\chi(1-2\underline V_1-b\overline{U}) + r\chi^2
\le 3r \chi.
\end{equation}
For $\xi \in [\hat\xi(t),L]$, $\overline{U}\leq \overline{U}_0$ and  $\underline{V}_1 = \underline{V}_0$. Using also the monotonicity of the nonlinearity $U\mapsto -rV(1-bU-V)$, 
\begin{equation}
\mathcal{N}_2[\overline{U},\underline{V}_1]= \mathcal{N}_2[\overline{U},\underline{V}_0]  \leq \mathcal{N}_2[\overline{U}_0,\underline{V}_0] \leq -\frac{\sqrt\delta}{1+ \delta t}\inf_{-1 \leq \xi \leq L}|(V_c^\alpha)'|, 
\end{equation}
where the last inequality follows from \eqref{e.0907.3}. 
Choosing $\chi^{\star}_0$ smaller if necessary so that
\begin{equation}\label{e.qsmall3}
3r \chi^{\star}_0\sup_{t \geq 0} e^{-\gamma t} (1 + \delta t)\le \sqrt\delta  \inf_{\xi \in [0,L]} (V_{c}^{\alpha})',
\end{equation}
we obtain
\begin{equation}
\mathcal{N}_2[\overline{U}, \underline{V}]
\le 
\mathcal{N}_2[\overline{U}, \underline{V}_1]
+ 3r \chi
\le -\frac{\sqrt\delta}{1+ \delta t}\inf_{-1 \leq \xi \leq L}|(V_c^\alpha)'| + 3r\chi \le 0
\qquad\text{for }0\le\xi\le L.
\end{equation}

\#\# Finally consider the case 
$\xi < \hat\xi(t)$. We only need to consider the case   
\begin{equation}
\underline{V}
=\underline{V}_1-\chi>0,
\quad
\underline{V}_1=\underline{V}_0(0,0)\theta.
\end{equation}
Otherwise  $\underline{V}\equiv 0$ and $\mathcal{N}_2\equiv 0$ holds trivially. 

Since $\xi<\hat\xi(t)<1$, we have
$\overline U=1+\frac{\alpha}{4}$.
Note that from Lemma~\ref{lem.bar},
$\mathcal{N}_2[\overline{U},\underline{V}_1]\leq 0$. Hence, using $\chi<\underline{V}_1$, we have 
\begin{align}\label{eqN2-1009}
\mathcal{N}_2[\overline{U}, \underline{V}]
&=\mathcal{N}_2[\overline{U}, \underline{V}_1]
+\gamma \chi+r\chi(1-2\underline{V}_1-b\overline{U})+r\chi^2.\\
&\leq
\chi\left[\gamma+r\left(1-b\left(1+\frac{\alpha}{4}\right)\right)-2r\underline{V}_1+r\chi
\right]\\
&\leq\chi\left[\gamma-r\left(b\left(1+\frac{\alpha}{4}\right)-1\right)\right]
<0,
\end{align}
where the last inequality holds if $\gamma<r(b-1)$.

To summarise, for any $\gamma$ satisfying \eqref{gamma-choice}, 
we can choose $\chi_0^{\star}\in(0,1]$ sufficiently small such that for any $\chi_0 \in (0,\chi_0^{\star}]$, the pair $(\overline{U},\underline{V})$ given in \eqref{e.101.1} forms a supersolution to \eqref{e.com.xi}.
\end{proof}

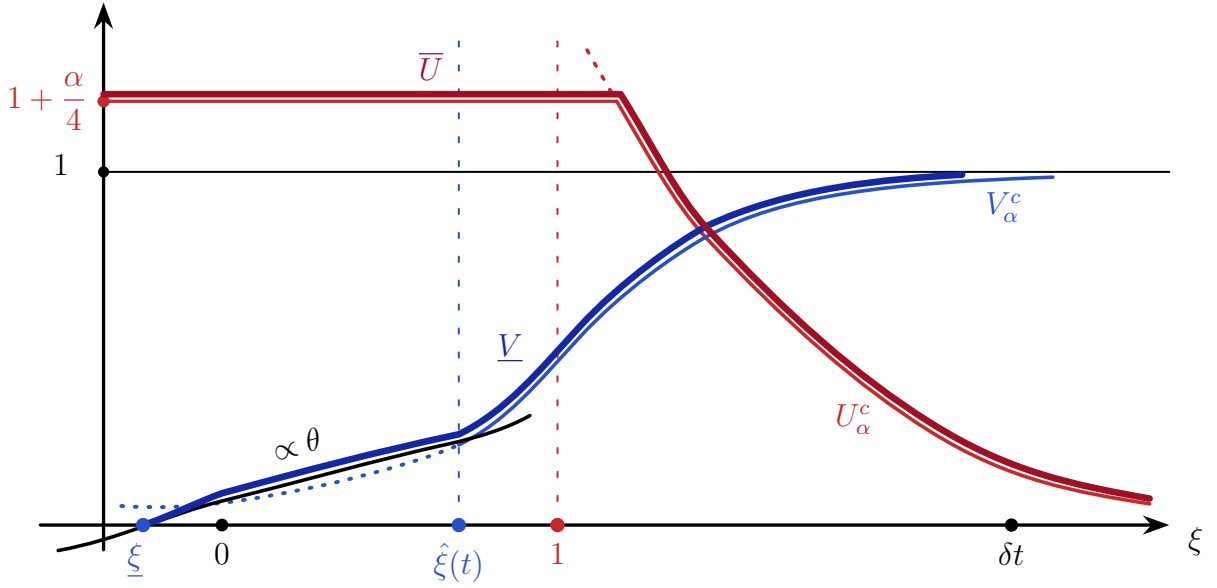
\begin{figure}[ht]
\begin{tikzpicture}[
    x=5.22cm, y=4.67cm,
    line cap=round,
    line join=round,
    font=\normalsize,
    axe/.style={
        black,
        line width=1.25pt,
        -{Stealth[length=3.5mm,width=2.5mm]}
    },
    courbe/.style={line width=1.35pt},
    courbebarre/.style={line width=2.5pt},
    rappel/.style={
        red,
        line width=0.75pt,
        dash pattern=on 2pt off 7pt
    },
    rappelblue/.style={
        blue,
        line width=0.75pt,
        dash pattern=on 2pt off 7pt
    },
    pointille/.style={
        blue,
        line width=1.45pt,
        dash pattern=on 0.7pt off 4.4pt
    }
]
  
    \def\xUn{0.85}

    \path[use as bounding box]
        (-0.64,-0.29) rectangle (2.53,1.56);

    \draw[axe] (-0.46,0) -- (2.40,0);
    \draw[axe] (-0.30,-0.07) -- (-0.30,1.48);
    \node[anchor=west] at (2.42,-0.045) {$\xi$};

    \draw[black, line width=0.75pt]
        (-0.30,1) -- (2.40,1);
    \draw[red, line width=1.25pt]
        (-0.30,1.20) -- (1,1.20);

    \draw[rappel] (\xUn,0) -- (\xUn,1.37);

     \draw[rappelblue] (0.60,0) -- (0.60,1.37);

    \fill[black] (-0.30,1) circle[radius=2.1pt];
    \fill[red] (-0.30,1.20) circle[radius=2.3pt];
    \node[anchor=east] at (-0.36,1.02) {$1$};
    \node[
        red, anchor=east,
        font=\normalsize, inner sep=1.5pt
    ] at (-0.33,1.20) {$1+\dfrac{\alpha}{4}$};

    \draw[
        red, courbe,
        dash pattern=on 1.2pt off 4.5pt
    ]
        (0.925,1.345)
        .. controls (0.950,1.295) and (0.974,1.249)
            .. (1.000,1.200);

    \draw[red,courbe]
        (1.000,1.200)
        .. controls (1.075,1.059) and (1.145,0.910)
            .. (1.225,0.815)
        .. controls (1.431,0.570) and (1.610,0.394)
            .. (1.793,0.266)
        .. controls (1.993,0.126) and (2.160,0.094)
            .. (2.350,0.060);

    \draw[pointille]
        (-0.255,0.054)
        .. controls (-0.092,0.041) and (0.067,0.066)
            .. (0.285,0.120)
        .. controls (0.400,0.151) and (0.516,0.186)
            .. (0.600,0.227);

    \draw[blue,courbe]
        (0.600,0.227)
        .. controls (0.720,0.292) and (0.800,0.403)
            .. (0.908,0.534)
        .. controls (1.012,0.660) and (1.150,0.768)
            .. (1.225,0.815)
        .. controls (1.420,0.937) and (1.702,0.970)
            .. (2.105,0.985);

      \draw[black,courbe]
        (-0.415,-0.072)
        .. controls (-0.263,-0.042) and (-0.137,0.030)
            .. (0.000,0.068)
        .. controls (0.202,0.124) and (0.385,0.183)
            .. (0.575,0.230)
        .. controls (0.655,0.250) and (0.725,0.275)
            .. (0.780,0.310);

      \draw[bluebar,courbebarre]
        (-0.200,0.000)
        .. controls (-0.142,0.02) and (-0.065,0.064)
            .. (0.000,0.090)
        .. controls (0.200,0.151) and (0.380,0.205)
            .. (0.600,0.257)
        .. controls (0.720,0.322) and (0.800,0.433)
            .. (0.908,0.564)
        .. controls (1.012,0.690) and (1.150,0.798)
            .. (1.225,0.845)
        .. controls (1.430,0.955) and (1.680,0.980)
            .. (1.875,0.992);

    \draw[redbar,courbebarre]
        (-0.300,1.22) -- (1.010,1.22)
        .. controls (1.087,1.084) and (1.15,0.935)
            .. (1.229,0.840)
        .. controls (1.431,0.595) and (1.610,0.419)
            .. (1.793,0.291)
        .. controls (1.993,0.149) and (2.160,0.113)
            .. (2.350,0.075);

   \node[redbar, font=\normalsize]
        at (0.53,1.3) {$\overline{U}$};
    \node[bluebar, font=\normalsize]
        at (0.73,0.50) {$\underline{V}$};
    \node[red, font=\normalsize]
        at (1.6,0.3) {$U_\alpha^c$};
    \node[blue, font=\normalsize]
        at (1.98,0.9) {$V_\alpha^c$};
    \node[rotate=15, font=\normalsize]
            at (0.19,0.23) {$\propto\theta$};

    \fill[blue]  (-0.20,0) circle[radius=2.7pt];
    \fill[black] ( 0.00,0) circle[radius=2.4pt];
    \fill[blue]  ( 0.60,0) circle[radius=2.7pt];
    \fill[red] (\xUn,0) circle[radius=2.7pt];
    \fill[black] ( 2.00,0) circle[radius=2.5pt];

    \node[blue, anchor=north]
        at (-0.22,-0.02) {$\underline{\xi}$};
    \node[anchor=north]
        at (0.00,-0.02) {$0$};
    \node[blue, anchor=north]
        at (0.60,-0.02) {$\hat{\xi}(t)$};
    \node[red, anchor=north]
        at (\xUn,-0.02) {$1$};
    \node[anchor=north]
        at (2.00,-0.02) {$\delta t$};
\end{tikzpicture}
\caption{Illustration of the supersolution formed by the couple $(\overline{U},\underline{V})$.}
\end{figure}

\subsection{Supersolution for the case \texorpdfstring{$0<a,b<1$}{0<a,b<1}}

We consider the weak-weak competition case $0<a,b<1$. We state a counterpart of \Cref{lem.qsmall}.

\begin{proposition}
\label{lem.qsmall'} Suppose $0<a,b<1$.
For each $c > c^0_{\rm LLW}$, there exists $\hat\delta>0$ such that for any  $\delta \in (0,\hat\delta]$ and for any
 $\rho\in \mathscr{C}^1([0,\infty))$ satisfying \eqref{e.0816.2}, the following statements hold.

\begin{itemize}
    \item[{\rm(a)}]There exists 
a supersolution
$(\overline{U}_2,\underline V_2)$ presented in the shifted frame $\xi = x-ct - \rho(t)$ 
satisfying \eqref{e.com.xi} in
\begin{equation}
\Omega
=
\left\{
(t,\xi):
t\ge 0,\;
-\infty< \xi\le\delta t
\right\}.
\end{equation}
Moreover, with $\alpha>0$ as given in \Cref{l.0907.1}, the following properties are satisfied:
\begin{itemize}
    \item[{\rm(i)}] On $\left\{(t,\xi):~t>0,~-\infty<\xi \leq 1\right\}$, \begin{equation}\label{e.919.4} 
\overline{U}_2 \geq  \frac{1-a}{\tfrac1{1+\frac{\alpha}{m+1}}-ab}
\quad \text{ and }\quad \underline{V}_2 \leq  \frac{\tfrac{1}{1+\frac{\alpha}{m}} -b}{\tfrac{1}{1+\frac{\alpha}{m}} -ab},
\end{equation}
for 
some constants $m,\alpha$ such that
\begin{equation}\label{e.919.7}
ab < \frac{m}{m+1}, \quad \text{ and }\quad 0< \alpha < \min\left\{\frac{\tfrac{m}{m+1}-ab}{\frac{ab}{m+1}}, \frac{m(1-b)}{b}\right\}.
\end{equation}

\item[{\rm(ii)}] $\overline{U}_2$ has the same spatial decay as the travelling-wave profile $U_c^\alpha$:
\begin{equation}
\limsup_{t\to\infty} \sup_{ \frac{\delta}{2}t \leq \xi \leq \delta t} e^{{\lambda_c} \xi}\overline{U}_2 < +\infty \quad \text{ with }\quad \lambda_c = \tfrac12 \left(c-\sqrt{c^2 - 4\left(1-a\right)}\right). 
\end{equation}
\end{itemize}

\item[{\rm(b)}]For each $\gamma \in \left(0,r\right)$\footnote{Only needed for \eqref{e.912.1}.}, there exists a small constant $0<\chi^{\star}_0\leq 1$ such that for any $\chi_0 \in (0,\chi_0^{\star}]$, the functions
\begin{equation}
(\overline{U}_3, \underline{V}_3)
=(\overline{U}_2, (\underline{V}_2-\chi)_+),
\quad \text{ with }\quad 
\chi(t)=\chi_0e^{-\gamma t},
\end{equation}
form a supersolution to \eqref{e.com.xi}  on $\Omega$. In particular, assertions (i), (ii) and the following hold for $(\overline{U}_3, \underline{V}_3)$:

\begin{itemize}
 \item[{\rm(iii)}] On the ray $\left\{(t,x):~ t>0,~\xi = \delta t\right\}$, we have
\begin{equation}
\underline{V}_3 \leq 1- \chi_0 e^{-\gamma t} \quad \text{ for }t \geq 0.
\end{equation}
\end{itemize}
    
\end{itemize}

\end{proposition}
\begin{proof}[{\bf Proof of \Cref{lem.qsmall'}}]
Choose
{$m>1$ so that $ab< \frac{m}{1+m}$ and then choose 
$\alpha>0$ sufficiently small  such that \eqref{e.919.7}, $c>c^{\alpha}_{\rm LLW}$ and $\frac{1}{1+\alpha}>b$}
hold. Then, 
denote the constant values
\begin{equation}
{(U^{\alpha}_{-\infty},V^{\alpha}_{-\infty})=\left( \frac{1-a}{\tfrac1{1+\alpha}-ab}, \frac{\tfrac{1}{1+\alpha} -b}{\tfrac{1}{1+\alpha} -ab}\right)},\quad 
(U_{\mathrm{con}},V_{\mathrm{con}})=\left( \frac{1-a}{\tfrac1{1+\frac{\alpha}{m+1}}-ab}, \frac{\tfrac{1}{1+\frac{\alpha}{m}} -b}{\tfrac{1}{1+\frac{\alpha}{m}} -ab}\right).  
\end{equation}
Then the choice of $m,\alpha$ implies 
\begin{align}\label{shift}
U^{\alpha}_{-\infty} > U_{\mathrm{con}} > \frac{1-a}{1-ab},\quad V^{\alpha}_{-\infty}< V_{\mathrm{con}} < \frac{1-b}{1-ab}
\end{align}
and 
\begin{equation}\label{e.919.3}
    U_{\mathrm{con}}(1-U_{\mathrm{con}}-aV_{\mathrm{con}}) < 0 < rV_{\mathrm{con}}(1-bU_{\mathrm{con}}-V_{\mathrm{con}}).
\end{equation}

The proof of (a) follows by first repeating the proof of \Cref{lem:U0V0} for the case $0<a,b<1$, to get $(\bar{U}_0,\underline{V}_0)$ (using the same notation) defined on $\left\{0 \leq \xi \leq \delta t\right\}$ and satisfying \eqref{e.919.4}, by starting with an appropriate translation of the travelling wave solutions and choosing $\delta$ small enough, due to \eqref{shift}. 
Then we can obtain $(\overline{U}_2,\underline{V}_2)$ as the truncation of $(\overline{U}_0,\underline{V}_0)$:
\begin{equation}
\overline{U}_2(t,\xi) := \begin{cases}
\min\left\{ U_{\mathrm{con}},\overline{U}_0(t,\xi)\right\}  &\text{ for }\xi \geq 0,\\
U_{\mathrm{con}} &\text{ for }\xi <0,
\end{cases}
\end{equation}
\begin{equation}
\underline{V}_2(t,\xi) := \begin{cases}
\max\left\{ V_{\mathrm{con}},\underline{V}_0(t,\xi)\right\}  &\text{ for }\xi \geq 0,\\
V_{\mathrm{con}}&\text{ for }\xi <0.
\end{cases}
\end{equation}
Using \eqref{e.919.3} and the monotonicity of the nonlinearity, 
it is easy to see that the truncated and extended functions form a supersolution pair in the domain $\left\{(t,\xi):~t>0,~-\infty < \xi \leq \delta t\right\}$.

Part (b) follows from repeating the proof of \Cref{lem.qsmall}. Note that the only place we need to restrict the rate $\gamma$ is at \eqref{e.912.1}.
\end{proof}

\begin{figure}[H]
\begin{tikzpicture}[
    x=5.35cm, y=5.15cm,
    line cap=round,
    line join=round,
    font=\tiny,
    axe/.style={black,line width=1.35pt,
        -{Stealth[length=3.7mm,width=2.7mm]}},
    courbe/.style={line width=1.35pt},
    courbeepaisse/.style={line width=2.55pt},
    niveau/.style={line width=0.8pt},
    rappel/.style={red,line width=0.85pt,
        dash pattern=on 6pt off 9pt}
]
    \def\xzero{0.10}
    \def\xone{0.92}
    \def\xdelta{2.02}
    \def\yOne{1.00}
    \def\yRhaut{0.88}
    \def\yRhautbar{0.9}
    \def\yRbas{0.79}
    \def\yBhaut{0.305}
    \def\yBmil{0.215}
    \def\yBbas{0.115}
    \def\yBbar{0.095}

    \path[use as bounding box] (-0.52,-0.15) rectangle (2.45,1.18);
    \draw[axe] (-0.24,0) -- (2.26,0);
    \draw[axe] (0,-0.13) -- (0,1.16);

    \draw[black,niveau] (0,\yOne) -- (2.25,\yOne);

    \draw[red,niveau] (0,\yRbas) -- (2.14,\yRbas);
    \draw[blue,niveau]  (0,\yBhaut) -- (2.20,\yBhaut);
    \draw[rappel] (\xone,0) -- (\xone,1.14);

    \draw[red,courbe] (0,\yRhaut) -- (\xone,\yRhaut)
        .. controls (0.955,0.845) and (1.015,0.735) .. (1.120,0.580)
        .. controls (1.280,0.375) and (1.445,0.205) .. (1.645,0.112)
        .. controls (1.780,0.052) and (1.885,0.026) .. (1.945,0.018);

     \draw[redbar,courbeepaisse] (0,\yRhautbar) -- (0.93,\yRhautbar)
        .. controls (0.970,0.862) and (1.030,0.752) .. (1.135,0.603)
        .. controls (1.29,0.398) and (1.445,0.228) .. (1.650,0.133)
        .. controls (1.790,0.070) and (1.900,0.039) .. (1.945,0.036);

    \draw[blue,courbe] (0,\yBbar) -- (0.9,\yBbar)
        .. controls (0.935,0.165) and (0.985,0.255) .. (1.045,0.372)
        .. controls (1.142,0.555) and (1.34,0.730) .. (1.56,0.848)
        .. controls (1.7,0.92) and (1.8,0.95) .. (1.985,0.970)
        .. controls (2.135,0.98) and (2.180,0.98) .. (2.22,0.98);

    \draw[bluebar,courbeepaisse] (0,\yBbas) -- (0.885,\yBbas)
        .. controls (0.910,0.140) and (0.950,0.230) .. (1.010,0.350)
        .. controls (1.108,0.535) and (1.262,0.700) .. (1.495,0.838)
        .. controls (1.675,0.944) and (1.840,0.986) .. (2.055,0.995)
        .. controls (2.120,0.997) and (2.170,0.997) .. (2.220,0.997);

     \fill[black] (0,\yOne) circle[radius=2.2pt];
    \fill[red] (0,\yRhaut) circle[radius=2.5pt];
    \fill[red] (0,\yRbas)  circle[radius=2.5pt];
    \fill[blue]  (0,\yBhaut) circle[radius=2.5pt];
    \fill[blue]  (0,\yBbar)  circle[radius=2.5pt];
 
    \fill[black] (\xzero,0) circle[radius=2.7pt];
    \fill[red] (\xone,0)  circle[radius=3.0pt];
    \fill[black] (\xdelta,0) circle[radius=2.8pt];

    \node[anchor=north,font=\tiny] at (\xzero,-0.04) {$0$};
    \node[red,anchor=north,font=\tiny] at (\xone,-0.04) {$1$};
    \node[anchor=north,font=\tiny] at (\xdelta,-0.04) {$\delta t$};
    \node[anchor=west,font=\tiny] at (2.24,-0.005) {$\xi$};

    \node[black,anchor=west,font=\tiny] at (-0.1,1.0) {$1$};
    \node[red,anchor=east,font=\tiny] at (-0.01,0.902)
        {$\dfrac{1-a}{\frac{1}{1+\frac{\alpha}{m+1}}-ab}$};
    \node[red,anchor=north east,font=\tiny] at (-0.01,0.85)
        {$\dfrac{1-a}{1-ab}$};
    \node[blue,anchor=south east,font=\tiny] at (-0.01,0.22)
        {$\dfrac{1-b}{1-ab}$};
    \node[blue,anchor=north east,font=\tiny] at (-0.01,0.23)
        {$\dfrac{\frac{1}{1+\frac{\alpha}{m}}-b}{\frac{1}{1+\frac{\alpha}{m}}-ab}$};

    \node[redbar,font=\normalsize] at (0.53,0.95) {$\overline{U}$};
    \node[bluebar,font=\normalsize] at (1.29,0.585) {$\underline{V}$};
    \node[red,font=\normalsize] at (1.50,0.11) {$U_\eta^c$};
    \node[blue,font=\normalsize] at (1.99,0.865) {$V_\eta^c$};
\end{tikzpicture}

    \caption{Illustration of the supersolution formed by the couple $(\overline{U},\underline{V})$.
    }
    \label{fig:supersolution2}
\end{figure}
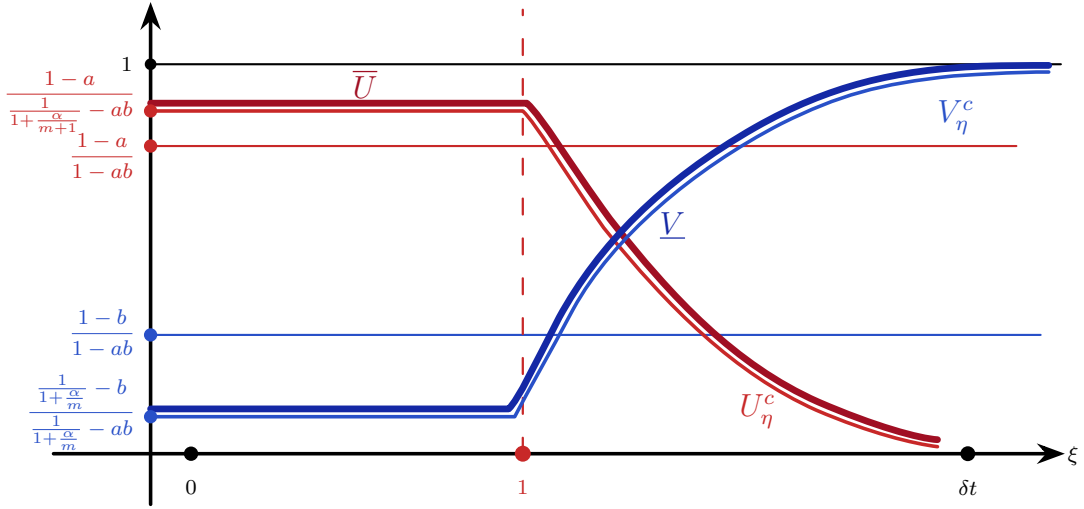

\section{Competition system: proofs of Theorems \texorpdfstring{\ref{thm:syst1}}{thm:syst1} and \texorpdfstring{\ref{thm:syst2}}{thm:syst2}}\label{sec:proofsyst}

\begin{lemma}\label{lem.0815.4}
Fix $c \in (c^0_{\mathrm{LLW}},\infty)$, let $\lambda$ be the smaller root of $c\lambda = \lambda^2 + 1-a$, and let $\underline{c}, \overline{c}$ be such that 
\begin{equation}\label{e.0816.a}
c < \underline c <  \overline{c} \leq +\infty. 
\end{equation}
 Suppose that there are $\eta>0$ and $q\in\mathbb{R}$ such that
\begin{equation}\label{e.0816.6}
\sup_{ (\underline{c}-\eta)t \leq x \leq (\overline c + \eta)t} \left(|u(t,x)| + | v(t,x)-1|\right) \to 0 \quad \text{ as }t \to \infty,
\end{equation}
\begin{equation}\label{e.0816.7}
\limsup_{t\to\infty} \, t^{-q} e^{\lambda(\overline c -c)t}u(t,\overline c t) <+\infty \quad \text{ when }\bar{c}<\infty,
\end{equation}
\begin{equation}\label{e.0816.7'}
{\limsup_{x\to\infty} \, x^{-q} e^{\lambda x}u(T',x) <+\infty \quad \text{ for some } T'>0 } \text{ when }\bar{c}=\infty.
\end{equation}
Then there exist $T>0$ and $C_T>0$ such that
\begin{equation}\label{e.0816.8}
u(t,x) \leq  C_T\left[
e^{-\underline \lambda (x-  \underline{c} t)} + (x-2\lambda t)^q e^{-\lambda (x - c t)}
\right] 
\end{equation}
in     $\Omega = \left\{(t,x):~ t >T,~ \underline{c} t< x <\overline{c} t\right\},$ where $\underline{\lambda}$ is the larger root of $\underline c \underline{\lambda}  = \underline{\lambda}^2 + 1-a$.
\end{lemma}
 
\begin{proof}[{\bf Proof of \Cref{lem.0815.4}}]
First, it follows from \cite[Lemma 1.12]{wu2023sharp} that there is $\eta'>0$ such that
\begin{equation}
\lim_{t\to\infty} e^{\eta' t}\sup_{(\underline c-\eta') t \leq x \leq (\bar c + \eta')t} \left( u(t,x) + |v(t,x)-1|\right) = 0.
\end{equation}
Hence, up to multiplication by a factor 
\begin{equation}
\beta(t) = \exp
\left(\pm C e^{-\frac{\eta'}{2} t} \right)
\end{equation}
which is bounded uniformly by positive constants, we may assume without loss of generality 
that $u$ satisfies
\begin{equation}
u_t - u_{xx} \leq (1-a) u \quad \text{ for }t>0,~ \underline{c} t \leq x \leq \bar{c} t.
\end{equation}
Next, choose a supersolution of the form 
\begin{equation}
\overline{u} = C\left[
e^{-\underline \lambda (x-  \underline{c} t)}+ \psi(t,x)
\right],
\end{equation}
where $\psi$ is the solution to the homogeneous equation 
\begin{equation}
\psi_t - \psi_{xx} = (1-a) \psi  \quad \text{ for }t>0,~ x \in \mathbb R,
\end{equation}
with initial data $ (1 + x_+)^q e^{-\lambda \max\left\{1,x\right\} }$. 
Since $c=\lambda+\tfrac{1-a}{\lambda}$ and $\underline{c}>c>2\lambda$, by 
\Cref{lem:C0},
\begin{equation}
\psi(t,x) \gtrsim
(x-2\lambda t)^q e^{-\lambda(x-ct)}
\qquad \text{ for }t>0,\  x \geq  \underline{c} t+1. 
\end{equation}

On the left boundary, by taking $C>1$, we have 
$\bar{u}(t,\underline{c} t)\geq C \geq  u(t,\underline{c} t)$ for $t\geq 0$.
On the right boundary, again by \Cref{lem:C0},
\begin{equation}
\bar{u}(t, \bar c t) \gtrsim \psi(t, \bar ct) \gtrsim t^{q} e^{-\lambda (\bar c - c)t} \gtrsim u(t,\bar c t) \quad \text{ for } t\gg1,
~\text{ in case }\bar c <\infty,
\end{equation}
and
\begin{equation} 
\psi(T,x)
\asymp x^q e^{-\lambda x} \gtrsim u(T,x) \quad \text{ as } x\to\infty, \text{ in case }\bar c = \infty.
\end{equation}
By choosing $T=0$ in case $\bar{c}<\infty$ (or $T=T'$ in case $\bar{c}=\infty$) and the constant multiple $C>1$ in $\overline u$ large enough, we can verify in either case that $\overline{u} \geq u$ on the parabolic boundary of $\Omega$. The conclusion follows by the comparison principle. 
\end{proof}

This lemma will be applied below with ($c=c_\lambda$,  $\bar{c}=+\infty$, and with $\underline{c} \in (c_\lambda, c_{\lambda} + 1)$) and ($c = c_{\mathrm{GL}}$ and $\underline c,\,\overline c \in (c_{\mathrm{GL}}, 2\sqrt{dr})$).

\begin{proof}[{\bf Proof of \Cref{thm:syst1}}]

Let $\lambda \in (0,\lambda_{\mathrm{LLW}})$. Then $c_\lambda=\lambda + \frac{1-a}{\lambda} > c^0_{\mathrm{LLW}}.$  We mainly treat the case $0<a<1<b$, and comment on the necessary modifications to handle the case $0<a,b<1$.

\medskip

\noindent {\bf \# Step 0.} The lower bound of the level set $\xi_{\kappa}[u](t)$ of $u$ immediately follows from Lemma \ref{lem:C1} and $u_t \geq u_{xx}+ u(1-a-u)$ (thanks to $v \leq 1$). Hence, \Cref{lem:C1}(ii) implies there exists $L>0$ such that 
\begin{equation}
\liminf_{t\to\infty} \inf_{x\in\R} \left( u(t,x) -\Phi_{1-a}(x-c_{\lambda} t-\frac{q}{\lambda}\ln (1+t)+L)\right) \geq 0.
\end{equation}
This gives both the lower bound for the level set as well as the necessary tail bound for the $u$-component when we pass to the limiting entire solution.

It remains to show the upper bound.

First, we fix $\alpha>0$ as in \Cref{l.0907.1}. Then, we fix $\delta\in \left(0, \left(\tfrac{\lambda}{|q|+1}\right)^2\right)$ so that the conclusion of \Cref{lem.qsmall} holds.

\medskip
\noindent {\bf \# Step 1.} 
Let $\psi$ be defined in Lemma~\ref{lem:C0}.
Choose $K>0$ such that $u_0\le K\psi_0$. Since $u_t-u_{xx}\leq u$,
by comparison, $u(t,x)\le Ke^{at}\psi(t,x)$. Hence, by Lemma~\ref{lem:C0}, for any sufficiently large $t$,
\begin{equation}
u(t,x)\lesssim x^qe^{-\lambda x}
\qquad\text{as }x\to+\infty.
\end{equation}
Here the implicit constant may depend on $t$.

\medskip

\noindent {\bf \# Step 2.}
Since $u_t-u_{xx}\le u$, comparison
with the scalar linear equation gives,
for every $\hat c>\lambda+\frac{1}{\lambda}$,
\begin{equation}
\sup_{x\ge\hat ct}u(t,x)\to0
\qquad(t\to\infty).
\end{equation}
By the far-field stability result for
the second component, for every
$\varepsilon>0$,
\begin{equation}
\sup_{x\ge(\hat c+\varepsilon)t}
|v(t,x)-1|\to0.
\end{equation}

We next apply \Cref{lem:perturbed-TW}.
For each $\varepsilon>0$, choose its
comparison-wave speed
\begin{equation}
c_\varepsilon\in
\left(c_\lambda,c_\lambda+\frac{\varepsilon}{2}\right).
\end{equation}
Then we can compare with a translated version of the supersolution pair $(\mathfrak U, \mathfrak V)$ constructed in \Cref{lem:perturbed-TW}. 
The comparison principle and the
exponential right-tail decay of the
comparison profiles then yield
$C_\varepsilon,\gamma_\varepsilon>0$
such that
\begin{equation}
u(t,x)+|1-v(t,x)|
\le C_\varepsilon e^{-\gamma_\varepsilon t},
\qquad
t\gg1,\quad
x\ge(c_\lambda+\varepsilon)t.
\end{equation}
In particular, taking
$\varepsilon=\frac{\delta}{8}$ gives
\begin{equation}\label{e.0829.4}
0\leq  1-v(t,x) \leq \frac{1}{\gamma_1}e^{-\gamma_1 t}  \quad \text{ for }t\gg 1,~x \geq (c_\lambda + \tfrac{\delta}7)t. 
\end{equation}

\medskip
\noindent {\bf \# Step 3.} We claim that for the $\delta>0$ chosen in \# Step 0, there exists $T$ such that 
\begin{equation}\label{e.0828.1}
u(t,x) \lesssim  |x-2\lambda t|^q e^{-\lambda (x - c_\lambda t)},\quad \text{ for } x\geq (c_\lambda+\tfrac{\delta}{5})t,~t\geq T.
\end{equation}
To this end, apply \Cref{lem.0815.4} with $c=c_\lambda$,  $\bar{c}=+\infty$, and with $\underline{c} \in (c_\lambda, c_{\lambda} + 1)$ chosen such that
\begin{equation}\label{e.0828.2}
\lambda + \underline\lambda < c_\lambda + \frac{\delta}5,
\end{equation}
where $\underline\lambda$ is the larger root of $\underline{c} \lambda = \lambda^2 + 1-a$. 
The last condition is possible since $\lambda + \underline\lambda \to c_\lambda$ as $\underline{c} \searrow c_\lambda$ (the sum of two roots of $-c_\lambda \lambda + \lambda^2 + 1-a$ is $c_\lambda$). 
Then, by \Cref{lem.0815.4},
\begin{equation} 
u(t,x) \lesssim \left[
e^{-\underline \lambda (x-  \underline{c} t)} + |x-2\lambda t|^q e^{-\lambda (x - c t)}
\right] \quad \text{ for } x \geq \underline{c}t.
\end{equation}
Since the second term dominates uniformly on $\{\tfrac{x}{t} \geq \lambda + \underline\lambda \} \supseteq \{\tfrac{x}{t}\geq c_\lambda + \frac{\delta}5\}$. This completes the \# Step 3.

\medskip
\noindent {\bf \# Step 4.}  For $t_0 \gg 1$, we consider $(\widetilde{u},\widetilde{v})(t,\xi) = (u_{\mathrm{new}},v_{\mathrm{new}})(t,\xi + c_\lambda t + \rho_{\mathrm{new}}(t))$, where
\begin{equation}
(u_{\mathrm{new}},v_{\mathrm{new}})(t,x) = (u,v)(t + t_0,x + (c_\lambda + \delta)t_0) \quad \text{ and }\quad \rho_{\mathrm{new}}(t) = \frac{q}{\lambda}\ln(t_0 + t),
\end{equation}
is a fixed translation of $(u,v)$ in the $(t,x)$ variables. 
Note that (recall \eqref{e.0828.1}),
\begin{equation}\label{e.0829.3a}
\widetilde{u}(t,\xi) = u(t+t_0, \xi + c_\lambda t + (c_\lambda+\delta)t_0 + \tfrac{q}{\lambda}\ln(t+t_0)) \lesssim e^{-\lambda\xi  -\lambda \delta t_0} \qquad \text{ when }\xi = \delta t.
\end{equation}
Moreover, by choosing $t_0 \gg 1$, the function $\rho_{\mathrm{new}}$
satisfies \eqref{e.0816.2}. So 
the hypothesis on $\rho_{\mathrm{new}}$ in \Cref{lem:U0V0} is satisfied.  
Choose $\gamma \in (0,\min\left\{r,r(b-1),\gamma_1\right\})$. By  \Cref{lem.qsmall}, there exists $\chi_0>0$ such that $(\overline{U}, \underline{V})$ is a supersolution in \begin{equation}
\Omega = \left\{(t,\xi):~ t \geq 0,~ \xi \leq \delta t\right\}.\end{equation} We claim that the following hold for $t_0 \gg 1$. 

We claim that, for sufficiently large $t_0$, 
\begin{itemize}
    \item[{\rm(i)}] $\displaystyle e^{\lambda \xi} \widetilde{u}(t,\xi) \leq  e^{\lambda \xi}  \overline{U}(t,\xi)$ for all $t>0$,~$\xi = \delta t.$
    \item[{\rm(ii)}]  $1- \widetilde{v}(t,\delta t) \leq 1- \underline{V}(t,\delta t)$   for $t \geq 0$.
    \item[{\rm(iii)}] $\displaystyle  \sup_{x\in \mathbb{R}}\widetilde{u}(t,x) \leq 1+ \tfrac{\alpha}9$  for $t \geq 0$.
\end{itemize}
Assertion (i) follows from \eqref{e.0829.3a} and the explicit construction of $\overline{U}_0$.

For assertion (ii), it follows from  \eqref{e.0829.4} and \Cref{lem.qsmall}(iii$''$) that, for $t_0 \gg 1$,
\begin{equation}
1-\widetilde{v}(t,\xi)\big|_{\xi = \delta t} \leq 
1- \underline{V}(t,\xi) \big|_{\xi = \delta t},
\end{equation}
This proves assertion (ii).
 Assertion (iii) is standard, since $u$ is a subsolution to the Fisher-KPP equation with nonlinearity $u(1-u)$. Based on (i)--(iii) and \Cref{lem.qsmall}(i$'$), 
 we can apply the comparison principle in the domain $\Omega$ to deduce that
 \begin{equation}\label{e.110.2}
 \widetilde{u}(t,\xi) \leq \overline{U} \quad \text{ and }\quad \widetilde{v}(t,\xi) \geq \underline{V} \quad \text{ for }t\geq 0,~ \xi \leq \delta t.
 \end{equation}
 It follows that the level sets satisfy
 \begin{equation}
 \xi_{\kappa}[\widetilde{u}](t) \leq O(1),\qquad 
\xi_{\kappa}[u_{\mathrm{new}}](t)  \leq c_\lambda t + \rho_{\mathrm{new}}(t) + O(1),
\end{equation}
 and hence
 \begin{equation}
 \xi_{\kappa}[u](t+ t_0) \leq c_\lambda t + \frac{q}{\lambda}\ln (t_0 +t) + O(1).
 \end{equation}
This completes the estimate of the level set of $u$ up to a bounded error.

Let $t_n\to\infty$ and extract a locally convergent subsequence of
\begin{equation}
(\widetilde{u}_n(t,y), \widetilde{v}_n(t,y)):= (\widetilde{u}(t+t_n,y),\widetilde{v}(t+t_n,y)).
\end{equation}
Then the limit $(\widetilde{u}_\infty, \widetilde{v}_\infty)$ in $\mathscr{C}_{\mathrm{loc}}(\mathbb{R}^2)$ satisfies the parabolic system \eqref{eq:LV} in $\mathbb{R}^2$, with $c=c_\lambda$.

By \eqref{e.110.2} and the
construction of $\overline U$,
every subsequential limit satisfies
\begin{equation}\label{e.110.3}
\widetilde u_\infty(t,y)
\le U_{c_\lambda}^\alpha(y)
\lesssim e^{-\lambda y},
\qquad t\in\mathbb R,\quad y\ge0.
\end{equation}
Moreover, thanks to \# Step 0,  
\begin{equation}\label{e.919.6a}
 \widetilde{u}_\infty(t,y) \gtrsim e^{-\lambda y} \quad \text{ for }t\in\mathbb R,~y \geq 0, \qquad 
\inf_{t\in \mathbb{R}, \xi \leq 0}\widetilde{u}_\infty >0,
\end{equation}
($\widetilde{u}_\infty$ is bounded from below by a travelling wave solution $\Phi_{1-a}$ of the KPP problem $c_\lambda \Phi' + \Phi'' + \Phi(1-a - \Phi) = 0$)
and 
\begin{equation}\label{e.919.6b}
\inf_{t\in \mathbb{R}, \xi \geq 0}\widetilde{v}_\infty >0,
\end{equation}
since $\widetilde{v} \geq \underline{V}$, so that $\widetilde{v}_\infty \geq V^\alpha_c(\xi)$, where the latter is the travelling wave solution appearing in \eqref{e.bar}.

The estimates \eqref{e.110.3}, \eqref{e.919.6a} and \eqref{e.919.6b} allow us to apply \Cref{thm:liouville} to conclude the convergence to the travelling profile, which is \eqref{sys1-TW-conv}. 

For the weak-weak competition case $0<a,b<1$,
the argument is similar,  by using the supersolution constructed in \Cref{lem.qsmall'}.
\end{proof}

We now turn to \Cref{thm:syst2}. We first need an additional reduction result. 

\begin{proposition}\label{lem:1C}
Let $\rho \in (0,1)$, $T>0$ and $\alpha>0$ be given. Set $c_v= 2\sqrt{dr}$, $\eta_v = \frac32 \sqrt{\frac{d}{r}}$ and let $w$ be a positive solution to  $$w_t - dw_{xx} = rw(1 - \rho e^{-\alpha t} -w)$$ in the domain 
 \begin{equation}
 \Omega = \left\{(t,x):~ t>T,~  x >\underline{c}t\right\} \quad \text{ for some }\underline c \in (0,c_v).
 \end{equation}
 If there exists $\delta >0$ small such that
 \begin{equation}\label{e.102.6}
 \lim_{t\to+\infty} e^{\delta t} |w(t,\underline{c}t) -1| = 0 ,
 \end{equation}
 then there exist $T_1 \geq T$, $\chi_0 \geq 1$ and $0<\chi_1 <1$ and $L_0 \in \mathbb{R}$ such that
 \begin{equation}
 w(t,x) \geq \frac{1}{1 + \chi_0 e^{-\chi_1 t}} \Phi_1^\star (x - s(t) + L_0) \quad \text{ for }\Omega_1:= \{(t,x):~ t\geq T_1, \underline c t \leq x \leq s(t)\},\label{e.102.3}
  \end{equation}
where $s(t) = c_v t - \eta_v \ln(1+t)$, and  $\Phi_{1}^\star$ is the travelling wave with spatial exponent $\lambda_v$,
\begin{equation}
-c_v\Phi_1'-d\Phi_1''=r\Phi_1(1-\Phi_1).
\end{equation}
\end{proposition}

\begin{proof}[{\bf Proof of \Cref{lem:1C}}]
Let $\widetilde w(t,x) = \exp\left( r \rho \int_0^t \exp(-\alpha s) \, \dd s\right) w(t,x)$; then it is a positive supersolution to
\begin{equation}\label{e.116.2}
\widetilde{w}_t \geq d\widetilde{w}_{xx} + r\widetilde{w}(1-\widetilde w) \qquad \text{ in }\Omega.
\end{equation}
Let $\underline\zeta = \mathcal{R}\underline\Psi$ be the subsolution given as in Proposition 4.3, with $\g \equiv 0$, and let 
\begin{equation}
\underline{w}(t,x) = w_1\left(rt, x \sqrt{\frac{r}{d}}\right),\qquad w_1(t,x) = \delta_1 \underline \zeta\left(t,x -2t + \frac{3}{2}\ln(1+t) - L_1\right).
\end{equation}
Then for any $0<\delta_1\ll 1$ and  $L_1\gg 1$, $\underline{w}$ is a subsolution to \eqref{e.116.2}. Hence, it follows by comparison in the domain $\Omega_1$ that for each $L_2 >L_1$ large enough,
\begin{equation}
w(t,s(t) + L_2) \asymp \widetilde{w}(t,s(t) + L_2) \geq \delta_1 \underline \zeta\left(rt,(s(t) + L_2) \sqrt{\frac{r}{d}} -2t + \tfrac32 \ln(1+t)-L_1\right) \qquad \text{ for }t \gg 1.
\end{equation}
By the definition of $\underline\zeta$,  the right-hand side is uniformly bounded from below by a positive number as $t\to\infty$.
Hence, for any $T'>T$ large enough,
\begin{equation}
\delta:= \min\Big\{ 1, \inf_{t \geq T'}w(t,\underline c),~ \inf_{cT'\leq x \leq s(T') + L_2} w,~
\inf_{t \geq T'} w(t,s(t)+L_2)\Big\} >0.    
\end{equation}
It then follows by the comparison principle that 
\begin{equation}\label{e.1016.10}
w(t,x) \geq \delta \qquad \text{ for }t \geq T',~ \underline c t \leq x \leq s(t) + L_2.
\end{equation}
Fix
\begin{equation}
0<\chi_1<\min\left\{\alpha,\delta,r\delta\right\},
\qquad
B(t):=\frac{1}{1+\chi_0e^{-\chi_1t}},
\end{equation}
while letting $T_1\geq T$ be chosen later, 
so that
\begin{equation}
B'(t)=\chi_1B(t)(1-B(t)),
\end{equation}
and (by \eqref{e.102.6}) we have 
\begin{equation}
1-B(t) \asymp e^{-\chi_1 t},  \qquad 1-w(t,\underline ct) = o(e^{-\delta t}).
\end{equation}
Hence, 
there exists $T_1\geq T'$ such that
\begin{equation}\label{e.102.6b}
1-B(t) \geq 1-w(t,\underline ct) \quad \text{ for }t \geq T_1.   
\end{equation}

Let
\begin{equation}
\underline w(t,x)=B(t) \Phi_1^\star(x- s(t) +L_0) \qquad \text{ in }\Omega,
\end{equation}
where $\Phi^\star_1$ is as above, and $L_0$ is chosen such that $\Phi^\star_1(L_0) = \delta$.

A direct computation gives
\begin{align}
\underline w_t-d\underline w_{xx}
-r\underline w\left(1-\rho e^{-\alpha t}-\underline w\right)
&=\frac{\eta_v}{1+t}B(\Phi^\star_1)'
+rB\Phi^\star_1\left[
(1-B)(\tfrac{\chi_1}{r}-\Phi^\star_1)+\rho e^{-\alpha t}
\right]\\
&\leq rB \Phi^\star_1 \left[ (1-B) (\tfrac{\chi_1}{r} - \delta) + \rho e^{-\alpha t}\right],
\end{align}
where we used $(\Phi^\star_1)'\leq0$ and $\Phi^\star_1\geq\delta>\tfrac{\chi_1}{r}$ in $\Omega$ for the second last inequality. Since $r\delta-\chi_1>0$, we can take $T_1\gg 1$ such that, for all $\chi_0 \geq 1$, 
\begin{equation}
\frac{e^{-\alpha t}}{1-B(t)} = e^{-\alpha t}\left( \frac1{\chi_0}e^{\chi_1 t} + 1 \right) \leq \frac{\delta - \tfrac{\chi_1}{r}}
{\rho} \quad \text{ for }t \geq T_1.
\end{equation} 
As a consequence,
\begin{align}
\underline w_t-d\underline w_{xx}
-r\underline w\left(1-\rho e^{-\alpha t}-\underline w\right)\leq 0 \quad \text{ in }\Omega, 
\end{align}
provided $t\geq T_1$. 
Moreover, using \eqref{e.1016.10},
 \begin{align}
\underline{w}(t,s(t)) =  B(t)\Phi^\star_1(L_0) \leq \delta \le  \inf_{\underset{t \geq T}{\underline{c}t< x < s(t)}} w(t, x)\leq w(t,s(t)).    
 \end{align}
Next, using \eqref{e.102.6b} we have
\begin{equation}\label{e.0823.7b}
     \underline w(t,\underline{c} t) \leq  B(t)\leq w(t,\underline{c} t) \qquad \text{ for }t \geq T_1.
 \end{equation}
Lastly, up to taking a larger $\chi_0 \geq 1$, 
\begin{equation}\label{e.0823.7c}
    \underline w(T_1,x) \leq w(T_1,x) \quad \text{ for }\underline{c}T_1 \leq x \leq s(T_1).
 \end{equation}
Gathering all previous steps, $\underline w$ is a subsolution, and the parabolic comparison principle
yields \eqref{e.102.3}. This completes the proof.
\end{proof}

\begin{proof}[{\bf Proof of \Cref{thm:syst2}}]
We first treat the case $0<a<1<b$,
for which the spreading-speed estimates
were established in
\cite{girardin2019invasion}.
We then briefly discuss the modifications
needed for $0<a,b<1$, whose spreading
speed was determined in
\cite{liu2019asymptotic}.  We denote the speed of $u$ by $c_{\mathrm{GL}}$, that of $v$ by $c_v= 2\sqrt{dr}$. In particular, for each sufficiently small $\eta>0$, 
\begin{equation}\label{e.spu}
\lim_{t\to\infty} \sup_{x > (c_{\rm GL} + \eta)t} u(t,x) = 0 \qquad  \liminf_{t\to\infty} \inf_{0 \leq x <  (c_{\rm GL} - \eta)t} u(t,x) >0  
\end{equation}
\begin{equation}\label{e.spv}
\lim_{t\to\infty} \sup_{x > (c_v + \eta)t} v(t,x) = 0 \qquad  \liminf_{t\to\infty} \inf_{(c_{\rm GL} + \eta)t < x <  (c_v - \eta)t} v(t,x) >0. 
\end{equation}
Under the intended parameter regime, we have
\begin{equation}
c^0_{\rm LLW}<c_{\mathrm{GL}}<2, \qquad 2<c_v<\fraksc.
\end{equation}
Moreover, let $\lambda_v = \sqrt{\frac{r}{d}}$, $\lambda_{\mathrm{GL}} = \sqrt{dr} - \sqrt{a}$,  $\eta_{v} = \frac3{2\lambda_v}$ and $q_u = \frac32 (\sqrt \frac{ad}{r} -1)$ as defined in the introduction. 
We establish the sharp level-set estimates.

Let 
\begin{equation}\label{e.sss}
s(t) = c_v t - \eta_{v} \ln (1+t).\end{equation}  We claim that there exists $y_1\geq 1$ such that
\begin{equation}\label{e.921.1}
    av(t,y + s(t)) \leq \overline\sigma(y):= a\min\left\{\exp\left(- \frac{\lambda_v}{2} (y - y_1)\right),1\right\} \quad \text{ in }t\geq 0,~ y \in \mathbb{R}.
\end{equation}
Indeed, $v$ is a subsolution to the homogeneous Fisher-KPP problem, i.e. $v_t - d v_{xx} \leq rv(1-v)$ for $t \geq 0$, $x \in \mathbb{R}$ with compactly supported initial data, so \eqref{e.921.1} follows from \cite{bramson1983convergence,hamel2013short}. In fact, by using a sub- and supersolution technique, there exists $y_0 >1$ such that
\begin{equation}
v(t,y + y_0 + s(t)) \leq \min\left\{(1+y)\exp\left(- \frac23{\lambda_v} y\right),1\right\} \quad \text{ in } t\geq 0,~ y \geq 0.
\end{equation}
Together with $v\le1$, this implies
\eqref{e.921.1} upon increasing $y_1$.

It follows that $u$ satisfies
\begin{equation}
\begin{cases}
    u_t \geq u_{xx} + u(1- \overline{\sigma}(x-s(t)) -u) \quad &\text{ for } t>0,~ x \in \mathbb{R}, \\
    u(t,s(t)) >0 \quad &\text{ for }t >0.
\end{cases}
\end{equation}
Since
\begin{equation}
\int_0^\infty y\overline{\sigma}(y) \dd y <+\infty, \qquad \int_{-\infty}^0 |\overline\sigma(y) -a| \dd y <+\infty,
\end{equation}
it follows from \Cref{thm:pulling} that the level set of $u$ satisfies
\begin{equation}\label{e.106.1}
\xi_{\kappa}[u](t) \geq c_{\mathrm{GL}} t+\frac{q_u}{\lambda_{\mathrm{GL}}}\ln t+O(1)\quad \text{ as } t\to\infty.
\end{equation}

Next, we prove the upper bound for level sets. 
In view of \Cref{lem.0815.4}, it is enough to show
\begin{equation}\label{e.0829.z}
  \limsup_{t\to\infty} \,t^{-q_u} e^{\lambda_{\mathrm{GL}}(\overline{c} - c_{\mathrm{GL}} )t} u(t,\overline{c} t) \lesssim 1,
\end{equation}
for some $\overline c \in (c_{\mathrm{GL}}, 2\sqrt{dr})$. After that, the level-set estimate and 
the remaining convergence argument follow the strategy in the proof of \Cref{thm:syst1}, once the corresponding comparison and Liouville hypotheses are verified. We now proceed to show \eqref{e.0829.z}.

\medskip
\noindent {\bf \# Step 1.} 
 (Lower bound for the leading $v$-front.)

Let $\Phi^\star_1$ be as in \Cref{lem:1C} and fix any $\underline{c} \in (c_{\mathrm{GL}}, 2)$. We claim that there exist $L_0 \in \mathbb{R}$, $\chi_0,\chi_1>0$ such that
\begin{equation}\label{e.0829.21}
    v(t,x) \geq   \Phi_1^\star(x-s(t) + L_0)\mathbf{1}_{\{x \leq s(t)\}} - \chi_0 e^{-\chi_1 t} \quad \text{ for }x \geq \underline{c} t,~t\geq 0.
\end{equation}
Indeed, by \eqref{e.spu}, 
since $u \to 0$ exponentially uniformly in $$\Omega_1= \left\{(t,x):~ t \geq T,~ \underline c t \le  x \right\}$$ as $t\to\infty$, we may use \Cref{lem:1C} to deduce 
the existence of some $\chi_0,\chi_1,T_1>0$ and $L_0 \in \mathbb{R}$ such that \eqref{e.0829.21} holds for $t \geq T_1$. By enlarging $\chi_0>0$, we can take $T_1 = 0$.

\medskip
\noindent {\bf \# Step 2.} (Reduction to a scalar shifting-habitat problem.)

We compare the $u$-component with a scalar equation in a shifting habitat. 
\begin{equation}\label{eggg}
\g(y) = \begin{cases}
    1-a \Phi_1^\star(y+L_0)  \quad &\text{ for }y \leq 0,\\
    1 &\text{ for }y > 0,
\end{cases}, 
\end{equation}
Combining these estimates together with the spreading speed estimates \eqref{e.spu}--\eqref{e.spv} yields, for each $\underline{c} \in (c_{\mathrm{GL}}, 2)$, 
\begin{equation}\label{e.0829.22}
\begin{cases}
    u_t \leq u_{xx} + u(\g(x-s(t)) + \chi(t) -u) \quad &\text{ for }t \geq0,~x \ge \underline{c}t,\\
 u(t,x) \leq 1 &\text{ for }t \geq 0,~ x \geq \underline{c}t,
\end{cases}
\end{equation}
where $\chi(t) = a\chi_0 e^{-\chi_1 t}$ is an integrable function.

\medskip
\noindent {\bf \# Step 3.} (Removal of the integrable perturbation.) Let 
\begin{equation}
\widetilde{u}(t,x) = u(t,x) \exp\left(-\int_0^t \chi(s) \dd s\right)
\end{equation}
then $\widetilde{u}$ is a subsolution to \eqref{eq:mainshift}, i.e. for each $\underline{c} \in (c_{\rm GL},2)$, we have
\begin{equation}\label{e.0829.22'}
\begin{cases}
    \widetilde{u}_t \leq \widetilde{u}_{xx} + \widetilde{u}(\g(x-s(t))  -\widetilde u) \quad &\text{ for }t \geq 0,~\underline{c}t \leq x \leq s(t),\\
 \widetilde u(t,x) \leq 1 &\text{ for }t \geq 0,~x \geq \underline{c}t.
\end{cases}
\end{equation}
Since $u \asymp \widetilde u$ uniformly in the region $\{t>0,~2t \leq x \leq s(t)\}$, it is equivalent to estimate $\widetilde u$.

\medskip

\noindent {\bf \# Step 4.} (Choice of the auxiliary parameters.)
We claim that 
\begin{equation}\label{e.103.1e}
    \frac{c_v^2}4-1 < (1+\sqrt{a})(c_v - 2).
\end{equation}
Indeed,
\begin{align}
   \frac{1}{c_v-2}\left[\frac{c_v^2}4-1 - (1+\sqrt{a})(c_v - 2) \right]&= \frac{c_v +2}4 -1-\sqrt a\\
   &\leq \frac12\left[\sqrt{1-a} -\sqrt{a}-1\right]<0,
\end{align}
    where we used  $0<a<1$ and $\frac{c_v}2 \in (1,\sqrt{a}+\sqrt{1-a})$.

Using  \eqref{e.103.1e} and (G1), we choose $L>0$ and $a' \in (0,a)$ but close to $a$ such that 
\begin{equation}\label{e.1011.1}
  \s = c_v,\qquad   \frac{c_v^2}4-1 < (1+\sqrt{a'})(\s-2), \qquad
      2 < \s < 2(\sqrt{a'} + \sqrt{1-a'}),  \end{equation}
      \begin{equation}
        \g(y) \leq 1 -a'  
    \qquad \text{ for }y \leq -L.
    \end{equation}

\medskip
\noindent {\bf \# Step 5.} (Upper bound at the environmental interface.) We claim that 
\begin{equation}\label{e.1011.6}
    u(t,s(t))\asymp \widetilde{u}(t,s(t))\lesssim (1+t)^{-\tfrac32 + \tfrac{\s \eta}2} e^{-\left[ \tfrac{\s^2}4 -1\right]t}.
\end{equation}
Define \begin{equation}
\widetilde\g = 1+\frac{a'}{a}(\g-1),\end{equation} then 
$
\widetilde\g\geq \g
$
and $\widetilde\g$ satisfies the conditions (G1)--(G2) with $a$ replaced by $a'$. 
Now, let $n$ (resp. $\widetilde n$) be
the corresponding solution of \eqref{eq:mainshift}, with $\g$ (resp. $\widetilde\g$ in place of $\g$) and with common initial data $N_0$ satisfying \eqref{e.ini} and
\begin{equation}
N_0 \geq u_0, \qquad N_0(0)>0.
\end{equation}
Then $\widetilde{n} \geq n$  and $\widetilde n$ has spreading speed
\begin{equation}
\widetilde{c}_{\rm GL} = \tfrac{\s}2-\sqrt{a'} + \frac{1-a'}{\tfrac{\s}2-\sqrt{a'}}  \in (c_{\rm GL},2).
\end{equation}
Hence for each $\underline{c} \in (c_{\rm GL}, \widetilde{c}_{\rm GL})$, there exists  $\widetilde{k}\ge 1$, such that
\begin{equation}
    \begin{cases}
        (\widetilde{k}\widetilde n)_t - (\widetilde{k}\widetilde n)_{xx} - \widetilde{k}\widetilde n(\g(x-s(t)) - \widetilde{k}\widetilde n) \geq 0 &\text{ for }t \geq 0,~ x \in \mathbb{R},\\[8pt]
        \widetilde{k}\widetilde{n}(t,\underline c t) \geq 1 &\text{ for }t \geq 0,\\[8pt]
        \widetilde k\widetilde{n}(t,s(t))  
        \asymp(1+t)^{-\tfrac32 + \tfrac{\s \eta}2} e^{-\left[ \tfrac{\s^2}4 -1\right]t}  &\text{ for }t \gg 1.
    \end{cases}
\end{equation}
Here, the first differential inequality is straightforward;  the second condition follows from the fact $\widetilde{n}(t,\underline{c}t) >0$ for all $t$ and $\widetilde{n}(t,\underline{c}t) \to 1-a'$ as $t\to\infty$;   
sharp asymptotics at $x=s(t)$ follow by \Cref{e.1011.5}. The estimate \eqref{e.1011.6} follows by comparison in the domain $\{(t,x):~t>0,~x \geq \underline{c} t\}$ and then restricting the resulting inequality $u(t,x) e^{-\int_0^t \chi}= \widetilde{u} \leq \widetilde k \widetilde n$ on the ray $\{x = s(t)\}$.

\medskip
\noindent {\bf \# Step 6.} (Construction of a supersolution behind the interface.)

Let $n$ be the solution of \eqref{eq:mainshift} with the same initial data as $\widetilde{n}$ and with  $\g$ and $s(t)$ (as in \eqref{eggg} and \eqref{e.sss}) and define
\begin{equation}
\bar{u}(t,x) = K(n(t,x) + \alpha(t,x)), \qquad \alpha(t,x) = e^{-(1+\sqrt{a'})(x-2t)}.
\end{equation}
We claim that for every $K\ge1$, $\bar{u}$ 
satisfies
\begin{equation}\label{e.103.1a}
     \bar{u}_t - \bar{u}_{xx} - \bar{u}(\g - \bar u) \geq 0 \qquad \text{ for }t \geq T',~ 2t \leq x \leq s(t)-L,
\end{equation}
\begin{equation}\label{e.103.1b}
    \bar{u}(t,2t) \geq 1 \qquad \text{ for }t \geq T',
\end{equation}
\begin{equation}\label{e.103.1c}
    \frac{\bar{u}(t,s(t)-L)}{n(t,s(t)-L)} \lesssim 1 \qquad \text{ for } t \geq T',
\end{equation}

The differential inequality \eqref{e.103.1a} follows from 
\begin{align}
    \bar{u}_t - \bar{u}_{xx} - \bar{u}(\g - \bar u) &\geq K\left[n_t - n_{xx} - n(\g - Kn)\right] + K\left[\alpha_t - \alpha_{xx} - \alpha(\g - 2Kn -K\alpha)\right]\\
    &\geq  K\alpha(2n + \alpha)\geq 0,
\end{align}
where we used \eqref{eq:mainshift} and 
\begin{equation}
\frac1{\alpha}\left[\alpha_t - \alpha_{xx} - \g(x-s(t))\alpha\right]  \geq 2(1+\sqrt{a'}) - (1+\sqrt{a'})^2 - (1-a') = 0 \qquad \text{ for }x-s(t) \le -L.
\end{equation}

The inequality \eqref{e.103.1b} follows from $\alpha(t,2t) = 1$. Next, observe that
\begin{equation}
n(t,s(t)-L)\approx n(t,s(t)) \gtrsim (1+t)^{-\tfrac32 + \tfrac{\s \eta}2} e^{-\left[ \tfrac{\s^2}4 -1\right]t} \gtrsim 
{\alpha(t,s(t)-L)},
\end{equation}
where the first estimate follows from Harnack's inequality, the second from \Cref{e.1011.5}, and the latter follows from the first condition in \eqref{e.1011.1}. This establishes \eqref{e.103.1c}.

\medskip
\noindent {\bf \# Step 7.}(Completion of the front estimate.) 
Setting $\bar c = c_v- \delta$ and taking $x = \bar ct$, we deduce from \eqref{e.102.13} that
\begin{equation}\label{e.111.a}
     n(t,\bar c t) \lesssim  t^{q_u} e^{-\lambda_{\rm GL}(\bar c - c_{\rm GL})t},
\end{equation}
\eqref{e.102.13} in the end with $q_B =  -\tfrac32 +\sqrt a \eta_v$, $\lambda_\star = \lambda_{\mathrm{GL}}$, $c_\star =c_{\rm GL}= \lambda_{\mathrm{GL}} + \tfrac{1-a}{\lambda_{\mathrm{GL}}}.$ 
Next, note that $\alpha$ satisfies
\begin{equation}\label{e.111.b}
 \alpha(t,\bar ct) = e^{-(1+\sqrt{a'})(c_v -\delta-2 )t} \lesssim (1+t)^{-\tfrac32 + \tfrac{\s \eta}2} e^{-\left[ \tfrac{\s^2}4 -1\right]t} \lesssim t^{q_u} e^{-\lambda_{\rm GL}(\bar c - c_{\rm GL})t}.
\end{equation}

By \eqref{e.103.1a}--\eqref{e.103.1c}, we may choose $K>1$ sufficiently large to apply comparison to obtain
\begin{equation}
\widetilde u(t,x) \leq \overline{u}(t,x) \qquad \text{ in }\{(t,x):~t \geq T',~ 2t \le x \le s(t)-L\}.
\end{equation}
Specialising to $x= \bar{c}t$ with $\bar c = \s -\delta$, we have in particular the estimate
\begin{equation}
u(t,\overline{c} t) \asymp \widetilde u(t,\overline{c} t)\lesssim n(t,\overline c t) + \alpha(t,\overline c t)
\lesssim 
t^{q_u} e^{-\lambda_{\mathrm{GL}}(\overline{c} - c_{\mathrm{GL}} )t}
\end{equation}
where we used \eqref{e.111.a}--\eqref{e.111.b}.
This proves \eqref{e.0829.z}.

Having proved \eqref{e.0829.z}, we may repeat the proof of \Cref{thm:syst1} to obtain the upper estimate of the level set of $u$ to be
\begin{equation}
\xi_{\kappa}[u](t) \leq c_{\mathrm{GL}} t + \frac{q_u}{\lambda_{\rm GL}} \ln(1+t) + O(1),
\end{equation}
as well as the convergence to the travelling-wave profile. 
This completes the proof. 
\end{proof}

\section{Rigidity results for competition systems: proofs of \texorpdfstring{Theorems \ref{t.uniqueness} and \ref{thm:liouville}}{thm:liouville}}\label{sec:liouville}

In this section, we establish the rigidity
and uniqueness results for the competition
system. We present the main arguments
for $0<a<1<b$ and indicate the
modifications needed for $0<a,b<1$.
The proofs are self-contained and rely on two maximum principles
on unbounded spatial domains, which we
establish at the end of the section.

\begin{lemma}\label{l.liouville.1}
    Assume $0 < a < 1 < b$. Under the assumptions of \Cref{thm:liouville},  for each $A \ge 1$, there exist positive constants $k_A,K_A$ such that 
        \begin{equation}\label{e.918.1}
            0<k_A \leq \inf_{t\in \mathbb R, |y|\leq A} v(t,y) \leq \sup_{{t\in \mathbb R}, |y|\leq A} v(t,y)\leq K_A<1.
        \end{equation}
Furthermore,
\begin{equation}\label{e.917.0}
\lim_{y \to +\infty} v(t,y) = 1 \quad \text{ holds uniformly in }t\in \mathbb{R}.
\end{equation}
Moreover,
    \begin{equation}\label{e.917.1}
\lim_{y \to -\infty} (u(t,y), v(t,y)) = (1,0), \quad \text{ holds uniformly for }t\in \mathbb{R}. 
\end{equation}

\end{lemma}

In the case $0<a,b<1$, the conclusions are modified as follows. 
\begin{align}\label{eq:changesab}
\text{(i)} &\quad \eqref{e.918.1} \text{ holds with }\tfrac{1-b}{1-ab}< k_A \leq K_A <1, \\
\text{(ii)} &\quad \eqref{e.917.0} \text{ holds without change},
\end{align}
and \eqref{e.917.1} needs to be replaced by
    \begin{equation}\label{e.917.1'}
\lim_{y \to -\infty} (u(t,y), v(t,y)) = (\tfrac{1-a}{1-ab},\tfrac{1-b}{1-ab}), \quad \text{ holds uniformly for }t\in \mathbb{R}.
\end{equation}
The condition \eqref{e.917.-2} (in addition to \eqref{e.917.-1}) is needed for the validity of \eqref{e.917.1'}.

\begin{proof}[{\bf Proof of \Cref{l.liouville.1}}]
Since $u$ and $v$ are nonnegative and uniformly bounded, the inequalities $0 \leq u \leq 1$ and $0 \leq v \leq 1$ follow from standard comparison with spatially homogeneous solutions of \eqref{eq:LV}. 
Since $u$ is neither identically $0$ nor identically $1$ (due to \eqref{e.917.2}), we have $0 < u < 1$ thanks to the strong maximum principle. Similarly, since $v \not\equiv 0$ (due to \eqref{e.917.-1}) and $v\not\equiv 1$ (otherwise, the second equation of \eqref{eq:LV} implies $u\equiv 0$ and contradicts \eqref{e.917.2}), we have $0 < v < 1$.

Next, we prove \eqref{e.918.1}. The uniform lower bound follows from \eqref{e.917.-1} and the Harnack inequality. It remains to prove the upper bound. 
Suppose, to the contrary, that there is a sequence $(t_n,y_n)$ such that $y_n \to \bar y$ and $v(t_n,y_n) \to 1$ (or, resp., $v(t_n,y_n) \to 0$). Then, by compactness and the strong maximum principle, the sequence of functions $(\widetilde{u}_n,\widetilde{v}_n)(t,y) = (u,v)(t+t_n,y)$ converges (up to passing to a subsequence) to $(\widetilde u(t,y),1)$ (resp. $(\widetilde u(t,y),0)$). 
In case $\widetilde{v}_n \to 1$, we can pass to the limit in the second equation of \eqref{eq:LV} to obtain $\widetilde u \equiv 0$, which is impossible, since $\widetilde u$ satisfies \eqref{e.917.2}. We have proved \eqref{e.918.1}.

Next, we show \eqref{e.917.1}. Indeed, for any $t_n \in \mathbb R$ and $y_n \to -\infty$, we have (up to passing to a subsequence), 
\begin{equation}
(u_n,v_n)(t,y) := (u,v)(t + t_n, y + y_n) \to (\hat u, \hat v)(t,y) \quad \text{ in } \mathscr{C}_{\mathrm{loc}}(\mathbb R^2),
\end{equation}
where $(\hat u, \hat v)$ is another bounded entire solution of \eqref{eq:LV} such that $\inf_{\mathbb{R}^2} \hat u >0$.
We deduce that 
\begin{equation}
(\hat u, \hat v) \equiv (1,0) \qquad \text{ in }\mathbb{R}^2,
\end{equation}since every bounded, nonnegative entire solution of the homogeneous competition system whose $u$-component has 
a strictly positive infimum value is identically $(1,0)$, in the regime $0<a<1<b$ (see \cite[Lemma 2.3]{liu2019asymptotic} or \cite[Proof of Theorem 1.2]{girardin2019invasion}).
Since $(t_n,y_n)$ are arbitrary sequences with $y_n \to -\infty$, we have proved \eqref{e.917.1}. The proof of \eqref{e.917.0} is similar, as every positive bounded entire solution of the Fisher-KPP equation whose infimum value is positive is identically $1$ in $\mathbb R^2$.
\end{proof}

For $h
\in \mathbb R$ and  $\xi \in \mathbb{R}$, define
\begin{equation}
u^\xi(t,y) = u(t+\xi, y + h - \xi c), \quad 
(p^\xi(t,y), q^\xi(t,y)) = (u^\xi - u, v^\xi - v).
\end{equation}
Then the differences $(p^\xi, q^\xi)$ satisfy the linear problem:
\begin{equation}
\begin{cases}
    p_t- p_{yy} -cp_y= p(1-u^\xi - u - av) - au^\xi q  &\quad \text{ in }\mathbb{R}^2\\
    q_t - dq_{yy} - cq_y = -rb v^\xi p + rq(1-bu - v^\xi - v)&\quad \text{ in }\mathbb{R}^2
\end{cases}
\end{equation}
We need the following lemmas, which apply in both case (i) $0<a<1<b$ and case (ii) $0<a,b<1$, and which are derived from the stability of the homogeneous states. Their proofs are postponed to the end of the section to improve clarity of the overall argument.
\begin{lemma}\label{l.liouville.2}
There exists $A_1\geq 0$  such that for any $\xi > \frac{h}{c}$ and any $\bar y \geq A_1$, we have 
\begin{equation}
\inf_{y \geq \bar y} p^\xi(t,y)\geq 0,\quad q^\xi(t,\bar y)<0 ~\mathrm{ for } \; t\in \mathbb{R}\quad \Longrightarrow  \quad q^\xi(t,y)<0 \quad \mathrm{ for } \; t\in \mathbb R,~y \geq \bar y.
\end{equation}
\end{lemma}

\begin{lemma}\label{l.liouville.3}
There exists $A_2\geq 0$  such that for any $\xi >\frac{h}{c}$ and $\bar y \leq -A_2$, if 
\begin{equation}\label{e.918.8a}
p^\xi(t,\bar y)>0 > q^\xi(t,\bar y) \; \mathrm{ for } \; t\in \mathbb{R},
\end{equation}
then
\begin{equation}
p^\xi(t,y)>0 > q^\xi(t,y) \quad \; \mathrm{ for } \; t\in \mathbb R,~y \leq \bar y.
\end{equation}
\end{lemma}

Now, we fix $h\in \mathbb{R}$ and define the following set (see also \cite{hamel2016logarithmic})
\begin{equation}
\Xi:= \left\{\xi'>\frac{h}{c}: p^\xi > 0 > q^\xi ~\text{ in }\mathbb{R}^2 ~\text{ for every }\xi \geq \xi'\right\}.
\end{equation}

\begin{lemma}\label{l.liouville.4}
  Define $\xi_* := \inf \Xi$. Then  $\xi_* =  \frac{h}{c}.$
\end{lemma}

\begin{proof}[{\bf Proof of \Cref{l.liouville.4}}]
By Lemmas \ref{l.liouville.2} and \ref{l.liouville.3}, it follows that $\Xi$ 
is non-empty. Since $\Xi$ is automatically bounded from below by construction, we deduce that $\xi_*$ is well defined. Indeed, based on the uniform convergence to $(0,1)$ and $(1,0)$ as $y \to \pm \infty$ and the detailed decay property \eqref{e.917.2}, we can choose $\xi_1 \gg 1$ such that, uniformly in $\xi \geq \xi_1$ and $t\in \mathbb{R}$,  
\begin{equation}\label{e.919.1}
p^\xi >0 \quad \text{ in }[-A_2,\infty), \quad \text{ and }\quad q^\xi <0 \quad \text{ in }[-A_2,A_1].
\end{equation}
Thanks to  \Cref{l.liouville.3} (resp. \Cref{l.liouville.2}), we can further replace $-A_2$ by $-\infty$ (resp. $A_1$ by $+\infty$) in \eqref{e.919.1}. This proves that $[\xi_1,\infty) \subset \Xi$. Hence it is nonempty.

Suppose on the contrary that $\xi_* > \frac{h}{c}$. By continuity, $p^{\xi_*} \geq  0\geq   q^{\xi_*} $ in $\mathbb{R}^2$. 
We claim that 
    \begin{equation}\label{e.917.3a}
    p^{\xi_*} > 0 >  q^{\xi_*} \quad \text{ in }\mathbb{R}^2.
    \end{equation}
If equality holds anywhere, by the strong maximum principle for the cooperative, irreducibly coupled system satisfied by $(p^\xi, -q^\xi)$, we have $(p^\xi,-q^\xi) \equiv (0,0)$ in $\mathbb{R}^2$, and hence $(u^{\xi_*},v^{\xi_*}) \equiv (u,v)$. In particular, 
   \begin{equation}\label{e.917.3b}
  u(t+ k\xi_*, x + k(h- \xi_*c)) \equiv u(t,x) \quad \text{ in }\mathbb{R}^2  
  \end{equation}
  holds for $k=1$.  Iterating, we deduce that \eqref{e.917.3b} holds for all $k \in \mathbb{N}$, and this contradicts \eqref{e.917.2}. Hence, \eqref{e.917.3a} holds, which also implies that $\xi_* \in \Xi$.
  
  Next, we claim the following.
      \begin{equation}\label{e.917.3}
    \inf_{t \in \mathbb{R},~ |y|\leq M} p^{\xi_*} > 0 >  \sup_{t \in \mathbb{R},~ |y|\leq M} q^{\xi_*} \quad \text{ in }\mathbb{R}^2,~\text{ for every }M>0.
    \end{equation}
    Indeed, suppose not; then (still fixing $\xi = \xi_*$) we may pass to a sequence $(t_n,y_n)$ such that $y_n \to y_\infty$ exists and is finite, and \begin{equation}
    \begin{cases}
    (u,v)(t+t_n, y) \to (u_\infty,v_\infty)(t,y) &\quad \text{ in }\mathscr{C}_{\mathrm{loc}}(\mathbb{R}^2),\\
     (p^{\xi_*},q^{\xi_*})(t+t_n,y) \to (p_\infty,q_\infty)(t,y) &\quad \text{ in }\mathscr{C}_{\mathrm{loc}}(\mathbb{R}^2),\\
     p_\infty(0,y_\infty) =0 \quad \text{ or }\quad q_\infty(0,y_\infty) = 0. &
     \end{cases}\end{equation}
     Note that $u_\infty$ satisfies \eqref{e.917.2} and $v_\infty >0$ thanks to \eqref{e.917.-1}. 
   Now, the differences satisfy the non-strict inequalities $p_\infty \geq 0 \geq q_\infty$ in $\mathbb{R}^2$ and at least one equality holds at $(t,y) = (0,y_\infty)$. By repeating the above argument, we deduce that $u_\infty$ satisfies \eqref{e.917.3b} for $k=1$ and, by iteration, $k \in \mathbb{N}$. This is again in contradiction with the quantitative tail estimate \eqref{e.917.2}.

Thanks to \eqref{e.917.3} and \Cref{l.liouville.3}, for each $A>0$,  there exists $\delta_A>0$ sufficiently small
such that $\xi_*-\delta_A >\frac{h}{c}$ and 
\begin{equation}\label{ineq:2sides}
p^\xi > 0 > q^\xi \quad \text{ for }\xi \in (\xi_*-\delta_A,\xi_*),~t\in \mathbb{R},~ y \in (-\infty,A].
\end{equation}
Hence, we may choose a sequence $\underline{\xi}_n \nearrow \xi_*$ such that for all $\xi \in [\underline{\xi}_n,\xi_*)$, 
\begin{equation}\label{e.110.5}
p^{\xi} > 0 > q^{\xi}\quad \text{ in }\R\times (-\infty,n].
\end{equation}
Note that $\underline{\xi}_n<\inf \Xi$, so there exist $\xi_n \in [\underline{\xi}_n,\xi_*)$ and 
 $(t_n,y_n)$ with $y_n>n$ such that
\begin{align}\label{p-xi-n<0}
p^{\xi_n}(t_n,y_n)\leq 0 \quad \text{for all } n.
\end{align}
Indeed, suppose no such $\xi_n,t_n,y_n$ exist. 
Then for every $\xi \in [\underline\xi_n,\xi_*)$,  $p^\xi>0$ in $\mathbb{R}^2$ (using also \eqref{e.110.5} here). Moreover, by \Cref{l.liouville.2}, we have $q^{\xi}<0$ in $\R\times [A_1,\infty)$ for every $\xi \in [\underline\xi_n,\xi_*)$. Together with \eqref{ineq:2sides} and the fact that $\xi_* \in \Xi$, we have $\underline{\xi}_n\in \Xi$. We have reached a contradiction. Therefore, \eqref{p-xi-n<0} holds and $y_n\to\infty$.

Next, define
\begin{equation}
U_n(t,y):=  \frac{u(t+t_n,y + y_n )}{u(t_n,y_n)}, \quad \text{ and }\quad \hat{U}_n(t,y):= \frac{u(t+t_n + \xi_*,y + y_n +h - \xi_* c)}{u(t_n,y_n)}.
\end{equation}
Thanks to the uniform estimate \eqref{e.917.2}, we have
\begin{equation}
0 \leq U_n \lesssim e^{-\lambda y} \quad \text{ for }t\in \mathbb{R},~ y \in [ C_1 - y_n, C_1 + 2y_n]
\end{equation}
This estimate and parabolic regularity guarantee that, after extracting a common subsequence, both sequences converge locally uniformly to limits
\begin{equation}
U = \lim U_n \quad \text{ and }\quad \hat U = \lim \hat{U}_n \quad \text{ in }\mathscr{C}_{\mathrm{loc}}(\mathbb{R}^2).
\end{equation} 

Since $y_n\to\infty$, by \Cref{l.liouville.1} and \eqref{e.917.2}, $U$ and $\hat{U}$ both solve the linear equation
\begin{equation}\label{U-eq0923}
U_t-U_{yy}-cU_y=(1-a)U\quad \text{ in } \R^2.
\end{equation}
Furthermore, $U(0,0)=1$ and $\hat{U}\geq U$ because $p^{\xi_*}\geq 0$ in $\R^2$. On the other hand, by \eqref{p-xi-n<0}, we have
\begin{equation}
 \widehat{U}_n\left(\xi_n-\xi_*,c(\xi_*-\xi_n)
    \right)=
    \frac{
      u(t_n+\xi_n,y_n+h-c\xi_n)
    }{u(t_n,y_n)}
    \le1.
\end{equation}
Taking $n\to\infty$, $\hat{U}(0,0)\leq 1=U(0,0)$. Consequently, $\hat{U}(0,0)=U(0,0)$.
Since $\hat{U}-U\geq 0$ solves \eqref{U-eq0923}, the strong maximum principle gives $\hat{U}\equiv U$ in $\R^2$.

We now can use an iteration argument to prove that, for each fixed $k\in \mathbb{N}$,
\begin{align}\label{iteration-k}
\frac{u(t_n + k\xi_*, y_n + k(h-\xi_*c))}{u(t_n,y_n)} \to 1 \quad \text{ as } n\to\infty.
\end{align}
To see this, since $\hat{U}\equiv U$ in $\R^2$, we have 
\begin{align}\label{limit0923}
 \widehat U_n(t,y)-U_n(t,y)\to0 \quad \text{ for all } (t,y)\in\R^2.    
\end{align}
Now, fix any $k\geq 1$. Evaluating \eqref{limit0923} at $(t,y)=((k-1)\xi_*, (k-1)(h-\xi_*c))$ gives
\begin{equation}
\frac{u(t_n+k\xi_*,y_n{+}k(h-\xi_*c))}{u(t_n,y_n)}-\frac{
  u\left(
    t_n+(k-1)\xi_*,
    y_n{+}(k-1)(h-\xi_*c)
  \right)
}{u(t_n,y_n)}\to 0 \quad \text{ as } n\to\infty,
\end{equation}
for any fixed $k$. Inductively in $k$, we obtain \eqref{iteration-k}.

But this violates the tail decay hypothesis \eqref{e.917.2}. Indeed, by \eqref{e.917.2}, there exists
$R>0$ such that
\begin{equation}
u(t,y) \asymp (\ell_0+\ell y)e^{-\lambda y},\quad t\in\R, \ y\geq R.
\end{equation}
For any fixed $k$, since $y_n{+}k(h-\xi_*c)\to\infty$, we have 
\begin{align}
1=\liminf_{n\to\infty} \frac{u(t_n+k\xi_*,y_n+k(h-\xi_*c))}{u(t_n,y_n)}&\gtrsim
e^{-\lambda k(h-\xi_*c)}\lim_{n\to\infty}\frac{\ell_0+\ell(y_n+k(h-\xi_*c))}{\ell_0+\ell y_n}\\
&\gtrsim e^{{\lambda k(\xi_*c-h)}}.        
\end{align}
Taking $k\gg1$ such that the latter is strictly greater than $1$ leads to a contradiction with \eqref{iteration-k}.
Therefore, $\xi_*>\frac{h}{c}$ is impossible.
We have proven $\xi_* = \frac{h}{c}$.
\end{proof}

\begin{proof}[{\bf Proof of \Cref{thm:liouville}}]
Let $h\in \mathbb{R}$ be given; the sliding argument implies (by taking $\xi = \frac{h}{c}$)
\begin{equation}\label{e.918.2}
u\left(t + \frac{h}{c}, y\right) \geq u(t,y) \quad \text{ and }\quad v\left(t + \frac{h}{c}, y\right) \leq v(t,y) \quad \text{ in }\mathbb{R}^2.
\end{equation}
Repeating the argument with $h$ replaced by $-h$, we obtain the reverse inequality.
Hence, we obtain that 
 $p^{\frac{h}{c}} \equiv 0$ and $q^{\frac{h}{c}} \equiv 0$ in $\mathbb{R}^2$, i.e.
\begin{equation}
(u,v)\left(t + \frac{h}{c}, y\right) = (u,v)(t,y) \quad \text{ for }(t,y) \in \mathbb{R},~ h \in \mathbb{R}.
\end{equation}
Since $h$ is arbitrary, $(u,v)$ is independent of $t$ and hence is a travelling-wave profile. 
Moreover, as a by-product of the sliding method, we have in fact proved that $U_{\mathrm{TW}}$ is decreasing in $y$ and $V_{\mathrm{TW}}$ is increasing in $y$ and they connect $(1,0)$ at $-\infty$ to $(0,1)$ at $+\infty$. We have finished the proof for \Cref{thm:liouville} in the case $0<a<1<b$, except for the uniqueness of travelling wave solutions (\Cref{t.uniqueness}), which we will present separately.

For the case $0<a,b<1$, we use \Cref{l.liouville.1} in conjunction with \eqref{eq:changesab} and \eqref{e.917.1'}. Assuming the validity of Lemmas \ref{l.liouville.2} and \ref{l.liouville.3}, the sliding argument is completely analogous to the former case, and we omit the details.
\end{proof}

Next, we show that the above sliding argument implies a general uniqueness result for travelling-wave solutions.
\begin{proof}[{\bf Proof of \Cref{t.uniqueness}}]
Fix $c>0$, and 
suppose $(U,V)$ and $(\widetilde U, \widetilde V)$ are wave profiles satisfying \eqref{TW-sys} with $E^+=(0,1)$ and $E^- = (1,0)$ in case $b>1$ and $(u^\star,v^\star) = (\tfrac{1-a}{1-ab}, \tfrac{1-b}{1-ab})$ in case $0<b<1$. By a translation, we may assume that
\begin{equation}\label{e.921.10}
\lim_{y\to+\infty} \frac{e^{\lambda y}}{\ell_0 + \ell_1 y}U(y) =1 = \lim_{y\to+\infty} \frac{e^{\widetilde{\lambda} y}}{\widetilde\ell_0 + \widetilde\ell_1 y}\widetilde U(y)     
\end{equation}
 for some $\lambda,\widetilde\lambda>0$ and $\ell_i, \widetilde\ell_i \geq 0$ with $\ell_0 + \ell_1>0$ and $\widetilde\ell_0 + \widetilde\ell_1 >0$; see \cite[Appendix A]{girardin2019invasion}. 
Thanks to \eqref{e.921.10}, the tail of travelling waves must be ordered. Therefore, by swapping $(U,V)$ and $(\widetilde U,\widetilde V)$ if necessary, we may reduce to the following cases:
\begin{equation}
\textup{(i)} \quad  \frac{\widetilde U(y)}{{U}(y)} \to +\infty \quad \text{ or }\quad \textup{(ii)} \quad  \frac{\widetilde{U}(y)}{U(y)} \to 1 \quad \text{ as }y\to+\infty.
\end{equation}
To this end, define
\begin{equation}
    \widetilde{p}^\xi(y) = \widetilde{U}(y-c\xi) - U(y) \quad \text{ and }\quad \widetilde{q}^\xi(y) = \widetilde{V}(y-c\xi) - V(y).
\end{equation}
and define
\begin{equation}\label{e.921.7}
\widetilde{\Xi}:= \{\xi' \in \mathbb{R}:~ \widetilde{p}^\xi >0 > \widetilde{q}^\xi \quad \text{ in }\mathbb{R}~\text{ for every }\xi \geq \xi'\}.   
\end{equation}
(Note that in case (ii), $\widetilde\Xi \subset [0,\infty)$.) The proofs of \Cref{l.liouville.2,l.liouville.3}
also apply to differences of two
travelling-wave profiles with the same
speed and limiting equilibria.
Indeed, the resulting linearised
difference system has the same
cooperative structure and the same
limiting matrices at spatial infinity.
We use these extensions below.

\medskip

\noindent {\bf \# Step 1. } We claim that $\widetilde\Xi$ is nonempty and bounded below. Indeed, by the limits of $(U,V)$ and $(\widetilde U, \widetilde V)$ at $\pm\infty$, it is clear that $\widetilde\Xi$ is bounded below. To see it is nonempty, take $\xi \gg 1$, so that 
\begin{equation}
\widetilde p^\xi >0 \quad \text{ in }[-A_2,\infty) \quad \text{ and }\quad \widetilde q^\xi <0 \quad \text{ in }[-A_2,A_1].
\end{equation}
Then it follows from Lemmas \ref{l.liouville.2} and \ref{l.liouville.3} (note that $(U, V)$ are automatically bounded entire solutions of \eqref{eq:LV}, so these lemmas are applicable) that $\xi \in \widetilde\Xi$ for all $\xi \gg 1$. 

In particular, $\xi_* = \inf \widetilde\Xi$ is well defined. 

\medskip

\noindent {\bf \# Step 2. }
At $\xi = \xi_*$, we claim that either $(\widetilde U(\cdot - c\xi_*),\widetilde V(\cdot - c\xi_*)) \equiv (U,V)$ (then we are done and have ruled out case (i) along the way) or 
\begin{equation}\label{e.921.8}
\widetilde{p}^{\xi_*} >  0 > \widetilde{q}^{\xi_*} \quad \text{ in }\mathbb{R}.\end{equation}
To see that, observe that the non-strict version of  \eqref{e.921.8} follows from continuity. The dichotomy follows from the strong maximum principle. Henceforth we proceed assuming \eqref{e.921.8}.

\medskip

\noindent {\bf \# Step 3.}
Assuming either (i) $\frac{\widetilde{U}}{U} \to +\infty$ or (ii') $\frac{\widetilde{U}}{U} \to 1$  and $\xi_* >0$, then
\begin{equation}
\liminf_{y \to +\infty} \frac{\widetilde{U}(y-c\xi_*)}{U(y)} >1,
\end{equation}
and 
we can find $\xi_{**} < \xi_*$ such that for any $\xi \in [\xi_{**},\xi_*)$, 
\begin{equation}
\widetilde{p}^{\xi}>0 \quad \text{ in }[-A_2,\infty) \quad \text{ and }\quad \widetilde{q}^{\xi}<0 \quad \text{ in }[-A_2,A_1].
\end{equation}
It follows from Lemmas \ref{l.liouville.2} and \ref{l.liouville.3} that  for any $\xi \in [\xi_{**},\xi_*)$, \begin{equation}
\widetilde{p}^{\xi}>0 > \widetilde{q}^{\xi} \quad \text{ in }\mathbb{R},
\end{equation} i.e. 
$\xi_{**} \in \widetilde\Xi$. This is a contradiction. Hence, we have ruled out the possibility of case (i) $\frac{\widetilde{U}}{U} \to +\infty$, i.e. any two travelling wave solutions with the same speed $c$ must have the same asymptotics as in case (ii) $\frac{\widetilde{U}}{U} \to 1$. Moreover, we must have $\xi_* \leq 0$, i.e. 
\begin{equation}
\widetilde{U}(y) \geq U(y) \quad \text{ and }\quad \widetilde{V}(y) \leq V(y) \quad \text{ in }\mathbb{R}.
\end{equation}

\medskip

\noindent {\bf \# Step 4.} 
Interchanging the roles of $\widetilde U$ and $U$ (since $\frac{U}{\widetilde{U}} \to 1$ at $+\infty$ is a symmetric relation), we show that the reverse inequalities hold, so $(\widetilde{U},\widetilde{V}) \equiv (U,V)$.
\end{proof}

We shall finish the section with the proofs of \Cref{l.liouville.2} and \Cref{l.liouville.3}, which have been used in the proof of \Cref{t.uniqueness,thm:liouville} in an essential way.

\begin{proof}[{\bf Proof of \Cref{l.liouville.2}}] 
The argument is the same in both cases $b >1$ and $0<b<1$.
Let $h \in \mathbb{R}$ be fixed and let $\xi > \frac{h}{c}$. 
    Choose $A_1>1$ (which is independent of $\xi$) such that 
    \begin{equation}\label{e.918.6}
    v(t,y) \geq \frac34 \quad \text{ for }t\in \mathbb R,~ y\geq A_1.
    \end{equation}
    Let $\bar y \geq A_1$. 
Suppose $q^\xi(t,\bar y) <0$ for $t \in \mathbb R$ and $p^\xi(t,y)\geq 0$ for $t\in \mathbb R$, $y \geq \bar y$. We claim that 
\begin{equation}\label{e.918.7}
    q^\xi(t,y) <0 \quad \text{ for }t\in \mathbb R,~ y\geq \bar y.
\end{equation}
Suppose not; then $M:= \sup_{\mathbb R \times [\bar y,\infty)} q^\xi >0.$ Since $q^\xi(t,+\infty) \to 0$ (thanks to \eqref{e.917.0}) uniformly in $t$, the maximising sequence $(t_n,y_n)$ satisfies $\sup |y_n|<+\infty$. Passing to a subsequence, we have $y_n \to  y^\star$ and $({u}_n,{v}_n)(t,y):= (u,v)(t+t_n,y)$ converge to some entire solution $(u_\infty , v_\infty)$ in $\mathscr{C}_{\mathrm{loc}}(\mathbb R^2)$. Moreover, $y^\star \neq \bar y$ since $q(t,\bar y) <0$, so $y^\star > \bar y$. Then the corresponding differences 
\begin{equation}
(p^\xi_n,q^\xi_n)(t,y)=(p^\xi,q^\xi)(t+t_n,y) \to (p_\infty, q_\infty)(t,y)
\end{equation}
with 
\begin{equation}
q_\infty(0,y^\star) = M = \sup_{\mathbb{R} \times [\bar y,\infty)} q_\infty.
\end{equation}
Hence,
\begin{equation}\label{e.918.5}
    (q_\infty)_t - d(q_\infty)_{yy} - c(q_\infty)_y \geq 0 \quad \text{ at }(0,y^\star).
\end{equation}
But $M>0$, so by \eqref{e.918.6}, we have
\begin{equation}
v^\xi_\infty(0,y^\star) = v_\infty(0,y^\star) + M \geq v_\infty(0,y^\star) \geq \frac34.
\end{equation}
The equation of $q_\infty$ then gives
\begin{align}
(q_\infty)_t - d(q_\infty)_{yy} - c(q_\infty)_y  &= -rbv^\xi_\infty p_\infty + rq_\infty(1-bu_\infty - v^\xi_\infty - v_\infty)\\
    &\leq rq_\infty(1- \frac34 - \frac34) = -\frac{rM}2 \quad \text{ at }(0,y^\star). 
\end{align}
This contradicts \eqref{e.918.5}, which proves 
\begin{equation}\label{e.918.7a}
q^\xi(t,y) \leq 0 \quad \text{ for }t\in \mathbb{R},~ y \geq \bar y.
\end{equation}
By the strong maximum principle and the fact that $q^\xi(t,\bar y) <0$, \eqref{e.918.7a} implies \eqref{e.918.7}.
\end{proof}
\begin{proof}[{\bf Proof of \Cref{l.liouville.3}}] We first handle the case $0<a<1<b$ in detail. 
Set $P = p^\xi$ and $Q = -q^\xi$; we have the cooperative system
\begin{equation}\label{e.918.9}
        Z_t - DZ_{yy} -c Z_y = B^\xi(t,y) Z, \quad Z = \begin{pmatrix}P\\Q\end{pmatrix}, 
    \end{equation}
    where $D = \mathrm{diag}(1,d)$ and  $B^\xi$ is a (Metzler) cooperative matrix
    \begin{equation}
        B^\xi = \begin{pmatrix}1 - u^\xi- u - av & au^\xi \\ rbv^\xi & r(1-bu - v^\xi - v)  \end{pmatrix} .
    \end{equation}
Since $\xi > \frac{h}{c}$, we deduce that $(u,v)(t,y)$ and $(u^\xi, v^\xi)(t,y)$ converge to $(1,0)$ as $y \to -\infty$ uniformly in $t$ and uniformly in $\xi>\frac{h}{c}$. Consequently, \begin{equation}
\lim_{y\to-\infty} B^\xi(t,y) = B_-:= \begin{pmatrix} -1  & a \\ 0 & -r(b-1)  \end{pmatrix} \quad \text{ uniformly in }t \text{ and }\xi>\frac{h}{c}.
\end{equation}
Choose $\vec{\mathbf{e}}=(1,1)^\top$; then 
\begin{equation}\label{e.921.2}
B_- \vec{\mathbf{e}} = \begin{pmatrix} -1 +a \\ -r(b-1)\end{pmatrix} \ll_{\mathbb{R}^2} 0.
\end{equation}
Hence, we can choose $\mu >0$ and  $A_2 \gg 1$ (independently of $\xi>\frac{h}{c}$) such that
\begin{equation}\label{e.918.3a}
    B^\xi(t,y) \vec{\mathbf{e}} \leq_{\mathbb{R}^2} - \mu \vec{\mathbf{e}} \quad \text{ for }t \in \mathbb R,~ y \leq -A_2.
\end{equation}
Now, suppose \eqref{e.918.8a} holds for some $\bar y \leq -A_2$. Then $Z(t,\bar y) \gg_{\mathbb{R}^2} 0$ for all $t$. Note also that 
\begin{equation}
Z(t,y) \to 0 \quad \text{ as }y \to -\infty,~\text{ uniformly in }t.
\end{equation}
Again, (thanks to the strong maximum principle) it suffices to show $Z \geq 0$ on $\mathbb R \times (-\infty, \bar y]$.

Fix an arbitrary point $(t_0,y_0) \in \mathbb R \times (-\infty,\bar y]$, and an arbitrary constant $\varepsilon>0$. 
Since $Z$ is globally bounded, we can choose $M>0$ such that
\begin{equation}\label{e.918.11}
Z(t,y) + M\vec{\mathbf{e}} \geq 0 \quad \text{ for all }t\in \mathbb R,~ y \leq \bar y.
\end{equation}
Then choose $ L=L(\varepsilon) \gg 1$ such that
\begin{equation}\label{e.918.10}
-L < y_0, \quad Z(t,-L) + \varepsilon \vec{\mathbf{e}} \geq 0 \quad \text{ for all }t.    
\end{equation}

For $T>t_0$, define on the domain
$Q_{T,L} := (-T,t_0] \times [-L,\bar y]$
the barrier function
\begin{equation}
W = Z + \varepsilon \vec{\mathbf{e}} + M e^{-\alpha(t+T)} \vec{\mathbf{e}},
\end{equation}
with $\alpha \in (0,\mu)$.
We compute 
\begin{equation}
\mathcal{L} W := W_t - D W_{yy} - c W_y - B^\xi W.
\end{equation}
Using $\mathcal{L}Z = 0$, we have
\begin{align}
\mathcal{L}W
&= -\varepsilon B^\xi \vec{\mathbf{e}}
   + M e^{-\alpha(t+T)}
     \left(-\alpha \vec{\mathbf{e}} - B^\xi \vec{\mathbf{e}}\right) \notag\\
&\geq \mu\varepsilon \vec{\mathbf{e}}
   + M(\mu-\alpha)e^{-\alpha(t+T)}\vec{\mathbf{e}}
   \gg 0
\end{align}
where we used \eqref{e.918.3a}, 
i.e. $W$ is a strict supersolution of the cooperative system \eqref{e.918.9}. 

Now, we check that $W \geq 0$ on the parabolic boundary of $Q_{T,L}$:
\begin{equation}
W(t,\bar y) \gg 0 \qquad\text{ since }Z(t,\bar y)\gg 0.
\end{equation}
\begin{equation}
    W(t,-L) \geq 0 \qquad \text{ by }\eqref{e.918.10}.
\end{equation}
\begin{equation}
    W(-T,y) = Z(-T,y) + \varepsilon \vec{\mathbf{e}} + M \vec{\mathbf{e}} \geq 0 \quad \text{ by }\eqref{e.918.11}.
\end{equation}
Hence, by the  comparison principle, we have $W \geq 0$ in $Q_{T,L}$. In particular,
\begin{equation}
Z(t_0,y_0) + \varepsilon \vec{\mathbf{e}} + M e^{-\alpha(t_0 + T)} \vec{\mathbf{e}} \geq 0 \quad \text{ in }Q_{T,L}.
\end{equation}
Letting $T\to\infty$ (note that $L,M$ are independent of $T$) and then $\varepsilon \searrow 0$, we have $Z(t_0,y_0) \geq 0$. Since $(t_0,y_0)$ is arbitrary, $P \geq 0$ and $Q\geq 0$ for $y \leq \bar y$. Again, the strict inequality holds thanks to the strong maximum principle. The lemma is proved for the case $b>1$.

For the case $0<b<1$, we have $(u,v) \to (u^{\star},v^{\star}):= (\tfrac{1-a}{1-ab},\tfrac{1-b}{1-ab})$ as $y\to-\infty$, and observe  
\begin{equation}
B_- = \lim_{y \to -\infty} B^\xi =\frac1{1-ab} \left( \begin{array}{cc}-(1-a) & a(1-a) \\ rb(1-b) & -r(1-b) \end{array} \right),
\end{equation}
which again has two negative eigenvalues, and satisfies again \eqref{e.921.2} with the same choice $\vec{\mathbf{e}}=(1,1)^T$:
\begin{equation}
{
B_-\mathbf e
=-\frac1{1-ab}
\begin{pmatrix}
(1-a)^2\\
r(1-b)^2
\end{pmatrix}
\ll0.
}
\end{equation}
Thus \eqref{e.918.3a} also holds
in this case, and the preceding
barrier argument applies unchanged.
\end{proof}

\section*{Acknowledgements} 
KYL is supported in part by a grant from the Simons Foundation (SFI-MPS-TSM-00025569) and the National Science Foundation (DMS-2602291).
CHW is partly supported by the Taiwan National Science and Technology Council (NSTC 113-2628-M-A49-004-MY4, 113-2918-I-A49-002) and the National Centre for Theoretical Sciences. CHW also thanks the Mathematical Research Institute (MRI) at Ohio State University for hospitality during his sabbatical
visit in 2024-2025.
EB and KYL are grateful to the Banff International Research Station for its warm welcome during an event in November 2025, at which this work was initiated. 

\section*{Declaration of generative AI use}
The paper being rather long, generative AI tools were used during the preparation of this manuscript
for language editing and for streamlining mathematical
arguments to improve readability. All the new mathematical ideas, however, come from the authors. The authors have independently written 
all mathematical statements, proofs, computations, and
references, and 
take full responsibility for the content of the manuscript.

\appendix

\section{A particular supersolution to the competition system}\label{app:auxiliary}

\begin{lemma}
\label{lem:perturbed-TW}
Let $0<a<1<b$. 
For every $c>c^0_{\mathrm{LLW}}$, there exists
$\alpha_0>0$ such that, for each
$\alpha\in(0,\alpha_0]$, there exist
$K,\chi_0^\star,\chi_1>0$ such that, for any 
$0<\chi_0\le\chi_0^\star$, and $L\in\mathbb R$,
the functions
\begin{equation}
\chi(t):=\chi_0 e^{-\chi_1 t},
\qquad
p(t):=K\int_0^t \chi(\tau)\, \dd \tau
=\frac{K\chi_0}{\chi_1}
(1-e^{-\chi_1 t}),
\end{equation}
and
\begin{equation}
(\mathfrak{U}, \mathfrak{V}) 
=
\left( U_{c}^{\alpha}
\left(x-c t-L-p(t)\right),\,\left[V_{c}^{\alpha}
\left(x-c t-L-p(t)\right)
-\chi(t)
\right]_+ \right)
\end{equation}
form a generalised supersolution of \eqref{e.com}
with respect to the competitive ordering.
\end{lemma}
\begin{proof}[{\bf Proof of \Cref{lem:perturbed-TW}}]
By continuity of
$\alpha\mapsto c^\alpha_{\rm LLW}$,
we may choose $\alpha_0>0$ such that
$c>c^\alpha_{\rm LLW}$ whenever
$0<\alpha\le\alpha_0$. Fix such an $\alpha$, and set
\begin{equation}
\xi=x-ct-L-p(t),\qquad
U=U_c^\alpha(\xi),\quad
V=V_c^\alpha(\xi).
\end{equation}
Then
\begin{equation}
\mathfrak U=U,\qquad
\mathfrak V=(V-\chi)_+,
\qquad
p'=K\chi,\quad
\chi'=-\chi_1\chi.
\end{equation}

\medskip

\noindent {\bf \# First component.} 
Using 
$\mathfrak{V}=(V-\chi)_+$ and 
\eqref{e.TWeta},
\begin{align}
\mathcal{N}_1[\mathfrak{U}, \mathfrak{V}]&= U_t- U_{xx}
- U(1- U-a(V-\chi)_+)
\\
&=
U\left(1-\frac{{U}}{1+\alpha}-a
V\right)
-p'(t){U}'
-{U}\left(1-{U}-a({V}-\chi)_+\right)
\\
&=
\frac{\alpha}{1+\alpha}{U}^2
+a [-V + (V - \chi)_+] U
-K \chi U'\\
&\geq   \frac{\alpha}{1+\alpha} U^2
-a \chi U
-K\chi U'.
\end{align}
Since $U'<0$, the last term is nonnegative.

Moreover, by the exponential behaviour of the travelling wave
at $+\infty$,
\begin{equation}
-\frac{{U}'(\xi)}{{U}(\xi)}
\longrightarrow\lambda:= \frac12 \left( c - \sqrt{c^2 - 4(1-a)}\right)
\qquad\text{as }\xi\to+\infty.
\end{equation}
Hence there exists $M>0$ such that
\begin{equation}
-{U}'(\xi)\ge\frac{\lambda}{2}{U}(\xi),
\qquad \xi\ge M.
\end{equation}
Choose $K$ sufficiently large so that
\begin{equation}
K\frac{\lambda}{2}\ge a.
\end{equation}
Then, for $\xi\ge M$,
\begin{align}
\frac{\alpha}{1+\alpha} U^2
-a\chi U
-K \chi U'
&\ge
\frac{\alpha}{1+\alpha} U^2
+\chi U
\left(
K\frac{\lambda}{2}-a
\right) \ge0.
\end{align}
On the other hand, for $\xi\le M$, monotonicity gives
\begin{equation}
 U(\xi)\ge m_M:= \inf_{\xi \leq M}{U}>0.
\end{equation}
Consequently,
\begin{align}
\frac{\alpha}{1+\alpha}U^2-a\chi U-K \chi U'
&\ge
U\left(\frac{\alpha}{1+\alpha}m_M-a\chi\right)
\ge0,
\end{align}
provided $\chi_0>0$ is chosen sufficiently small so
that
\begin{equation}
a \chi_0
\le
\frac{\alpha}{1+\alpha}m_M.
\end{equation}
Thus $\mathcal{N}_1[\mathfrak{U}, \mathfrak{V}]\geq 0$ holds.

\medskip

\noindent {\bf \# Second component.} On $\{V>\chi\}$, we have
$\mathfrak V=V-\chi$. Hence
\begin{align}
\mathcal N_2[\mathfrak U,\mathfrak V]
&=(V-\chi)_t-d(V-\chi)_{xx}-r(V-\chi)(1-V+\chi-bU)\\
&=rV(1-V-bU)-K\chi V'
+\chi_1\chi -r(V-\chi)(1-V+\chi-bU)\\
&=\chi\left[
\chi_1+r(1-bU-2V)-KV'
\right]+r\chi^2.
\end{align}

We claim that $K$ may be chosen so that for some $\gamma>0$, we have
\begin{equation}
r(1-b{U}-2{V})-K V'
\le -\gamma
\qquad\text{on }\mathbb R.
\end{equation}

Indeed, since
\begin{equation}
({U},{V})(-\infty)=(1+\alpha,0),
\end{equation}
and $b>1$,
\begin{equation}
r(1-b{U}-2{V})
\longrightarrow
r(1-b(1+\alpha))<0
\qquad\text{as }\xi\to-\infty.
\end{equation}
Similarly,
\begin{equation}
r(1-b{U}-2{V})
\longrightarrow-r<0
\qquad\text{as }\xi\to+\infty.
\end{equation}
Therefore the set on which
$r(1-b U-2 V)$ can be nonnegative is contained
in a compact interval. Since the travelling wave is strictly
monotone, $ V'>0$ there. Increasing $K$ if necessary,
we may thus find $\gamma>0$ such that
\begin{equation}
r(1-b U-2 V)-K  V'
\le-\gamma
\qquad\text{for all }\xi\in\mathbb R.
\end{equation}

It follows that
\begin{equation}
\mathcal N_2
\le
\chi(-\gamma+\chi_1+r \chi).
\end{equation}
Choose $\chi_1$ and $\chi_0$ sufficiently small
so that
\begin{equation}
\chi_1+r\chi_0\le\gamma.
\end{equation}
Then
\begin{equation}
\mathcal{N}_2[\mathfrak{U}, \mathfrak{V}]\le0
\qquad\text{on }\left\{{V}>\chi\right\}.
\end{equation}

Finally, on $\left\{{V}\le \chi\right\}$ we have
$\mathfrak{V}=0$, which satisfies the $v$-inequality.
Therefore the positive-part construction gives a generalised
subsolution in the second component.
\end{proof}

\section{Results from previous papers \texorpdfstring{{\rm \cite{bramson1983convergence,hamel2016logarithmic,LamWu_preprint}}}{s = 2}}

\label{app:previouspaper}

The following lemma improves \cite[Lemma 2.3]{LamWu_preprint}.

\begin{lemma}\label{lem:C0}
Let $0<a<1$ and $\psi$ be the solution to  $$\psi_t - \psi_{xx} = (1-a)\psi$$ in $(0,\infty)\times \mathbb R$, with initial data $\psi_0$ satisfying, for some $q\in \mathbb{R}$, $\lambda \in (0,\sqrt{1-a}]$ and $x_0 \geq 1$,
\begin{equation}\label{e.26.i}
0 \leq \psi_0 \in \mathsf{L}^\infty(\mathbb{R}),\quad \liminf_{x\to-\infty} \psi_0(x)>0 \quad \text{ and }\quad \psi_0(x) = x^q e^{-\lambda x} \quad \text{ for }x \geq x_0.
\end{equation} 
Then for each  
$\delta>0$, 
\begin{equation}\label{e.26.22a'}
\lim_{t\to\infty} \sup_{x\geq (2\lambda +\delta)t} \Big| (x-2\lambda t)^{-q} e^{\lambda(x-c_\lambda t)}\psi(t,x)-1 \Big|=0, \end{equation}
where 
$c_\lambda = \lambda + \frac{1-a}{\lambda}$. 
In particular, there exist $C>1$ and $T>0$ such that
\begin{equation}\label{e.1004}
   C^{-1}(x-2\lambda t)^q e^{-\lambda(x-c_\lambda t)}\leq  \psi(t,x)\le C(x-2\lambda t)^q e^{-\lambda(x-c_\lambda t)},
 \end{equation}
for $t\ge T$ and  $x\geq  (2\lambda+\delta)t$,
\end{lemma}
\begin{proof}[{\bf Proof of \Cref{lem:C0}}]
Define $\overline{y}(t,x) =x-2\lambda t$ and $\zeta(y)=e^{\lambda y}\psi_0(y)$. Write
\begin{equation}
G(t,z)=\frac{1}{\sqrt{4\pi t}}e^{-\frac{z^2}{4t}}.
\end{equation}
Then by completing the square in the heat kernel, we have
\begin{equation}
\psi(t,x)=e^{-\lambda(x-c_{\lambda }t)} \int_{\R}G(t,y-\overline{y}(t,x))\zeta(y)dy,
\end{equation} 
where we used
\begin{equation}
(1-a)t-\frac{(x-y)^2}{4t}-\lambda y=-\lambda(x-c_{\lambda }t)-\frac{(y-(x-2\lambda t))^2}{4t}.
\end{equation}
Choose $T>0$ such that
$\delta t\geq \max\{1,2x_0\}$ and  $\delta\sqrt{t}\geq 1$ for all $t\geq T$.
Note that $\zeta(y)=y^q$ for $y\geq x_0$. For $x\geq (2\lambda+\delta)t $, we can write
\begin{align}
\overline{y}^{-q}e^{\lambda(x-c_{\lambda}t)}\psi &= \overline{y}^{-q}\int_{-\infty}^{\frac{\overline{y}}{2}}G(t,y-\overline{y})\zeta(y) \, \dd y + r^{-q}\int_{\frac{\overline{y}}{2}}^{\infty}G(t,y-\overline{y})y^q \, \dd y\\
&:=J_1(t,x)+J_2(t,x),
\end{align}
\medskip
\noindent \# We first prove that 
\begin{align}\label{J1-0}
\lim_{t\to\infty}\sup_{x\geq (2\lambda+\delta)t } |J_1(t,x)|=0.    
\end{align}
By \eqref{e.26.i}, we have
$0\leq \zeta(y)\lesssim (1+y_+)^{q_+}$ for $y\in\R$.
Using $y\leq \frac{\overline{y}}{2}$ and $\overline{y}\geq \delta t\geq 1$, we obtain $\overline{y}^{-q}\zeta(y) \lesssim \overline{y}^{q}(1+\frac{\overline{y}}{2})^{q_+} \lesssim \overline{y}^{ q_+ - q} = \overline{y}^{ q_-}$ and thus
\begin{align}
0\leq J_1(t,x)& \lesssim \overline{y}^{q_-}\int_{-\infty}^{\frac{\overline{y}}{2}}G(t, y-\overline{y})dy= \overline{y}^{q_-}\int_{-\infty}^{-\frac{\overline{y}}{4\sqrt{t}}}\frac{1}{\sqrt{\pi}}e^{-\xi^2}d\xi\\
&\lesssim \overline{y}^{q_-}e^{-\frac{\overline{y}^2}{16t}}\lesssim \overline{y}^{q_-}e^{-\delta \frac{\overline{y}}{16}},
\end{align}
where the last inequality holds since $\overline{y}\geq \delta t$. Therefore, 
\begin{equation}
\sup_{x\geq (2\lambda+\delta)t } |J_1(t,x)|\lesssim \sup_{\overline{y}\geq \delta t}\overline{y}^{q_-}e^{-\delta \frac{\overline{y}}{16}}\to 0,\quad \text{ as } t\to0,
\end{equation}
and \eqref{J1-0} holds.

\medskip
\noindent \# Next, we prove 
\begin{align}\label{J2-0}
\lim_{t\to\infty}\sup_{x\geq (2\lambda+\delta)t } |J_2(t,x)-1|=0.    
\end{align}
For this, we set $y=\overline{y}+z\sqrt{t} $. It follows that
\begin{equation}
J_2(t,x)
=\int_{-\frac{\overline{y}}{2\sqrt{t}}}^\infty
G(1,z)\left(1+\frac{\sqrt t}{\overline{y}}z\right)^q \, \dd z.
\end{equation}
Using $\overline{y}\geq \delta t$, we have $\frac{\sqrt t}{\overline{y}} \leq \frac{1}{\delta\sqrt{t}}$. By the Lebesgue dominated
convergence theorem, we immediately get \eqref{J2-0}. Combining \eqref{J1-0} and \eqref{J2-0}, we obtain \eqref{e.26.22a'} and thus \eqref{e.1004} holds. This completes the proof.

\end{proof}

\begin{lemma}\label{lem:C1}
Let $0<a<1$ and $w$ be the solution to the nonlinear equation,
$$w_t - w_{xx} = w(1-a-w)$$ in $(0,\infty)\times \mathbb R$, with initial data $w(0,\cdot)=\psi_0$ satisfying \eqref{e.26.i} for some $q\in \mathbb{R}$, $\lambda \in (0,\sqrt{1-a}]$ and $x_0 \geq 1$. 
Set
\begin{equation}
c_\lambda = \lambda + \frac{1-a}{\lambda}.
\end{equation}
Then the following statements hold.
\begin{itemize}
    \item [{\rm(i)}] For each $\delta>0$, we have 
\begin{equation}\label{e.26.22a}
\lim_{t\to\infty} \sup_{x\geq (c_\lambda + \delta)t }
\Big| (x-2\lambda t)^{-q} e^{\lambda(x-c_\lambda t)}w(t,x)-1 \Big|=0.
\end{equation}
    \item [{\rm(ii)}] There exists a bounded function $\Lambda(t)$ such that
\begin{equation}\label{e.26.22b}
\sup_{x\in \mathbb R} \left|w(t,x) - \Phi_{1-a}(x- m(t) - \Lambda(t)) \right\vert \to 0 \quad \text{ as }t\to\infty,
\end{equation}
where  $m(t) = c_\lambda t + \frac{q}{\lambda} \ln (1+t)$ if $\lambda \in (0,\sqrt{1-a})$ and $m=c_\lambda t -\ell$ in \eqref{m^w-2} if $\lambda = \sqrt{1-a}$.
\end{itemize}
\end{lemma}
\begin{proof}[{\bf Proof of \Cref{lem:C1}}] 
It follows from \cite[Proposition 3.1]{LamWu_preprint} and Lemma~\ref{lem:C0} that 
\eqref{e.26.22a} holds. For \eqref{e.26.22b}, see \cite[Theorem A and Example 1]{bramson1983convergence} for the subcritical and \cite[Theorem A, Examples 3 and 4]{bramson1983convergence} for the critical case.
\end{proof}

Next, we recall \cite[Lemma 4.6]{LamWu_preprint}, cited using the notation of the present paper, for the convenience of the reader.

\begin{lemma}
\label{lem:liminf-case}
Let $0 < a < 1$, $\lambda\in (0,\sqrt{1-a}]$ be given and set $c_\lambda = \lambda + \frac{1-a}{\lambda}$. 
{Set $s(t)=\s-\eta\ln(1+t)$ with $\s>c_{\lambda}$ and 
$\eta\in\R$.}
Let $n$ be a bounded nonnegative supersolution such that
\begin{equation}
    n_t - n_{xx} \geq n(1-a-n) \quad \text{ in  } \Omega_0:= \left\{(t,x):~t \geq t_0,~ x \leq s(t)\right\},
\end{equation}
and
\begin{align}
\liminf_{x \to -\infty}n(t_0,x) >0,\qquad  \liminf_{t \to +\infty} \left[(x - 2\lambda t)^{-q} e^{\lambda x - (\lambda^2 +1-a)t} n(t,x)\right]_{x=s(t)} >0\label{liminf}
\end{align}
hold for some $t_0 \geq 0$ and $q \in \mathbb R$. Then the following statements hold:
\begin{itemize}
    \item[{\rm(i)}] If $\lambda \in (0,\sqrt{1-a})$, then there is a constant $C\in \mathbb{R}$ such that 
    \begin{equation}
\liminf_{t\to\infty }\left[\inf_{ -\infty< x \leq s(t)}\left(n(t,x )  -\Phi_{\lambda,1-a}\left(x- c_\lambda t - \frac{q}{\lambda} \ln (1+t) -  C\right)\right)\right] \geq  0.
\end{equation}   
\item[{\rm(ii)}] If $\lambda = \sqrt{1-a}$, then there is a constant $C\in \mathbb{R}$ such that
    \begin{equation}
\liminf_{t\to\infty }\left[\inf_{ -\infty< x \leq s(t)}\left(n(t,x )  -\Phi_{1-a}(x-c_\lambda t + \ell(t)-  C)\right)\right] \geq  0,
\end{equation}  
where the expression of $\ell$, depending on the parameter value $\eta = \frac{1}{\sqrt a}(q+\tfrac32)$, is given in \eqref{m^w-2}.
\end{itemize}

\end{lemma}
\begin{proof}[{\bf Proof of \Cref{lem:liminf-case}}] 
Since $\s>c_\lambda$, there exist $\varepsilon,M>0$ such that
\begin{equation}
c_\lambda+\varepsilon\leq s'(t)\leq M \quad\text{ for all large } t.
\end{equation}
Then the conclusions follow from \cite[Lemma 4.6 and Remark 4.7]{LamWu_preprint} with $R=1-a$.
\end{proof}
Next, we summarise the Liouville result developed in \cite{beresycki2007generalized,hamel2016logarithmic} stated in terms of the homogeneous scalar equation. 
\begin{proposition}\label{p.109.1}
Let $c,\nu>0$ and let $w$ be a bounded positive entire function  satisfying
\begin{equation}
\begin{cases}
w_t - w_{yy} - cw_y = w(\nu - w) &\text{ in }\mathbb{R}^2,\\
w(t,y) \asymp (\ell_0 + \ell_1 y)e^{-\lambda y} &\text{ in }\mathbb{R}\times [y_0,\infty),
\end{cases}\end{equation}
for some $y_0,\lambda >0$ and $\ell_0,\ell_1 \geq 0$ such that $\ell_0 + \ell_1 >0$. Then {there exists $h\in\R$ such that $w(t,\cdot) = \Phi_c(\cdot-h)$}, where $\Phi_c$ is the unique travelling wave function satisfying
\begin{equation}
\Phi'' + c\Phi' + \Phi(\nu - \Phi) = 0 \quad \text{ in }\mathbb{R},\quad 
{\phi_c(0)=\frac{\nu}2},\quad 
\Phi_c(-\infty) = \nu,\quad \Phi_c(+\infty) = 0. 
\end{equation}
\end{proposition}
\begin{proof}[{\bf Proof of \Cref{p.109.1}}]
    Without loss of generality, we can assume $\nu = 1$. If $w$ is assumed to be trapped between two translates of $\Phi_c$, then the result follows from
    \cite[Theorem 3.5]{beresycki2007generalized}. The proof for the weakened assumption as above is due to \cite[Lemma 8.2]{hamel2016logarithmic}, where it is proved for the critical case $c_* = 2\sqrt{\nu}$ where $\ell_1>0$. Their argument works for the supercritical case without change. See also \Cref{thm:liouville} and \Cref{app:rigid} for the extension of the proof to systems.
\end{proof}

\section{Rigidity results for general monotone systems}\label{app:rigid}

In this section, we generalise \Cref{t.uniqueness,thm:liouville} to formulate and prove a rigidity result for monostable waves in general monotone parabolic systems. Such a uniqueness result for bistable travelling wave solutions can be found in Volpert et al. \cite[pp. 15, Theorem 2.1]{Volpert1994travelling}. Here we establish a result of a similar flavour for monostable situations.

\begin{definition}
For two vectors $\vec{a}=(a_1,...,a_k) \in \mathbb{R}^k$ and $\vec{b}=(b_1,...,b_k) \in \mathbb{R}^k$, $\vec{a} \leq_{\mathbb{R}^k} \vec{b}$ (resp. $\vec{b} \ll_{\mathbb{R}^k} \vec{b}$) means $a_j \leq b_j$ for all $j$ (resp. $a_j < b_j$ for all $j$).
\end{definition}

Let $c>0$ be given, and 
suppose $w=(u,v)$, with $u : \mathbb{R}^2 \to \mathbb{R}^n$ and $v : \mathbb{R}^2 \to \mathbb{R}^m$, is a bounded entire solution of a  system of PDEs with monotone structure:
\begin{equation}\label{e.923.1}
   \begin{cases}
       u_t = \widetilde{D}_u u_{yy} + c u_{y} + F(u,v) &\text{ for }(t,y)\in \mathbb{R}^2  
       \\
    v_t = \widetilde{D}_v v_{yy} + cv_y + G(u,v) &\text{ for }(t,y)\in \mathbb{R}^2.
   \end{cases} 
\end{equation}

We assume that, for $k\in \{u,v\}$, $\widetilde{D}_k = {\rm Diag}(d_{k,i})_i$ is a diagonal
matrix with positive diagonal entries and the vector-valued function $H(u,v) = (F(u,v),G(u,v)): \mathbb{R}^{n+m} \to \mathbb{R}^{n+m}$ is of class $\mathscr{C}^2$ and satisfies the following:

\begin{itemize}

\item[{\rm(H1)}]
The states $E^+=0$ and $E^-$ are equilibria, i.e. $H(0)=H(E^-)=0$, and
\begin{equation}
E^- \gg_{\mathbb R^{n+m}}0.
\end{equation}
We denote by $\mathcal I:=[0,E^-]$
the corresponding order interval.

\item[{\rm(H2)}]
The system is cooperative in $\mathcal I$: the matrix $DH(z)$ is of Metzler type for every
$z=(u,v)\in\mathcal I$. Moreover, the matrix $DH(z)$ is irreducible for every
$\displaystyle 
    0\ll_{\mathbb R^{n+m}} z
    \ll_{\mathbb R^{n+m}}E^-.
$

\item[{\rm(H3)}]
The equilibrium $E^-$ is strictly stable in a positive direction:
\begin{equation}
    DH(E^-)\mathbf a_-
    \leq_{\mathbb R^{n+m}}
    -\mu_-\mathbf a_-,
\end{equation}
for some 
$
    \mathbf a_- \gg_{\mathbb R^{n+m}}0$ and 
    $
    \mu_->0.
$

\item[{\rm(H4)}]
The face $\{u=0\}$ is invariant for the reaction term, namely
\begin{equation}
    F(0,v)=0
\end{equation}
for every $v$ such that $(0,v)\in\mathcal I$.
Furthermore,
$B_+:=D_uF(0,0)$
is an irreducible Metzler matrix\footnote{Condition (H4) marks the distinction from the kind of co-invasion travelling wave connecting $(0,0)$ to $(u^\star, v^\star)$ by Tang-Fife.}.

    \item[{\rm(H5)}] There exist $\delta_2>0$ and a vector
$\mathbf a_+=(a_{+,1},\ldots,a_{+,m})\gg_{\mathbb R^m}0$
such that the following property holds.
For every $(u,v)$ satisfying
\begin{equation}
    \|(u,v)-E^+\|<\delta_2,
\end{equation}
every $p\in\mathbb R^n$ with $p\ge_{\mathbb R^n}0$, and every
$q\in\mathbb R^m$ such that both $(u,v)$ and $(u+p,v+q)$ belong
to the order interval $\mathcal I$, if for some $k\in\{1,\ldots,m\}$,
\begin{equation}\label{e.H2prime}
    \frac{q_k}{a_{+,k}}
    =
    \min_{1\le j\le m}\frac{q_j}{a_{+,j}}
    <0,
\end{equation}
then
\begin{equation}\label{e.H2prime2}
    G_k(u+p,v+q)-G_k(u,v)>0.
\end{equation}

\end{itemize}

\begin{remark}
    If the $v$ component in problem \eqref{e.923.1} is void, then (H5) vacuously holds, and \Cref{t.c.1,t.c.2} hold. The latter provides a shorter proof of the main result in \cite{huang2007} without assuming irreducibility of $DH(E_-)$.
\end{remark}
\begin{theorem}\label{t.c.1}
Assume (H1)--(H5).
Let $w=(u,v)$ be a bounded entire solution  of \eqref{e.923.1} satisfying the two conditions:
\begin{itemize}
\item[{\rm(W1)}] $\displaystyle E^- \gg_{\mathbb{R}^{n+m}}w(t,y)\gg_{\mathbb{R}^{n+m}} E^+$ in $\mathbb{R}^2$ and  
\begin{equation}
\lim_{y \to -\infty} \sup_{t\in\mathbb R}\|w(t,y) - E^-\|= 0 = \lim_{y \to +\infty} \sup_{t\in\mathbb R}\|w(t,y) - E^+\|.
\end{equation}
\item[{\rm(W2)}]
There exist $C_1>1$, $\lambda>0$, and $r\in\mathbb N \cup \{0\}$,
independent of $i$, such that
\begin{equation}
    \frac1{C_1}
    \le
    \frac{e^{\lambda y}}{y^r}
    \left(u_i(t,y)-(E^+)_i\right)
    \le
    C_1,
    \qquad
    t\in\mathbb R,\quad y\ge C_1,\quad i=1,\ldots,n.
\end{equation}
\end{itemize}
Then there exists $\sigma \in \mathbb{R}$ such that
\begin{equation}
w(t,y) \equiv W_c(y  - \sigma),
\end{equation}
where $W_c$ is a monotone travelling wave solution of \eqref{e.923.1} such that 
\begin{equation}
W(-\infty) = E^-,\quad W(+\infty) = E^+.
\end{equation}

\end{theorem}

We point out that the hypothesis (W2) is consistent with  Harnack's inequality \cite[Theorem 3.9]{Foldes2009cooperative}. Assuming in addition the following hyperbolicity condition, we can derive a general uniqueness of monostable travelling wave solutions.
\begin{itemize}

\item[{\rm(H6)}]The equilibrium $E^+=0$ is hyperbolic for the
spatial dynamical system associated with the travelling-wave equation.
More precisely, we assume that
\begin{equation}
    \det\!\left(
        \mu^2\widetilde D+c\mu I+J_+
    \right)\neq0 \qquad \text{ for all } \mu \in \mathbb C\text{ such that }\operatorname{Re}\mu=0,
\end{equation}
where
\begin{equation}
    \widetilde D
    :=
    \begin{pmatrix}
        \widetilde D_u & 0\\
        0 & \widetilde D_v
    \end{pmatrix},
    \qquad
    J_+:=DH(0).
\end{equation}

\end{itemize}

   \begin{remark}
If both $u$ and $v$ are scalar, then, in view of {\rm(H4)},
\begin{equation}
    DH(0,0)
    =
    \begin{pmatrix}
        \partial_uF(0,0)&0\\
        \partial_uG(0,0)&\partial_vG(0,0)
    \end{pmatrix}.
\end{equation}
Consequently, for $c>0$, hypothesis {\rm(H6)} is equivalent to
\begin{equation}
    \partial_uF(0,0)\neq0,
    \qquad
    \partial_vG(0,0)\neq0.
\end{equation}
In particular, for the Lotka--Volterra system
\eqref{eq:LV}, after the change of variables $w=1-v$, 
\begin{equation}
    DH(0,0)
    =
    \begin{pmatrix}
        1-a&0\\
        rb&-r
    \end{pmatrix}.
\end{equation}
Thus {\rm(H6)} holds whenever
$a\neq1, r\neq0$, 
and these conditions are automatically satisfied under the
monostable assumption $0<a<1$, $r>0$.
\end{remark}
\begin{theorem}

Assume (H1)--(H6).
Let $c>0$ be given. 
    If $W$ and $\widetilde{W}$ are two travelling wave solutions of \eqref{e.923.1} with the same speed $c>0$, both connecting $E^-$ and $E^+=0$, and
    \begin{equation}
E^- \gg  W, \widetilde W \gg E^+
    \end{equation}
on $\mathbb{R}$, then there exists $\sigma\in \mathbb{R}$ such that
    \begin{equation}
    W(y) = \widetilde{W}(y-\sigma) \quad \text{ for } y \in \mathbb{R}.
    \end{equation}
    \label{t.c.2}
\end{theorem}

Next, we remark that this theorem actually contains the rigidity of \eqref{eq:LV} as well as systems with more than two species.
\begin{remark}[A two-species Lotka--Volterra system]\label{r.LV2}
Hypotheses {\rm (H1)--(H6)} are satisfied by the
Lotka--Volterra system \eqref{eq:LV}. Here we verify condition {\rm(H5)}. (See the above remark for {\rm(H6)}.) 
We are interested in entire solutions of \eqref{eq:LV} connecting $(u^\star,v^\star)$ and $(0,1)$. Introduce $w=1-v$; then the equilibrium $(0,1)$ is transformed into $E^+=(0,0)$ and the equations for $(u,w)$ become \eqref{e.923.1} with 
\begin{equation}\label{e.926.7a}
A_1=1,\quad A_2=d,\quad F(u,w)=u(1-a-u + aw), \quad G(u,w) = r(1-w)(bu -w),
\end{equation}
where the relevant order interval is $\mathcal{I}=\{(u,w): 0\leq u \leq 1,~0 \leq w \leq 1\}$. Let $p \geq 0$ and $q <0$ and $(u,w)\approx (0,0)$; then 
\begin{align}
    G(u+p,w+q) - G(u,w) &= r(1-w-q)(bp-q)-rq(bu-w)  \\
    &\geq -rq\left[1-w-q - |bu-w| \right] >0.
\end{align}
\end{remark}

\begin{remark}[A three-species Lotka--Volterra system]\label{r.LV3}
Hypotheses {\rm (H1)--(H6)} are satisfied by a natural three-species
Lotka--Volterra system in which one species competes with two mutually
cooperative species. 
Consider
\begin{equation}\label{e.LV3-original}
\left\{
\begin{aligned}
 u_t
 &=
 u_{yy}+c u_y
 +u\left(1- u-a_1v_1-a_2v_2\right),\\
 (v_1)_t
 &=
 d_1(v_1)_{yy}+c(v_1)_y
 +r_1v_1\left(1-b_1u-v_1+\gamma_{12}v_2\right),\\
 (v_2)_t
 &=
 d_2(v_2)_{yy}+c(v_2)_y
 +r_2v_2\left(1-b_2u+\gamma_{21}v_1-v_2\right),
\end{aligned}
\right.
\end{equation}
where all coefficients are positive and
\begin{equation}
    \gamma_{12}\gamma_{21}<1, \qquad \min\{b_1,b_2\} >1, \qquad 1 - a_1\bar{v}_1 - a_2 \bar{v}_2 >0, 
\end{equation}
where $\bar v=(\bar v_1,\bar v_2)$  is the positive equilibrium of the $v$-subsystem:
\begin{equation}
    1-\bar v_1+\gamma_{12}\bar v_2=0,
    \qquad
    1+\gamma_{21}\bar v_1-\bar v_2=0.
\end{equation}
In such a case, $E^-=(1,\bar{v}_1,\bar{v}_2)$ and $E^+=(0,0,0)$.
\end{remark}

Before turning to the proof of the rigidity theorems,  we first develop two maximum principles for large domains reminiscent of the initial step in the moving-plane method for elliptic problems \cite{Gidas1981symmetry}. For $h
\in \mathbb R$ and  $\xi \in \mathbb{R}$, define
\begin{equation}
u^\xi(t,y) = u(t+\xi, y + h - \xi c), \qquad  v^\xi(t,y) = v(t+\xi, y + h - \xi c).
\end{equation}

Then the differences $(p^\xi, q^\xi)= (u^\xi - u, v^\xi - v)$ satisfy the linear problem:
\begin{equation}
\begin{cases}
    p_t- \widetilde{D}_u p_{yy} -cp_y= pF_1+ qF_2
      &\quad \text{ in }\mathbb{R}^2\\
    q_t - \widetilde{D}_v q_{yy} - cq_y = pG_1 + qG_2&\quad \text{ in }\mathbb{R}^2
\end{cases}
\end{equation}
where 
\begin{equation}
\left(\begin{array}{c}
pF_1+ qF_2 \\ pG_1 + qG_2
\end{array}\right)
= B^\xi \left(\begin{array}{c}
p \\ q
\end{array}\right) =\left(
\int_0^1 DH(W + s(W^\xi-W))\,ds
\right)\left(\begin{array}{c}
p \\ q
\end{array}\right)
\end{equation}
Note that $B^\xi(t,y)$ is Metzler.

\begin{lemma}\label{l.liouville.2'}
There exists $A_1\geq 0$  such that for any $\xi > \frac{h}{c}$ and any $\bar y \geq A_1$, if 
\begin{equation}\label{e.925.2a}
\inf_{y \geq \bar y} p^\xi(t,y)\geq_{\mathbb{R}^n} 0,\quad q^\xi(t,\bar y)\gg_{\mathbb{R}^m}0 ~\mathrm{ for } \, t\in \mathbb{R},
\end{equation}
then 
\begin{equation}\label{e.925.2b}
q^\xi(t,y)\gg_{\mathbb{R}^m}0 \quad \mathrm{ for } \, t\in \mathbb R,~y \geq \bar y.
\end{equation}
\end{lemma}
\begin{proof}[{\bf Proof of \Cref{l.liouville.2'}}]
By the uniform convergence
\begin{equation}
    (u(t,y),v(t,y))\to E^+
    \qquad\text{as }y\to+\infty,
\end{equation}
uniformly in $t\in\mathbb R$, we may choose $A_1>0$ such that
\begin{equation}\label{e.closeE2}
    \|(u(t,y),v(t,y))-E^+\|<\frac{\delta_2}2 
    \qquad
    \text{for all }t\in\mathbb R,\ y\ge A_1.
\end{equation}
Notice that $A_1$ is independent of $\xi$, since
\eqref{e.closeE2} only involves the unshifted solution.

Fix $\xi>\frac{h}{c}$ and $\bar y\ge A_1$, and suppose that
\eqref{e.925.2a} holds.
We first prove that
\begin{equation}
    q^\xi(t,y)\ge_{\mathbb R^m}0,
    \qquad
    \text{for all }t\in\mathbb R,\ y\ge\bar y.
\end{equation}
Suppose otherwise. Set
\begin{equation}\label{e.weighted-min}
    M:=
    \min_{1\le i\le m}
    \inf_{\substack{t\in\mathbb R\\ y\ge\bar y}}
    \frac{q_i^\xi(t,y)}{a_{+,i}}<0.
\end{equation}
Then we can pass to a sequence $(t_n,y_n,k_n)$ such that for some fixed $k$,
\begin{equation}
    \frac{q_{k}^\xi(t_n,y_n)}{a_{+,k}}\to M.
\end{equation}
Since
\begin{equation}
    q^\xi(t,y)\to0,
    \qquad\text{as }y\to+\infty,
\end{equation}
uniformly in $t$, the sequence $(y_n)$ is bounded.
Passing to a subsequence, we may assume that
\begin{equation}
    k_n=k,\qquad y_n\to y_*\ge\bar y.
\end{equation}

Using parabolic estimates and passing to a sequence, we can show that the time-translated functions converge
\begin{equation}
    (u_n,v_n)(t,y)
    :=
    (u,v)(t+t_n,y)
\to(u_\infty,v_\infty) \quad \text{ in }\mathscr{C}_{\mathrm{loc}}(\mathbb{R}^2),
\end{equation}
together with the corresponding differences
\begin{equation}
    (p_n^\xi,q_n^\xi)\to(p_\infty^\xi,q_\infty^\xi) \quad \text{ in }\mathscr{C}_{\mathrm{loc}}(\mathbb{R}^2).
\end{equation}
Then
\begin{equation}\label{e.weighted-contact}
    \frac{(q_{\infty,k}^\xi)(0,y_*)}{a_{+,k}}
    =M,
    \qquad
    \frac{(q_{\infty,j}^\xi)(0,y_*)}{a_{+,j}}
    \ge M
    \quad\text{for all }j.
\end{equation}
Moreover, $y_*>\bar y$. Indeed, if $y_*=\bar y$, then
\eqref{e.925.2a}, together with local convergence, would give
$q_{\infty,k}^\xi(0,\bar y)\ge0$, contradicting
\eqref{e.weighted-contact} and $M<0$.

Thus $(0,y_*)$ is an interior minimum point of the $k$-th
component, and hence
\begin{equation}\label{e.minimumineq}
    (q_{\infty,k}^\xi)_t
    -d_{2,k}(q_{\infty,k}^\xi)_{yy}
    -c(q_{\infty,k}^\xi)_y
    \le0
    \qquad\text{at }(0,y_*).
\end{equation}

On the other hand, the difference equation gives
\begin{align}
    &(q_{\infty,k}^\xi)_t
    -d_{2,k}(q_{\infty,k}^\xi)_{yy}
    -c(q_{\infty,k}^\xi)_y =
    G_k\left(
        u_\infty+p_\infty^\xi,
        v_\infty+q_\infty^\xi
    \right)
    -
    G_k(u_\infty,v_\infty).
    \label{e.qdifference}
\end{align}
Now, \eqref{e.925.2a} gives
$p_\infty^\xi(0,y_*)\ge0$, and, 
by \eqref{e.closeE2},
\begin{equation}
    \|(u_\infty(0,y_*),v_\infty(0,y_*))-E^+\|
    <\delta_2.
\end{equation}
Finally, \eqref{e.weighted-contact} is precisely the condition
\eqref{e.H2prime}. Therefore {\rm (H5)} yields
\begin{equation}
    G_k\left(
        u_\infty+p_\infty^\xi,
        v_\infty+q_\infty^\xi
    \right)
    -
    G_k(u_\infty,v_\infty)
    >0.
\end{equation}
This contradicts \eqref{e.minimumineq}. Hence
\begin{equation}\label{e.qnonnegative}
    q^\xi(t,y)\ge_{\mathbb R^m}0,
    \qquad
    \text{for all }t\in\mathbb R,\ y\ge\bar y.
\end{equation}

It remains to prove strict positivity. By {\rm (H2)}, the linear
system satisfied by $(p^\xi,q^\xi)$ is cooperative. In view of
$p^\xi\ge0$ and \eqref{e.qnonnegative}, each component
$q_i^\xi$ is a nonnegative supersolution of a scalar linear
parabolic equation. By \eqref{e.qnonnegative} and 
    $q_i^\xi(t,\bar y)>0$, 
the strong maximum principle implies
\begin{equation}
    q_i^\xi(t,y)>0
    \qquad
    \text{for all }t\in\mathbb R,\ y\ge\bar y,
\end{equation}
for every $i=1,\ldots,m$.
\end{proof}

\begin{lemma}\label{l.liouville.3'}
There exists $A_2\geq 0$  such that for any $\xi >\frac{h}{c}$ and $\bar y \leq -A_2$, if 
\begin{equation}
p^\xi(t,\bar y)\gg_{\mathbb{R}^n} 0 \qquad \mathrm{ and }\qquad  q^\xi(t,\bar y)\gg_{\mathbb{R}^m} 0~\quad\mathrm{ for } \quad t\in \mathbb{R},
\end{equation}
then
\begin{equation}\label{e.918.8b}
p^\xi(t,y)\gg_{\mathbb{R}^n} 0 \qquad \mathrm{ and }\quad  q^\xi(t,y)\gg_{\mathbb{R}^m} 0 \quad \mathrm{ for } \quad t\in \mathbb R,~y \leq \bar y.
\end{equation}
\end{lemma}
\begin{proof}[{\bf Proof of \Cref{l.liouville.3'}}]
The proof relies on the hypothesis (H3) and a positive vector $\vec{a}_- \gg_{\mathbb{R}^{n+m}} 0$ and is completely analogous to that of Lemma \ref{l.liouville.3} so we omit it here. Note that for $\xi > \frac{h}{c}$, the convergences 
    \begin{equation}
    W^\xi(t,y) \to E^- \quad \text{ and }\quad W(t,y) \to E^- \quad \text{ as }y \to -\infty
    \end{equation} are uniform in
    the two parameters 
    \begin{equation}
    (t,\xi) \in \mathbb{R}\times \left(\frac{h}{c},+\infty\right).
    \end{equation}
This is easier to handle than Lemmas \ref{l.liouville.2} and \ref{l.liouville.2'}, whose proofs required the additional structural condition (H5).
\end{proof}

\begin{proof}[{\bf Proof of \Cref{t.c.1}}]

Fix $h\in \mathbb R$ and define
    \begin{equation}
    u^\xi(t,y) = u(t+ \xi, y + h - c\xi), \quad  v^\xi(t,y) = v(t+ \xi, y + h - c\xi),\end{equation}
    \begin{equation}W^\xi(t,y) = (u^\xi(t,y),v^\xi(t,y)),\quad 
    p^\xi = u^\xi - u, \quad \text{ and }\quad q^\xi = v^\xi - v.
    \end{equation}
This allows us to define the sliding parameter set:
    \begin{equation}
    \Xi:= \left\{\eta > \frac{h}{c}:~ W^\xi  > W \quad \text{ in }\mathbb{R}^2 ~\text{ for every }\xi \in [\eta,\infty)\right\}.
    \end{equation}

\medskip
\noindent {\bf \# Step 1.} First, we claim that there exist $\xi_1 \gg 1$ and $y_2>A_1$ such that for $\xi \geq \xi_1$,
    \begin{equation}\label{e.926.5}
    u^\xi(t,y) \gg_{\mathbb{R}^n} u(t,y) \quad \text{ for }~t\in \mathbb{R},~ y \geq y_2.
    \end{equation}
Indeed, by (W2) we can choose $\xi_1 >1$ such that for any $k$, and every $\xi \geq \xi_1$, 
\begin{equation}\label{e.926.6a}
    u_k(t',y+h- c\xi ) > u_k(t,y) \quad \text{ for every }t,t' \in \mathbb{R},~ y \geq C_1~\text{ such that }y+h - c\xi \geq C_1.
\end{equation}
Next, take $\delta_u = \min_{k} \inf_{\mathbb{R}\times (-\infty, C_1]} u_k(t,y)$. We claim that $\delta_u>0$. Otherwise for fixed $k$, we find a sequence $(t_n,y_n) \in \mathbb{R}\times (-\infty, C_1]$ such that $u_k(t_n,y_n) \to 0$. By (W1), $y_n$ is bounded, so we may further suppose $y_n \to y_\infty$. Now, $F(0,v) = 0$ implies $F(u,v) = B(u,v)u$. Since $B(u,v)$ is Metzler,
\begin{equation}
(u_k)_t - \widetilde{d}_k (u_k)_{yy} - c (u_k)_y = B_{kj}u_j \geq B_{kk} u_k.
\end{equation}
Then passing to the limit $\widetilde{u}_n(t,y) = u_k(t+t_n) \to \widetilde{u}(t,y)$, we see that the limit function $\widetilde{u}$ is a nonnegative supersolution of some linear parabolic equation. Then $\widetilde{u}(0,y_\infty) = 0$ implies that $\widetilde{u}\equiv 0$ thanks to the strong maximum principle. This contradicts (W2), and proves that $\delta_u>0$.

Next, choose $y_2 \geq C_1$ such that 
\begin{equation}\label{e.926.6b}
\max_{k} \sup_{\mathbb{R} \times [y_2,\infty)} u_k(t,y) <\delta_u. 
\end{equation}
Then for any $\xi \geq \xi_1$, \eqref{e.926.5}  holds thanks to \eqref{e.926.6a} if $y-c\xi \geq C_1$ and thanks to \eqref{e.926.6b} if $y - c\xi < C_1$.

\medskip

\noindent {\bf \# Step 2.}     
    (Initiation of sliding) Choose $y_2 = \max\{A_1,A_2,C_1\}$ and define
    \begin{equation}
    \delta_* = \min_j \inf_{\mathbb{R} \times [-y_2,y_2]} [(E^-)_j - W_j(t,y)].
    \end{equation}
    Then again $\delta_*>0$ upon performing a time-translation compactness/strong-comparison argument, which we omit. Using $\delta_*>0$ and $W(t,-\infty) = E^-$,  there exists $\xi_2 > \xi_1$ such that for $\xi \geq \xi_2$,
    \begin{equation}
    W^\xi(t,y) > W(t,y) \quad \text{ for }t\in \mathbb{R},~ |y| \leq y_2.
    \end{equation}
Next, we can use \Cref{l.liouville.2'} to propagate the inequality to the $V$ components in $y \geq y_2$, and then use \Cref{l.liouville.3'} to propagate 
the full positive ordering to $-\infty$.  In conclusion, we have $\xi \in \Xi$ for all $\xi \geq \xi_2$. 
Hence, $\xi_* = \inf\Xi$ is well defined. We will establish in Steps \#3a, 3b and 3c that $\xi_* = \frac{h}{c}$.

\medskip

\noindent {\bf \# Step 3a.}  Suppose, to the contrary, that $\xi_* > \frac{h}{c}$, 
we claim that \begin{equation}\label{e.926.1}
W^{\xi_*} = (u^{\xi_*},v^{\xi_*}) \gg_{\mathbb{R}^{n+m}}(u,v)=W  \quad \text{ in }\mathbb{R}^2.\end{equation}
Indeed, the non-strict inequalities hold by continuity and definition of $\xi_*$. Suppose equality holds for \eqref{e.926.1} at some component and some $(\bar t, \bar y)$, then the irreducibility of $DH$ as stated in (H2) implies that $W^{\xi_*} \equiv W$. Since $W^{\xi_*}$ is a translation of $W$, this implies that
\begin{equation}\label{e.926.6c}
W(t + k\xi_*, y + k(h-c\xi_*)) \equiv W(t,y) \quad \text{ for all }(t,y)\in\mathbb{R}^2\text{ and every }k\in \mathbb{N}.
\end{equation}
Taking $k \to +\infty$ in \eqref{e.926.6c}, we deduce from $h-c\xi_* <0$ and the first part of (W1) that $W \equiv E^-$, which  contradicts the second part of (W1) which states that $W(t,+\infty) = E^+$ since $E^- \neq E^+$. This proves \eqref{e.926.1} (in case $\xi_*>\frac{h}{c}$).

\medskip

\noindent {\bf \# Step 3b.}  Fix $\ell\in\mathbb N$, such that $\ell \geq A_1$. By \eqref{e.926.1} and a compactness
argument, there exists
    $\bar\xi_\ell\in\left(\frac hc,\xi_*\right)$ 
such that, for every $\xi\in[\bar\xi_\ell,\xi_*]$,
\begin{equation}\label{e.926.8}
    W^\xi(t,y)\gg_{\mathbb R^{n+m}}W(t,y)
    \qquad
    \text{for }(t,y)\in
    \mathbb R\times[-A_2,\ell].
\end{equation}
Indeed, otherwise there would exist
$\xi_j\to\xi_*$, $t_j\in\mathbb R$ and
$y_j\in[-A_2,\ell]$ at which the strict ordering fails.
Passing to a time-translated limit yields a nonnegative
difference at the critical shift having an interior contact.
The strong maximum principle then gives translation invariance and a contradiction with (W1) can be derived as in Step \#3a.

In particular, the inequality \eqref{e.926.8} holds at $y=-A_2$.
Hence, by Lemma \ref{l.liouville.3'}, after increasing
$\bar\xi_\ell$ (though still satisfies $\bar{\xi}_\ell \in \left(\tfrac{h}{c},\xi_*\right)$) if necessary, we may assume $\bar{\xi}_\ell \nearrow \xi_*$ and that 
\begin{equation}
    W^\xi(t,y)\gg_{\mathbb R^{n+m}}W(t,y)
    \qquad
    \text{for all }t\in\mathbb R,\ y\le\ell,~\text{ and every }\xi\in[\bar\xi_\ell,\xi_*].
\end{equation}

Since $\bar{\xi}_\ell<\xi_*$, we have $\bar{\xi}_\ell \not\in \Xi$. Hence, for each $\ell$, there exist a number $\xi_\ell \in [\bar\xi_\ell,\xi_*)$, a point $(t_\ell,y_\ell)$ with $y_\ell \geq \ell$, and a component $k$ such that
\begin{equation}
    W^{\xi_\ell}_k(t_\ell,y_\ell) \leq W_k(t_\ell,y_\ell)\quad \text{ for }\ell \to +\infty.
\end{equation}
Note that $\xi_\ell \to \xi_*$ since $\bar\xi_\ell \to \xi_*$. Now, it cannot happen that $u^{\xi_\ell} \gg u$ for $y\in [\ell,\infty)$; otherwise 
Lemma \ref{l.liouville.2'} implies that $v^{\xi_\ell} \gg v$. Thus we may always choose $(t_\ell, y_\ell)$ such that $y_\ell > \ell$ and that 
\begin{equation}\label{e.926.2}
    u^{\xi_\ell}_k(t_\ell,y_\ell) \leq u_k(t_\ell,y_\ell) \quad \text{ for }\ell \to +\infty.
\end{equation}

\medskip

\noindent {\bf \# Step 3c.} 

Define
\begin{equation}
U_\ell(t,y) := \frac{u(t+t_\ell, y + y_\ell)}{u_k(t_\ell, y_\ell)}, \quad \text{ and }\quad \hat U_\ell(t,y) := \frac{u(t+t_\ell + \xi_*, y + y_\ell + h - c\xi_*)}{u_k(t_\ell, y_\ell)},
\end{equation}
and consider $U = \lim U_\ell$ and $\hat U = \lim \hat U_\ell$, which exist  in $\mathscr{C}_{\mathrm{loc}}(\mathbb{R}^2)$ after passing to a subsequence, thanks to the condition (W2) and parabolic regularity theory. Since $y_\ell \to +\infty$, we have 
\begin{equation}\label{e.926.4}
    \hat{U}(t,y) \equiv  U(t+\xi_*, y + (h - c\xi_*)) \quad \text{ and }\quad  e^{\lambda y } U_i(t,y) \asymp 1 \quad \text{ in }\mathbb{R}^2,~\text{ for all }i.
\end{equation}
Using $F(0,v) = 0$, we write $\displaystyle F(u,v) = \left(\int_0^1 D_u F(su,v)\,ds\right)u$, then observe that $U_\ell$ satisfies
\begin{equation}
(U_\ell)_t - \widetilde{D}_u(U_\ell)_{yy} - c(U_\ell)_y = B_\ell(t,y) (U_\ell),
\end{equation}
where
\begin{equation}
B_\ell(t,y) = \int _0^1 D_u F(su(t+ t_\ell, y + y_\ell), v(t+t_\ell, y + y_\ell))\,ds. 
\end{equation}
Since $y_\ell \to +\infty$ and $W(t,+\infty) = E_+$ uniformly in $t$, we see that 
\begin{equation}
B_\ell \to D_uF(0,0) \quad \text{ in }\mathscr{C}_{\mathrm{loc}}(\mathbb{R}^2).
\end{equation}
Hence, $U$, and $\hat{U}$ for the same reason,  satisfy globally the linear cooperative system with constant coefficients:
\begin{equation}
    U_t - \widetilde{D}_u U_{yy} - c U_y = D_uF(E^+)U \quad \text{ in }\mathbb{R}^2.
\end{equation}

Furthermore, $U_k(0,0) = 1$ and $\hat{U} \geq_{\mathbb{R}^n} U$ since $u^{\xi_*} \geq_{\mathbb{R}^n} u$. On the other hand, by \eqref{e.926.2}, we have
\begin{equation}
(\hat{U}_\ell)_k(\xi_\ell - \xi_*, c(\xi_* - \xi_\ell)) = \frac{u_k(t_\ell + \xi_\ell, y_\ell + h - c\xi_\ell)}{u_k(t_\ell,y_\ell)} \leq 1 \quad \text{ for all }\ell.
\end{equation}
Taking $\ell \to +\infty$, we have $\hat{U}_k(0,0) \leq 1 = U_k(0,0)$. Put $Z = \hat{U} - U$; then $Z \geq 0$, $Z_k(0,0) = 0$ and it satisfies the linear system of equations
\begin{equation}
Z_t - \widetilde{D}_u Z_{yy} - c Z_y = B_+ Z \quad \text{ in }\mathbb{R}^2.
\end{equation}
Since $B_+=D_uF(E^+)$ is an irreducible Metzler matrix, 
the strong maximum principle for irreducible cooperative systems, followed by uniqueness for the Cauchy problem, implies the vector identity  $Z = \hat{U} - U \equiv 0$ in $\mathbb{R}^2$.

By the first half of \eqref{e.926.4}, we deduce that
\begin{equation}\label{e.926.4b}
U(t + j \xi_*, y + j(h-c\xi_*)) \equiv U(t,y) \quad \text{ for }(t,y) \in \mathbb{R}^2,~\text{ for every }j \in \mathbb{N}.
\end{equation}
Setting $t=y=0$ in \eqref{e.926.4b}, and using the second half of \eqref{e.926.4}, we have
\begin{equation}
\frac{1}{C} e^{-\lambda j(h-c\xi_*)} \leq C \quad \text{ for every }j \geq 1.
\end{equation}
But this is impossible if $\xi_* > \frac{h}{c}$. 
Therefore, we have established that $\xi_* = \frac{h}{c}$.

\medskip
\noindent {\bf \# Step 4.} To summarise, for a given $h\in \mathbb{R}$, we have proved that $\inf\Xi = \frac{h}{c}$. It follows by continuity that 
\begin{equation}
    W\left(t + \frac{h}{c}, y\right) \geq_{\mathbb{R}^{n+m}} W(t,y) \quad \text{ in }\mathbb{R}^2.
\end{equation}
Repeating the argument with $h$ replaced by $-h$, we obtain the reverse inequality. Hence, we have 
\begin{equation}
    W\left(t + \frac{h}{c}, y\right) \equiv  W(t,y) \quad \text{ in }\mathbb{R}^2.
\end{equation}
Since $h \in \mathbb R$ is arbitrary, it follows that
\begin{equation}
    W(t , y) \equiv  W_{\mathrm{TW}}(y) \quad \text{ in }\mathbb{R}^2,
\end{equation}
where $W_{\mathrm{TW}}$ is a travelling wave solution of \eqref{e.923.1}. Using the fact that $\xi \in \Xi$ implies $[\xi,\infty) \subset \Xi$ and that $\inf \Xi = \frac{h}{c}$, it follows that
$\Xi = (\frac{h}{c},\infty)$. Thus $W_{\mathrm{TW}}(y)$ is strictly decreasing in $y \in \mathbb{R}$.

The uniqueness of $W_{\mathrm{TW}}$ will be established next, by an adaptation of the above sliding argument.
\end{proof}
\begin{proof}[{\bf Proof of \Cref{t.c.2}}] Let $W= (U,V)$ and $\widetilde{W}= (\widetilde U, \widetilde V)$ be two travelling wave solutions of \eqref{e.923.1} with $W(-\infty) = E^-$ and $W(+\infty) = E^+$. By
the Harnack inequality and up to  
a translation, we may assume that there are $\lambda, \widetilde\lambda>0$ and $r \in \mathbb{N}\cup\{0\}$ (both are independent of $i$) such that
\begin{equation}
    \lim_{y\to+\infty} y^{-r}{e^{\lambda y}} U(y) = (a_i)  \quad \text{ and }\quad \lim_{y\to+\infty} y^{-\widetilde r}{e^{\widetilde \lambda y}} \widetilde U(y) = (\widetilde a_i) 
\end{equation}
where the form $y^r e^{-\lambda y}$ for each component follows from the hyperbolicity assumption (H6) and the standard asymptotic theory for solutions approaching a hyperbolic equilibrium of an autonomous ODE. Moreover, the fact that the same $\lambda$ and $P(y)$ can be chosen for all components is due to Harnack's inequality \cite[Theorem 3.9]{Foldes2009cooperative}. Moreover, the right hand side limits satisfy 
\begin{equation}
(a_i),(\widetilde a_i) \gg_{\mathbb{R}^n} 0. \end{equation}

\medskip

\noindent {\bf \# Step 0.}
By the asymptotic expansions above, after interchanging $W$ and
$\widetilde W$ if necessary, either
\begin{equation}
    \frac{\widetilde U_i(y)}{U_i(y)}\to+\infty
    \qquad\text{for every }i,
\end{equation}
or $\widetilde\lambda=\lambda$ and $\widetilde P$ and $P$ have the
same degree, in which case
\begin{equation}
    \frac{\widetilde U_i(y)}{U_i(y)}
    \longrightarrow \rho_i\in(0,\infty)
    \qquad\text{for every }i.
\end{equation}
In the latter case, translating $\widetilde W$ if necessary, we may
normalise
\begin{equation}
    \min_{1\le i\le n}\rho_i=1.
\end{equation}
Thus, without loss of generality, one of the following two cases holds:
\begin{itemize}
    \item[{\rm(i)}]  $\displaystyle \lim_{y \to +\infty} \frac{\widetilde U_k(y)}{U_k(y)} =+\infty$ holds for every $k$,
    \item[{\rm(ii)}]  $\displaystyle \lim_{y \to +\infty} \frac{\widetilde U_k(y)}{U_k(y)} \geq 1$ holds for every $k$ and equality holds for at least one $k$.   
\end{itemize}
Now, recall that $\widetilde{W} = (\widetilde{U}, \widetilde V)$, $W = (U,V)$ and define 
\begin{equation}
\widetilde{p}^\xi = \widetilde{U}(y-c\xi) - U(y) \quad \text{ and }\quad \widetilde{q}^\xi = \widetilde{V}(y-c\xi) - V(y), 
\end{equation}
and define the sliding set 
\begin{equation}
\widetilde\Xi:= \left\{\eta \in \mathbb R:~ \widetilde p^\xi \gg_{\mathbb{R}^{n}} 0,~  \widetilde q^\xi \gg_{\mathbb{R}^{m}} 0 \quad \text{ in }\mathbb{R} \quad \text{ for every }\xi\in [ \eta,\infty)\right\}.
\end{equation}
Note that in case (ii), $\widetilde\Xi \subset [0,\infty)$.
\medskip

\noindent {\bf \# Step 1. } We claim that $\xi_*:= \inf\widetilde\Xi$ is well defined. 

It suffices to show that $\widetilde\Xi$  
is nonempty and bounded from below. Indeed, at any fixed $y_0$, 
\begin{equation}
\lim_{\xi \to -\infty}\widetilde W(y_0 -c\xi) = E^+ \ll_{\mathbb{R}^{n+m}} W(y_0), 
\end{equation}
it follows that
$\widetilde\Xi$ is automatically bounded from below.  To see it is nonempty, take $\xi_0 \gg 1$ so that for any $\xi \in [\xi_0,\infty)$,
\begin{equation}
\widetilde{p}^\xi \gg_{\mathbb{R}^n} 0 \quad \text{ in }[-A_2,\infty) \quad \text{ and }\quad \widetilde{q}^\xi \gg_{\mathbb{R}^m} 0 \quad \text{ in }[-A_2,A_1].
\end{equation}
Then it follows from Lemmas \ref{l.liouville.2'} and \ref{l.liouville.3'} that $[\xi_0,\infty) \subseteq \widetilde\Xi$. 

\medskip

\noindent {\bf \# Step 2. } It follows from continuity that 
\begin{equation}
\widetilde p^{\xi_*} \geq_{\mathbb{R}^{n}} 0,~  \widetilde q^{\xi_*} \geq_{\mathbb{R}^{m}} 0 \quad \text{ in }\mathbb{R}.
\end{equation}
If equality holds for some point in some component $k$, then (thanks to the irreducibility assumption in (H2)) the strong maximum principle implies that $\widetilde W(\cdot - c \xi_*)\equiv W(\cdot)$ and we are done (and as a consequence we can rule out (i)). Hence, it remains to consider the case when 
there is no contact if $\xi = \xi_*$. Since we already know $(\xi_*,\infty) \subset \widetilde \Xi$, 
it remains to consider the case when $\xi_* \in \widetilde\Xi$.

\medskip

\noindent {\bf \# Step 3. } 
Continuing under the additional assumption $\xi_* \in \widetilde\Xi$, we claim that case (i) is impossible and in case (ii), we must have $\xi_* = 0$. Indeed, if we are either in case (i) or in case (ii) with additionally $\xi_*>0$, we can use the right-side tail estimate and the fact that $\xi_* \in \widetilde \Xi$ to find $\delta>0$ such that for any $\xi \in [\xi_*-\delta, \xi_*]$: 
\begin{equation}
\widetilde p^{\xi} \gg_{\mathbb{R}^{n}} 0\quad \text{ in }\mathbb{R}\times [-A_2, \infty), \quad ~  \widetilde q^{\xi} \gg_{\mathbb{R}^{m}} 0 \text{ in }\mathbb{R}\times [-A_2, A_1].
\end{equation}
Then
we can apply Lemmas \ref{l.liouville.2'} and \ref{l.liouville.3'} to propagate the strict inequality to  $y \in \mathbb{R}$, implying that $\xi_* -\delta \in \widetilde\Xi$ for some $\delta>0$. (Note that the specialisation to the time-independent versions of these lemmas no longer requires $\xi>\frac{h}{c}$ in case (i).) This contradicts the minimality of $\xi_*$. 

\medskip

\noindent {\bf \# Step 4. } By \# Step 3, case (ii) must hold and $\xi_* = 0$. Hence, $\widetilde\lambda = \lambda$ and  $\widetilde{W} \geq_{\mathbb{R}^{n+m}} W$ in $\mathbb{R}^2$. Moreover, 
$$
\widetilde{M}(a_i) = 0\quad \text{  and } \quad \widetilde{M}(\widetilde a_i)=0$$
where
\begin{equation}
\widetilde{M}(\lambda):=\lambda^2 \widetilde{D}_u -c\lambda I + D_u F(E^+)
\end{equation}
which is irreducible and Metzler since its off diagonal entries are equal to $D_u F(E^+)$ and the latter is irreducible and Metzler. It follows from the uniqueness of the positive eigenvector of irreducible Metzler matrices and the normalisation that $(a_i) \equiv s(\widetilde{a}_i)$ for some $s>0$ \cite[Chapter 8]{MR2978290}. Now, $s=1$ since we are in case (ii).

Now that case (ii) holds with $(a_i) \equiv (\widetilde{a}_i)$, we can repeat the whole argument while interchanging the roles of $\widetilde W$ and $W$ to conclude that $W \geq_{\mathbb{R}^{n+m}} \widetilde W$. This proves that $\widetilde{W} \equiv W$ up to a translation that we performed in dividing into cases (i) and (ii) back in \# Step 0.
\end{proof}

\bibliographystyle{siam}
\bibliography{references}

\begin{thebibliography}{10}

\bibitem{alfaro2024bramson}
{\sc M.~Alfaro, T.~Giletti, and D.~Xiao}, {\em The {Bramson} correction in the {Fisher}-{KPP} equation: from delay to advance}, Mathematische Annalen, 392 (2025), pp.~5275--5316.

\bibitem{aronson2006nonlinear}
{\sc D.~G. Aronson and H.~F. Weinberger}, {\em Nonlinear diffusion in population genetics, combustion, and nerve pulse propagation}, in Partial Differential Equations and Related Topics, J.~A. Goldstein, ed., vol.~446 of Lecture Notes in Mathematics, Springer, Berlin, Heidelberg, 1975, pp.~5--49.

\bibitem{berestycki2009can}
{\sc H.~Berestycki, O.~Diekmann, C.~J. Nagelkerke, and P.~A. Zegeling}, {\em Can a species keep pace with a shifting climate?}, Bulletin of Mathematical Biology, 71 (2009), pp.~399--429.

\bibitem{berestycki2018forced}
{\sc H.~Berestycki and J.~Fang}, {\em Forced waves of the {Fisher}--{KPP} equation in a shifting environment}, Journal of Differential Equations, 264 (2018), pp.~2157--2183.

\bibitem{beresycki2007generalized}
{\sc H.~Berestycki and F.~Hamel}, {\em Generalized travelling waves for reaction-diffusion equations}, in Perspectives in Nonlinear Partial Differential Equations: In Honor of {Ha{\"i}m Brezis}, H.~Berestycki, M.~Bertsch, F.~E. Browder, L.~Nirenberg, L.~A. Peletier, and L.~V{\'e}ron, eds., vol.~446 of Contemporary Mathematics, American Mathematical Society, Providence, RI, 2007, pp.~101--123.

\bibitem{bramson1983convergence}
{\sc M.~Bramson}, {\em Convergence of solutions of the {Kolmogorov} equation to travelling waves}, Memoirs of the American Mathematical Society, 44 (1983).

\bibitem{carrere2018spreading}
{\sc C.~Carr{\`e}re}, {\em Spreading speeds for a two-species competition-diffusion system}, Journal of Differential Equations, 264 (2018), pp.~2133--2156.

\bibitem{conley1984application}
{\sc C.~Conley and R.~Gardner}, {\em An application of the generalized {Morse} index to travelling wave solutions of a competitive reaction-diffusion model}, Indiana University Mathematics Journal, 33 (1984), pp.~319--343.

\bibitem{davis1981quaternary}
{\sc M.~B. Davis}, {\em Quaternary history and the stability of forest communities}, in Forest Succession: Concepts and Application, D.~C. West, H.~H. Shugart, and D.~B. Botkin, eds., Springer Advanced Texts in Life Sciences, Springer, New York, 1981, pp.~132--153.

\bibitem{Ducrot2015}
{\sc A.~Ducrot}, {\em On the large time behaviour of the multi-dimensional {Fisher}-{KPP} equation with compactly supported initial data}, Nonlinearity, 28 (2015), pp.~1043--1076.

\bibitem{fisher1937wave}
{\sc R.~A. Fisher}, {\em The wave of advance of advantageous genes}, Annals of Eugenics, 7 (1937), pp.~355--369.

\bibitem{Foldes2009cooperative}
{\sc J.~F\"{o}ldes and P.~Pol\'{a}\v{c}ik}, {\em On cooperative parabolic systems: {Harnack} inequalities and asymptotic symmetry}, Discrete and Continuous Dynamical Systems, 25 (2009), pp.~133--157.

\bibitem{gardner1982existence}
{\sc R.~A. Gardner}, {\em Existence and stability of travelling wave solutions of competition models: a degree theoretic approach}, Journal of Differential Equations, 44 (1982), pp.~343--364.

\bibitem{Gidas1981symmetry}
{\sc B.~Gidas, W.~M. Ni, and L.~Nirenberg}, {\em Symmetry of positive solutions of nonlinear elliptic equations in {$\mathbb{R}^n$}}, in Mathematical analysis and applications, {P}art {A}, L.~Nachbin, ed., vol.~7 of Advances in Mathematics, Supplementary Studies, Academic Press, New York, London, 1981, pp.~369--402.

\bibitem{girardin2024spreading}
{\sc T.~Giletti, L.~Girardin, and H.~Matano}, {\em Spreading properties of the {Fisher}--{KPP} equation when the intrinsic growth rate is maximal in a moving patch of bounded size}.
\newblock Accepted for publication in Annales de l'Institut Fourier; arXiv:2407.21549v3, 2025.

\bibitem{girardin2019invasion}
{\sc L.~Girardin and K.-Y. Lam}, {\em Invasion of open space by two competitors: spreading properties of monostable two-species competition-diffusion systems}, Proceedings of the London Mathematical Society, 119 (2019), pp.~1279--1335.

\bibitem{hamel2013short}
{\sc F.~Hamel, J.~Nolen, J.-M. Roquejoffre, and L.~Ryzhik}, {\em A short proof of the logarithmic {Bramson} correction in {Fisher}-{KPP} equations}, Networks and Heterogeneous Media, 8 (2013), pp.~275--289.

\bibitem{hamel2016logarithmic}
\leavevmode\vrule height 2pt depth -1.6pt width 23pt, {\em The logarithmic delay of {KPP} fronts in a periodic medium}, Journal of the European Mathematical Society, 18 (2016), pp.~465--505.

\bibitem{holzer2014accelerated}
{\sc M.~Holzer and A.~Scheel}, {\em Accelerated fronts in a two-stage invasion process}, SIAM Journal on Mathematical Analysis, 46 (2014), pp.~397--427.

\bibitem{MR2978290}
{\sc R.~A. Horn and C.~R. Johnson}, {\em Matrix analysis}, Cambridge University Press, Cambridge, second~ed., 2013.

\bibitem{hosono1998minimal}
{\sc Y.~Hosono}, {\em The minimal speed of traveling fronts for a diffusive {Lotka}-{Volterra} competition model}, Bulletin of Mathematical Biology, 60 (1998), pp.~435--448.

\bibitem{huang2007}
{\sc W.~Huang and M.~Puckett}, {\em A note on uniqueness of monotone mono-stable waves for reaction-diffusion equations}, International Journal of Qualitative Theory of Differential Equations and Applications, 1 (2007), pp.~157--170.

\bibitem{kan1995parameter}
{\sc Y.~Kan-On}, {\em Parameter dependence of propagation speed of travelling waves for competition-diffusion equations}, SIAM Journal on Mathematical Analysis, 26 (1995), pp.~340--363.

\bibitem{kan1997fisher}
\leavevmode\vrule height 2pt depth -1.6pt width 23pt, {\em {Fisher} wave fronts for the {Lotka}-{Volterra} competition model with diffusion}, Nonlinear Analysis: Theory, Methods \& Applications, 28 (1997), pp.~145--164.

\bibitem{kolmogorov1937study}
{\sc A.~N. Kolmogorov, I.~G. Petrovski, and N.~S. Piskunov}, {\em {\'E}tude de l'{\'e}quation de la diffusion avec croissance de la quantit{\'e} de mati{\`e}re et son application {\`a} un probl{\`e}me biologique}, Bull. Univ. Moscou, S{\'e}r. Internat., Sect. A, 1 (1937), pp.~1--25.

\bibitem{lam2022introduction}
{\sc K.-Y. Lam and Y.~Lou}, {\em Introduction to reaction-diffusion equations: Theory and applications to spatial ecology and evolutionary biology}, Lecture Notes on Mathematical Modelling in the Life Sciences, Springer, Cham, 2022.

\bibitem{lam2024asymptoticspreadingkppreactive}
{\sc K.-Y. Lam, G.~Nadin, and X.~Yu}, {\em Asymptotic spreading of {KPP} reactive fronts in heterogeneous shifting environments {II}: Flux-limited solutions}, Mathematische Annalen, 392 (2025), pp.~5429--5481.

\bibitem{LamWu_preprint}
{\sc K.-Y. Lam and C.-H. Wu}, {\em On the logarithmic correction of transition fronts in shifting environments}.
\newblock Accepted for publication in Indiana University Mathematics Journal; arXiv:2509.11521v2, 2026.

\bibitem{lam2026survey}
{\sc K.-Y. Lam, C.-H. Wu, and X.~Yu}, {\em Invasion fronts in shifting habitats and competition systems: A {Hamilton}-{Jacobi} approach and nonlocal effects}, Discrete and Continuous Dynamical Systems, 58 (2027), pp.~511--554.

\bibitem{Lam2022asymptotic}
{\sc K.-Y. Lam and X.~Yu}, {\em Asymptotic spreading of {KPP} reactive fronts in heterogeneous shifting environments}, Journal de Math{\'e}matiques Pures et Appliqu{\'e}es, 167 (2022), pp.~1--47.

\bibitem{lau1985nonlinear}
{\sc K.-S. Lau}, {\em On the nonlinear diffusion equation of {Kolmogorov}, {Petrovsky}, and {Piscounov}}, Journal of Differential Equations, 59 (1985), pp.~44--70.

\bibitem{Leung2011travelling}
{\sc A.~W. Leung, X.~Hou, and W.~Feng}, {\em Traveling wave solutions for {Lotka}-{Volterra} system re-visited}, Discrete and Continuous Dynamical Systems - B, 15 (2011), pp.~171--196.

\bibitem{lewis2002spreading}
{\sc M.~A. Lewis, B.~Li, and H.~F. Weinberger}, {\em Spreading speed and linear determinacy for two-species competition models}, Journal of Mathematical Biology, 45 (2002), pp.~219--233.

\bibitem{li2005spreading}
{\sc B.~Li, H.~F. Weinberger, and M.~A. Lewis}, {\em Spreading speeds as slowest wave speeds for cooperative systems}, Mathematical Biosciences, 196 (2005), pp.~82--98.

\bibitem{lin2012asymptotic}
{\sc G.~Lin and W.-T. Li}, {\em Asymptotic spreading of competition diffusion systems: the role of interspecific competitions}, European Journal of Applied Mathematics, 23 (2012), pp.~669--689.

\bibitem{liu2019asymptotic}
{\sc Q.~Liu, S.~Liu, and K.-Y. Lam}, {\em Asymptotic spreading of interacting species with multiple fronts {I}: A geometric optics approach}, Discrete and Continuous Dynamical Systems, 40 (2020), pp.~3683--3714.

\bibitem{liu2021asymptotic}
{\sc S.~Liu, Q.~Liu, and K.-Y. Lam}, {\em Asymptotic spreading of interacting species with multiple fronts {II}: Exponentially decaying initial data}, Journal of Differential Equations, 303 (2021), pp.~407--455.

\bibitem{matano2019asymptotic}
{\sc H.~Matano, Y.~Mori, and M.~Nara}, {\em Asymptotic behavior of spreading fronts in the anisotropic {Allen}-{Cahn} equation on {$\mathbb{R}^n$}}, Annales de l'Institut Henri Poincar{\'e} C, Analyse Non Lin{\'e}aire, 36 (2019), pp.~585--626.

\bibitem{nolen2017convergence}
{\sc J.~Nolen, J.-M. Roquejoffre, and L.~Ryzhik}, {\em Convergence to a single wave in the {Fisher}-{KPP} equation}, Chinese Annals of Mathematics, Series B, 38 (2017), pp.~629--646.

\bibitem{peng2021sharp}
{\sc R.~Peng, C.-H. Wu, and M.~Zhou}, {\em Sharp estimates for the spreading speeds of the {Lotka}-{Volterra} diffusion system with strong competition}, Annales de l'Institut Henri Poincar{\'e} C, Analyse Non Lin{\'e}aire, 38 (2021), pp.~507--547.

\bibitem{potapov2004climate}
{\sc A.~B. Potapov and M.~A. Lewis}, {\em Climate and competition: the effect of moving range boundaries on habitat invasibility}, Bulletin of Mathematical Biology, 66 (2004), pp.~975--1008.

\bibitem{Rothe}
{\sc F.~Rothe}, {\em Convergence to travelling fronts in semilinear parabolic equations}, Proceedings of the Royal Society of Edinburgh Section A: Mathematics, 80 (1978), pp.~213--234.

\bibitem{shigesada1997biological}
{\sc N.~Shigesada and K.~Kawasaki}, {\em Biological invasions: theory and practice}, Oxford University Press, Oxford, 1997.

\bibitem{stavrianidi2026longtimebehaviorsystemsfisherkpp}
{\sc A.~Stavrianidi}, {\em Long-time behavior for systems of {F}isher-{KPP} type with interacting components}, 2026.

\bibitem{tang1980propagating}
{\sc M.~M. Tang and P.~C. Fife}, {\em Propagating fronts for competing species equations with diffusion}, Archive for Rational Mechanics and Analysis, 73 (1980), pp.~69--77.

\bibitem{MR3243083}
{\sc G.~Teschl}, {\em Mathematical methods in quantum mechanics: With applications to {Schr{\"o}dinger} operators}, vol.~157 of Graduate Studies in Mathematics, American Mathematical Society, Providence, RI, second~ed., 2014.

\bibitem{uchiyama1978behavior}
{\sc K.~Uchiyama}, {\em The behavior of solutions of some non-linear diffusion equations for large time}, Journal of Mathematics of Kyoto University, 18 (1978), pp.~453--508.

\bibitem{Volpert1994travelling}
{\sc A.~I. Volpert, V.~A. Volpert, and V.~A. Volpert}, {\em Traveling wave solutions of parabolic systems}, vol.~140 of Translations of Mathematical Monographs, American Mathematical Society, Providence, RI, 1994.
\newblock Translated from the Russian manuscript by James F. Heyda.

\bibitem{Wang2019uniqueness}
{\sc Y.~Wang, L.~Shi, G.~Liu, and Z.~Ma}, {\em The uniqueness and stability of travelling fronts of competitive systems with different diffusive coefficients}, IMA Journal of Applied Mathematics, 84 (2019), pp.~171--196.

\bibitem{wu2023sharp}
{\sc C.-H. Wu, D.~Xiao, and M.~Zhou}, {\em Sharp estimates for the spreading speeds of the {Lotka}-{Volterra} competition-diffusion system: The strong-weak type with pushed front}, Journal de Math{\'e}matiques Pures et Appliqu{\'e}es, 172 (2023), pp.~236--264.

\bibitem{wxz2024}
\leavevmode\vrule height 2pt depth -1.6pt width 23pt, {\em Complete classification of traveling waves and resolution of linear conjecture in monostable systems}.
\newblock arXiv:2409.12463v2, 2025.
\newblock Version 2, 26 October 2025.

\bibitem{zhang2025sharpasymptoticskppequation}
{\sc M.~Zhang}, {\em Sharp asymptotics for the {KPP} equation with some front-like initial data}, arXiv,  (2025).

\bibitem{Zhao2011existence}
{\sc G.~Zhao and S.~Ruan}, {\em Existence, uniqueness and asymptotic stability of time periodic traveling waves for a periodic {Lotka}-{Volterra} competition system with diffusion}, Journal de Math{\'e}matiques Pures et Appliqu{\'e}es, 95 (2011), pp.~627--671.

\end{thebibliography}

\end{document}